\documentclass[a4paper,reqno]{amsart}

\usepackage{amsmath}
\usepackage{amsthm}
\usepackage{amssymb}
\usepackage{amscd} 

\usepackage{bbm} 
\usepackage{mathtools,mathrsfs} 

\usepackage{verbatim} 

\usepackage{enumitem}

\usepackage{xy}
\xyoption{matrix} \xyoption{arrow} \xyoption{arc} \xyoption{color}

\newtheoremstyle{newremark}
  {5pt}
  {5pt}
  {\rmfamily}
  {}
  {\rmfamily\bf}
  {.}
  {.5em}
  {}

\newtheorem{theorem}{Theorem}
\newtheorem{lemma}[theorem]{Lemma}
\newtheorem{corollary}[theorem]{Corollary}
\newtheorem{proposition}[theorem]{Proposition}

\theoremstyle{newremark}
\newtheorem{remark}[theorem]{Remark}
\newtheorem{definition}[theorem]{Definition}

\newtheorem*{definition*}{Definition} 
\newtheorem*{notations*}{Notations}

\numberwithin{theorem}{section}
\numberwithin{equation}{section}

\newcommand{\N}{\mathbb{N}} 
\newcommand{\R}{\mathbb{R}} 

\newcommand{\C}{\mathbb{C}} 
\newcommand{\D}{\mathbb{D}} 

\newcommand{\bbC}{\mathbb{C}}
\newcommand{\bbD}{\mathbb{D}}

\newcommand{\bbS}{\mathbb{S}}

\DeclareMathOperator{\rmdiv}{\mathrm{div}} 

\DeclareMathOperator{\rmLip}{\mathrm{Lip}} 

\DeclareMathOperator{\rmsing}{\mathrm{sing}} 
\DeclareMathOperator{\SO}{SO} 
\DeclareMathOperator{\On}{O}

\def\XXint#1#2#3{{%
\setbox0=\hbox{$#1{#2#3}{\int}$}
\vcenter{\hbox{$#2#3$}}\kern-.5\wd0}}

\newcommand{\veps}{\varepsilon}

\renewcommand{\leq}{\leqslant}
\renewcommand{\geq}{\geqslant}
\renewcommand{\subset}{\subseteq}
\renewcommand{\supset}{\supseteq}

\newcommand{\bb}[1]{{\mathbb #1}}

\newcommand{\eps}{\varepsilon}

\newcommand{\abs}[1]{\left| #1 \right|} 
\newcommand{\pd}[2]{\frac{\partial #1}{\partial #2}} 

\newcommand{\trans}{\mathsf{t}}
\newcommand{\trac}{{\rm tr}}

\newcommand{\eo}{{\bf e}_0}  
\newcommand{\euu}{{\bf e}^{(1)}_1} 
\newcommand{\eud}{{\bf e}^{(1)}_2} 
\newcommand{\edu}{{\bf e}^{(2)}_1} 
\newcommand{\edd}{{\bf e}^{(2)}_2} 

\begin{document}


\title[Torus-like solutions for the LDG model]{Torus-like  solutions for the Landau-de Gennes model. \\ Part III: torus vs split minimizers}

\author{Federico Dipasquale}
\address{Dipartimento di Informatica, Universit\`{a} di Verona, Strada Le Grazie 15, 37134 Verona, Italy}
\email{federicoluigi.dipasquale@univr.it}

\author{Vincent Millot}
\address{LAMA, Univ Paris Est Creteil, Univ Gustave Eiffel, UPEM, CNRS, F-94010, Cr\'{e}teil, France}
\email{vincent.millot@u-pec.fr}

\author{Adriano Pisante}
\address{Dipartimento di Matematica, Sapienza Universit\`{a} di Roma, P.le Aldo Moro 5, 00185 Roma, Italy}
\email{pisante@mat.uniroma1.it}

\date{\today}


\begin{abstract}
We study the behaviour of global minimizers of a continuum Landau-de Gennes energy functional for nematic liquid crystals, in three-dimensional axially symmetric domains domains diffeomorphic to a ball (a nematic droplet) and in a restricted class of $\bbS^1$-equivariant configurations. It is known from our previous paper \cite{DMP2} that, assuming smooth and uniaxial (e.g. homeotropic) boundary conditions and a physically relevant norm constraint in the interior (Lyuksyutov constraint), minimizing configurations are either of \emph{torus} or of \emph{split} type. Here, starting from a nematic droplet with the homeotropic boundary condition, we show how singular (split) solutions or smooth (torus) solutions (or even both) for the Euler-Lagrange equations do appear as energy minimizers by suitably deforming either the domain or the boundary data. As a consequence, 
we derive symmetry breaking results for the minimization among all competitors.
\end{abstract}



\maketitle


\tableofcontents


\section{Introduction}\label{sec:intro}

The present  article is the third of a series in  a project on the analysis of 
the Landau-de Gennes model for nematic liquid crystal. Relying  on our previous results \cite{DMP1,DMP2} (see \cite{Pi} for a short overview),   we pursue our investigations on the qualitative properties of  minimizers  of the Landau-de Gennes functional restricted to a class of axially symmetric configurations with pointwise unit norm (the Lyuksyutov constraint). We refer to  \cite{DMP1,DMP2} and the references therein for an extensive discussion on this model and its physical background. For the sake of concision, we shall simply recall the main elements and basic features of the model.

As customary in LdG $Q$-tensor theory (see, e.g., \cite{MoNe,Vi}),   we consider $\mathscr{M}_{3\times3}(\R)$ the vector space made of $3\times3$-matrices with real entries and its $5$-dimensional subspace of admissible matrices
\[
	\mathcal{S}_0:=\Big\{Q=(Q_{ij})\in\mathscr{M}_{3\times 3}(\R) : Q=Q^\trans \, , \; \trac(Q)=0 \Big\}\,.
\]
Here $Q^\trans$ denotes the transpose of $Q$, $\trac(Q)$ the trace of $Q$. The space $\mathcal{S}_0$ is endowed with the usual (Frobenius) inner product. 
 As in \cite{DMP1,DMP2}, the  indicator function of physical interest is provided by the {\em signed biaxiality} parameter
 \begin{equation}
\label{signedbiaxiality}
\widetilde{\beta}(Q):=\sqrt{6} \,\frac{{\rm{tr}}( Q^3)}{| Q|^{3}}  \in [-1,1] \, , \qquad Q \neq 0 \, .
\end{equation}
For a matrix $Q$ satisfying $|Q|=1$, the extremal values $\widetilde{\beta}(Q)=\pm 1$ occur iff the minimal/maximal eigenvalue of $Q$ is double which corresponds to the purely positive/negative uniaxial phase in the language of liquid crystals. In turn, the case $-1< \widetilde{\beta}(Q)<1$ corresponds to the biaxial phase, and it is maximal for $\widetilde{\beta}(Q)=0$ (i.e., maximal gap between the distinct eigenvalues).

Following \cite{DMP1,DMP2}, (rescaled) liquid crystal configurations occupying a given  bounded domain $\Omega \subset \mathbb{R}^3$ (with $C^{1}$-smooth boundary at least) 
 are described through Sobolev maps $Q \in W^{1,2}(\Omega; \bb S^4)$. The choice of the target $\bb S^4 \subset \mathcal{S}_0$,  the unit sphere of $\mathcal{S}_0$, encodes the Lyuksyutov constraint typical of soft biaxial nematics~\cite{LoTr} (see also~\cite{MuNi}).
As first suggested in \cite{DMP1}, the  qualitative properties of a smooth (or merely Sobolev) configuration  $Q:\Omega \to \bb S^4$  can be described by means of the 
 signed biaxiality function $\widetilde \beta\circ Q$, through the \emph{biaxiality regions}, i.e.,
\begin{equation}\label{biaxialityregions}
	\big\{\beta \leq t\big\} := \big\{ x \in \overline{\Omega} : \widetilde{\beta} \circ Q(x) \leq t \big\}\,,\quad \big\{\beta \geq t\big\} := \big\{ x \in \overline{\Omega} : \widetilde{\beta} \circ Q(x) \geq t\big\}\,,\quad t \in [-1,1]\,, 
\end{equation}
and the corresponding \emph{biaxial surfaces} $\big\{ \beta = t \big\} :=\big \{ x \in \overline{\Omega} : \widetilde{\beta} \circ Q(x) = t \big\}$. Among these sets, a crucial role 
is played by $\big\{ \beta=-1\big\}$, which should correspond to the experimentally observed \emph{disclination lines}, where eigenvalues exchange occurs (see, e.g., \cite{KV,KVZ}).
   
 After rescaling and under the Lyuksyutov constraint, the reduced LdG energy functional obtained in \cite{DMP1}  takes the form 
\begin{equation}
\label{LDGenergytilde}
\mathcal{E}_\lambda(Q):=\int_\Omega \frac{1}{2} |\nabla Q|^2+\lambda W(Q)\, dx\,,
\end{equation}
for a material-dependent constant  $\lambda>0$. It reduces to the Dirichlet integral $\mathcal{E}_0$ for maps into $\bb S^4$ when $\lambda=0$. The parameter $\lambda^{-1/2}$ is known as  the \emph{biaxial coherence length}.  The  functional $\mathcal{E}_\lambda$ formally corresponds to   a LdG energy  with quartic potential in the 1-constant approximation for the elastic energy and in the regime of zero \emph{uniaxial correlation length} reflecting the norm constraint (see the discussion in \cite[Section 1]{DMP1}).
  The reduced potential $W:\mathcal{S}_0\to\R$, when restricted to unit norm matrices, is given by 
\begin{equation}
\label{redpotential}
W(Q)=\frac{1}{3\sqrt{6}}\Big(1-\widetilde{\beta}(Q)\Big) \qquad \forall Q\in \mathbb{S}^4\,.
\end{equation}
Hence $W$ is nonnegative on $\mathbb{S}^4$. Its set of minima is called the vacuum manifold $\mathcal{Q}_{\rm min}:=\{W=0\}\cap\mathbb{S}^4$, and $\nabla_{\rm tan} W (Q)=0$ for any $Q \in \mathcal{Q}_{\rm min}$.  The minimum of $W$ is achieved when the signed biaxiality is maximal, so that
 $W(Q)=0$ iff $Q \in \mathcal{Q}_{\rm min}=\bb RP^2 \subset \bb S^4$, where we regard the projective plane $\bb RP^2 \subset \bb S^4$  embedded as the set of positive uniaxial matrices 
\begin{equation}
\label{vacuum}
\mathcal{Q}_{\rm min}= \left\{ Q \in \bb S^4:  Q=\sqrt{\frac32} \left( n \otimes n -\frac13 {\rm Id} \right) \, , \quad n \in \mathbb{S}^2  \right\} \, .
\end{equation}
Since $\mathcal{Q}_{\rm min}=\mathbb{R}P^2$,  it has nontrivial topology{{,}} and both homotopy groups $\pi_2(\mathcal{Q}_{\rm min})=\mathbb{Z}$ and $\pi_1(\mathcal{Q}_{\rm min})=\mathbb{Z}_2$  play a role in the presence of defects, especially in the restricted class of axisymmetric configurations. 
A critical point $Q_\lambda\in W^{1,2}(\Omega;\mathbb{S}^4)$ of $\mathcal{E}_\lambda$ among $\mathbb{S}^4$-valued maps satisfies in the weak sense the Euler-Lagrange equations
\begin{equation}\label{MasterEq}
\Delta Q_\lambda+|\nabla Q_\lambda|^2Q_\lambda = \lambda \nabla_{\rm tan} W ( Q_\lambda) \, ,
\end{equation}
where the tangential gradient of $W$ at  $Q\in\mathbb{S}^4 \subset \mathcal{S}_0$ is given by
\[  
	\nabla_{\rm tan} W (Q)=- \Big(Q^2-\frac{1}{3}{\rm Id} - {\rm tr}(Q^3)Q\Big)  \,.
\]
 The left hand side in \eqref{MasterEq} is the tension field of the $\mathbb{S}^4$-valued map $Q_\lambda$ as in the theory of harmonic maps, see e.g. \cite{LiWa}.

Symmetry ans\"atze  have been considered in several recent articles dedicated to Landau-de Gennes models in dimension two or three, see e.g. \cite{INSZ2,INSZ3,ABL,INSZ,Yu,ABGL, TaYu, ACS}. In the present paper, we consider the LdG functional  $\mathcal{E}_\lambda$ restricted to a class of $\mathbb{S}^1$-equivariant configurations,  continuing the analysis initiated in \cite{DMP2}.  As reviewed in Section \ref{sec:preliminaries}, we identify the group $\mathbb{S}^1$ with the subgroup of $\SO(3)$   made of rotations around the vertical axis of $\mathbb{R}^3$, and we consider the induced action on $\mathcal{S}_0$ given by $\mathcal{S}_0 \ni A \mapsto R A R^\trans \in \mathcal{S}_0$. Assuming that the open set $\Omega \subset \mathbb{R}^3$ is bounded, smooth, and $\mathbb{S}^1$-invariant, i.e., $R \cdot \Omega=\Omega$ for any $R \in \mathbb{S}^1$, we restrict ourselves to maps $Q \colon \Omega \to \mathcal{S}_0$ which are $\mathbb{S}^1$-equivariant, i.e.,
\begin{equation}
\label{S1equivariance}
Q(R x)= R Q(x) R^\trans  \quad \mbox{for a.e. } x \in \Omega \, , \quad \forall R \in \mathbb{S}^1 \, ,
\end{equation}
with the obvious analogue definition for maps defined on the boundary.  
Following our notations from \cite{DMP1,DMP2}, given an $\bbS^1$-equivariant Dirichlet boundary data $Q_{\rm b}:\partial\Omega\to\mathbb{S}^4$, we set 
$$ \mathcal{A}_{Q_{\rm b}}(\Omega):=\Big\{Q\in W^{1,2}(\Omega;\mathbb{S}^4): Q=Q_{\rm b}\text{ on $\partial\Omega$}\Big\}\,,$$
and 
\begin{equation}
\label{S1admissibleconf}
\mathcal{A}^{{\rm sym}}_{Q_{\rm b}}(\Omega):=\Big\{ Q \in \mathcal{A}_{Q_{\rm b}}(\Omega) : Q \text{ is $\bb S^1$-equivariant }\Big\} \subsetneq \mathcal{A}_{Q_{\rm b}}(\Omega) \, .
\end{equation}
We are then interested in minimizers $Q_\lambda$ of  $\mathcal{E}_\lambda$  over the restricted class $\mathcal{A}^{{\rm sym}}_{Q_{\rm b}}(\Omega)$. 
As already discussed in \cite[Theorem~1.1]{DMP2} (and reviewed in the next sections), if $\partial\Omega$ and $Q_{\rm b}$ are smooth enough, then minimizers always exist and they are smooth up to a singular set, denoted by ${\rm sing}(Q_\lambda)$, made of (at most) finitely many interior point singularities located on the symmetry axis.  When present, these singular points are due either to a topological obstruction related to the equivariance constraint or to an energy efficiency mechanism.
\vskip3pt

The main  purpose of this article is to shed some light on the delicate interplay between the geometry of the boundary and the properties of the Dirichlet boundary  condition in determining the  qualitative properties of the corresponding  minimizers.  As initiated in \cite{DMP2}, 
we   investigate nonexistence vs existence of singularities for maps minimizing $\mathcal{E}_\lambda$ over the symmetric class $\mathcal{A}^{{\rm sym}}_{Q_{\rm b}}(\Omega)$ for  a boundary data $Q_{\rm b}$ exploiting the topology of the vacuum manifold $\mathcal{Q}_{\rm min}=\bb RP^2$. In this way, the topology of minimizers will be either of torus type or of split type respectively, according to our qualitative description of the biaxiality regions and surfaces in \cite[Section~7]{DMP2} and \cite[Theorems~1.4 \&~1.5]{DMP2}. For the sake of concision, we refer to \cite{DMP2} for the relevant definitions and the precise results. In few words, the {\sl torus type} holds whenever the minimizer $Q_\lambda$ is smooth (i.e., ${\rm sing}(Q_\lambda)=\emptyset$) and our terminology refers to the fact that a biaxial surface must have a connected component of positive genus enclosing at least a circle of negative uniaxiallity,  i.e., a (invariant) disclination ring. In turn, the {\sl split type} holds if the minimizer is singular (i.e., ${\rm sing}(Q_\lambda)\not=\emptyset$) and in this case, singularities come in pairs and biaxiality equals $-1$ in  between (i.e., there are disclination segments on the vertical axis), while biaxial surfaces contain spheres with poles at the singular points.  
\vskip3pt

For   simplicity, we restrict ourselves to axisymmetric cylinder-type domains diffeomorphic to a ball (see Definition~\ref{def:smooth-cyl}), or to the model case of a nematic droplet, i.e., the unit ball $\Omega=B_1 \subset \bb R^3$. Concerning the boundary data, a natural choice is to take it smooth (at least of class $C^1$) and valued in the vacuum manifold, i.e., $Q_{\rm b}\in C^1(\partial \Omega; \mathbb{R}P^2)$.  Since $\partial \Omega \simeq \bb S^2$, every such map can be written in the form
\begin{equation}
  \label{Qblift}
  Q_{\rm b}(x)=\sqrt{\frac32} \left( v(x) \otimes v(x) -\frac13 {\rm Id}\right) \qquad \hbox{for all }  x\in \partial \Omega \, , \quad v\in C^1(\partial \Omega;\mathbb{S}^2) \, .
\end{equation}
 Since $\Omega$ is axisymmetric, such map $Q_{\rm b}$ is $\mathbb{S}^1$-equivariant if and only if its lift $v$ is itself $\bbS^1$-equivariant (with respect to the obvious action of $\mathbb{S}^1$ on $\mathbb{S}^2\subset \mathbb{R}^3$ by rotation). Then, the topological nontriviality of $Q_{\rm b}$ introduced in \cite{DMP2} and required here amounts to the assumption that the topological degree $\deg \, v \in \bb Z$ of the lift is odd (this assumption only depends on $Q_{\rm b}$ and not on a particular choice of the lift $v$). 
For instance, if $\partial \Omega$ is of class $C^2$ and $v$ in \eqref{Qblift} is the  outer unit normal field on $\partial\Omega$ (i.e., $v(x)=\overset{\to}{n}(x)$), then we obtain the so-called {\em homeotropic boundary condition} (see \eqref{eq:radial-anchoring}) which is $\bb S^1$-equivariant and its lift $v$ satisfies $\deg \,  v=  1$, i.e., it satisfies our topological requirement.
A main example entering in our discussion below is the case  a nematic droplet  $\Omega=B_1$ with homeotropic boundary data. Then $\overset{\to}{n}(x)=\frac{x}{|x|}$ and $Q_{\rm b}(x)=H(x)$ where $H$ is the {\em unit-norm hedgehog}
\begin{equation}
  \label{Hsphere}
  H(x)=\sqrt{\frac32} \left( \frac{x}{|x|}\otimes \frac{x}{|x|} -\frac13 {\rm Id}\right) \qquad \hbox{for all } x\in \partial B_1 \, ,
\end{equation}
which is  actually equivariant in the sense of \eqref{S1equivariance} with respect to the full orthogonal group $\On(3)$.

Besides $\R P^2$-valued maps, we shall also consider more general $\bbS^4$-valued boundary data.  
According to \cite{DMP1,DMP2},  we shall always assume that $\Omega$ and $Q_{\rm b}$ are  smooth enough, axisymmetric, and satisfying the conditions:
\vskip2pt
\begin{itemize}
	\item[($HP_1$)] $\bar{\beta}:=\min_{x \in \partial \Omega} \widetilde{\beta}\circ Q_{\rm b}(x)>-1$; 
	\vskip5pt
	\item[($HP_2$)] $\Omega$ is 
	diffeomorphic (equivariantly and up to the boundary) to a ball;
	\vskip5pt
	\item[($HP_3$)] $\deg(v, \partial \Omega)$  is odd;
\end{itemize}
\vskip2pt
where $(HP_3)$ has to be understood in the following way. 
 In view of ($HP_1$), the maximal eigenvalue $\lambda_{\rm max}(x)$ of $Q_{\rm b}(x)$ is simple and the function $\lambda_{\rm max}:\partial\Omega\to \R$ is smooth, hence there is a well defined and smooth eigenspace map $V_{\rm max}: \partial \Omega \to \mathbb{R}P^2$ (which inherits equivariance). Since $\partial \Omega \simeq \bb S^2$ by ($HP_2$), the mapping $V_{\rm max}$ has a (nonunique) smooth lifting $v: \partial \Omega \to \mathbb{S}^2${{,}} which is required to satisfy ($HP_3$). In the case $\Omega=B_1$, besides the radial hedgehog $H$, the main examples of boundary data satisfying our general assumptions are the $\mathbb{S}^1$-equivariant harmonic spheres $ \mathbf{\omega}_{\mu_1,\mu_2} : \bb S^2 \to \bb S^4$, for positive parameters $\mu_1$ and $\mu_2$ (see the full classification \cite[Proposition 3.8 and proof of Theorem 1.3]{DMP2}).
 \vskip3pt
 
 In \cite[Theorem~1.4 \&~1.5]{DMP2}, we have shown that under assumptions  ($HP_1$)-($HP_3$), minimizers of $\mathcal{E}_\lambda$ over the class $\mathcal{A}^{{\rm sym}}_{Q_{\rm b}}(\Omega)$ must be   either of torus type (when smooth) or of split type (when singular), in agreement with some physical expectations based on  numerical simulations (e.g., \cite{GaMk, KV, KVZ, DLR, HQZ}).  To complement this result, \cite[Theorem~1.2 \&~1.3]{DMP2} provide in the case $\Omega=B_1$ two explicit continuous deformations\footnote{$C^{2,\alpha}_{\rm sym}(\partial B_1;\mathbb{S}^4)$ stands for the subset of $C^{2,\alpha}(\partial B_1;\mathbb{S}^4)$ made of all $\mathbb{S}^1$-equivariant maps. More generally, we shall use the sub/supscript sym on a functional space to indicate that the mappings involved are $\mathbb{S}^1$-equivariant.} $\Gamma:[0,1]\to C^{2,\alpha}_{\rm sym}(\partial B_1;\mathbb{S}^4)$ of the hedgehog map $H$  along which ($HP_1$)-($HP_3$) are preserved and such that minimizers corresponding to the final map $Q_{\rm b}=\Gamma(1)$ are  either all of  torus type or all of split type respectively  (see also Remark~\ref{rmk:curve-coex}). Our first main result actually shows that both type of minimizers coexist for the same boundary data when suitably chosen at some intermediate stage of one these deformations.
 
\begin{theorem}
\label{main-coexistence}
 Let $\alpha\in(0,1)$, $\lambda\geq 0$, and $\Gamma:[0,1]\to C^{2,\alpha}_{\rm sym}(\partial B_1;\mathbb{S}^4)$  a continuous curve along which ($HP_1$)-($HP_3$) are satisfied.  Assume that for $Q_{\rm b}=\Gamma(0)$ and $Q_{\rm b}=\Gamma(1)$, the minimizers of $\mathcal{E}_\lambda$ over $\mathcal{A}^{{\rm sym}}_{Q_{\rm b}}(B_1)$ are all of torus type and all of split type, respectively. Then there exist $0<t_1\leq t_2<1$ such that 
\begin{enumerate}
\item[(i)] for every $0\leq t<t_1$ and $Q_{\rm b}=\Gamma(t)$, any minimizer of $\mathcal{E}_\lambda$ over $\mathcal{A}^{{\rm sym}}_{Q_{\rm b}}(B_1)$ is smooth and thus of torus type; 
\vskip5pt
\item[(ii)] for every $t_2<t\leq 1$ and $Q_{\rm b}=\Gamma(t)$, any minimizer of $\mathcal{E}_\lambda$ over $\mathcal{A}^{{\rm sym}}_{Q_{\rm b}}(B_1)$ is singular and thus of split type; 
\vskip5pt

\item[(iii)] for $t=t_1$ or $t=t_2$  and $Q_{\rm b}=\Gamma(t)$,  there exist a smooth and a singular minimizer of $\mathcal{E}_\lambda$ over the class $\mathcal{A}^{{\rm sym}}_{Q_{\rm b}}(B_1)$, hence of torus and split type respectively.
\end{enumerate} 

\noindent As a consequence, there exists $Q_{\rm b} \in C^{2,\alpha}_{\rm sym}(\partial B_1;\mathbb{S}^4)$ satisfying ($HP_1$)-($HP_3$) which yields coexistence of torus and split  minimizers of $\mathcal{E}_\lambda$ over $\mathcal{A}_{Q_{\rm b}}^{\rm sym}(B_1)$.
\end{theorem}

 The proof of Theorem~\ref{main-coexistence},  in Section~\ref{sectcoexball}, essentially relies on the interior and boundary regularity theory   developed in \cite{DMP1,DMP2}   and suitably presented in Section \ref{secregth}. Along with further refinements, it follows that both smoothness and presence of singularities persist under strong $W^{1,2}$-convergence as the pair $(Q_{\rm b}, \lambda)$ varies in the space of data $C^{2,\alpha}_{\rm sym}(\partial \Omega;\bb S^4)\times [0,\infty)$, in analogy with \cite{AlLi,HL2} in the case of minimizing harmonic maps into  $\bbS^2$. Using these properties together with unique continuation arguments, we prove  in Theorem~\ref{decompobdspacethm}  a decomposition of the space of data into two open sets, for which all the minimizers are of the same type (smooth or singular), and their common boundary, where coexistence  occurs. Then Theorem~\ref{main-coexistence} follows as a direct consequence (see Corollary \ref{corocoexist}) as the two open sets are not empty by \cite[Theorem~1.2 \&~1.3]{DMP2} and  there exists an explicit continuous path connecting them (as already mentioned). 
 It is a natural open question to understand if for such explicit path deforming the data used in \cite[Theorem~1.2]{DMP2} into the one used in \cite[Theorem~1.3]{DMP2}  and passing through the hedgehog $H$,  
the coexistence parameters given in Theorem~\ref{main-coexistence} are precisely those of the hedgehog, i.e., if $Q_{\rm b}=H$ yields coexistence of torus and split minimizers in the class $\mathcal{A}_{H}^{\rm sym}(B_1)$.

Our coexistence property is somehow related to a similar result established in the recent article~\cite{Yu}, although the methods employed are completely different.  As already commented in more details in \cite[Section~7]{DMP2}, the analysis in \cite{Yu} is performed to the case   $\Omega=B_1$ with boundary condition given by the unit norm hedgehog $H$,  and the minimization is restricted  to the strictly smaller class of $\On(2)\times \mathbb{Z}_2$-equivariant configurations (the extra $\bb Z_2$-symmetry corresponding to the reflection across the horizontal plane). In this restricted class, the author performs a clever further  constrained minimization which yields coexistence of minimizers of ``torus'' and ``split'' type, although these notions are in a sense weaker than ours in \cite{DMP2}. However,  their energy minimality in the full symmetric class $\mathcal{A}_{H}^{\rm sym}(B_1)$ remains unclear. 
 \vskip5pt

The second part of the article is dedicated to minimizers of $\mathcal{E}_\lambda$ over the equivariant class \eqref{S1admissibleconf} with homeotropic boundary conditions on axisymmetric domains $\Omega \subseteq \bb R^3$ diffeomorphic to the unit ball. Here the goal is to show that the presence of smooth or singular minimizers and even their coexistence depends in a subtle way on the shape of $\Omega$. To capture the essence of these phenomena, we restrict ourselves to an explicit family of axisymmetric cylinder-type domains denoted by $\mathfrak{C}^h_{\ell,\rho}$ and obtained as a regularization (near the angles) of vertical cylinders of height $2h$ and radius~$\ell$, the parameter $\rho$ being the smoothing parameter (see Definition~\ref{def:smooth-cyl}). The boundary condition $Q_{\rm b}$ is the homeotropic boundary data given by \eqref{Qblift} with $v=\overrightarrow{n}$  the outer unit normal field. 
Under these choices of $\Omega$ and $Q_{\rm b}$, assumptions ($HP_1$)-($HP_3$) above are satisfied and the results in \cite{DMP2} apply. Exploiting these facts, we discuss here the nature of minimizers, i.e., smooth or singular, and thus their type, torus or split, as the characteristic lengths $h$ and $\ell$ vary.    
Borrowing a terminology from physics (see, e.g., \cite{AZB}, for the case of Bose-Einstein condensates in trapping potentials), we are interested in two opposite regimes, namely: (i) the case $h \gg \ell$ of long and thin cylinders, (the ``cigar shape''), and its opposite, i.e., (ii) the case $h \ll \ell$ of flat and very large cylinders (the ``pancake shape''). Both cases are somehow natural, as they are a mathematical idealization of the case in which the liquid crystals occupy a long pipe or it is arranged as a thin film respectively. 

We shall see in Theorem \ref{thm:vertical-cylinders} below that, 
in the asymptotic regime $h\gg \ell$, a 2D-reduction phenomenon occurs and the 3D-minimizers in $\mathfrak{C}^h_{\ell,\rho}$ tend to minimize the 2D-energy on most of the horizontal cross-sections of the domain. To present and describe this dimension reduction, it is useful to anticipate and analyse the effective 2D-variational problem which involves maps defined on a generic horizontal cross-section $\bb D_\ell$ of the smoothed cylinder $\mathfrak{C}^h_{\ell,\rho}$. This 2D-minimization problem is of independent interest and ressembles the one considered in \cite{INSZ} without the norm constraint. 

For simplicity, we rescale the disc $\bb D_\ell$ of radius $\ell$ to the unit disc $\bb D$, and to distinguish the 2D from the 3D case, we shall use the notation $E_\lambda$ (instead of $\mathcal{E}_\lambda$) to refer to the LdG energy in two dimensions. In other words, we consider for each $\lambda\geq 0$ the 2D-LdG energy
\begin{equation}\label{eq:def-2d-energy}
	E_\lambda(Q) := \int_{\bbD} \frac12\abs{\nabla Q}^2 + \lambda W(Q) \,dx\,,
\end{equation}
defined for configurations in the class $W^{1,2}(\bbD; \bbS^4)$.  Note that in the case $\lambda=0$, the energy $E_0$ still  reduces to the Dirichlet integral. 

In the 2D-problem, we are interested in minimizers of $E_\lambda$ over the $\mathbb{S}^1$-equivariant class
\begin{equation}\label{eq:def-2d-adm}
	\mathcal{A}^{\rm sym}_{\overline{H}}(\bbD) := \big\{ Q \in W^{1,2}_{\rm sym}(\bbD;\bbS^4) : Q  = \overline{H} \text{ on $\partial\bbD$} \big\} \, ,
\end{equation}
where  $\overline H:\R^2\setminus\{0\}\to\mathbb{R}P^2 \subset \mathbb{S}^4$ is the radial anchoring map (or constant norm hedgehog), i.e., 
\begin{equation}\label{defbarH}
\overline H(x):=\sqrt{\frac{3}{2}}\left( \frac{1}{|x|^2}
\begin{pmatrix}
x_1 \\
x_2\\
0
\end{pmatrix} \otimes \begin{pmatrix}
x_1 \\
x_2\\
0
\end{pmatrix}
-\frac{1}{3}{\rm Id}\right)\,. 
\end{equation}
 The restriction of $\overline{H}$ to $\partial \bb D_\ell$ corresponds precisely to the homeotropic boundary condition at the boundary of the cross-section $\partial \bb D_\ell \subseteq \partial \mathfrak{C}^h_{\ell,\rho}$ where the outer normal is horizontal. We observe that maps belonging to $\mathcal{A}^{\rm sym}_{\overline{H}}(\bbD)$ 
are continuous in $\overline{\bb D}$ (see Section~\ref{sec:preliminaries}), hence there is a natural decomposition $\mathcal{A}_{\overline{H}}^{\rm sym}(\bbD) = \mathcal{A}_{\rm N} \cup \mathcal{A}_{\rm S}$ (with disjoint union) according to the   respective value at the origin $Q(0)=\pm \eo$, where $\eo$ is the matrix given by \eqref{eq:basis-S0}. Indeed, $\pm \eo$ are the only unit norm matrices invariant under the action of $\bbS^1$ on $\bbS^4$, so that  equivariance,  norm constraint, and continuity imply this decomposition.
\vskip3pt

Our second main result discusses  the nature of 2D-minimizers as the parameter $\lambda \geq 0$ varies, that is the belonging to the class $ \mathcal{A}_{\rm N}$ or to the class $\mathcal{A}_{\rm S}$. Note that fixing the cross-section of the sample and varying the biaxial coherence length $\lambda^{-1/2}$ is mathematically equivalent, by rescaling, to fixing the material-dependent length $\lambda^{-1/2}$ and varying the width of the sample, which is physically more realistic.

\begin{theorem}
\label{2d-biaxial-escape}
There exist $0<\lambda_0<\lambda_*<\lambda^*$ such that the following statements hold.

\begin{itemize}
\item[(i)] The maps $\overline Q\simeq \bar u$ with $\bar{u}(z)=g_{\overline{H}}(\pm z)$  explicitly given by \eqref{eq:Hbar}, 
are (positively) uniaxial, they are minimizers of $E_\lambda$ over $\mathcal{A}_{\rm N}$,  and   local minimizers of $E_\lambda$ over $\mathcal{A}^{\rm sym}_{\overline{H}}(\bbD)$ for every $\lambda \geq 0$. In addition, these maps are the unique absolute minimizers for $\lambda \in( \lambda_*,\infty)$.
\vskip5pt

\item[(ii)] If $\lambda \in [0,\lambda^*)$ then there exist minimizers $Q_\lambda$ of $E_\lambda$ over $\mathcal{A}_{\rm S}$. Moreover, 
they are local minimizers of $E_\lambda$ over $\mathcal{A}^{\rm sym}_{\overline{H}}(\bbD)$, and  satisfy $\widetilde\beta\circ Q_\lambda(\overline\bbD)=[-1,1]$. In addition, if $\lambda \in[0, \lambda_*)$, then minimizers over $\mathcal{A}_{\rm S}$ are the  the only minimizers of $E_\lambda$ over $\mathcal{A}^{\rm sym}_{\overline{H}}(\bbD)$, and uniqueness holds for $\lambda<\lambda_0$.
\vskip5pt

\item[(iii)] If $\lambda=\lambda_*$, then the maps $\overline Q$ in (i) and $Q_\lambda$ in (ii) are both minimizers of $E_\lambda$ over $\mathcal{A}^{\rm sym}_{\overline{H}}(\bbD)$.
\end{itemize}
  	
\end{theorem}

The previous theorem provides a purely energetic explanation of the biaxial escape phenomenon in 2D-biaxial nematics, at least under norm and axial symmetry constraints. The escape mechanism here is explained  in a completely different way   compared to \cite{Can1}, where complete biaxial escape in 2D is inferred in the low-temperature limit. In our case, the boundary data \ref{defbarH} is trivial in $\pi_1(\bb RP^2)$, while in  
\cite{Can1}, its nontriviality implies that almost uniaxial extensions cannot exist, even without the norm constraint. Indeed, according to claim (i), if $\lambda$ is very large (equivalently, when the size of the sample is large compared to the characteristic length $\lambda^{-1/2}$), then energy minimizers are purely positively uniaxial (and even explicit, due to the norm constraint), because of the strong penalization of the biaxial phase induced by the potential $W$. On the other hand,  claim (ii) shows that reducing $\lambda$ to smaller values (equivalently,  reducing the size the sample compared to $\lambda^{-1/2}$) makes uniaxiality non necessarily   favorable. Indeed, for $\lambda$ below the coexistence threshold $\lambda_*$, the biaxiality parameter of minimizers attains its full range $[-1,1]$, and complete \emph{biaxial escape} occurs.   

The proof of Theorem~\ref{2d-biaxial-escape} is presented in Section~\ref{2Dminimization}. As commented in more details there, the   cornerstone is Theorem~\ref{thm:gap-u} which gives  an energy gap phenomenon  between the infimum of Dirichlet integral $E_0$ over the class $\mathcal{A}_{\rm N}$ and the class $\mathcal{A}_{\rm S}$ together with a   complete classification of the corresponding optimal maps following the lines of \cite{DMP2}.  The main difficulties come from the conformal invariance of the Dirichlet integral in 2D and the associated concentration/compactness alternative with possible bubbling-off of harmonic spheres along minimizing sequences (see the proof of Proposition~\ref{ASlambdaminimization}) as the pointwise constraints $Q(0)=\pm \eo$ are not weakly closed.   This intermediate step and Theorem~\ref{2d-biaxial-escape} can be seen  as analogues of the construction of small and large solutions for $\bbS^2$-valued harmonic maps in two dimensions, see \cite{BrCo,Jo}. Borrowing the terminology from the $\bbS^2$-valued case, the {\em large solutions} $\overline{Q}$ in (i) escape from the (small) spherical cap of $\bb S^4$ centered at $-\eo$ containing the image of the boundary datum $ \overline{H}$, as opposed to {\em smalls solutions} $Q_\lambda$ in (ii)  (at least for $\lambda$ small enough) for which the escape phenomenon does not happen. 
\vskip5pt

In Section \ref{SectSplit}, we take advantage of the previous 2D result to describe the asymptotic behaviour of minimizers  in the 3D cylindrical domains $\mathfrak{C}^h_{\ell,\rho}$ with homeotropic boundary condition in the regime $h\gg \ell$. Our third main result below shows that, for such long ``cigar shaped'' domains,   any minimizing configuration must be singular, hence of split type in the sense of \cite{DMP2}.

\begin{theorem}
\label{thm:vertical-cylinders}

Let $\lambda \geq 0$ be a fixed number and $\lambda_0$, $\lambda_*$ the values provided by Theorem~\ref{2d-biaxial-escape}. 
 Given $0<2\rho<\ell$  and a sequence $h_n\to+\infty$ satisfying $h_n>\ell$, set $\Omega_n:=\mathfrak{C}^{h_n}_{\ell,\rho}$ and let $Q_{\rm b}^{(n)}$ be the homeotropic boundary data on  $\partial \Omega_n$. If, for each $n$,  $Q^{(n)}$ is a minimizer of $\mathcal{E}_\lambda$ over $\mathcal{A}^{\rm sym}_{Q_{\rm b}^{(n)}}(\Omega_n)$, then the following statements hold for $n$ large enough.
\vskip5pt

\begin{enumerate}
\item[(i)] (Split Structure) If $\ell<\sqrt{\lambda_*/\lambda}$, then $\rmsing{(Q^{(n)})} \neq \emptyset$. As a consequence, $Q^{(n)}$ is   of split type 
 and $\beta_n:=\widetilde{\beta}\circ Q^{(n)}$ satisfies $\beta_n(\overline{\Omega_n})=[-1,1]$. 
\vskip5pt

\item[(ii)]  (2D-reduction)   If $\ell<\sqrt{\lambda_0/\lambda}$ and  $\widehat{Q}_\ell$ denotes the unique minimizer of $E_\lambda(\,\cdot \,; \bb D_\ell)$ over  
$\mathcal{A}^{\rm sym}_{\overline{H}}(\bb D_\ell)$, then $Q^{(n)}\to \widehat{Q}_{\ell}$ strongly in $W^{1,2}_{\rm loc}(\mathfrak{C}^\infty_\ell)$ and in fact, locally smoothly in $\overline{\mathfrak{C}^\infty_\ell}$ as $n\to+\infty$.
\vskip5pt

\item[(iii)] (Singularities Ejection)  If $\ell<\sqrt{\lambda_0/\lambda}$, then   $\rmsing{(Q^{(n)})} \cap\{x_3\geq 0\}$ and  $\rmsing{(Q^{(n)})} \cap\{x_3< 0\}$ are both nonempty, each one of them contains an odd number of points,   $ \rmsing{(Q^{(n)})}\subset \{x_3\text{-axis}\}\cap\{h_n-\alpha \leq |x_3|\leq h_n-\frac{1}{\alpha} \}$ for some constant $\alpha\geq1$ independent of $n$, and   
$Q^{(n)} = -\eo$ on $\{x_3\mbox{-axis}\}\cap\{|x_3|< h_n-\alpha\}$.  
In addition, ${\rm Card}\big(\rmsing{(Q^{(n)})}\big)$ remains bounded as $n\to\infty$.
\end{enumerate}
\end{theorem}

This theorem shows that singularities occur purely for reasons of energy efficiency, in analogy with the case of minimizing harmonic maps into $\bbS^2$ first described in \cite{HL}. 
 Claim (ii) in the  theorem above states that minimizers tend to become two-dimensional (i.e., independent of $x_3$) on each fixed   bounded portion of the (smoothed) cylinder as the height goes to infinity.  For sufficiently thin cylinders (below the critical threshold $\sqrt{\lambda_0/\lambda}$), 2D minimizers on the cross sections assume the value $-{\bf e}_0$ at the origin by Theorem~\ref{2d-biaxial-escape}, so that negative uniaxiality must occur  on the symmetry axis for 3D minimizers. This property, in combination with the boundary data, forces the presence of point singularities,  and thus the split structure. Finally, according to (iii),  singularities have to escape to infinity along the symmetry axis in a certain quantitative way, whereas full regularity on each fixed bounded portion of the cylinders is inherited from the limiting map.   From the presence of singularities, we derive in Corollary~\ref{inst+symmbreaking} the instability 
of minimizers over  $\mathcal{A}^{\rm sym}_{Q_{\rm b}^{(n)}}(\Omega_n)$ in the full class $ \mathcal{A}_{Q_{\rm b}^{(n)}}(\Omega_n)$. As a consequence, minimizers of $\mathcal{E}_\lambda$ over $ \mathcal{A}_{Q_{\rm b}^{(n)}}(\Omega_n)$ are not symmetric and non uniqueness holds, in  
analogy with our previous result \cite[Corollary~7.15]{DMP2}.  Such symmetry breaking phenomena were already proved in \cite{AlLi} and \cite{HKL}   for minimizing harmonic maps into $\mathbb{S}^2$ (i.e., for the Frank-Oseen model).  Hence, our result is a natural counterpart for the Landau-de Gennes model, in agreement with the numerical simulations in \cite{DLR}. 

 The proof of Theorem~\ref{thm:vertical-cylinders} relies on various energy identities leading to uniform a priori bounds and compactness properties. But the heart of  the matter is a 2D-rigidity result for local minimizers in infinite cylinders, see Proposition~\ref{vertical-liouville}. Relying on the 2D-uniqueness property in Theorem~\ref{2d-biaxial-escape}, we obtain $x_3$-independence by constructing comparison maps  with optimal energy growth, and to this purpose  it is crucial to assume that the cylinders are sufficiently thin.   Our analysis also shows that the number of singularities is bounded and that, near each tip of the cylinder, there must be an odd number of them. It remains an open question whether or not there is exactly one singular point near each tip for $h$ large enough.
\vskip5pt

The next result describes the asymptotic behaviour of   minimizers over the equivariant class in the opposite regime $h \ll \ell$. It shows that for such ``pancake shaped'' domains the minimizing configurations must be smooth, hence of torus type in the sense of \cite{DMP2}.

\begin{theorem}\label{thmlargecyl}
	Let $\lambda \geq 0$ be a fixed number.  Given $0<2\rho<h$ and an increasing sequence $\ell_n\to+\infty$ satisfying $\ell_n>\sqrt{2}h$, set $\Omega_n:=\mathfrak{C}^h_{\ell_n,\rho}$ and let $Q_{\rm b}^{(n)}$ be the homeotropic boundary data on~$\partial \Omega_n$. If, for each $n$,  $Q^{(n)}$ is a minimizer of $\mathcal{E}_\lambda$ over $\mathcal{A}^{\rm sym}_{Q_{\rm b}^{(n)}}(\Omega_n)$, then the following statements hold for $n$ large enough. 
\vskip5pt

\begin{enumerate}
\item[(i)] (Torus Structure) We have $\rmsing{(Q^{(n)})} = \emptyset$. As a consequence, $Q^{(n)}$ is of torus type,  
$\beta_n:=\widetilde{\beta}\circ Q^{(n)}$ satisfies $\beta_n(\overline{\Omega_n})=[-1,1]$, and the level set $\{\beta_n = -1\}$ contains an invariant horizontal circle mutually linked to the vertical axis. 
\vskip5pt

\item[(ii)]  (Asymptotic Behaviour)  $Q^{(n)}\to \eo$ strongly in $W^{1,2}_{\rm loc}(\mathfrak{C}^h_\infty)$ and in fact, locally smoothly in $\overline{\mathfrak{C}^h_\infty}$ as $n\to+\infty$.
\vskip5pt

\item[(iii)] (Biaxiality Ejection) For any $t \in [-1,1)$, there exist $n_t \in \N$ and a value $d_t > 0$ independent of $n$ such that $\{\beta_n \leq t \} \cap \mathfrak{C}_{\ell_n-d_t}^h=\emptyset$ for any $n \geq n_t$. 
\end{enumerate}
\end{theorem}

According to claim (ii),  minimizers approach the constant map $\eo$ on each fixed bounded portion of the cylinder as the width increases to infinity. Indeed, 
the influence of the nonconstant part of the boundary data, which is present only on the curved part of $\partial \Omega_n$, fades as this curved part is sent to infinity when $\ell_n\to+\infty$. Then, full regularity near the symmetry axis (and hence everywhere) is inherited from the limiting map, whence the torus type structure. 
 Furthermore,  the local smooth convergence to a constant uniaxial map pushes the biaxial sets to infinity, in such a way that they remain at finite distance from the lateral boundary.

The proof of Theorem~\ref{thmlargecyl} also relies on monotonicity formulae, local energy bounds and compactness arguments.  The first key estimate is a linear law for the growth of the total energy with respect to $\ell$ obtained through comparison maps. Refining it into a sublinear estimate slightly in the interior (see Lemma \ref{lemma:lin-sublin})  leads to the constancy of the limiting map and to a uniform bound for the distance of the biaxial sets from the lateral boundary.

\medskip
In   our last main result, we discuss how the nature of  minimizers of $\mathcal{E}_\lambda$ over the symmetric class changes under deformations of the domain, in analogy with and complementing Theorem~\ref{main-coexistence} when varying the boundary data.   Theorem \ref{thmcoexmovdom} below refines the conclusions in Theorem~\ref{thm:vertical-cylinders} \&~\ref{thmlargecyl}, and it shows how the transition from the ``cigar shape'' to the ``pancake shape'' naturally leads to coexistence of torus and split minimizers under homeotropic boundary data for domains of suitable limiting size.   More precisely, starting from a cigar shape domain provided by Theorem~\ref{thm:vertical-cylinders} where any minimizer is of split type, and then enlarging it sufficiently we arrive at a pancake shape where any minimizer is of torus type by~Theorem \ref{thmlargecyl}. Then Theorem~\ref{thmcoexmovdom} shows that split and torus minimizers must coexist in some domains of  intermediate size.The proof is similar in  spirit to the one for Theorem~\ref{main-coexistence} and it is still based on persistence of smoothness and persistence of singularities.

\begin{theorem}\label{thmcoexmovdom}
Let $\lambda\geq 0$ and $h,\ell_0,\rho>0$ be fixed numbers such that $2\rho<\ell_0/6$ and $\ell_0<3h$. For  $\ell\geq \ell_0$, set $\Omega_\ell:=\mathfrak{C}^h_{\ell,\rho}$ and let $Q^{(\ell)}_{\rm b}$ be the homeotropic boundary data on $\partial\Omega_\ell$. 
Assume that every minimizer of $\mathcal{E}_\lambda$ over $\mathcal{A}^{\rm sym}_{Q_{\rm b}^{(\ell_0)}}(\Omega_{\ell_0})$ is of split type (i.e., it has a non empty singular set). 
Then there exist numbers $\ell_2\geq \ell_1>\ell_0$ such that 
\begin{enumerate}
\item[(i)] for every $\ell_0\leq \ell<\ell_1$, every minimizer of $\mathcal{E}_\lambda$ over $\mathcal{A}^{\rm sym}_{Q_{\rm b}^{(\ell)}}(\Omega_{\ell})$ is of split type (i.e., singular); 
\vskip5pt
\item[(ii)] for every $\ell>\ell_2$, every minimizer of $\mathcal{E}_\lambda$ over $\mathcal{A}^{\rm sym}_{Q_{\rm b}^{(\ell)}}(\Omega_{\ell})$ is of torus type (i.e., smooth); 
\vskip5pt
\item[(iii)] for $\ell\in\{\ell_1,\ell_2\}$, $\mathcal{E}_\lambda$ admits both a split and a torus minimizer over $\mathcal{A}^{\rm sym}_{Q_{\rm b}^{(\ell)}}(\Omega_{\ell})$. 
\end{enumerate}
\end{theorem}

In the previous statement, we actually expect that $\ell_1=\ell_2$, i.e., only one critical size of the domain provides the coexistence property, but it seems to be a quite difficult problem.  
Existence of a singular minimizer in the symmetric class at the intermediate sizes $\ell=\ell_1$ and $\ell=\ell_2$ indicates once again that  a symmetry breaking occurs for global minimizers of $\mathcal{E}_\lambda$ over the global class  $\mathcal{A}_{Q_{\rm b}^{(\ell)}}(\Omega_{\ell})$. We shall prove in Corollary \ref{symbreakintermcyl} that symmetry breaking still occurs in a neighborhood of $\ell=\ell_1$ and  $\ell=\ell_2$, even for $\ell>\ell_2$ when all minimizers in the symmetric class are smooth. This fact enlightens the difficulty of proving or disproving axial symmetry of  minimizers over the full class. For instance, it would be already very interesting to determine whether or not minimizers of $\mathcal{E}_\lambda$ over $\mathcal{A}_{Q_{\rm b}^{(\ell)}}(\Omega_{\ell})$ are actually $\mathbb{S}^1$-equivariant for $\ell\gg\ell_2$ large enough. 
\vskip3pt

To conclude, we would like to mention that all the results presented here should have an analogue when the Lyuksyutov constraint is replaced by the Lyuksyutov (asymptotic)  regime as in \cite[Section 4]{DMP1},  and isotropic points playing the role of singular points. This will be the object of future investigations. 
\vskip10pt

\noindent \textbf{Acknowledgements.} F.D. and A.P. would like to warmly thank V.M. for kind invitation to visit Universit\'e Paris Diderot (currently, Universit\'e de Paris) in March~2019, where this third part of our project was initiated. 



\section{Axisymmetric domains, symmetric criticality, and Euler-Lagrange equations}\label{sec:preliminaries}

	\subsection{Axially symmetric domains}

In  this preliminary subsection, we define the relevant class of cylindrical domains of interest in the present paper. For geometric and topological properties of arbitrary axisymmetric domains $\Omega \subset \R^3$, we refer to \cite[Section 2]{DMP2}.
\vskip3pt

 First, we recall that  the unit circle $\mathbb{S}^1$ is identified with the subgroup of $\SO(3)$ made of all rotations around the vertical $x_3$-axis (see \eqref{rotmatr}),
  so that a matrix $R\in\mathscr{M}_{3\times 3}(\mathbb{R})$ represents a rotation of angle $\theta$ around the vertical axis iff it writes
\begin{equation}\label{rotmatr}
R=\begin{pmatrix} \widetilde R & 0 \\ 
  0 & 1 \end{pmatrix}
\quad\text{with}\quad \widetilde R:= \begin{pmatrix} \cos\theta & -\sin\theta  \\ \sin\theta & \cos\theta  \end{pmatrix}\,. 
\end{equation}
Axisymmetry is defined  accordingly.
 
 \begin{definition}
 \label{domain-section}
 A set $\Omega\subset \R^3$ is said to be \emph{axisymmetric} (or \emph{$\bbS^1$-invariant}, or \emph{rotationally symmetric}) if it is invariant under the action of $\mathbb{S}^1$, i.e., $R\cdot\Omega=\Omega$ for every $R\in \mathbb{S}^1$. Equivalently, $\Omega$ is axisymmetric if  
 $$\Omega=\bigcup_{R\in\mathbb{S}^1}R\cdot\mathcal{D}_\Omega \quad\text{where}\quad\mathcal{D}_\Omega:=\Omega\cap\{x_2=0\}\,.$$
 \end{definition}
 
For such domains, it is also useful to consider 
 the (relatively) open subsets 
\begin{equation}\label{vertsectsymdom}
\mathcal{D}^+_\Omega:=\mathcal{D}_\Omega\cap\{x_1>0\} \text{ and } \mathcal{D}^-_\Omega:=\mathcal{D}_\Omega\cap\{x_1<0\}
\end{equation}
 of the vertical plane $\{x_2=0\}$, so that $R_\pi\mathcal{D}^{\pm}_\Omega=\mathcal{D}^{\mp}_\Omega$. Indeed, if $I=\Omega \cap \{ x_3\hbox{-axis}\}$ then the following obvious identities hold:

\begin{equation}
\label{reconstruction}	
 \Omega \setminus I=\mathbb{S}^1 \cdot \mathcal{D}^+_\Omega \,  , \qquad \partial \Omega \cup I=  \mathbb{S}^1 \cdot \partial \mathcal{D}^+_\Omega \, , \qquad \overline{\Omega} =\mathbb{S}^1 \cdot \overline{\mathcal{D}^+_\Omega} \, , 
\end{equation}
with $\partial \mathcal{D}^+_\Omega \subset \overline{\mathcal{D}^+_\Omega} \subset \{ x_2=0\}$. Note that if $\Omega\subset \R^3$ is a bounded and smooth open set
then $\mathcal{D}_\Omega$ (or $\mathcal{D}^\pm_\Omega$) is a bounded and smooth (resp. piecewise smooth and Lipschitz) relatively open subset of the plane $\{x_2=0\}$. 

 \begin{remark}[{\bf homeotropic boundary data}]
We observe that if $\Omega$ is axisymmetric and $C^3$-smooth (resp. $C^{k,\alpha}$-smooth with $k\geq 3$), then the same property holds for the function  given by the signed distance to the boundary. Hence its gradient  is an $\mathbb{S}^1$-equivariant map, and in particular the outer normal field $\overrightarrow{n}(x)$ along $\partial \Omega$ is $C^2$-smooth (resp. $C^{k-1,\alpha}$-smooth) and equivariant. As a consequence, the corresponding {\sl homeotropic boundary data}  given by 
\begin{equation}
\label{eq:radial-anchoring}
		Q_{\rm b}(x) := \left( \overrightarrow{n}(x) \otimes \overrightarrow{n}(x) - \frac{1}{3}{\rm Id} \right)
	\end{equation}
	is $C^2$-smooth (resp. $C^{k-1,\alpha}$-smooth) and equivariant.  
\end{remark}

We shall be mainly concerned with axisymmetric domains $\Omega \subset \R^3$ which are homeomorphic to a cylinder. To define properly those domains, let us first set some useful notations. 
\vskip5pt

\noindent{\bf Notation} {\bf (rectangles \& cylinders).} Let $h,\ell\in(0,\infty]$ and $y\in\R^3$. 
\begin{itemize}
\item[(i)] The rectangle $\mathfrak{R}_\ell^h$ centered at the origin and the rectangle $\mathfrak{R}_\ell^h(y)$ centered at $y \in \{x_2=0\}$ are the sets 
\begin{equation}
\label{def:rectangles}	
\mathfrak{R}_\ell^h :=(-\ell,\ell)\times \{0\} \times (-h,h) \quad\text{and}\quad \mathfrak{R}_\ell^h (y):=y+ \mathfrak{R}_\ell^h \,.
\end{equation}
\item[(ii)] The  cylinder $\mathfrak{C}_\ell^h$ centered at the origin and the  cylinder  $\mathfrak{C}_\ell^h(y)$  centered at $y\in \R^3$ are the sets 
\begin{equation}
\label{def:cylinders}	
\mathfrak{C}_\ell^h := \big\{ x_1^2+x_2^2<\ell^2\big\} \times\{ |x_3|<h \}\, , \qquad \mathfrak{C}_\ell^h (y):=y+ \mathfrak{C}_\ell^h \,.
\end{equation}
\end{itemize}
We shall refer to $h$ as the \emph{height} and $\ell$ as the \emph{thickness} (or radius) of a cylinder.
 \vskip5pt

In order to apply our boundary regularity theory in \cite{DMP2} for energy minimizers under $\mathbb{S}^1$-symmetry constraint,  we need to consider some regularized version of the cylinders in \eqref{def:cylinders}. 
To define those, we first recall that for $p\in (1,\infty)$, a \emph{$p$-disc}  centered at $y=(y_1,0,y_3)$ and radius $\rho>0$ included in the vertical plane $\{x_2 = 0\}$ is a set of the form
$$ D^{(p)}_\rho(y) := \big\{ x=(x_1,0,x_3) \in \R^3 : \left( \abs{x_1-y_1}^p + \abs{x_3-y_3}^p \right)^{1/p} < \rho \big\}\,.$$
We shall use $p$-discs with $p=4$ to obtain inner $C^3$-regularizations of rectangles and cylinders. The scale of regularization $\rho>0$ will usually be a fixed number to be explicitly specified in terms of $h$ and $\ell$ in the calculations.
 
 \begin{definition}[\bf smoothed rectangles \& cylinders]\label{def:smooth-cyl}
 Let $h,\ell>0$ and $0<2\rho< \min \{h,\ell\}$.
 \begin{itemize}
	\item[($i$)] For vertical rectangles $\mathfrak{R}_\ell^h$ (resp. $\mathfrak{R}_\ell^h(y)$) as in \eqref{def:rectangles}, the corresponding \emph{smoothed $\rho$-rectangle} $\mathfrak{R}_{\ell,\rho}^h$ (resp. $\mathfrak{R}_{\ell,\rho}^h(y)$) is the union of all $4$-discs $D_\rho^{(4)}(z)$, $z=(z_1,z_3) \in \{x_2=0\}$, contained in $\mathfrak{R}_\ell^h$ (resp. $\mathfrak{R}_\ell^h(y)$). 
 	\item[($ii$)] For vertical cylinders $\mathfrak{C}_\ell^h$ and $\mathfrak{C}_\ell^h(y)$ as in \eqref{def:cylinders}, the corresponding \emph{smoothed $\rho$-cylinder} $\mathfrak{C}^{h}_{\ell,\rho}$ and  $\mathfrak{C}^{h}_{\ell,\rho}(y)$, $y\in \R^3$, are defined as
 	 \[
 	 	\mathfrak{C}^h_{\ell,\rho} := \bigcup_{R \in \bbS^1} R \cdot \mathfrak{R}^h_{\ell,\rho}\, , \qquad  \mathfrak{C}^{h}_{\ell,\rho}(y):=y+ \mathfrak{C}^{h}_{\ell,\rho} \, .
 	 \]
 \end{itemize}
The radius $\rho$ is called \emph{smoothing scale} of $\mathfrak{R}^h_\ell$ and $\mathfrak{C}^h_\ell$. When it is not relevant, we shall simply speak of \emph{smoothed rectangles} and \emph{smoothed cylinders}.
 \end{definition}
In view of the previous definition, $\mathfrak{C}^{h}_{\ell,\rho}$ is axially symmetric and the same holds for $\mathfrak{C}^{h}_{\ell,\rho}(y)$ if and only if $y$ belongs to the vertical axis, i.e., $y=(0,0,y_3)$, $y_3\in \R$. Moreover, $\mathfrak{C}^{h}_{\ell,\rho}(y) \cap \{x_2=0\}=\mathfrak{R}^h_{\ell,\rho}(y)$ whenever $y \in \{x_2=0\}$.
 
 \begin{remark}
 \label{rmk:rho-fine-reg} 
The boundary of a smooth rectangle is of class $C^{3,1}$ by our choice of $D^{(p)}_\rho(y)$ with $p=4$  (more generally, it is of class $C^{p -1,1}$ for each integer $p\geq 2$). The radius $\rho>0$ of the approximating discs gives the size of the region near the angles on which smoothing takes place. In addition, it is straightforward to check that
 $\mathfrak{R}^h_{\ell,\rho} \uparrow \mathfrak{R}^h_\ell$ and $\mathfrak{C}^h_{\ell,\rho} \uparrow \mathfrak{C}^h_\ell$ (and similarly for their translated counterparts) in the Hausdorff distance as $\rho \downarrow 0$ as a consequence of the elementary inclusions (recall that $0<2\rho<\min \{h,\ell\}$)
\begin{equation}\label{eq:RC-inclusions} \mathfrak{R}^h_{\ell-\rho} \cup \mathfrak{R}^{h-\rho}_\ell\subset \mathfrak{R}^h_{\ell,\rho} \subset \mathfrak{R}^h_\ell \, , \qquad \mathfrak{C}^h_{\ell-\rho} \cup \mathfrak{C}^{h-\rho}_\ell\subset \mathfrak{C}^h_{\ell,\rho} \subset \mathfrak{C}^h_\ell\, , \end{equation}
and the obvious analogues for their translated counterparts.
 \end{remark}


\subsection{Decomposition of $\mathcal{S}_0$ into invariant subspaces}

In order to give an efficient description of $\bbS^1$-equivariant  configurations, we will use the following decomposition results from \cite[Section~2]{DMP2} for the space $\mathcal{S}_0$ of admissible tensors.

\begin{lemma}[{\cite[Lemmas~2.1 \& 2.2, and Remark~2.3]{DMP2}}]\label{lemma:decomposition-S0}
There is a distinguished orthonormal basis $\big\{\eo, \euu, \eud, \edu, \edd\big\}$ of $\mathcal S_0$ given by 
\begin{multline}\label{eq:basis-S0}
\eo := \frac{1}{\sqrt 6}\begin{pmatrix} -1 & 0 & 0 \\ 0 & -1 & 0 \\ 0 & 0 & 2 \end{pmatrix}\,,\;\euu := \frac{1}{\sqrt 2}\begin{pmatrix} 0 & 0 & 1 \\ 0 & 0 & 0 \\ 1 & 0 & 0 \end{pmatrix}\,, \;\eud := \frac{1}{\sqrt 2}\begin{pmatrix} 0 & 0 & 0 \\ 0 & 0 & 1 \\ 0 & 1 & 0 \end{pmatrix}\,,\\[5pt]
 \edu:= \frac{1}{\sqrt 2}\begin{pmatrix} 1 & 0 & 0 \\ 0 & -1 & 0 \\ 0 & 0 & 0 \end{pmatrix}\,, \; \edd: = \frac{1}{\sqrt 2}\begin{pmatrix} 0 & 1 & 0 \\ 1 & 0 & 0 \\ 0 & 0 & 0 \end{pmatrix}\,,\qquad\qquad\qquad\qquad
\end{multline}
such that the subspaces
$$L_0 := \bb{R}\eo\,,\quad L_1:= \bb{R} \euu \oplus \bb{R} \eud\,,\quad L_2 := \bb{R} \edu \oplus \bb{R} \edd\,,$$
are invariant under the induced action of $\bb{S}^1$ on $\mathcal{S}_0$, namely, $\mathcal{S}_0 \ni A \mapsto RAR^{\rm t} \in \mathcal{S}_0$, and 
\begin{equation}
		\label{eq:decomposition-S0}
		\mathcal{S}_0 =L_0 \oplus L_1 \oplus L_2 \simeq \R \oplus \C \oplus \C\,.
\end{equation}
	Moreover, the $\bb{S}^1$-action on $\mathcal{S}_0$ corresponds to an $\bbS^1$-action on each $L_k$ by rotations of degree $k$, in the sense that the induced $\bbS^1$-action on $\R \oplus \C \oplus \C$ is given by
	\begin{equation}\label{eq:deg-action}
		R_\alpha \cdot (t, \zeta_1, \zeta_2) = (t, e^{i\alpha} \zeta_1, e^{2i\alpha}\zeta_2) \quad \forall R_\alpha \in \bbS^1\,.
	\end{equation}
\end{lemma}
\vskip5pt

As a straightforward consequence of the decomposition~\eqref{eq:decomposition-S0} in the orthonormal basis \eqref{eq:basis-S0}, we derive the following explicit formulas for a tensor $Q$ and its determinant.

\begin{lemma}
\label{lemma:Q-detQ}
Elements $Q \in \mathcal{S}_0$ are in one-to-one (linear) correspondence with elements  $u=(u_0,u_1,u_2) \in \R \oplus \C \oplus \C$. This correspondence, denoted as $Q\simeq u$, is given by
\begin{equation}\label{eq:correspondence}
		 Q= \frac{1}{\sqrt{2}} \left( 
\begin{array}{ccc}
 -\frac{u_0}{\sqrt{3}}+ {\rm Re} (u_2) & {\rm Im} (u_2) & {\rm Re} (u_1) \\
 {\rm Im} (u_2) & -\frac{u_0}{\sqrt{3}}-{\rm Re} (u_2) & {\rm Im}(u_1) \\
 {\rm Re} (u_1) & {\rm Im} (u_1) & \frac{2 u_0}{\sqrt{3}}
\end{array}
\right) \, .
 \end{equation}
 In addition, it is isometric, i.e., $|Q|^2={\rm Tr}(Q^2) =|u|^2= u_0^2 + \abs{u_1}^2 + \abs{u_2}^2$, and
 \begin{equation}
 \label{detQ}
 {\rm det} \,Q= \frac{1}{2\sqrt{2}}\left[ \frac{2u_0}{\sqrt{3}} \left( \frac{u_0^2}{3}+\frac12|u_1|^2-|u_2|^2\right)+{\rm Re} (u_1^2 \overline{u_2})\right]\, .
\end{equation}
 
\end{lemma}
 
The previous lemmas yield in the obvious way a (linear, isometric) correspondence between $Q$-tensor fields on $\Omega$ and maps from $\Omega$ into $\R \oplus \C \oplus \C$. The following corollary is a direct consequence of \eqref{eq:basis-S0}, \eqref{eq:decomposition-S0}, and \eqref{eq:correspondence}. The proof is elementary and  left to the reader.

\begin{corollary}\label{cor:dec}
Let $\Omega$ be an open subset of  $\R^d$.  Elements $Q \in W^{1,2}(\Omega;\mathcal{S}_0)$ are in one-to-one (linear)  correspondence with elements  $u=(u_0,u_1,u_2) \in W^{1,2}(\Omega; \R \oplus \C \oplus \C)$. This correspondence, still denoted as $Q\simeq u$, is given by relation \eqref{eq:correspondence} holding a.e. in $\Omega$. In addition, if $Q\simeq u$, then $|Q|^2=|u|^2$ and $\abs{\nabla Q}^2 = \abs{\nabla u}^2$ a.e. in $\Omega$. In particular, $Q \in W^{1,2}(\Omega;\bbS^4)$ if and only if $u \in W^{1,2}(\Omega ; \bb S^4)$.
\end{corollary}


\subsection{$\bbS^1$-equivariant $Q$-tensor fields}
We now specialize our previous discussion to $\bbS^1$-equivariant $Q$-tensor fields on rotationally invariant bounded open sets. It is natural to describe such sets and $Q$-tensor fields in terms of cylindrical coordinates $(r,x_3,\phi)$ (which of course reduce to polar coordinates $(r,\phi)$ in the case of horizontal discs). This description yields the following refinement of the decomposition in Corollary~\ref{cor:dec}.

\begin{lemma}\label{lemma:dec-equiv}
Let $\Omega \subset \R^3$ be a bounded and axisymmetric open set and $\mathcal{D}_\Omega^+$ its vertical section given by \eqref{vertsectsymdom}. If $Q \in W^{1,2}_{\rm sym}(\Omega;\mathcal{S}_0)$ and  $Q\simeq u = (u_0,u_1,u_2) \in W^{1,2}(\Omega;\R \oplus \C \oplus \C)$ is the corresponding map in the sense of Corollary~\ref{cor:dec},  then $u$ is $\bbS^1$-equivariant  with respect to the action \eqref{eq:deg-action} on $\R \oplus \C \oplus \C$. As a consequence, for each $k \in \{0,1,2\}$, $u_k$ can be decomposed as 
$$u_k(x) = f_k(r,x_3) e^{i k \phi}\,,$$ 
for functions $f_k \in W^{1,2}(\mathcal{D}_\Omega^+,rdrdx_3)$ which are $\C$-valued for $k=1,2$, and $\R$-valued for $k=0$. 
Thus, 
\begin{equation}\label{eq:equiv-nabla-2}
	\abs{\nabla Q}^2 = \abs{\nabla f_0}^2 + \abs{\nabla f_1}^2 + \abs{\nabla f_2}^2 + \frac{\abs{f_1}^2 + 4\abs{f_2}^2}{r^2}\quad\text{a.e. in $\Omega$}\,,
\end{equation}
where $\abs{\nabla f_k}^2 := \abs{\partial_r f_k}^2 + \abs{\partial_{x_3} f_k}^2$. In particular, $\abs{\nabla Q}^2$ does not depend on $\phi$, and
\begin{equation}\label{Eq:equiv-energy-E0}
	 \mathcal{E}_0(Q) =\pi \int_{\mathcal{D}_\Omega^+} \left(\abs{\nabla f_0}^2 + \abs{\nabla  f_1}^2 + \abs{\nabla f_2}^2 + \frac{\abs{f_1}^2 + 4\abs{f_2}^2}{r^2}\right)\,r\,drdx_3<\infty\,.
\end{equation}
\end{lemma}

\begin{proof}
	In view of \eqref{eq:deg-action}, the $\bbS^1$-equivariance of $Q$ translates into the identities
	\[
		u_0(R_\alpha x) = u_0(x) \,,\quad u_1(R_\alpha x) = e^{i\alpha} u_1(x) \,, \quad u_2(R_\alpha x) = e^{2i\alpha} u_2(x)\,,
	\]
	which hold for every $R_\alpha \in \bbS^1$ and a.e. $x \in \Omega$. In terms of cylindrical coordinates, those identities imply
	\begin{equation}
	\label{eq:dec-equiv}
		u_0(x) = f_0(r,x_3) \,,\quad u_1(x) = f_1(r,x_3) e^{i\phi} \,,\quad u_2(x) = f_2(r,x_3) e^{2i\phi}\,,
	\end{equation}
	a.e. in $\Omega$. Hence $f_k \in W^{1,2}(\mathcal{D}_\Omega^+,rdrdx_3)$ since $u_k \in W^{1,2}(\Omega)$ for each $k \in \{0,1,2\}$.
	
	 Moreover, Corollary \ref{cor:dec} yields
	\[
	 	\abs{\nabla Q}^2 = \sum_{k=0}^2 \abs{\nabla u_k}^2 = \sum_{k=0}^2 \abs{\nabla f_k}^2 + \frac{k^2 \abs{f_k}^2}{r^2}\quad\text{a.e. in $\Omega$}\,,
	\]
	which proves \eqref{eq:equiv-nabla-2}. Finally,
	since the right hand side above only  depends on $(r,x_3)$, applying Fubini's theorem leads to 
	\[
		\int_{\Omega} \abs{\nabla u_k}^2 \,dx = 2\pi \int_{\mathcal{D}_\Omega^+} \abs{\nabla f_k}^2 + \frac{k^2 \abs{f_k}^2}{r^2} \,r\,dr dx_3< +\infty\,. 
	\]
Then \eqref{Eq:equiv-energy-E0} follows summing this equality over $k=0,1,2$. 
\end{proof}

\begin{remark}
\label{2d-dec-equiv}
It is straightforward to check that the previous lemma also holds in two dimensions, i.e., if $\Omega=\bb D_\rho \subset \R^2 $ is a disc of radius $\rho>0$ centered at the origin. In this case, if  $Q \in W^{1,2}_{\rm sym}(\bb D_\rho;\mathcal{S}_0)$ and  $Q\simeq u = (u_0,u_1,u_2)$, then 
\begin{equation} \label{eq:dec-equiv2D} 
u_k(x)=f_k(r)e^{ik\phi}
\end{equation}
where $(r,\phi)$ are the polar coordinates, and each $f_k$ belongs to $W^{1,2}((0,\rho), rdr)$. In addition, 
\eqref{eq:equiv-nabla-2} and \eqref{Eq:equiv-energy-E0} still holds under the forms 
\begin{equation}\label{eq:equiv-nabla-22D}
\abs{\nabla Q}^2=|\nabla u|^2= \abs{f^\prime_0}^2 + \abs{f^\prime_1}^2 + \abs{f^\prime_2}^2 + \frac{\abs{f_1}^2 + 4\abs{f_2}^2}{r^2}\quad\text{a.e. in $\mathbb{D}_\rho$}\,,
\end{equation}
and
\begin{equation}\label{Eq:equiv-energy-E02D}
\frac{1}{2}\int_{\mathbb{D}_\rho}|\nabla Q|^2\,dx=\pi \int_0^\rho \left(\abs{f^\prime_0}^2 + \abs{f^\prime_1}^2 + \abs{f^\prime_2}^2 + \frac{\abs{f_1}^2 + 4\abs{f_2}^2}{r^2}\right)\,r\,dr<\infty\,,
\end{equation}
respectively. 
\end{remark}

The next result describes a  fine property of the space $W^{1,2}_{\rm sym}(\bb D_\rho;\bbS^4)$ in the 2D-case $D_\rho \subset \R^2$. Symmetry and norm constraints yield the inclusion $W^{1,2}_{\rm sym}(\bb D_\rho;\bbS^4) \subset C^0(\overline{\bb D_\rho};\bb S^4)$, a property which will be of crucial importance for the 2D-minimization problems discussed in Section \ref{2Dminimization}. Up to a rescaling, we may assume without loss of generality in the following statement that $\rho = 1$.

\begin{lemma}\label{lemma:s1eq-emb}
	Let $\bb D \subset \R^2$ be the unit disc. If $Q \in W^{1,2}_{\rm sym}(\bbD;\bbS^4)$, then 
	\begin{itemize}
		\item[($i$)] $Q \in C^0(\overline{\bbD}; \bb S^4)$ and either $Q(0) = \eo$ or $Q(0) = -\eo$.
	\end{itemize}
	Moreover, for $\{Q_n\} \subset W^{1,2}_{\rm sym}(\bbD;\bbS^4)$ and $Q_* \in W^{1,2}(\bbD;\mathcal{S}_0)$, the following statements hold.
	\begin{itemize}
		\item[($ii$)] If $Q_n \rightharpoonup Q_*$ weakly in $W^{1,2}(\bbD)$, then $Q _*\in W^{1,2}_{\rm sym}(\bbD;\bbS^4)$
		and $Q_n \to Q_*$ in $C^0_{\rm loc}(\overline{\bbD} \setminus \{0\})$. In particular, $Q_* \vert_{\partial \bbD}\to Q_n \vert_{\partial \bbD}$ uniformly on $\partial \bbD$. 
		\vskip5pt
		\item[($iii$)] If $Q_n \to Q_*$ strongly in $W^{1,2}(\bbD)$ then $Q _*\in W^{1,2}_{\rm sym}(\bbD;\bbS^4)$, 
		$Q_*(0)\equiv Q_n(0)$ for $n$ large enough, and $Q_n \to Q_*$ uniformly on $\overline{\bbD}$. 
	\end{itemize}
 	Claims ($i$), ($ii$), and ($iii$) still hold replacing $Q$, $Q_n$, and $Q_*$ with the corresponding maps with values into $\R \oplus \C \oplus \C$. 
\end{lemma}	
	
\begin{proof}
	($i$) According to Corollary~\ref{cor:dec}, we write $Q\simeq u=(u_0,u_1, u_2)$ with $u_k \in W^{1,2}(\bb D)$, $k=0,1,2$. By Remark~\ref{2d-dec-equiv} above, each function $f_k$ in \eqref{eq:dec-equiv2D} belongs to $W^{1,2}((0,1),rdr)$. Then the 1D-Sobolev embedding implies that $f_k \in C^{0,\frac{1}{2}}_{\rm loc}((0,1])$, and in turn $Q \in C^{0,\frac{1}{2}}_{\rm loc}(\overline{\bbD} \setminus \{0\}) \subset C^0(\overline{\bbD} \setminus \{0\}) $ by $\bbS^1$-equivariance. Then it only remains to prove continuity at the origin. To this purpose, we fix $0 < \rho'<\rho<1$. Combining Young's inequality with \eqref{Eq:equiv-energy-E02D} and Remark~\ref{2d-dec-equiv}, we compute
\begin{align}
\nonumber	\abs{\abs{f_1(\rho)}^2-\abs{f_1(\rho')}^2 }+& \abs{\abs{f_2(\rho)}^2-\abs{f_2(\rho')}^2} =\abs{\int_{\rho'}^\rho \partial_r \abs{f_1}^2 dr}+ \abs{\int_{\rho'}^\rho \partial_r \abs{f_2}^2 dr} \\
\nonumber	&\leq   \int_{\rho'}^\rho \left(\abs{ f^\prime_1}^2 + \frac{\abs{f_1}^2}{r^2} \right)\,r\,dr +  \int_{\rho'}^\rho \left(\abs{f^\prime_2}^2 + \frac{\abs{f_2}^2}{r^2} \right)\,r\,dr\\
\nonumber	& \leq \int_{\rho'}^\rho \left(\abs{ f^\prime_0}^2 + \abs{f^\prime_1}^2 + \abs{ f^\prime_2}^2 + \frac{\abs{f_1}^2+4\abs{f_2}^2}{r^2} \right)\,r\,dr\\
\label{eq:modica-trick}	&  \leq \frac{1}{2\pi}\int_{\bbD_\rho}|\nabla Q|^2\,dx\, . 
\end{align}	
Since $Q$ belongs to $W^{1,2}(\bbD)$, we have $\int_{\bbD_\rho}|\nabla Q|^2\,dx \to 0$ as $\rho \to 0$. Hence both $\ell_1:=\lim_{r \to 0} |f_1(r)|$ and $\ell_2:=\lim_{r \to 0} |f_2(r)|$ exist. On the other hand, it follows  
from \eqref{Eq:equiv-energy-E02D}  
that $\ell_1=\ell_2=0$. Thus, both $f_1$ and $f_2$ extend by continuity to elements of $C^0([0,1];\C)$ with $f_1(0)=f_2(0)=0$. In turn, \eqref{eq:dec-equiv2D} yields $u_k \in C^0(\overline{\bb D};\C)$ with $u_k(0)=0$ for $k=1,2$. 
	
	Finally, combining Corollary~\ref{cor:dec} with \eqref{eq:dec-equiv2D} leads to $|Q(re^{i \phi})|^2=|f_0(r)|^2+|f_1(r)|^2+|f_2(r)|^2\equiv 1$. Since $f_1(0)=f_2(0)=0$, we have $|f_0(r)|\to 1$ as $r \to 0$. Moreover, either $f_0(r) \to 1$ or $f_0(r) \to -1$ as $r \to 0$. Indeed, if the limit does not exist, then $\liminf_{r \to 0}f_0(r)=-1<1=\limsup_{r \to 0} f_0(r)$. By continuity, it would imply the existence of $r_n \downarrow 0$ such that $f_0(r_n)\equiv 0$, and leading to the identity $1 \equiv |f_0(r_n)|^2+|f_1(r_n)|^2+|f_2(r_n)|^2\to 0$ as $n \to \infty$, a contradiction. Thus, $f_0$ extends by continuity to a function in $C^0([0,1];\R)$ with $f_0(0)=\pm 1$, and in turn $u \in C^0(\overline{\bb D};\bb S^4)$ with $u(0)=(\pm 1,0,0)$. As a consequence, $Q \simeq u$ in continuous on $\overline{\bb D}$, and $Q(0)=\pm \eo$ which proves{\it (i)}. 
\vskip5pt
	
	($ii$)	In view of ($i$) all the maps involved are continuous. Moreover, $\bbS^1$-equivariance allows us to use the continuous embedding $W^{1,2}((0,1),rdr) \hookrightarrow C^{0,\frac{1}{2}}_{\rm loc}((0,1])$  and the compact embedding $C^{0,\frac{1}{2}}_{\rm loc}((0,1]) \hookrightarrow C^0_{\rm loc}((0,1])$ to deduce that $Q_n \to Q_*$ locally uniformly on $\overline{\bbD} \setminus \{0\}$. As the convergence is also pointwise on $\overline{\bbD} \setminus \{0\}$, both equivariance and norm constraints persist, and we have $Q_* \in W^{1,2}_{\rm sym}(\bbD;\bbS^4)\cap C^0(\overline{\bbD})  $. Moreover $Q_* \vert_{\partial \bbD} \to Q_n \vert_{\partial \bbD}$ uniformly on $\partial \bbD$.
\vskip5pt
	
	($iii$) Assume now $Q_n \to Q_*$ strongly in $W^{1,2}(\bbD)$. By ($ii$), it only remains to prove uniform convergence in a (small) disc centered at the origin. 
	 To achieve this,  it suffices to show that $Q_n(0)\equiv Q_*(0)$ for $n$ large enough and that the sequence $\{Q_n\}$ is equicontinuous at the origin. To check these properties, we first notice that \eqref{eq:modica-trick} holds for each $Q_n$. By ($i$), we can choose $\rho^\prime=0$ and any fixed $\rho \in (0,1)$ to obtain from \eqref{eq:modica-trick}, 
	\begin{equation}
	\label{eq:hor-mod-cont}
	\abs{f^{(n)}_1(\rho)}^2+ \abs{f^{(n)}_2(\rho)}^2\leq \frac{1}{2\pi}\int_{\bb D_\rho}|\nabla Q_n|^2\,dx \, .
	\end{equation}
Letting $n \to \infty$ above, the same inequality holds for the components $f^*_k$, $k=1,2$, of $Q_*$. 
	
	By the Vitali-Hahn-Saks theorem (see e.g. \cite[Theorem 1.30]{AFP}), the strong $W^{1,2}$-convergence of the sequence $\{Q_n\}$ implies that $\{ \abs{\nabla Q_n}^2 \}$ is equiintegrable. Combining this fact with \eqref{eq:hor-mod-cont}, it follows that $\{f^{(n)}_1\}$ and $\{f^{(n)}_2\}$ are equicontinuous at the origin. Moreover, there exists $\bar{\rho}>0$ such that 
$$ \int_{\bbD_{\bar\rho}}|\nabla Q_n|^2\,dx + \int_{\bbD_{\bar\rho}}|\nabla Q_*|^2\,dx\leq  \frac{3\pi}{2} \quad\text{for $n$ large enough}\,.$$ 
Hence $\big|f_1^{(n)}\big|^2+ \big|f^{(n)}_2\big|^2\leq \frac34$ in $[0,\bar\rho]$ for $n$ large enough,  which in turn implies that $\big|f_0^{(n)}\big|\geq \frac12$  in $[0,\bar\rho]$ for $n$ large enough. By continuity, it follows that each $f_0^{(n)}$ has constant sign in $[0,\bar\rho]$  for $n$ large enough.   
The same property holds for $f^*_0$, and the sign of $f_0^{(n)}$ must be the same of $f_0^*$ for $n$ large enough because of the pointwise convergence in $\bb D_{\bar{\rho}}\setminus\{0\}$. This proves that $Q_n(0)\equiv Q_*(0)$ for $n$ large enough. 
	
	Finally, combining the pointwise inequalities on $f_0^{(n)}$, the norm constraint, and \eqref{eq:hor-mod-cont}, we have for every $0<\rho \leq \bar{\rho}$, 
\[ \abs{f^{(n)}_0(\rho)-f^{(n)}_0(0) }=\frac{ 1- \abs{f^{(n)}_0(\rho)}^2}{ \abs{f^{(n)}_0(\rho)+f^{(n)}_0(0) } } \leq \abs{f^{(n)}_1(\rho)}^2+ \abs{f^{(n)}_2(\rho)}^2\leq \frac{1}{2\pi} \int_{\bbD_\rho} |\nabla Q_n|^2\,dx \, .\]
Hence the sequence $\{f_0^{(n)}\}$ is also equicontinuous at the origin by the Vitali-Hahn-Saks theorem. Going back to \eqref{eq:dec-equiv2D}, we deduce that the maps $\{u^{(n)}\}$ are equicontinuous at the origin, and thus thus the same holds for $\{Q_n\}$ which  completes the proof of the uniform convergence. 

The final claim concerning the corresponding maps into $\R \oplus \C \oplus \C$ follows taking scalar products with the orthonormal basis in \eqref{eq:basis-S0}. 
\end{proof}

With Lemma~\ref{lemma:s1eq-emb} in hands, we can easily prove that  $W^{1,2}$-tensor fields on a 3D-axisymmetric domain $\Omega$ have a well-defined trace on the vertical axis.

\begin{corollary}
\label{verticaltrace}
Let $\Omega\subset \R^3$ be a bounded and axisymmetric open set with Lipschitz boundary, and set $I:=\Omega \cap \{x_3\hbox{-axis}\}$. There is a (strongly)  continuous trace operator ${\rm Tr} : W^{1,2}_{\rm sym}(\Omega;\bbS^4) \to L^1(I;\{ \pm \eo\})$ satisfying ${\rm Tr}\,Q=Q_{\vert_I}$ whenever $Q \in W^{1,2}_{\rm sym}(\Omega;\bbS^4)\cap C^0(\overline{\Omega})$.
\end{corollary}

\begin{proof}
We first notice that for $\ell>0$ small enough and $h>0$ large enough, the set $\overline{\Omega \cap \mathfrak{C}_\ell^h}$ is (equivariantly) biLipschitz homeomorphic to a finite union of disjoint $\bb S^1$-invariant closed cylinders, the homeomorphism being the identity on the vertical axis. 
Hence, up to {{a}} change of variables, it is enough to construct the trace operator when the domain is an arbitrary cylinder to have a well defined induced operator ${\rm Tr} : W^{1,2}_{\rm sym}(\Omega \cap \mathfrak{C}_\ell^h;\bbS^4) \to L^1(I;\{ \pm \eo\})$.  In turn, the conclusion follows by composition with the continuous restriction operator $W^{1,2}(\Omega) \to W^{1,2}(\Omega\cap \mathfrak{C}_\ell^h) $.

Assuming now that $\Omega=\mathfrak{C}_\ell^h=\bb D_\ell \times (-h,h)=\bb D_\ell\times I$, then we have 
$$W^{1,2}_{\rm sym}(\mathfrak{C}_\ell^h;\bb S^4)\subset L^2\big(I; W^{1,2}_{\rm sym}(\bb D_\ell;\bb S^4)\big) \subset L^1\big(I; W^{1,2}_{\rm  sym}(\bb D_\ell;\bb S^4)\big)$$ 
with continuous inclusions. In view of Lemma~\ref{lemma:s1eq-emb} the mapping $W^{1,2}_{\rm sym}(\bb D_\ell;\bb S^4) \ni Q \mapsto Q(0) \in \{ \pm \eo\}$ is well defined and (strongly) continuous. Hence, by composition of this map with the inclusion maps above, we have a well defined and (strongly) continuous map ${\rm Tr} : W^{1,2}_{\rm sym}(\mathfrak{C}_\ell^h;\bbS^4) \to L^1(I;\{ \pm \eo\})$ with all the desired properties.
\end{proof}

\subsection{Existence of minimizers and Euler-Lagrange equations}
We recall from \cite{DMP2} the following results about ``symmetric criticality'' and existence of minimizers over the class $\mathcal{A}_{Q_{\rm b}}^{\rm sym}(\Omega)$. Even if the results were stated in case of 3D  domain, their proofs hold with obvious modifications in the planar case, i.e., when $\Omega$ is disc a centered at the origin. 

\begin{proposition}[{\cite[Proposition~6.1 and 6.2]{DMP2}}]\label{prop:symmetric-criticality}
Let $\Omega \subset \R^3$ be a bounded and axisymmetric open set. 
\begin{itemize}
	\item[(i)] If $Q_\lambda \in W^{1,2}_{\rm sym}(\Omega;\mathbb{S}^4)$ is a critical point of $\mathcal{E}_\lambda$ over $W^{1,2}_{\rm sym}(\Omega;\bbS^4)$, then $Q_\lambda$ is a critical point of $\mathcal{E}_\lambda$ among all maps $W^{1,2}(\Omega;\bbS^4)$.
	\vskip5pt
	
	\item[(ii)] If $\partial \Omega$ is Lipschitz regular and $Q_{\rm b} \in \rmLip(\partial \Omega;\bbS^4)$ is $\mathbb{S}^1$-equivariant, then $\mathcal{A}_{Q_{\rm b}}^{\rm sym}(\Omega)$ is not empty and there exists at least one minimizer of $\mathcal{E}_\lambda$ over $\mathcal{A}_{Q_{\rm b}}^{\rm sym}(\Omega)$. 
\end{itemize}
\end{proposition}

In other words, critical points of  $\mathcal{E}_\lambda$ among equivariant compactly supported perturbations preserving  the $\mathbb{S}^4$-constraint are critical points with respect to every compactly supported perturbation still  preserving  the $\mathbb{S}^4$-constraint  (i.e., even with respect to those which are not equivariant). In other words, they are weak solutions to \eqref{MasterEq}. 

\begin{remark}
For a map $Q\in W^{1,2}_{\rm sym}(\Omega;\bb S^4)$, the energy functional \eqref{LDGenergytilde} can be rewritten in terms the correspondence $Q\simeq u=(f_0, f_1 e^{i \phi},f_2 e^{i2\phi})$ in Lemma \ref{lemma:dec-equiv}. By  \eqref{detQ}, we have 
\begin{equation}\label{eq:beta-f}
	\widetilde{\beta}(f):=\widetilde{\beta}(Q) = 3 \sqrt{6} \det Q = f_0\left( f_0^2 +\frac{3}{2}\abs{f_1}^2 - 3 \abs{f_2}^2 \right)+ \frac{3\sqrt{3}}{2}{\rm Re}(f_1^2 \overline{f_2})\,,
\end{equation}
where $f:=(f_0,f_1,f_2)$. Combining identity \eqref{redpotential} with \eqref{Eq:equiv-energy-E0} yields 
\begin{equation}
\label{eq:equiv-energy-f}
\mathcal{E}_\lambda(Q)=\pi \int_{\mathcal{D}^+_\Omega}\left( \abs{\nabla f}^2+ \frac{\abs{f_1}^2 + 4\abs{f_2}^2}{r^2} +2\lambda \frac{1-\tilde{\beta}(f)}{3\sqrt{6}}    \right) \, rdrdx_3 \, .\end{equation}
If $Q$ is a critical point  of  $\mathcal{E}_\lambda$ among equivariant compactly supported perturbations (preserving  the $\mathbb{S}^4$-constraint), then $Q$ weakly solves \eqref{MasterEq}  from the proposition above. To rephrase the equations in terms of $f$, we may project \eqref{MasterEq} onto the orthonormal frame \eqref{eq:basis-S0} or, equivalently, take variations in the energy functional \eqref{eq:equiv-energy-f}.  The criticality condition \eqref{MasterEq} then translates into the following nonlinear system for $f=(f_0,f_1,f_2)\in W^{1,2}(\mathcal{D}^+_\Omega;\bb S^4,rdrdx_3)$, namely, 
\begin{equation}\label{eq:EL-symm-PDE}
	\left\{
		\begin{aligned}
			\partial^2_{r} f_0 + \frac{1}{r} \partial_r f_0 + \partial^2_{x_3} f_0 &= - \abs{\nabla Q}^2 f_0 +  \frac{\lambda}{\sqrt{6}} \left(\abs{f_2}^2-f_0^2 - \frac{1}{2}\abs{f_1}^2 +  \widetilde{\beta}(f) f_0 \right) \,,\\
			\partial^2_{r} f_1 + \frac{1}{r} \partial_r f_1 + \partial^2_{x_3} f_1 &= - \abs{\nabla Q}^2 f_1 - \frac{1}{r^2} f_1 + \frac{\lambda}{\sqrt{6}} \left(- \sqrt3f_2 \overline{f_1}-f_0 f_1+\widetilde{\beta}(f)  f_1    \right) \,,\\
			\partial^2_{r} f_2 + \frac{1}{r} \partial_r f_2 + \partial^2_{x_3} f_2 &= - \abs{\nabla Q}^2 f_2 - \frac{4}{r^2}f_2 + \frac{\lambda}{\sqrt6} \left( -\frac{\sqrt3}2 f_1^2 +  2f_0 f_2+ \widetilde{\beta}(f) f_2 \right) \,,
		\end{aligned}
	\right.
\end{equation}
with $\widetilde{\beta}(f)$ as in \eqref{eq:beta-f}  and $|\nabla Q|^2$ as in \eqref{eq:equiv-nabla-2}, both depending only on $f=(f_0,f_1,f_2)$.

\end{remark}

\begin{remark}[\bf 2D-case]\label{rmk:2d-EL-eq}
In Section~\ref{2Dminimization} (mostly), we shall consider the two dimensional case $\Omega=\bb D_\rho \subset \R^2$. To differentiate the 2D from the 3D case, we shall use the notation $E_\lambda(Q,\mathbb{D}_\rho)$ (instead of $\mathcal{E}_\lambda$) for the {\sl 2D-energy} of a configuration $Q\in W^{1,2}_{\rm sym}(\bbD_\rho;\mathbb{S}^4)$. In view of Remark~\ref{2d-dec-equiv}, and as in \eqref{eq:equiv-energy-f}, the energy of $Q\simeq u=(f_0, f_1 e^{i \phi},f_2 e^{i2\phi})$ can be written in terms of $f$, leading to 
$$E_\lambda(Q,\mathbb{D}_\rho)=\pi\int_0^\rho \left( \abs{ f^\prime}^2+ \frac{\abs{f_1}^2 + 4\abs{f_2}^2}{r^2} +2\lambda \frac{1-\tilde{\beta}(f)}{3\sqrt{6}}    \right) \, rdr\,.$$
Then the criticality condition (in terms of $f$) for the functional $E_\lambda$ is almost identical to \eqref{eq:EL-symm-PDE}. It is obtained from it simply neglecting in each equation the terms $\partial^2_{x_3}$  and $\partial_{x_3}$  in the left hand side and the right hand side respectively.
\end{remark}



\section{Coexistence of smooth and singular minimizers}\label{sectcoexball}

\subsection{Regularity theory}\label{secregth}
The purpose of this subsection is to gather (and slightly refine) the main regularity results and tools obtained in \cite{DMP1,DMP2}  to have them at disposal in the most convenient form when they will be repeatedly used in the next subsections. To this end, let us recall the usual definition of {\sl singular set} for a map $Q$ defined on an open set $\Omega$. It is then defined as 
$${\rm sing}(Q):=\Omega\setminus\big\{x\in\Omega: \text{$Q$ is continuous in a neighborhood of $x$}\big\}\,.$$
\vskip3pt

The following interior regularity theorem, even if not explicitly stated in \cite{DMP2}, is a direct consequence of the discussion in \cite[Section 6]{DMP2}. In particular, formula \eqref{singularitycost} below is a combination of the strong $W^{1,2}$-convergence of the rescaled maps $Q_\lambda^{\bar x,\rho}$ together with the explicit form \eqref{formtangmaps} of all possible blow-up limits at a singular point. In our statement below, we  require Lipschitz regularity of the boundary only to ensure that the $W^{1,2}$-trace operator on $\partial\Omega$ is well defined. 

\begin{theorem}[\bf \cite{DMP2}, interior regularity]\label{intregthm}
Let $\Omega\subset\R^3$ be a bounded and axisymmetric open set with Lipschitz boundary,  and $Q_\lambda \in W^{1,2}_{\rm sym}(\Omega;\mathbb{S}^4)$ minimizing $\mathcal{E}_\lambda$ among all $Q \in W^{1,2}_{\rm sym}(\Omega;\mathbb{S}^4)$ satisfying $Q=Q_\lambda$ on $\partial \Omega$. Then $Q_\lambda\in C^\omega(\Omega\setminus {\rm sing}(Q_\lambda))$ and ${\rm sing}(Q_\lambda)\subset \{x_3\text{-axis}\}\cap\Omega$ is locally finite in $\Omega$. In addition, for every $\bar x\in {\rm sing}(Q_\lambda)$, there exist a rotation $R_\alpha\in\mathbb{S}^1$ and  $Q_*\in\{\pm Q^{(\alpha)}\}$   such that 
\begin{itemize}
\item[(i)] $Q_\lambda^{\bar x,\rho}\to Q_*$ strongly in $W^{1,2}_{\rm loc}(\R^3)$ as $\rho\to 0$;
\vskip5pt
\item[(ii)] $\|Q_\lambda^{\bar x,\rho}-  Q_*\|_{C^2(\overline B_2\setminus B_1)}=O(\rho^\nu)$ as $\rho\to 0$ for some $\nu>0$;  
\end{itemize}
where $Q_\lambda^{\bar x,\rho}(x):=Q_\lambda(\bar x+\rho x)$ and 
\begin{equation}\label{formtangmaps}
Q^{(\alpha)}(x):=R_\alpha \cdot \frac{1}{\sqrt{6}}\frac{1}{|x|}
\left(\begin{array}{ccc}
-x_3 & 0 & \sqrt{3} x_1\\
0 & -x_3 & \sqrt{3}x_2 \\
\sqrt{3} x_1 & \sqrt{3} x_2 & 2 x_3
\end{array}\right)\,.
\end{equation}
In particular,
\begin{equation}\label{singularitycost}
\lim_{\rho\to 0}\frac{1}{\rho}\int_{B_\rho(\bar x)} \frac{1}{2} |\nabla Q_\lambda|^2\,dx= \lim_{\rho\to 0}\frac{1}{\rho}\mathcal{E}_\lambda\big(Q_\lambda,B_\rho(\bar x)\big)=4\pi\quad \text{for every $\bar x\in{\rm sing}(Q_\lambda)$}\,.
\end{equation}
\end{theorem}
\vskip5pt

Regularity at the boundary holds whenever the boundary of $\Omega$ and the boundary data are smooth enough. In this case, the singular set is  made of finitely many points inside the domain $\Omega$.

\begin{theorem}[\bf \cite{DMP2}, regularity up to the boundary]\label{bdrregthm}
Let $\Omega\subset\R^3$ be a bounded and axisymmetric open set with boundary of class $C^3$, and let $Q_{\rm b}\in C^{1,1}(\partial\Omega;\mathbb{S}^4)$ be an $\mathbb{S}^1$-equivariant map. If $Q_\lambda$ is a minimizer of $\mathcal{E}_\lambda$ over $\mathcal{A}^{\rm sym}_{Q_{\rm b}}(\Omega)$, then $Q_\lambda\in C^\omega(\Omega\setminus{\rm sing}(Q_\lambda))\cap C^{1,\alpha}(\overline\Omega\setminus{\rm sing}(Q_\lambda))$ for every $\alpha\in(0,1)$ and ${\rm sing}(Q_\lambda)$ is a finite subset of $\Omega\cap\{x_3\text{-axis}\}$. Moreover, 
\begin{itemize}
\item[(i)] if $Q_{\rm b}\in C^{2,\alpha}(\partial\Omega)$, then $Q_\lambda\in C^{2,\alpha}(\overline\Omega\setminus{\rm sing}(Q_\lambda))$; 
\vskip3pt
\item[(ii)] if $\partial\Omega$ is of class $C^{k,\alpha}$ and $Q_{\rm b}\in C^{k,\alpha}(\partial\Omega)$ with  $k\geq 3$, then $Q_\lambda\in C^{k,\alpha}(\overline\Omega\setminus{\rm sing}(Q_\lambda))$; 
\vskip3pt
\item[(iii)] if $\partial\Omega$ is analytic and $Q_{\rm b}\in C^\omega(\partial\Omega)$, then $Q_\lambda\in C^{\omega}(\overline\Omega\setminus{\rm sing}(Q_\lambda))$. 
\end{itemize}
\end{theorem}

Those two theorems rest on analytical tools that we shall repeatedly use. The first one is a monotonicity formula for the energy on balls, and we have to distinguish between balls inside the domain and balls centered at the boundary. Our statement below about the interior monotonicity formula is slightly different from the one in \cite[Proposition 6.6]{DMP2} (in the sense that we do not impose here a smooth boundary data), but a quick inspection of the proof (which is based on \cite[Proposition 2.4]{DMP1}) reveals that smoothness at the boundary is only used to establish the boundary monotonicity formula. Concerning the boundary case, the formula in  \cite[Proposition 6.6]{DMP2}  involve constants depending only on $\lambda$, the domain $\Omega$, and the boundary data. We provide in Proposition~\ref{bdrmonoform} below a statement with a control on those constant which is transparent from the proof of  \cite[Proposition 6.6]{DMP2}. 

\begin{proposition}[\bf \cite{DMP1,DMP2}, interior monotonicity formula]\label{intmonoform}
Let $\Omega\subset\R^3$ be a bounded and axisymmetric open set with Lipschitz boundary,  and $Q_\lambda \in W^{1,2}_{\rm sym}(\Omega;\mathbb{S}^4)$ minimizing $\mathcal{E}_\lambda$ among all $Q \in W^{1,2}_{\rm sym}(\Omega;\mathbb{S}^4)$ satisfying $Q=Q_\lambda$ on $\partial \Omega$. Then, 
\begin{multline*}
\frac{1}{\rho}\mathcal{E}_\lambda\big(Q_\lambda,B_\rho(\bar x)\big) - \frac{1}{\sigma}\mathcal{E}_\lambda\big(Q_\lambda,B_\sigma(\bar x)\big) \\
= \int_{B_\rho(\bar x)\setminus B_\sigma(\bar x)}\frac{1}{|x-\bar x|}\bigg|\frac{\partial Q_\lambda}{\partial|x-\bar x|}\bigg|^2\,dx +2\lambda\int_\sigma^\rho\bigg(\frac{1}{t^2}\int_{B_t(\bar x)}W(Q_\lambda)\,dx\bigg)\,dt
\end{multline*}
for every $\bar x\in \Omega$ and $0<\sigma<\rho<{\rm dist}(\bar x,\partial\Omega)$. 
\end{proposition}

\begin{proposition}[\bf \cite{DMP1,DMP2}, boundary monotonicity formula]\label{bdrmonoform}
Let $\Lambda,L>0$ and $\Omega\subset\R^3$  a bounded and axisymmetric open set with boundary of class $C^3$. Let $Q_{\rm b}\in C^{1,1}(\partial\Omega;\mathbb{S}^4)$ be an $\mathbb{S}^1$-equivariant map satisfying 
$\|Q_{\rm b}\|_{C^1(\partial\Omega)}\leq L$. If $\lambda\in[0,\Lambda]$ and  $Q_\lambda$ is a minimizer of $\mathcal{E}_\lambda$ over $\mathcal{A}^{\rm sym}_{Q_{\rm b}}(\Omega)$, then
\begin{multline*}
\frac{1}{\rho}\mathcal{E}_\lambda\big(Q_\lambda,B_\rho(\bar x)\cap \Omega\big) - \frac{1}{\sigma}\mathcal{E}_\lambda\big(Q_\lambda,B_\sigma(\bar x)\cap\Omega\big)\geq -K_*(\rho-\sigma) \\
+ \int_{\big(B_\rho(\bar x)\setminus B_\sigma(\bar x)\big)\cap\Omega}\frac{1}{|x-\bar x|}\bigg|\frac{\partial Q_\lambda}{\partial|x-\bar x|}\bigg|^2\,dx +2\lambda\int_\sigma^\rho\bigg(\frac{1}{t^2}\int_{B_t(\bar x)\cap\Omega}W(Q_\lambda)\,dx\bigg)\,dt
\end{multline*}
for every $\bar x\in\partial\Omega$ and every $0<\sigma<\rho<\mathbf{r}_*$, where the radius $\mathbf{r}_*>0$ only depends on $\Omega$, and $K_*>0$ is a constant depending only on  $\Lambda$, $L$, and $\Omega$. 
\end{proposition}

The second main ingredient we need to emphasize is an {\sl epsilon-regularity result}, consequence of  a more general regularity theorem in \cite[Theorem 2.12 \& Proposition 2.18]{DMP1}. Here again, we have to distinguish the interior and the boundary case, and our statements  below provide a better control on the involved constants inherited from  their proofs (see \cite[Section 6]{DMP2} and \cite[Section~2]{DMP1}).

\begin{proposition}[\bf \cite{DMP1,DMP2}, interior $\boldsymbol{\varepsilon}$-regularity]\label{intepsregprop}
Let $\Lambda>0$ and $\Omega\subset\R^3$ be a bounded and axisymmetric open set with Lipschitz boundary.  Let $\lambda\in[0,\Lambda]$ and $Q_\lambda \in W^{1,2}_{\rm sym}(\Omega;\mathbb{S}^4)$ minimizing $\mathcal{E}_\lambda$ among all $Q \in W^{1,2}_{\rm sym}(\Omega;\mathbb{S}^4)$ satisfying $Q=Q_\lambda$ on $\partial \Omega$. There exist a universal constant $\boldsymbol{\varepsilon}_{\rm in}>0$ such that for every ball $B_r(x_0)\subset \Omega$ with $r$ small enough (depending only on $\Lambda$), the condition 
$$\frac{1}{r}\int_{B_r(x_0)}|\nabla Q_\lambda|^2\,dx\leq \frac{\boldsymbol{\varepsilon}_{\rm in}}{4} $$
implies $Q_\lambda\in C^\omega(B_{r/8}(x_0))$, and $\|\nabla^k Q_\lambda\|_{L^\infty(B_{r/16}(x_0))}\leq C_k r^{-k}$ for each $k\in\mathbb{N}$ and a constant $C_k$ depending only on $k$. 
\end{proposition}

\begin{proof}[Sketch of the proof]
We consider the universal constant  $\boldsymbol{\varepsilon}_{\rm in}>0$ provided by \cite[Corollary 2.19]{DMP1}. By Proposition~\ref{intmonoform}, $Q_\lambda$ satisfies the interior monotonicity formula which allows us to argue as in \cite[proof of Lemma~2.6]{DMP1} and obtain
\begin{multline*}
\sup_{B_\rho(\bar x)\subset B_{r/2}(x_0)} \frac{1}{\rho}\int_{B_\rho(\bar x)} \frac{1}{2} |\nabla Q_\lambda|^2\,dx\leq \sup_{B_\rho(\bar x)\subset B_{r/2}(x_0)} \frac{1}{\rho} \mathcal{E}_\lambda\big(Q_\lambda,B_\rho(\bar x)\big)\\
\leq \frac{2}{r}\mathcal{E}_\lambda\big(Q_\lambda,B_r(x_0)\big)\leq \frac{1}{r} \int_{B_r(x_0)}|\nabla Q_\lambda|^2\,dx+ \frac{\boldsymbol{\varepsilon}_{\rm in}}{4}\leq \frac{\boldsymbol{\varepsilon}_{\rm in}}{2}
\end{multline*}
for $r$ small enough (depending only on $\Lambda$). By Proposition \ref{prop:symmetric-criticality}, $Q_\lambda$ is a weak solution of \eqref{MasterEq} in~$\Omega$. Hence \cite[Corollary 2.19]{DMP1} applies, and we conclude that 
$Q_\lambda\in C^\omega(B_{r/8}(x_0))$ with the announced estimates. 
\end{proof}

Compared to \cite{DMP1,DMP2}, we provide below a localized version (in terms of the data) of the boundary epsilon-regularity. This statement will be of first importance when varying the domain $\Omega$. The arguments remain essentially the same so that we only sketch the main changes. The first version we state here holds under uniform smallness of the scaled Dirichlet integral.

\begin{proposition}[\bf \cite{DMP1,DMP2}, boundary $\boldsymbol{\varepsilon}$-regularity 1]\label{bdrepsregprop}
Let $\Lambda,L>0$ and $\Omega\subset\R^3$  a bounded and axisymmetric open set with boundary of class $C^3$. Let $Q_{\rm b}\in C^{1,1}(\partial\Omega;\mathbb{S}^4)$ be an $\mathbb{S}^1$-equivariant map. Let $\lambda\in[0,\Lambda]$ and $Q_\lambda$ be a minimizer of $\mathcal{E}_\lambda$ over $\mathcal{A}^{\rm sym}_{Q_{\rm b}}(\Omega)$. Let $x_*\in\partial\Omega$ and $r_*>0$ be such that $\|Q_{\rm b}\|_{C^{1,1}(\partial\Omega\cap B_{r_*}(x_*))}\leq L$.  There exist $\bar{\boldsymbol{\varepsilon}}_{\rm bd}>0$ and $\bar{\boldsymbol{\kappa}}\in(0,1)$ depending only on $\partial\Omega\cap B_{r_*}(x_*)$ such that 
for every $x_0\in\partial\Omega\cap B_{r_*/4}(x_*)$ and every radius $r\in(0,r_*/4)$ small enough (depending only on  $\partial \Omega\cap B_{r_*}(x_*)$, $\Lambda$, and $L$), the condition 
$$\sup_{B_\rho(\bar x)\subset B_{r}(x_0)}  \frac{1}{\rho} \int_{B_\rho(\bar x)\cap\Omega}|\nabla Q_\lambda|^2\,dx\leq \bar{\boldsymbol{\varepsilon}}_{\rm bd}$$
implies $Q_\lambda\in C^{\omega}(B_{\bar{\boldsymbol{\kappa}}r}(x_0)\cap\Omega)\cap  C^{1,\alpha}(B_{\bar{\boldsymbol{\kappa}}r}(x_0)\cap\overline\Omega)$ for every $\alpha\in(0,1)$ with the estimate 
$\|\nabla Q_\lambda\|_{L^\infty(B_{\bar{\boldsymbol{\kappa}}r}(x_0)\cap\Omega)}\leq C r^{-1}$ and a constant $C>0$ depending only on $\partial\Omega\cap B_{r_*}(x_*)$ and $L$.  In addition, if $\partial\Omega\cap B_{r_*}(x_*)$ is of class 
$C^{k,\alpha}$ (of class $C^3$ for $k=2$) and $Q_{\rm b}\in C^{k,\alpha}(\partial\Omega\cap B_{r_*}(x_*))$ with $k\geq 2$, then $Q_\lambda\in C^{k,\alpha}(B_{\bar{\boldsymbol{\kappa}}r/2}(x_0)\cap\overline\Omega)$ and 
$\|Q_\lambda\|_{C^{k,\alpha}(B_{\bar{\boldsymbol{\kappa}}r/2}(x_0)\cap\overline\Omega)}\leq C_{k,\alpha,r}$ for a constant $C_{k,\alpha, r}>0$ depending only on $r$,  $\partial\Omega\cap B_{r_*}(x_*)$,  and 
$\|Q_{\rm b}\|_{ C^{k,\alpha}(\partial\Omega\cap B_{r_*}(x_*))}$. 
\end{proposition}

\begin{proof}[Sketch of the proof]
Since $\partial\Omega$ is of class $C^3$, we can find $\delta>0$ such that the nearest point projection $\boldsymbol{\pi}_\Omega$  on $\partial\Omega$ is well defined and of class $C^2$ in the $2\delta$-tubular neighborhood of $\partial\Omega\cap B_{r_*/2}(x_*)$. Then we argue as in \cite[Section 2.2]{DMP1}, and we consider the  reflection of $Q_\lambda$ across $\partial\Omega$  given by \cite[(2.22)]{DMP1} and denoted by $\widehat Q_\lambda$. Then we choose  $r\in(0,\delta/2)$ small enough in such a way that $\boldsymbol{\pi}_\Omega(B_r(y))\subset \partial\Omega\cap B_{2r}(y)$ and  $\boldsymbol{\sigma}_\Omega(B_r(y))\subset B_{2r}(y)$
for every $y\in \partial\Omega\cap B_{r_*/2}(x_*)$, where $\boldsymbol{\sigma}_\Omega:=2\boldsymbol{\pi}_\Omega-{\rm id}$ is the geodesic reflection across $\partial\Omega$. 

Arguing as in the proof of \cite[Lemma 2.10]{DMP1}, there exists a constant $\boldsymbol{\kappa}\in(0,1)$ depending only on $\partial\Omega\cap B_{r_*}(x_*)$ such that 
$$\sup_{B_\sigma(z)\subset B_{\boldsymbol{\kappa}r}(x_0)}  \frac{1}{\sigma} \int_{B_\sigma(z)}|\nabla \widehat Q_\lambda|^2\,dx\leq \sup_{B_\rho(\bar x)\subset B_{r}(x_0)}  \frac{C_1}{\rho} \int_{B_\rho(\bar x)\cap\Omega}|\nabla Q_\lambda|^2\,dx + C_2 r \leq C_1 \bar{\boldsymbol{\varepsilon}}_{\rm bd} +C_2 r\,,$$
for a constant $C_1>0$ depending only on $\partial\Omega\cap B_{r_*}(x_*)$, and a constant $C_2>0$ depending only on  $\partial\Omega\cap B_{r_*}(x_*)$ and $L$. Then we choose $ \bar{\boldsymbol{\varepsilon}}_{\rm bd}$ and $r$ in such a way that $C_1 \bar{\boldsymbol{\varepsilon}}_{\rm bd} +C_2 r\leq \boldsymbol{\varepsilon}_{\rm bd}$, where $\boldsymbol{\varepsilon}_{\rm bd}>0$ is the constant provided by \cite[Corollary 2.17]{DMP1} (note that $\boldsymbol{\varepsilon}_{\rm bd}$ only depends on $\partial\Omega\cap B_{r_*}(x_*)$). By Proposition \ref{prop:symmetric-criticality}, $Q_\lambda$ is a weak solution of \eqref{MasterEq} in~$\Omega$. By our choice of $\boldsymbol{\varepsilon}_{\rm bd}>0$, the proofs of \cite[Corollary 2.17 and Corollary 2.20]{DMP1} apply and lead to the main conclusions with $\bar{\boldsymbol{\kappa}}:=\boldsymbol{\kappa}/4$. Once the gradient estimate is obtained, higher order estimates follow from standard elliptic theory (see e.g. \cite{GilbTrud}). 
\end{proof}

Combining Proposition \ref{bdrepsregprop} with the boundary monotonicity formula in Proposition \ref{bdrmonoform}, we recover the following (more usual) epsilon-regularity at the boundary, which holds under smallness of the Dirichlet integral in a neighborhood of a single point.

\begin{corollary}[\bf \cite{DMP1,DMP2}, boundary $\boldsymbol{\varepsilon}$-regularity 2]\label{bdrepsregcor}
Let $\Lambda,L>0$ and $\Omega\subset\R^3$  a bounded and axisymmetric open set with boundary of class $C^3$. Let $Q_{\rm b}\in C^{1,1}(\partial\Omega;\mathbb{S}^4)$ be an $\mathbb{S}^1$-equivariant map such that $\|Q_{\rm b}\|_{C^{1,1}(\partial\Omega)}\leq L$. Let $\lambda\in[0,\Lambda]$ and $Q_\lambda$ a minimizer of $\mathcal{E}_\lambda$ over $\mathcal{A}^{\rm sym}_{Q_{\rm b}}(\Omega)$.  
There exist $\bar{\boldsymbol{\varepsilon}}^\prime_{\rm bd}>0$ and $\bar{\boldsymbol{\kappa}}^\prime\in(0,1)$ depending only on $\Omega$ such that 
for every $x_0\in\partial\Omega$ and every radius $r>0$ small enough (depending only on  $\Omega$, $\Lambda$, and $L$), the condition 
$$ \frac{1}{r} \int_{B_r(x_0)\cap\Omega}|\nabla Q_\lambda|^2\,dx\leq \bar{\boldsymbol{\varepsilon}}^\prime_{\rm bd}$$
implies $Q_\lambda\in C^{\omega}(B_{\bar{\boldsymbol{\kappa}}^\prime r}(x_0)\cap\Omega)\cap  C^{1,\alpha}(B_{\bar{\boldsymbol{\kappa}}^\prime r}(x_0)\cap\overline\Omega)$ for every $\alpha\in(0,1)$ with the estimate 
$\|\nabla Q_\lambda\|_{L^\infty(B_{\bar{\boldsymbol{\kappa}}^\prime r}(x_0)\cap\Omega)}\leq C r^{-1}$ and a constant $C>0$ depending only on $\Omega$ and $L$.  In addition, if $\partial\Omega$ is of class 
$C^{k,\alpha}$ (of class $C^3$ for $k=2$) and $Q_{\rm b}\in C^{k,\alpha}(\partial\Omega)$ with $k\geq 2$, then $Q_\lambda\in C^{k,\alpha}(B_{\bar{\boldsymbol{\kappa}}^\prime r/2}(x_0)\cap\overline\Omega)$ and 
$\|Q_\lambda\|_{C^{k,\alpha}(B_{\bar{\boldsymbol{\kappa}}^\prime r/2}(x_0)\cap\overline\Omega)}\leq C_{k,\alpha}$ for a constant $C_{k,\alpha}>0$ depending only on  $\Omega$,  and 
$\|Q_{\rm b}\|_{ C^{k,\alpha}(\partial\Omega)}$. 
\end{corollary}

\begin{proof}
Using the boundary monotonicity formula in Proposition \ref{bdrmonoform}, we can argue as \cite[Proof of Lemma 2.6, Step 2]{DMP1} to show that 
$$ \sup_{B_\rho(\bar x)\subset B_{r/6}(x_0)}\frac{1}{\rho}\mathcal{E}_\lambda(Q_\lambda,B_\rho(\bar x)\cap\Omega)\leq \frac{4}{r}\mathcal{E}_\lambda(Q_\lambda,B_r(x_0)\cap\Omega)+C_1 r\,,$$
for a constant $C_1>0$ depending only on  $\Lambda$, $L$, and $\Omega$. Hence, 
$$ \sup_{B_\rho(\bar x)\subset B_{r/6}(x_0)}  \frac{1}{\rho} \int_{B_\rho(\bar x)\cap\Omega}|\nabla Q_\lambda|^2\,dx\leq \frac{4}{r} \int_{B_r(x_0)\cap\Omega}|\nabla Q_\lambda|^2\,dx +C_2 r\,,$$
for $r>0$ small and a further constant $C_2>0$ depending only on  $\Lambda$, $L$, and $\Omega$.  

Next we set $r_*:=4$, and we consider a finite covering of $\partial\Omega$ by balls $B_1(x_*^k)$, $k=1,\ldots, K$.  We denoted by  $\bar{\boldsymbol{\varepsilon}}^k_{\rm bd}$ and $\bar{\boldsymbol{\kappa}}_k$ the constants  provided by Proposition \ref{bdrepsregprop}  with $x_*=x_*^k$. Choosing  $\bar{\boldsymbol{\varepsilon}}^\prime_{\rm bd}:= \frac{1}{8}\min_k\bar{\boldsymbol{\varepsilon}}^k_{\rm bd}$,  $\bar{\boldsymbol{\kappa}}^\prime:=\frac{1}{6}\min_k\bar{\boldsymbol{\kappa}}_k$, and then $r>0$ small enough such that $C_2r\leq \bar{\boldsymbol{\varepsilon}}^\prime_{\rm bd}$ 
(depending only on $\Omega$, $\Lambda$, and~$L$),  we obtain that $Q_\lambda$ satisfies 
$$\sup_{B_\rho(\bar x)\subset B_{r/6}(x_0)}  \frac{1}{\rho} \int_{B_\rho(\bar x)\cap\Omega}|\nabla Q_\lambda|^2\,dx\leq \bar{\boldsymbol{\varepsilon}}^k_{\rm bd}\,,$$
for an index $k$ such that $x_0\in \partial\Omega\cap B_1(x_*^k)$. Then the conclusion follows from Proposition \ref{bdrepsregprop}.  
\end{proof}

\begin{remark}[\bf Locally flat geometry]\label{specgeomremreg}
In the following sections, we shall  consider the situation where, for some $x_*\in\{x_3\text{-axis}\}$ and $r_*>0$, $\Omega\cap B_{2r_*}(x_*)=x_*+\{\pm x_3>0\}\cap B_{2r_*}(0)$ and $Q_{\rm b}={\bf e}_0$ on   
$\partial \Omega\cap B_{2r_*}(x_*)=x_*+\{x_3=0\}\cap B_{2r_*}(0)$. 
According to \cite[Remark 2.5]{DMP1} (see the proof of \cite[Proposition 6.6]{DMP2}), if $Q_\lambda$ is as in Proposition~\ref{bdrepsregprop}, then
\begin{multline}\label{specifbdmonotform}
\frac{1}{\rho}\mathcal{E}_\lambda\big(Q_\lambda,B_\rho(\bar x)\cap\Omega\big) - \frac{1}{\sigma}\mathcal{E}_\lambda\big(Q_\lambda,B_\sigma(\bar x)\cap\Omega\big) \\
= \int_{\big(B_\rho(\bar x)\setminus B_\sigma(\bar x)\big)\cap\Omega}\frac{1}{|x-\bar x|}\bigg|\frac{\partial Q_\lambda}{\partial|x-\bar x|}\bigg|^2\,dx +2\lambda\int_\sigma^\rho\bigg(\frac{1}{t^2}\int_{B_t(\bar x)\cap\Omega}W(Q_\lambda)\,dx\bigg)\,dt
\end{multline}
for every $\bar x\in \partial \Omega\cap B_{2r_*}(x_*)$ and $0<\rho<\sigma<2r_*-|x_*-\bar x|$. 
As in \cite[Remark 2.7]{DMP1}, it implies that
$$ \sup_{B_\rho(\bar x)\subset B_{r/6}(x_0)}\frac{1}{\rho}\mathcal{E}_\lambda(Q_\lambda,B_\rho(\bar x)\cap\Omega)\leq \frac{4}{r}\mathcal{E}_\lambda(Q_\lambda,B_r(x_0)\cap\Omega)$$ 
for every $x_0\in \partial \Omega\cap B_{r_*}(x_*)$ and $0<r<r_*$. Repeating the proof of Corollary~\ref{bdrepsregcor}, we can apply Proposition \ref{bdrepsregprop} to obtain the existence of {\sl universal constants} 
$\boldsymbol{\varepsilon}^\sharp_{\rm bd}>0$ and  $\boldsymbol{\kappa}^\sharp\in(0,1)$, such that for every $x_0\in \partial \Omega\cap B_{r_*/4}(x_*)$ and $r\in(0,r_*/4)$ small enough (depending only on $\Lambda$), the 
condition 
$$\frac{1}{r}\int_{B_r(x_0)\cap \Omega}|\nabla Q_\lambda|^2\,dx\leq \boldsymbol{\varepsilon}^\sharp_{\rm bd}$$
implies the same conclusions as in Proposition \ref{bdrepsregprop} in $B_{\boldsymbol{\kappa}^\sharp r}(x_0)\cap \Omega$. 
\end{remark}


\subsection{Persistence of smoothness}

We now apply the regularity theory of the previous subsection to show that absence of singularities in energy minimizing configurations (within the equivariant class) is  a strongly $W^{1,2}$-open/closed  property.

\begin{lemma}\label{persissmoothloc}
Let $(Q_*,\lambda_*)$ and a sequence $\{(Q_n,\lambda_n)\}$ in $W^{1,2}_{\rm sym}(B_r;\mathbb{S}^4)\times[0,\infty)$ be such that $Q_n\to Q_*$ strongly in $W^{1,2}(B_r)$ and $\lambda_n\to\lambda_*$ as $n\to\infty$. Assume that  
$Q_*$ is  minimizing $\mathcal{E}_{\lambda_*}$ among all $Q \in W^{1,2}_{\rm sym}(B_r;\mathbb{S}^4)$ satisfying $Q=Q_*$ on $\partial B_r$, and that $Q_n$ is  minimizing $\mathcal{E}_{\lambda_n}$ among all $Q \in W^{1,2}_{\rm sym}(B_r;\mathbb{S}^4)$ satisfying $Q=Q_n$ on $\partial B_r$. 
\begin{itemize}
\item[(i)] If ${\rm sing}(Q_*)\cap B_r=\emptyset$, then for every $0<\rho<r$, there exists an integer $n_\rho$ such that ${\rm sing}(Q_n) \cap B_\rho=\emptyset$ whenever $n\geq n_\rho$.
\vskip5pt   
\item[(ii)] If ${\rm sing}(Q_n)\cap B_r=\emptyset$ for every $n$, then ${\rm sing}(Q_*)\cap B_r=\emptyset$.  
\end{itemize}
\end{lemma}		

\begin{proof}
 We start proving claim (i). Fix a radius $0<\rho<r$, and assume by contradiction that there exists a (not relabeled) subsequence such that ${\rm sing}(Q_n) \cap B_\rho\not=\emptyset$ for every $n$. Then we choose for each $n$ a point $x_n\in {\rm sing}(Q_n) \cap B_\rho$. Extracting a further subsequence if necessary, we may assume that $x_n\to x_*\in \overline B_\rho$. On the other hand, since $Q_*$ is smooth in $B_r$, we can find a small enough radius $0<\sigma<r-\rho$ such that  
$$\frac{1}{\sigma}\int_{B_\sigma(x_*)} |\nabla Q_*|^2\,dx\leq \frac{\boldsymbol{\varepsilon}_{\rm in}}{8}\,,$$
where the universal constant $\boldsymbol{\varepsilon}_{\rm in}>0$ is given by Proposition \ref{intepsregprop}. From the strong convergence of $Q_n$ toward $Q_*$, we deduce that 
$$\frac{1}{\sigma}\int_{B_\sigma(x_*)} |\nabla Q_n|^2\,dx\leq \frac{\boldsymbol{\varepsilon}_{\rm in}}{4}$$
for $n$ large enough. By Proposition \ref{intepsregprop}, it implies that $Q_n$ is smooth in $B_{\sigma/8}(x_*)$. Since $x_n\to x_*$, we have $x_n\in B_{\sigma/8}(x_*)$ for $n$ large enough, contradicting the fact that $Q_n$ is singular at $x_n$. 
\vskip3pt

We now prove claim (ii). To this purpose, it is enough to show that ${\rm sing}(Q_*)\cap B_\rho=\emptyset$ for every $0<\rho<r$. Hence we fix an arbitrary radius $0<\rho<r$, and we assume by contradiction that ${\rm sing}(Q_*)\cap B_\rho\not=\emptyset$. By Theorem \ref{intregthm}, ${\rm sing}(Q_*)\cap \overline B_\rho\subset \{x_3\mbox{-axis}\}$ is finite, and setting $I_\rho:=B_\rho\cap \{x_3\mbox{-axis}\}$, the trace of $Q_*$ on $I_\rho$ (see Corollary \ref{verticaltrace}) is a non trivial piecewise constant function with values in $\{\pm {\bf e}_0\}$ (since we are assuming that ${\rm sing}(Q_*)\cap B_\rho\not=\emptyset$). On the other hand,  $Q_n$ is smooth in $\overline B_\rho$, so that either $Q_n\equiv {\bf e}_0$ or 
$Q_n\equiv -{\bf e}_0$ on $I_\rho$. Extracting a subsequence if necessary, we may assume for instance that ${Q_n}_{|I_\rho}\equiv {\bf e}_0$ for every $n$. By the strong $W^{1,2}$-convergence of $Q_n$ and the continuity of the trace operator established in Corollary \ref{verticaltrace}, we infer that ${Q_n}_{|I_\rho}\to {Q_*}_{|I_\rho}$ in $L^1(I_\rho)$ as $n\to \infty$. Hence ${Q_*}_{|I_\rho}\equiv {\bf e}_0$ contradicting its non triviality. 
\end{proof}

\begin{corollary}\label{corolpersissmooth}
Let $\Omega\subset \R^3$ be a bounded and axisymmetric open set with boundary of class $C^3$. Let $(Q^*_{\rm b},\lambda_*)$ and a sequence $\{(Q^{(n)}_{\rm b},\lambda_n)\}$ in $C^2_{\rm sym}(\partial\Omega;\mathbb{S}^4)\times[0,\infty)$ be such that 
$Q^{(n)}_{\rm b}\to Q^*_{\rm b}$ in $C^2(\partial\Omega)$, and $\lambda_n\to\lambda_*$ as $n\to\infty$. For each $n\in\mathbb{N}$, let $Q_n$ be a minimizer of $\mathcal{E}_{\lambda_n}$ over $\mathcal{A}^{\rm sym}_{Q^{(n)}_{\rm b}}(\Omega)$, $Q_*$   a minimizer of $\mathcal{E}_{\lambda_*}$ over $\mathcal{A}^{\rm sym}_{Q^*_{\rm b}}(\Omega)$, and assume that $Q_n\to Q_*$ strongly in $W^{1,2}(\Omega)$. 
\begin{itemize}
\item[(i)] If ${\rm sing}(Q_*)=\emptyset$, then  there exists an integer $n_*$ such that ${\rm sing}(Q_n)=\emptyset$ whenever $n\geq n_*$.   
\vskip3pt
\item[(ii)] If ${\rm sing}(Q_n)=\emptyset$ for every $n$, then ${\rm sing}(Q_*)=\emptyset$.  
\end{itemize}
\end{corollary}

\begin{proof}
 We start proving claim (i). To prove it, it is enough to show that there exists $\delta>0$ independent of $n$ such that the $C^1$-norms of $Q_n$ are uniformly bounded in a $\delta$-neighborhood of $\partial\Omega$ (recall that ${\rm sing}(Q_n)$ coincides with the discontinuity points of $Q_n$). Indeed, in this case we have ${\rm sing}(Q_n)\subset \Omega\cap\{{\rm dist}(\cdot,\partial\Omega)\geq\delta\}$ for every $n$. Recalling that ${\rm sing}(Q_n)\subset \{x_3\text{-axis}\}$ by Theorem~\ref{intregthm}, we choose a finite covering of $\Omega\cap\{{\rm dist}(\cdot,\partial\Omega)\geq\delta\}\cap \{x_3\text{-axis}\}$ by open balls $B_{\delta/2}(x_1),\ldots,B_{\delta/2}(x_K)$. We apply Lemma \ref{persissmoothloc} in each $B_{\delta}(x_j)$ to find an integer $n_*$ such that ${\rm sing}(Q_n)\cap B_{\delta/2}(x_j)=\emptyset$ for each $j$ and every $n\geq n_*$. Hence ${\rm sing}(Q_n)=\emptyset$ for every $n\geq n_*$. 

To show that the $C^1$-norm of $Q_n$ remains bounded in a $\delta$-neighborhood of $\partial\Omega$, we shall make use of the regularity estimates from Section \ref{secregth}. By Theorem \ref{bdrregthm}, $Q_*$ is  of class $C^{1,\alpha}$ for every $\alpha\in(0,1)$ in a neighborhood of $\partial\Omega$. Hence, for a radius $\eta>0$ to be chosen small enough, we have 
$$\frac{1}{\eta} \int_{B_{\eta}(y)\cap \Omega}|\nabla Q_*|^2\,dx \leq \frac{\bar{\boldsymbol{\varepsilon}}^\prime_{\rm bd}}{2}\quad\text{for every $y\in\partial\Omega$}\,,$$
where the constant $\bar{\boldsymbol{\varepsilon}}^\prime_{\rm bd}>0$ (depending only on $\Omega$) is provided by Corollary \ref{bdrepsregcor}. Next we set $\Lambda:=\sup_n\lambda_n<\infty$, and 
$L:=\sup_n\|Q^{(n)}_{\rm b}\|_{C^{1,1}(\partial\Omega)}<\infty$. We now choose $\eta>0$ small enough (depending only on $\Lambda$, $L$, and $\Omega$) such that the conclusion of Corollary \ref{bdrepsregcor} holds. We also set $r_*:=\bar{\boldsymbol{\kappa}}^\prime\eta$ with constant $\bar{\boldsymbol{\kappa}}^\prime\in(0,1)$ still given by Corollary \ref{bdrepsregcor} (depending only on $\Omega$),  
and we consider a finite covering $B_{r_*}(y_1),\ldots, B_{r_*}(y_J)$ of $\partial \Omega$ with $y_j\in\partial\Omega$. Since $Q_n\to Q_*$ strongly in $W^{1,2}(\Omega)$, we can find an integer $n_*$ such that 
$$\frac{1}{\eta} \int_{ B_{\eta}(y_j)\cap \Omega}|\nabla Q_n|^2\,dx\leq \bar{\boldsymbol{\varepsilon}}^\prime_{\rm bd}\quad\text{for each $j=1,\ldots,J$ and every $n\geq n_*$}\,.$$ 
Applying Corollary \ref{bdrepsregcor}, we infer that $Q_n\in C^{1,\alpha}(B_{r_*}(y_j)\cap\Omega)$ for every $\alpha\in(0,1)$ and each $j$ with the estimate $\|\nabla Q_n\|_{L^\infty(B_{r_*}(y_j)\cap\Omega)}\leq C r_*^{-1}$ and a constant $C$ independent of $n$. Since the balls $B_{r_*}(y_1),\ldots, B_{r_*}(y_J)$ cover $\partial\Omega$, the $C^1$-norm of $Q_n$ remains bounded in a $\delta$-neighborhood of $\partial\Omega$ for some $\delta\in(0,r_*)$. 
\vskip3pt

We now prove claim (ii). Assume by contradiction that ${\rm sing}(Q_*)\not=\emptyset$, i.e., $Q_*$ has at least one  singular point $x_*\in \Omega${{,}} which must belong to the $\{x_3\text{-axis}\}$ by Theorem \ref{intregthm}. Choose a radius $r>0$ such that $B_r(x_*)\subset \Omega$. 
Since $Q_n\to Q_*$ strongly in $W^{1,2}(\Omega)$, we can apply Lemma \ref{persissmoothloc} in  the ball $B_r(x_*)$ to infer that $Q_*$ is smooth in $B_r(x_*)$, a contradiction. 
\end{proof}


\subsection{Persistence of singularities}

By analogy with the previous subsection, we now study the behavior of the singular set along strongly $W^{1,2}$-convergent sequences of minimizers, proving that singular points converge to singular points. The following result is the counterpart in the present context of \cite[Theorem~1.8]{AlLi} (see also \cite{HL2}).
		
\begin{proposition}\label{singgotosing}
Let $(Q_*,\lambda_*)$ and a sequence $\{(Q_n,\lambda_n)\}$ in $W^{1,2}_{\rm sym}(B_r;\mathbb{S}^4)\times[0,\infty)$ be such that $Q_n\to Q_*$ strongly in $W^{1,2}(B_r)$ and $\lambda_n\to\lambda_*$ as $n\to\infty$. 
Assume that  
$Q_*$ is  minimizing $\mathcal{E}_{\lambda_*}$ among all $Q \in W^{1,2}_{\rm sym}(B_r;\mathbb{S}^4)$ satisfying $Q=Q_*$ on $\partial B_r$, and that $Q_n$ is  minimizing $\mathcal{E}_{\lambda_n}$ among all $Q \in W^{1,2}_{\rm sym}(B_r;\mathbb{S}^4)$ satisfying $Q=Q_n$ on $\partial B_r$. Then, for every radius $\rho\in(0,r)$ such that ${\rm sing}(Q_*)\cap\partial B_\rho=\emptyset$ and ${\rm sing}(Q_*)\cap B_\rho=\{a_1^*,\ldots, a_K^*\}$, there exists an integer $n_\rho$ such that for every $n\geq n_\rho$, 
${\rm sing}(Q_n)\cap\partial B_\rho=\emptyset$ and ${\rm sing}(Q_n)\cap B_\rho=\{a_1^n,\ldots, a_K^n\}$ for some distinct points $a_1^n, \ldots,a_K^n\in B_\rho$ satisfying $|a_j^n-a_j^*|\to 0$ as $n\to \infty$ for $j = 1, \dots, K$. 
\end{proposition}		

\begin{proof}
By Theorem~\ref{intregthm}, ${\rm sing}(Q_*)$ and ${\rm sing}(Q_n)$ are made of locally finitely many points in $B_r\cap\{x_3\text{-axis}\}$. 
If ${\rm sing}(Q_*)\cap \partial B_\rho=\emptyset$, then $Q_*$ is smooth in a neighborhood of $\partial B_\rho$. 
Applying Lemma~\ref{persissmoothloc} at the north and south pole of $\partial B_\rho$, we infer that there exists an integer $\bar n_\rho$ such that $Q_n$ is smooth in a uniform neighborhood of $\partial B_\rho$  for every $n\geq \bar n_\rho$. Then we set $\Sigma^\rho_*:={\rm sing}(Q_*)\cap B_\rho$ and $\Sigma^\rho_n:={\rm sing}(Q_n)\cap B_\rho$. We claim that $\Sigma^\rho_n\to \Sigma^\rho_*$ in the Hausdorff distance.  To prove this claim, let us first consider $a_*\in \Sigma^\rho_*$, and prove that there exists  $a_n\in \Sigma^\rho_n$ such that $a_n\to a_*$. By contradiction, assume that $\Sigma^\rho_n$ remains at a positive distance from $a_*$ for $n$ large. Then we can find $\eta>0$ such that $B_\eta(a_*)\cap  \Sigma^\rho_n=\emptyset$ for $n$ large enough. Applying Lemma \ref{persissmoothloc} in $B_\eta(a_*)$, we deduce that $B_\eta(a_*)\cap  \Sigma^\rho_*=\emptyset$, a contradiction. The other way around, let $a_n\in \Sigma^\rho_n$ be a sequence converging to some point $a_*$. Since $\Sigma^\rho_n$ remains at a positive distance from $\partial B_\rho$, we have $a_*\in B_\rho$, and let us show that $a_*\in \Sigma^\rho_*$. Again by contradiction, assume that  $a_*\not\in \Sigma^\rho_*$. Then we can find $\eta>0$ such that $B_{2\eta}(a_*)\cap \Sigma^\rho_*=\emptyset$. Applying Lemma \ref{persissmoothloc} in $B_{2\eta}(a_*)$, we infer that $B_{\eta}(a_*)\cap \Sigma^\rho_n=\emptyset$ for $n$ large enough, which contradicts the fact that $a_n\to a_*$. Hence $\Sigma^\rho_n\to \Sigma^\rho_*$ in the Hausdorff distance. 
\vskip5pt

To complete the proof of Proposition~\ref{singgotosing}, we shall make use of the following key lemma, giving a lower bound on the mutual distance between singularities for minimizers, in the spirit of \cite[Theorem~2.1]{AlLi} for minimizing harmonic maps into $\bbS^2$.

\begin{lemma}\label{infdistsing}
Let $M,\Lambda>0$ and $\lambda\in[0,\Lambda]$. Assume that $Q_\lambda\in W^{1,2}_{\rm sym}(B_1;\mathbb{S}^4)$ is minimizing $\mathcal{E}_\lambda(\cdot,B_1)$ among all maps $Q\in W^{1,2}_{\rm sym}(B_1;\mathbb{S}^4)$ satisfying $Q=Q_\lambda$ on $\partial B_1$, and  that $\mathcal{E}_\lambda(Q_\lambda,B_1)\leq M$. Then there exists a constant $\kappa=\kappa(M,\Lambda)>0$ depending only on $M$ and $\Lambda$ such that 
$$|a-b|\geq \kappa\quad \text{for every $a,b\in{\rm sing}(Q_\lambda)\cap \overline B_{1/2}$, $a\not=b$}\,.$$
\end{lemma}		
		
\begin{proof}
We argue by contradiction assuming that there exists a sequence $\{Q_n\}$ in $W^{1,2}_{\rm sym}(B_1;\mathbb{S}^4)$ and $\lambda_n\in[0,\Lambda]$ such that $Q_n$ is minimizing $\mathcal{E}_{\lambda_n}(\cdot,B_1)$ among all maps $Q\in W^{1,2}_{\rm sym}(B_1;\mathbb{S}^4)$ satisfying $Q=Q_n$ on $\partial B_1$, and   $\mathcal{E}_{\lambda_n}(Q_{n},B_1)\leq M$, and such that there exists two distinct points $a_n,b_n\in {\rm sing}(Q_\lambda)\cap \overline B_{1/2}$ satisfying $r_n:=|a_n-b_n|\to 0$ as $n\to\infty$. Extracting a subsequence if necessary, we may assume that $\lambda_n\to \lambda_*\in [0,\Lambda]$, $\lim_{n\to\infty} a_n=\lim_{n\to\infty} b_n=  c_*\in\overline B_{1/2}$. By the compactness theorem in \cite[Theorem 5.1]{DMP2}, we can find a (not relabeled) subsequence such that $Q_n\to Q_*$ strongly in $W^{1,2}_{\rm loc}(B_1)$ for a map $Q_*\in W^{1,2}_{\rm sym}(B_1;\mathbb{S}^4)$ which is minimizing $\mathcal{E}_{\lambda_*}(\cdot,B_1)$ among all maps $Q\in W^{1,2}_{\rm sym}(B_1;\mathbb{S}^4)$ satisfying $Q=Q_*$ on $\partial B_1$. Arguing as in the proof of Proposition~\ref{singgotosing}, we infer that $c_*\in {\rm sing}(Q_*)$. Setting $c_n:=(a_n+b_n)/2$, we have $c_n\to c_*$, and we define for $x\in B_2$ and $n$ large enough, 
$$\bar Q_n(x):=Q_n(c_n+r_n x)\,.$$
Since $a_n,b_n\in \{x_3\text{-axis}\}$, we have $c_n\in\{x_3\text{-axis}\}$, and thus $\bar Q_n\in W^{1,2}_{\rm sym}(B_2;\mathbb{S}^4)$. From the minimality of $Q_n$ and a change of variables, we infer that $\bar Q_n$ minimizes $\mathcal{E}_{r_n^2\lambda_n}(\cdot,B_2)$ among all maps $Q\in W^{1,2}_{\rm sym}(B_2;\mathbb{S}^4)$ such that $Q=\bar Q_n$ on $\partial B_2$. Extracting a subsequence if necessary, we may assume that $p_1:=(a_n-c_n)/r_n=(0,0,1/2)$ and  $p_2:=(b_n-c_n)/r_n=(0,0,-1/2)$. Then, by construction, $p_1,p_2\in {\rm sing}(\bar Q_n)$. 

By the interior monotonicity formula in Proposition \ref{intmonoform}, we have for every $x_0\in B_2$, every $t\in(0,{\rm dist}(x_0,\partial B_2)]$ and $r \in (0,1)$, 
\begin{align}
\nonumber\frac{1}{t}\mathcal{E}_{r^2_n\lambda_n}\big(\bar Q_n,B_t(x_0)\big) & = \frac{1}{r_n t}\mathcal{E}_{\lambda_n}\big(Q_n,B_{r_nt}(c_n+r_n x_0))\big)\\
\nonumber&\leq \frac{1}{r-|c_n+r_nx_0-c_*|}\mathcal{E}_{\lambda_n}\big(Q_n,B_{r-|c_n+r_n x_0-c_*|}(c_n+r_n x_0)\big)\\
&\leq \frac{1}{r-|c_n+r_nx_0-c_*|}\mathcal{E}_{\lambda_n}\big(Q_n,B_{r+|c_n+r_n x_0-c_*|}(c_*)\big)\label{distsing1} 
\end{align}
whenever $n$ is large enough. Since $Q_n\to Q_*$ strongly in $W^{1,2}_{\rm loc}(B_1)$ and $\lambda_n\to \lambda_*$, we have 
\begin{equation}\label{distsing2}
\lim_{n\to\infty}\frac{1}{r-|c_n+r_nx_0-c_*|}\mathcal{E}_{\lambda_n}\big(Q_n,B_{r+|c_n+r_n x_0-c_*|}(c_*)\big)=\frac{1}{r}\mathcal{E}_{\lambda_*}\big(Q_*,B_{r}(c_*)\big)\,.
\end{equation}
In view of \eqref{distsing1}-\eqref{distsing2} with $x_0=0$ and $t=2$, we first deduce that $\sup_n\mathcal{E}_{r_n^2\lambda_n}(\bar Q_n,B_2) <\infty$. By the compactness result in \cite[Theorem 5.1]{DMP2}, we can find a (not relabeled) subsequence such that $\bar Q_n\to \bar Q_*$ strongly in $W^{1,2}_{\rm loc}(B_2)$ for a map $\bar Q_*\in W^{1,2}_{\rm sym}(B_2;\mathbb{S}^4)$ which is minimizing $\mathcal{E}_{0}(\cdot,B_2)$ among all maps $Q\in W^{1,2}_{\rm sym}(B_2;\mathbb{S}^4)$ satisfying $Q=\bar Q_*$ on $\partial B_2$. 

Letting $n\to \infty$ in \eqref{distsing1}, we infer from  \eqref{distsing2} that for every $x_0\in B_2$, every $t\in(0,{\rm dist}(x_0,\partial B_2))$, and $r \in (0,1/2)$ small enough, 
\begin{equation}\label{distsing3}
 \frac{1}{t}\mathcal{E}_{0}\big(\bar Q_*,B_t(x_0)\big)\leq \frac{1}{r}\mathcal{E}_{\lambda_*}\big(Q_*,B_{r}(c_*)\big)\,. 
 \end{equation}
On the other hand, since $c_*\in{\rm sing}(Q_*)$, Theorem \ref{intregthm} tells us that 
$$\lim_{r\to 0}\frac{1}{r}\int_{B_r(c_*)} \frac{1}{2} |\nabla Q_*|^2\,dx= 4\pi\,.$$
Letting now $r\to 0$ in \eqref{distsing3} yields 
\begin{equation}\label{distsing4}
 \frac{1}{t}\mathcal{E}_{0}\big(\bar Q_*,B_t(x_0)\big)\leq 4\pi\quad\text{for every $x_0\in B_2$ and $t\in(0,{\rm dist}(x_0,\partial B_2))$}\,.
 \end{equation}
On the other hand, $p_1$ and $p_2$ are singular points of $\bar Q_n$ for each $n$, and thus $p_1,p_2\in {\rm sing}(\bar Q_*)$ by Lemma \ref{persissmoothloc}. As a consequence, Theorem \ref{intregthm} and the interior  monotonicity formula in Proposition~\ref{intmonoform} imply that for $j=1,2$, 
\begin{equation}\label{distsing5}
\frac{1}{t}\mathcal{E}_{0}\big(\bar Q_*,B_t(p_j)\big)\geq 4\pi \quad \forall t\in(0,1)\,.
\end{equation}
Setting $y_t:=(0,0,t-1/2)$ for $t\in(0,1)$, since $B_t(p_1) \cup B_{1-t}(p_2) \subset B_1(y_t)$, we gather \eqref{distsing4} and \eqref{distsing5} to derive 
$$ 4\pi \geq \mathcal{E}_{0}\big(\bar Q_*,B_1(y_t)\big)\geq \mathcal{E}_{0}\big(\bar Q_*,B_t(p_1)\big)+  \mathcal{E}_{0}\big(\bar Q_*,B_{1-t}(p_2)\big)\geq 4\pi t+ 4\pi(1-t)=4\pi \quad \forall t\in(0,1)\,.$$
Therefore $|\nabla \bar Q_*|^2\equiv0$ a.e. in $B_1(y_t)\setminus\big(B_t(p_1)\cup B_{1-t}(p_2)\big)$ for every $t\in(0,1)$. Since
$$B_1\cap \bigcup_{0<t<1}\Big( B_1(y_t)\setminus\big(B_t(p_1)\cup B_{1-t}(p_2)\big)\Big)= B_1\setminus [p_1,p_2] \,,$$
we conclude that $|\nabla \bar Q_*|^2\equiv0$ a.e. in $B_1$. Thus $ \bar Q_*$ is constant in $B_1$, which contradicts the fact that $\bar Q_*$ is singular at $p_1,p_2\in B_1$. 
\end{proof}

\noindent {\it Proof of Proposition  \ref{singgotosing} Completed.} To complete the proof, it remains to show that there exists  an integer $n_\rho\geq\bar n_\rho$ such that ${\rm Card}\,\Sigma^\rho_n={\rm Card}\,\Sigma^\rho_*$ for $n\geq n_\rho$. Once again we argue by contradiction assuming that for some (not relabeled) subsequence, we have ${\rm Card}\,\Sigma^\rho_n \not ={\rm Card}\,\Sigma^\rho_*$. In view of the previous discussion, ${\rm Card}\,\Sigma^\rho_n >{\rm Card}\,\Sigma^\rho_*$ for $n$ large enough. As a consequence, there exist at least two points $a_n,b_n\in \Sigma^\rho_n$ such that $a_n\not=b_n$ and $\lim_n a_n= \lim_n b_n= c_*$ for a point $c_*\in \Sigma^\rho_*$. In particular, $|a_n-b_n|\to 0$. Then we choose a radius $\eta>0$ such that $B_\eta(c_*)\subset B_r$. For $n$ large enough, we  have $a_n,b_n\in B_{\eta/2}(c_*)$. Rescaling variables, we can apply Lemma \ref{infdistsing} in $B_\eta(c_*)$ to deduce that $|a_n-b_n|\geq \kappa\eta $ for some constant $\kappa>0$ depending only on $\sup_n\frac{1}{\eta}\mathcal{E}_{\lambda_n}\big(Q_n,B_\eta(c_*)\big)<\infty$ and $\sup_n \eta^2\lambda_n<\infty$, which contradicts the fact that $|a_n-b_n|\to 0$.
\end{proof}
		
The following result is the global counterpart of Proposition~\ref{singgotosing}.

\begin{corollary}\label{corosinggotosingglobal}		
Let $\Omega\subset \R^3$ be a bounded and axisymmetric open set with boundary of class~$C^3$. Let $(Q^*_{\rm b},\lambda_*)$ and a sequence $\{(Q^{(n)}_{\rm b},\lambda_n)\}$ in $C^2_{\rm sym}(\partial\Omega;\mathbb{S}^4)\times[0,\infty)$ be such that 
$Q^{(n)}_{\rm b}\to Q^*_{\rm b}$ in $C^2(\partial\Omega)$, and $\lambda_n\to\lambda_*$ as $n\to\infty$. For each $n\in\mathbb{N}$, let $Q_n$ be a minimizer of $\mathcal{E}_{\lambda_n}$ over $\mathcal{A}^{\rm sym}_{Q^{(n)}_{\rm b}}(\Omega)$, $Q_*$ a minimizer of $\mathcal{E}_{\lambda_*}$ in $\mathcal{A}^{\rm sym}_{Q^*_{\rm b}}(\Omega)$, and assume that $Q_n\to Q_*$ strongly in $W^{1,2}(\Omega)$. If ${\rm sing}(Q_*)=\{a^*_1,\ldots,a_K^*\}$, then there exists an integer $n_*$ such that for every $n\geq n_*$, 	${\rm sing}(Q_n)=\{a^n_1,\ldots,a_K^n\}$ for some distinct points $a_1^n,\ldots,a_K^n$ satisfying $|a_j^n-a_j^*|\to 0$ as $n\to\infty$.  
\end{corollary}
			
\begin{proof}
By Theorem \ref{bdrregthm}, $Q_*$ is smooth in $\Omega\setminus{\rm sing}(Q_*)$ and ${\rm sing}(Q_*)$ is a finite subset of $\Omega\cap\{x_3\text{-axis}\}$, i.e., ${\rm sing}(Q_*)=\{a^*_1,\ldots,a_K^*\}\subset \Omega\cap\{x_3\text{-axis}\}$. Let us fix  $\delta>0$ such that $B_{3\delta}(a_i^*)\cap B_{3\delta}(a_j^*)=\emptyset$ if $i\not=j$, and ${\rm dist}(a^*_j,\partial\Omega)\geq 3\delta$. We  set $K_\delta:=\{x\in\Omega: {\rm dist}(x,\partial\Omega)\geq\delta\}\setminus\cup_jB_{\delta}(a^*_j)$, and we claim that $Q_n \to Q_*$ in 
$C^2(K_\delta)$. Indeed, by smoothness of $Q_*$ away from  ${\rm sing}(Q_*)$, we can find a radius $r\in(0,\delta/2)$ such that 
$$\frac{1}{r}\int_{B_r(x_0)}|\nabla Q_*|^2\,dx\leq \frac{\boldsymbol{\varepsilon}_{\rm in}}{8} \quad\text{for every $x_0\in K_\delta$}\,,$$
where the universal constant $\boldsymbol{\varepsilon}_{\rm in}>0$ is provided by  Proposition \ref{intepsregprop}. 
Choosing $r$ smaller if necessary (depending only on $\Lambda:=\sup_n\lambda_n<\infty$), we may assume that the conclusion of Proposition \ref{intepsregprop} holds for every $\lambda_n$. 
Then we consider a finite covering $B_{r/16}(y_1),\ldots, B_{r/16}(y_J)$ of $K_\delta$. Since $Q_n\to Q_*$ strongly in $W^{1,2}(\Omega)$, we have for $n$ large enough, 
$$\frac{1}{r}\int_{B_r(y_j)}|\nabla Q_n|^2\,dx\leq \frac{\boldsymbol{\varepsilon}_{\rm in}}{4} \quad\text{for every $j=1,\ldots,J$}\,. $$
By Proposition~\ref{intepsregprop}, for $n$ large enough, $Q_n$ is  smooth in each  $B_{r/16}(y_j)$ and $\|Q_n\|_{C^3(B_{r/16}(y_j))}\leq C_r$ for some constant $C_r>0$. Therefore $Q_n$ remains bounded in $C^3(K_\delta)$ for $n$ large enough. From the $W^{1,2}$-convergence of $Q_n$ towards $Q_*$ and the Arzel\`a-Ascoli Theorem, we deduce that $Q_n\to Q_*$ in $C^2(K_\delta)$. 

Now we set $\Omega_\delta:=\Omega\setminus\cup_jB_{2\delta}(a^*_j)$ which is a bounded and axisymmetric open set with boundary of class~$C^3$. By our discussion above and the assumption on $Q^{(n)}_{\rm b}$, the restriction of $Q_n$ to $\partial\Omega_\delta$ converges in the $C^2$-topology to the restriction of $Q_*$ to $\partial\Omega_\delta$. Applying Corollary \ref{corolpersissmooth} in $\Omega_\delta$, we infer that ${\rm sing}(Q_n)\cap \Omega_\delta=\emptyset$ for $n$ large enough. Then we can apply Proposition \ref{singgotosing} in each ball $B_{3\delta}(a^*_j)$ with $\rho=2\delta$. It shows that for $n$ large enough, ${\rm sing}(Q_n)\cap\overline B_{2\delta}(a^*_j)=\{a^n_j\}$ for some point $a^n_j\to a^*_j$ as $n\to\infty$. 
\end{proof}


	\subsection{Coexistence results in a ball}

In this subsection, we take advantage of the results above to study the space $(\text{boundary condition})\times(\text{$\lambda$-parameter})$. We are interested in the nature of the sets of data leading to smooth or/and singular solutions. To motivate this question, we recall the results in~\cite{DMP2} showing the existence for $\lambda\geq 0$ arbitrary of  boundary conditions $Q^{\rm smooth}_{\rm b}$ and $Q^{\rm sing}_{\rm b}$ in $C^\infty_{\rm sym}(\partial B_1;\bbS^4)$ such that any minimizer of $\mathcal{E}_\lambda$ over $\mathcal{A}^{\rm sym}_{Q^{\rm smooth}_{\rm b}}(B_1)$, resp. over $\mathcal{A}^{\rm sym}_{Q^{\rm sing}_{\rm b}}(B_1)$, is smooth, resp. singular. 
To apply the results of the previous subsection, the  topology for the space of boundary conditions we shall working with is the $C^{2,\alpha}$-topology for some $\alpha\in(0,1)$. 

Given $\alpha\in(0,1)$, we consider the sets 	
$$BD^{\rm smooth}_\alpha:=\Bigg\{(Q_{\rm b},\lambda) \in C_{\rm sym}^{2,\alpha}(\partial B_1;\mathbb{S}^4)\times [0,\infty): {\rm sing}(Q_\lambda)=\emptyset\text{ for every }Q_\lambda\in \mathop{\rm argmin}_{\mathcal{A}^{\rm sym}_{Q_{\rm b}}(B_1)}\mathcal{E}_\lambda \Bigg\} \,,$$		 		
$$BD^{\rm sing}_\alpha:=\Bigg\{(Q_{\rm b},\lambda) \in C_{\rm sym}^{2,\alpha}(\partial B_1;\mathbb{S}^4)\times [0,\infty): {\rm sing}(Q_\lambda)\not=\emptyset\text{ for every }Q_\lambda\in \mathop{\rm argmin}_{\mathcal{A}^{\rm sym}_{Q_{\rm b}}(B_1)}\mathcal{E}_\lambda \Bigg\} \,,$$			
and		
$$BD^{\rm coexist}_\alpha:=\Big(C_{\rm sym}^{2,\alpha}(\partial B_1;\mathbb{S}^4)\times [0,\infty)\Big)\setminus \Big(BD^{\rm smooth}_\alpha\cup BD^{\rm sing}_\alpha\Big)\,.$$
As already mentioned,  $BD^{\rm smooth}_\alpha\not=\emptyset$ and 	$BD^{\rm sing}_\alpha\not=\emptyset$ by \cite[Theorem 1.2 \& Theorem 1.3]{DMP2}, and more precisely, 
$$BD^{\rm smooth}_\alpha\cap\big(C_{\rm sym}^{2,\alpha}(\partial B_1;\mathbb{S}^4)\times\{\lambda\}\big)\neq\emptyset \text{ and } BD^{\rm sing}_\alpha\cap\big(C_{\rm sym}^{2,\alpha}(\partial B_1;\mathbb{S}^4)\times\{\lambda\}\big)\neq\emptyset \text{ for every $\lambda\geq0$}\,.$$

The main result of this subsection is the following theorem whose proof is postponed to the end of the subsection. 	
	
\begin{theorem}\label{decompobdspacethm}
Let $\alpha\in(0,1)$. The (disjoint) sets $BD^{\rm smooth}_\alpha$ and $BD^{\rm sing}_\alpha$ are open in $C^{2,\alpha}_{\rm sym}(\partial B_1;\mathbb{S}^4)\times [0,\infty)$, and $BD^{\rm coexist}_\alpha$ coincides with their common boundary, i.e., 
$$\partial BD^{\rm smooth}_\alpha= BD^{\rm coexist}_\alpha=\partial BD^{\rm sing}_\alpha\,.$$ 
\end{theorem}

As a direct consequence of Theorem~\ref{decompobdspacethm}, we obtain the following corollary proving immediately claims (i), (ii), and (iii) of Theorem~\ref{main-coexistence}. With the aid of Remark~\ref{rmk:curve-coex} below, also the last claim of Theorem~\ref{main-coexistence}, and hence its full proof, follows at once from the corollary. 

\begin{corollary}\label{corocoexist}
Let $\alpha\in(0,1)$, $\lambda>0$, and $\Gamma:[0,1]\to C^{2,\alpha}_{\rm sym}(\partial B_1;\mathbb{S}^4)$  a continuous curve such that $(\Gamma(0),\lambda)\in BD^{\rm sing}_\alpha$ and $(\Gamma(1),\lambda)\in BD^{\rm smooth}_\alpha$.  There exist $0<t_1\leq t_2<1$ such that  
\begin{enumerate}
\item[(i)] $(\Gamma(t),\lambda)\in BD^{\rm sing}_\alpha$  for every $0\leq t<t_1$; 
\vskip5pt
\item[(ii)] $(\Gamma(t),\lambda)\in BD^{\rm smooth}_\alpha$  for every $t_2<t\leq 1$; 
\vskip5pt
\item[(iii)] $(\Gamma(t_1),\lambda), (\Gamma(t_2),\lambda)\in BD^{\rm coexist}_\alpha$.
\end{enumerate} 
\end{corollary}

\begin{proof}
Consider the continuous curve $\widehat \Gamma:[0,1]\to  C^{2,\alpha}_{\rm sym}(\partial B_1;\mathbb{S}^4)\times [0,\infty)$ given by $\widehat \Gamma(t):=(\Gamma(t),\lambda)$. Then $\widehat \Gamma(0)\in  BD^{\rm sing}_\alpha$ and $\widehat \Gamma(1)\in  BD^{\rm smooth}_\alpha$. Consider
$$t_1:=\sup\big\{t\in[0,1]: \widehat \Gamma(s)\in BD^{\rm sing}_\alpha\text{ for every $0\leq s\leq t$}\}\,. $$
By Theorem \ref{decompobdspacethm} and the continuity of $\widehat\Gamma$, we have $t_1\in(0,1)$ and $\widehat \Gamma(t_1)\in \partial BD^{\rm sing}_\alpha= BD^{\rm coexist}_\alpha$, so that (i) holds. 

Then we consider
$$t_2:=\inf\big\{t\in[0,1]: \widehat \Gamma(s)\in BD^{\rm smooth}_\alpha\text{ for every $t\leq s\leq 1$}\}\,. $$
Clearly $t_1\leq t_2$, and as above, Theorem \ref{decompobdspacethm} and the continuity of $\widehat\Gamma$ imply $t_2<1$ and $\widehat \Gamma(t_2)\in \partial BD^{\rm smooth}_\alpha= BD^{\rm coexist}_\alpha$ proving (ii), and completing the proof.
\end{proof}

\begin{remark}\label{rmk:curve-coex}
	As already alluded in the Introduction, there exists at least one curve $\Gamma$ with the properties required by Corollary~\ref{corocoexist}. This is obtained by concatenating the curves built in (the proofs of) \cite[Theorem~1.2 and Theorem~1.3]{DMP2}. Thus, the corollary shows in particular that $BD^{\rm coexist}_{\alpha}$ is not empty, clearly implying the last claim of Theorem~\ref{main-coexistence}, and concluding its proof.
\end{remark}

The proof of Theorem \ref{decompobdspacethm} rests on our regularity results together with the unique  continuation property for real analytic maps. This tool leads to the following uniqueness statement.

\begin{lemma}\label{uniqstatinside}
Let $Q_\lambda\in W^{1,2}_{\rm sym}(B_1;\mathbb{S}^4)$  be a minimizer of $\mathcal{E}_\lambda(\cdot,B_1)$ among all $Q\in W^{1,2}_{\rm sym}(B_1;\mathbb{S}^4)$ satisfying $Q=Q_\lambda$ on $\partial B_1$. For every radius $\rho\in(0,1)$ such that  ${\rm sing}(Q_\lambda)\cap\partial B_\rho=\emptyset$, the restriction of $Q_\lambda$ to $B_\rho$ is the unique minimizer of $\mathcal{E}_\lambda$ over the class $\mathcal{A}^{\rm sym}_{Q_{\lambda}}(B_\rho)$.
\end{lemma}

\begin{proof}
By Theorem \ref{intregthm} and since ${\rm sing}(Q_\lambda)\cap\partial B_\rho=\emptyset$, $Q_\lambda$ is (real) analytic in a neighborhood of $\partial B_\rho$. We fix a further radius $\rho^\prime\in(\rho,1)$ such that $Q_\lambda$ is analytic 
in the open annulus $A:=B_{\rho^\prime}\setminus \overline B_\rho$. 
Now, let $Q_\rho$ be a minimizer of $\mathcal{E}_\lambda$ over the class $\mathcal{A}^{\rm sym}_{Q_{\lambda}}(B_\rho)$. We consider the comparison map
$$\bar Q_\lambda(x):=\begin{cases} 
Q_\rho(x) & \text{if $x\in B_\rho$}\\
Q_\lambda(x) & \text{if $x\in B_1\setminus B_\rho$}\,,
\end{cases}$$
which belongs to $W^{1,2}_{\rm sym}(B_1;\mathbb{S}^4)$ and agrees with $Q_\lambda$ on $\partial B_1$. Hence, 
\begin{multline*}
\mathcal{E}_\lambda(Q_\rho,B_\rho)+\mathcal{E}_\lambda(Q_\lambda,B_1\setminus B_\rho) \leq \mathcal{E}_\lambda(Q_\lambda,B_\rho)+\mathcal{E}_\lambda(Q_\lambda,B_1\setminus B_\rho)\\
=\mathcal{E}_\lambda(Q_\lambda,B_1)
\leq \mathcal{E}_\lambda(\bar Q_\lambda,B_1)= \mathcal{E}_\lambda(Q_\rho,B_\rho)+\mathcal{E}_\lambda(Q_\lambda,B_1\setminus B_\rho)\,.
\end{multline*}
Thus $\mathcal{E}_\lambda(Q_\lambda,B_1)= \mathcal{E}_\lambda(\bar Q_\lambda,B_1)$ which in turn implies that $\bar Q_\lambda$ minimizes $\mathcal{E}_\lambda(\cdot,B_1)$ among all $Q\in W^{1,2}_{\rm sym}(B_1;\mathbb{S}^4)$ satisfying $Q=Q_\lambda$ on $\partial B_1$. Once again, by Theorem \ref{intregthm}, $\bar Q_\lambda$ is real analytic in $B_1\setminus {\rm sing}(\bar Q_\lambda)$ with ${\rm sing}(\bar Q_\lambda)$ a locally finite subset in $B_1$. As a consequence, the map $\bar Q_\lambda-Q_\lambda$ is analytic in the open set $B_1 \setminus \Sigma$ containing $A$, where $\Sigma = {\rm sing}(Q_\lambda)\cup {\rm sing}(\bar Q_\lambda)$ is a locally finite set in $B_1$, and $\bar Q_\lambda-Q_\lambda\equiv 0$ in $A$. 
Thus, by unique continuation for real analytic maps, $\bar Q_\lambda=Q_\lambda$ in $B_1 \setminus \Sigma$, which shows that $Q_\rho=Q_\lambda$ in $B_\rho$. 
\end{proof}

In order to apply the results of the previous subsections, we establish now the following strong $W^{1,2}$-compactness property of minimizers. 

\begin{lemma}\label{compacteasy}
Let  $\alpha\in(0,1)$ and $\{(Q^{(n)}_{\rm b},\lambda_n)\}$ be a converging sequence in $C^{2,\alpha}_{\rm sym}(\partial B_1;\mathbb{S}^4)\times[0,\infty)$, and denote by $(Q^*_{\rm b},\lambda_*)$ its limit. Every sequence $\{Q_n\}$ such that $Q_n$ minimizes $\mathcal{E}_{\lambda_n}$ over $\mathcal{A}^{\rm sym}_{Q^{(n)}_{\rm b}}(B_1)$ admits a subsequence strongly converging in $W^{1,2}(B_1)$ to some  $Q_*$ minimizing $\mathcal{E}_{\lambda_*}$ over $\mathcal{A}^{\rm sym}_{Q^*_{\rm b}}(B_1)$. 
\end{lemma}	
	
\begin{proof}
We consider the comparison map $\bar Q_n$ defined by $\bar Q_n(x):=Q^{(n)}_{\rm b}(x/|x|)$. A direct computation shows that $\sup_n\mathcal{E}_{\lambda_n}(\bar Q_n,B_1)<\infty$  since $Q^{(n)}_{\rm b}$ is bounded in the $C^{2,\alpha}$-topology. By minimality of $Q_n$, we have $\mathcal{E}_{\lambda_n}(Q_n,B_1)\leq \mathcal{E}_{\lambda_n}(\bar Q_n,B_1)\leq C$ for a constant $C$ independent of $n$. Applying \cite[Theorem 5.1]{DMP2}, we deduce that for a (not relabeled) subsequence, $Q_n\rightharpoonup Q_*$ weakly in $W^{1,2}(B_1)$ and $Q_n\to Q_*$ strongly in $W^{1,2}_{\rm loc}(B_1)$ for some $Q_*$ minimizing $\mathcal{E}_{\lambda_*}$ over $\mathcal{A}^{\rm sym}_{Q^*_{\rm b}}(B_1)$. Hence it remains to prove that  $Q_n\to Q_*$ strongly in $W^{1,2}(B_1)$. 

First we   notice that $Q_n\to Q_*$ strongly in $L^4(B_1)$ by the compact Sobolev embedding $W^{1,2}(B_1)\hookrightarrow L^4(B_1)$. Therefore, 
\begin{equation}\label{compeasy1}
\lambda_n\int_\Omega W(Q_n)\,dx\to \lambda_* \int_\Omega W(Q_*)\,dx\,.
\end{equation}
Now we fix an arbitrary small $\delta\in(0,1)$, and we define for $x\in B_1$, 
$$\widetilde Q_n(x):=\begin{cases} 
\displaystyle Q_*\bigg(\frac{x}{1-\delta}\bigg) & \text{if $x\in B_{1-\delta}$}\,,\\[10pt]
\displaystyle \bigg(1-\frac{1-|x|}{\delta}\bigg)Q_{\rm b}^n\bigg(\frac{x}{|x|}\bigg) + \frac{1-|x|}{\delta}Q_{\rm b}^*\bigg(\frac{x}{|x|}\bigg) & \text{if $x\in B_1\setminus B_{1-\delta}$}\,.
\end{cases}$$
Then $\widetilde Q_n\in W^{1,2}_{\rm sym}(B_1;\mathcal{S}_0)$ satisfies $\widetilde Q_n=Q_{\rm b}^{(n)}$ on $\partial B_1$. Since $Q_{\rm b}^{(n)}\to Q_{\rm b}^*$ in $C^{2,\alpha}(\partial B_1)$, $\widetilde Q_n$ converges to the mapping $ x\mapsto Q_{\rm b}^*(x/|x|)$ in $C^{2,\alpha}(\overline B_1\setminus B_{1-\delta})$. In particular, $|\widetilde Q_n|\geq 1/2$ for $n$ large enough which allows us to define 
$$\widehat Q_n:= \frac{\widetilde Q_n}{|\widetilde Q_n|} \in \mathcal{A}^{\rm sym}_{Q^n_{\rm b}}(B_1)\,.$$
By minimality of $Q_n$, we have $\mathcal{E}_{\lambda_n}(Q_n,B_1)\leq \mathcal{E}_{\lambda_n}(\widehat Q_n,B_1)$. Since $\widehat Q_n$ also converges to $ x\mapsto Q_{\rm b}^*(x/|x|)$ in $C^{2,\alpha}(\overline B_1\setminus B_{1-\delta})$, we have 
\begin{multline*}
\limsup_{n\to\infty} \mathcal{E}_{\lambda_n}(Q_n,B_1)\leq \lim_{n\to\infty} \mathcal{E}_{\lambda_n}(\widehat Q_n,B_{1-\delta})+ \lim_{n\to\infty} \mathcal{E}_{\lambda_n}(\widehat Q_n,B_1\setminus B_{1-\delta}) \\
= (1-\delta)\mathcal{E}_{(1-\delta)^2\lambda_*}(Q_*,B_{1}) + \mathcal{E}_{\lambda^*}\big( Q^*_{\rm b}(x/\abs{x}), B_1 \setminus B_{1-\delta}\big)\,, 
\end{multline*}
Now letting $\delta\to 0$, we deduce that 
$$\limsup_{n\to\infty} \mathcal{E}_{\lambda_n}(Q_n,B_1)\leq \mathcal{E}_{\lambda_*}(Q_*,B_{1}) \,.$$
In view of \eqref{compeasy1}, we thus have $\limsup_{n}\int_{B_1}|\nabla Q_n|^2\,dx\leq \int_{B_1}|\nabla Q_*|^2\,dx$. On the other hand, $\liminf_n\int_{B_1}|\nabla Q_n|^2\,dx\geq \int_{B_1}|\nabla Q_*|^2\,dx$  by lower semicontinuity of the Dirichlet energy. Hence $\lim_n\int_{B_1}|\nabla Q_n|^2\,dx= \int_{B_1}|\nabla Q_*|^2\,dx$, and the conclusion classically follows. 
\end{proof}	
	
\begin{proof}[Proof of Theorem \ref{decompobdspacethm}]	
{\it Step 1.} We first prove that $BD^{\rm smooth}_\alpha$ is open. Let $\{(Q^{(n)}_{\rm b},\lambda_n)\}$ be a sequence in $C^{2,\alpha}(\partial B_1;\mathbb{S}^4)\times [0,\infty)$ such that $\lambda_n\to \lambda_*$ and 
$Q^{(n)}_{\rm b}\to Q^*_{\rm b}$ in $C^{2,\alpha}(\partial B_1)$ for some $(Q^*_{\rm b},\lambda_*)\in BD^{\rm smooth}_\alpha$.  We aim to prove that $(Q^{(n)}_{\rm b},\lambda_n)\in BD^{\rm smooth}_\alpha$ for $n$ large enough. By contradiction, 
assume that $(Q^{(n)}_{\rm b},\lambda_n)\not\in BD^{\rm smooth}_\alpha$ for some (not relabeled) subsequence. Then, for each $n$ we can find  a minimizer $Q_n$ of $\mathcal{E}_{\lambda_n}$ over $\mathcal{A}^{\rm sym}_{Q^{(n)}_{\rm b}}(B_1)$ such that ${\rm sing}(Q_n)\not=\emptyset$. By Lemma \ref{compacteasy}, we can extract a further subsequence such that $Q_n\to Q_*$ strongly in $W^{1,2}(B_1)$ for some map $Q_*$ minimizing $\mathcal{E}_{\lambda_*}$ over $\mathcal{A}^{\rm sym}_{Q^*_{\rm b}}(B_1)$. Since $(Q^*_{\rm b},\lambda_*)\in BD^{\rm smooth}_\alpha$, we have ${\rm sing}(Q_*)=\emptyset$, and we infer from Corollary \ref{corolpersissmooth} that ${\rm sing}(Q_n)=\emptyset$ for $n$ large enough, a contradiction. 
\vskip5pt

\noindent {\it Step 2.} We now prove that $BD^{\rm sing}_\alpha$ is open following the same argument as above. Assume that $\{(Q^{(n)}_{\rm b},\lambda_n)\}$ is a sequence  in $C^{2,\alpha}(\partial B_1;\mathbb{S}^4)\times [0,\infty)$ converging to some $(Q^*_{\rm b},\lambda_*)\in BD^{\rm sing}_\alpha$. Assume also by contradiction that  $(Q^{(n)}_{\rm b},\lambda_n)\not \in  BD^{\rm sing}_\alpha$. Then we can find minimizers $Q_n$ of $\mathcal{E}_{\lambda_n}$ over $\mathcal{A}^{\rm sym}_{Q^{(n)}_{\rm b}}(B_1)$ such that ${\rm sing}(Q_n)=\emptyset$. Then  $Q_n\to Q_*$ strongly in $W^{1,2}(B_1)$ for some  $Q_*$ minimizing $\mathcal{E}_{\lambda_*}$ over $\mathcal{A}^{\rm sym}_{Q^*_{\rm b}}(B_1)$ (up to a subsequence). Since $(Q^*_{\rm b},\lambda_*)\in BD^{\rm sing}_\alpha$, we have ${\rm sing}(Q_*)\not=\emptyset$ which is in contradiction with ${\rm sing}(Q_n)=\emptyset$ and Corollary \ref{corolpersissmooth}. 
\vskip5pt

\noindent {\it Step 3.} To conclude the proof, it remains to prove that $BD^{\rm coexist}_\alpha$ is the common boundary of $BD^{\rm smooth}_\alpha$ and $BD^{\rm sing}_\alpha$. Let $(Q^*_{\rm b},\lambda_*)\in BD^{\rm coexist}_\alpha$, and $Q^{\rm t}_*$ and $Q^{\rm s}_*$ be two minimizers of $\mathcal{E}_{\lambda_*}$ over $\mathcal{A}^{\rm sym}_{Q^*_{\rm b}}(B_1)$ such that ${\rm sing}(Q^{\rm t}_*)=\emptyset$ and ${\rm sing}(Q^{\rm s}_*)\not=\emptyset$. By Theorem \ref{bdrregthm},  we have $Q^{\rm t}_*\in C^{2,\alpha}(\overline B_1)$ and we can find a radius $\rho_*\in(0,1)$ such that $Q^{\rm s}_*\in C^{2,\alpha}(\overline B_1\setminus B_{\rho_*})$ (i.e., ${\rm sing}(Q^{\rm s}_*)\subset \overline B_{\rho_*}$). We fix an arbitrary sequence $\{\rho_n\} \subset (\rho_*,1)$ such that $\rho_n\to1$ as $n\to\infty$, and we set for $x\in \partial B_1$, 
$$Q^{{\rm t},n}_{\rm b}(x):= Q^{\rm t}_*(\rho_n x)\quad\text{and}\quad Q^{{\rm s},n}_{\rm b}(x):= Q^{\rm s}_*(\rho_n x)\,.$$
Then  $(Q^{{\rm t},n}_{\rm b},\rho_n^2\lambda_*)\to (Q^*_{\rm b},\lambda_*)$ and $(Q^{{\rm s},n}_{\rm b},\rho_n^2\lambda_*)\to (Q^*_{\rm b},\lambda_*)$ in $C^{2,\alpha}(\partial B_1)\times[0,\infty)$. On the other hand, rescaling variables we infer from Lemma \ref{uniqstatinside} that the maps $Q^{\rm t}_n:x\mapsto Q^{\rm t}_*(\rho_n x)$ and $Q^{\rm s}_n:x\mapsto Q^{\rm s}_*(\rho_n x)$ are the unique minimizers of $\mathcal{E}_{\rho_n^2\lambda_*}$ over  
$\mathcal{A}^{\rm sym}_{Q^{{\rm t},n}_{\rm b}}(B_1)$ and $\mathcal{A}^{\rm sym}_{Q^{{\rm s},n}_{\rm b}}(B_1)$ respectively. Since ${\rm sing}(Q^{\rm t}_n)=\emptyset$ and ${\rm sing}(Q^{\rm s}_n)\not =\emptyset$, it shows that 
$(Q^{{\rm t},n}_{\rm b},\rho_n^2\lambda_*)\in BD^{\rm smooth}_\alpha$ and $(Q^{{\rm s},n}_{\rm b},\rho_n^2\lambda_*)\in BD^{\rm sing}_\alpha$ for $n$ large enough. Hence $(Q^*_{\rm b},\lambda_*)\in \partial BD^{\rm smooth}_\alpha \cap \partial BD^{\rm sing}_\alpha$, thus $BD^{\rm coexist}_\alpha \subset \partial BD^{\rm smooth}_\alpha \cap \partial BD_\alpha^{\rm sing}$. Now, to reach the claimed conclusion, it is enough to prove that $BD^{\rm coexists}_\alpha \supset \partial BD^{\rm smooth}_\alpha \cup \partial BD^{\rm sing}$. Indeed, this and the previous inclusion together imply as announced $BD^{\rm coexists}_\alpha = \partial BD^{\rm smooth}_\alpha = \partial BD^{\rm smooth}_\alpha$. To this end, notice that $C^{2,\alpha}(\partial B_1;\bbS^4) \times [0,\infty)$ is the union of the disjoint sets $BD_\alpha^{\rm smooth}$, $BD_\alpha^{\rm sing}$, and $BD^{\rm coexist}_\alpha$. Since the first two sets are open, $BD_\alpha^{\rm smooth} \cup BD^{\rm coexist}_\alpha$ and $BD_\alpha^{\rm sing} \cup BD^{\rm coexist}_\alpha$ are closed, hence they must contain the closures of $BD_\alpha^{\rm smooth}$ and $BD_\alpha^{\rm sing}$, respectively. In turn, this means that $BD_\alpha^{\rm coexist}$ contains both $\partial BD_\alpha^{\rm smooth}$ and $\partial BD_\alpha^{\rm sing}$ (indeed, $BD_\alpha^{\rm smooth}$ and $BD_\alpha^{\rm sing}$, being open, are disjoint from their boundaries). Therefore, $\partial BD_\alpha^{\rm smooth} \cup \partial BD_\alpha^{\rm sing} \subset BD_\alpha^{\rm coexist}$, and the conclusion follows.
\end{proof}

		

\section{Landau-de Gennes minimizers in 2D}\label{2Dminimization}

In this section, we examine a two dimensional minimization problem whose importance (beyond its own interest) will be revealed mostly in Section \ref{SectSplit}. We consider the minimization of the LdG energy among equivariant unit norm configurations defined on a two dimensional disc $\bbD_\rho := \left\{ z \in \C : \abs{z} < \rho \right\}$. We will always assume  $\rho=1$ (discarding the subscript for simplicity) which can always be  achieved by rescaling the domain. In view of Lemma~\ref{lemma:Q-detQ} and Corollary~\ref{cor:dec}, admissible configurations can be described as maps in the space $W_{\rm sym}^{1,2}(\bb D;\bb S^4)$ with the two equivalent forms: either in terms of tensors $Q \in \bb S^4 \subset \mathcal{S}_0$, or in terms of $u \in \bb S^4 \subset \R \oplus \C \oplus \C$. However, we shall mostly rely on the $\R \oplus \C \oplus \C$-description as it is more suited for our purposes. 
 \vskip5pt

We consider, for  fixed $\lambda\geq 0$, the 2D-LdG energy
$E_\lambda$ as in \eqref{eq:def-2d-energy} and we aim to minimize it over the $\mathbb{S}^1$-equivariant class $\mathcal{A}^{\rm sym}_{\overline{H}}(\bbD)$ defined in \eqref{eq:def-2d-adm},
where $\overline H:\partial\bbD\to\mathbb{R}P^2 \subset \mathbb{S}^4$ denotes the radial anchoring map defined in \eqref{defbarH}. According to the correspondence in Corollary \ref{cor:dec}, we have $\overline H\simeq  g_{\overline{H}}$ where $ g_{\overline{H}}:\partial \bbD\to \mathbb{S}^4\subset \R \oplus \C \oplus \C$ is given by 
\begin{equation}
\label{eq:hor-anchoring}
 g_{\overline{H}}(z):=\left( -\frac{1}{2},0,\frac{\sqrt{3}}{2}z^2\right)\,.
\end{equation}
Here and in the sequel, we make use of the complex variable $z:=x_1+ix_2$, identifying in this way $\R^2$ with the complex plane $\C$. 
\vskip5pt

As we announced in Theorem~\ref{2d-biaxial-escape} and proved in the present section, the uniaxial or biaxial character of any minimizer of $E_\lambda$ over $\mathcal{A}^{\rm sym}_{\overline{H}}(\bbD) $  
depends on $\lambda$ in a crucial way. More precisely, a sharp transition in the qualitative properties of minimizers occurs through a biaxial escape mechanism, as the strength parameter $\lambda$ of the confining potential $W$ in \eqref{eq:def-2d-energy} decreases. 

Note that, by Lemma~\ref{lemma:s1eq-emb}, we have   $\mathcal{A}_{\overline{H}}^{\rm sym}(\bbD) = \mathcal{A}_{\rm N} \cup \mathcal{A}_{\rm S}$ with disjoint union, and 
\begin{equation}\label{eq:AN}
		\mathcal{A}_{\rm N} := \big\{ Q \in \mathcal{A}^{\rm sym}_{\overline{H}}(\bbD) : Q(0) = \eo \big\}\simeq \big\{ u \in \widetilde{\mathcal{A}}^{\rm sym}_{ g_{\overline{H}}}(\bbD) : u(0) = (1,0,0)  \big\}=: \widetilde{\mathcal{A}}_{\rm N}\,,
	\end{equation}
	and
	\begin{equation}\label{eq:AS}
		\mathcal{A}_{\rm S} := \big\{ Q \in \mathcal{A}^{\rm sym}_{ \overline{H}}(\bbD) : Q(0)= -\eo \big\} \simeq \big\{ u \in \widetilde{\mathcal{A}}^{\rm sym}_{ g_{\overline{H}}}(\bbD) : u(0) = (-1,0,0)  \big\}=: \widetilde{\mathcal{A}}_{\rm S}\, .
	\end{equation}
We aim to describe precisely to which of these two components the minimizers of $E_\lambda$ over $\mathcal{A}^{\rm sym}_{\overline{H}}(\bbD) $ belong to as the parameter $\lambda$ varies. 
To tackle this question, we rely in an essential way on a gap phenomenon for the Dirichlet energy $E_0$ over the two components of the class $\mathcal{A}_{\overline{H}}^{\rm sym}(\bbD) = \mathcal{A}_{\rm N} \cup \mathcal{A}_{\rm S}$ which is of independent interest. This is the object of the next two subsections. 
By studying the minimization problem of $E_0$ in each class $\mathcal{A}_{\rm N}$ or $\mathcal{A}_{\rm S}$, we shall prove that the corresponding infima are different. Describing the set of minimizers for both the infima, we shall also make the energy gap fully explicit.


\subsection{Large equivariant harmonic maps in 2D}

In this subsection, we classify all critical points of the Dirichlet energy $E_0$ in the class $\mathcal{A}_{\overline{H}}^{\rm sym}(\bbD)$ satisfying $Q(0)= \eo$. According to Proposition~ \ref{prop:symmetric-criticality}, those are critical points of $E_0$ over $W^{1,2}(\bbD;\mathbb{S}^4)$, and thus equivariant (weakly) harmonic maps from $\bbD$ into $\bbS^4$ satisfying $Q=\overline{H}$ on $\partial \bbD$. In terms of the isometric correspondence $Q\simeq u$ from Corollary~\ref{cor:dec}, we are interested in equivariant (weakly) harmonic maps $u$ into $\bbS^4\subset \R \oplus \C \oplus \C$ satisfying $u=\ g_{\overline{H}}$ on $\partial \bbD$ and $u(0)=(1,0,0)$.  
As recalled in the Introduction, these harmonic maps are usually referred to in the literature as the {\em large solutions} of the harmonic map system. They escape from the (small) spherical cap containing the image of the boundary data $g_{\overline{H}}$ given in \eqref{eq:hor-anchoring}, as opposed to the {\em small solution} discussed in Proposition~\ref{ASminimization}, for which the escape phenomenon does not hold. 
\vskip5pt

Recall that $u \in W^{1,2}(\bbD;\bbS^4)$ is a \emph{weakly harmonic map} in $\bbD$ if $u$ is a critical point of the Dirichlet energy
\begin{equation}\label{eq:2d-LDGenergy-tilde}
	\widetilde{E}_0(u) := \int_{\bbD} \frac{1}{2}\abs{\nabla u}^2\,dx
\end{equation}
with respect to compactly supported perturbations preserving the $\bbS^4$-constraint. If, in addition, $u=g_{\overline{H}}$ on $\partial \bbD$ in the sense of traces, then $u$ is a distributional solution of the following boundary value problem  
\begin{equation}
\label{eq:harmonic-bdvp}
		\begin{cases}
			\Delta u + \abs{\nabla u}^2 u=0 & \mbox{ in } \bbD\,,\\
			u =  g_{\overline{H}} & \mbox{ on } \partial\bbD\,.
		\end{cases}
	\end{equation} 
By H\'{e}lein's theorem \cite{He} and the general analyticity results for elliptic systems from \cite[Chapter~6]{Morrey}, such a map $u$ is real analytic in the interior. 
Under the Dirichlet boundary condition $g_{\overline{H}}$, the map $u$ is actually real analytic up to the boundary by \cite{Morrey,Quing}. Hence it is harmonic in $\overline{\bbD}$ in the classical sense. According to \eqref{eq:dec-equiv2D}, an equivariant harmonic map $u$ has the form 
\begin{equation}\label{formequivmap}
u(re^{i\phi})=(f_0(r),f_1(r) e^{i \phi},f_2(r) e^{i2\phi})\,,
\end{equation} 
and the Euler-Lagrange equation in \eqref{eq:harmonic-bdvp} rewrites into a system of ODEs for $r \in (0,1]$ (see Remark~\ref{rmk:2d-EL-eq}), 
\begin{equation}
\label{eq:harmonic-odes}
		\begin{cases}
\displaystyle			f^{\prime\prime}_0+\frac1r f^\prime_0 = -\abs{\nabla u}^2 f_0 \,,\\[5pt]
			\displaystyle f^{\prime\prime}_1+\frac1r  f^\prime_1 = -\abs{\nabla u}^2 f_1 -\frac1{r^2}f_1 \,,\\[5pt]
				\displaystyle  f^{\prime\prime}_2+\frac1r  f^\prime_2 = -\abs{\nabla u}^2 f_2 -\frac4{r^2}f_2 \,.
		\end{cases}
	\end{equation} 
Here $f(r):=(f_0(r),f_1(r),f_2(r)) \in \R\oplus\bbC\oplus\C$, and by \eqref{eq:equiv-nabla-22D}, 
\begin{equation}
\label{eq:gradsquare}
 \abs{\nabla u}^2=\abs{\partial_r u}^2+\frac1{r^2}\abs{\partial_\phi u}^2=\abs{f^\prime}^2+\frac1{r^2}\left( \abs{f_1}^2 +4\abs{f_2}^2 \right) \, . 
\end{equation}
\vskip3pt

In order to describe the equivariant solutions to \eqref{eq:harmonic-bdvp} satisfying the condition $u(0)=(1,0,0)$, we shall combine \eqref{eq:harmonic-odes}  with the classification of    
equivariant harmonic spheres from \cite{DMP2}. 
Following \cite{DMP2}, it is convenient to use complex differentiation through the usual Wirtinger's operators
\[\partial_z=\frac12 \left( \frac{\partial}{\partial x_1} -i \frac{\partial}{\partial x_2}\right) \, , \quad \partial_{\bar{z}}=\frac12 \left( \frac{\partial}{\partial x_1} +i \frac{\partial}{\partial x_2}\right) \, . \] 
Since $\abs{\nabla u}^2=2 \abs{\partial_z u}^2+2 \abs{\partial_{\bar{z}} u}^2$ and $\Delta u=4 \partial_{\bar{z}z}u$,  
the harmonic map equation  \eqref{eq:harmonic-bdvp} rewrites as 
$$\partial_{\bar{z}z}u +\frac12 \big( \abs{\partial_z u}^2+ \abs{\partial_{\bar{z}} u}^2\big) u=0\,.$$ 

Let us now recall the classical notions of conformality and isotropy. 

\begin{definition}\label{defconformisot}
A smooth map $u:\Omega \to \bb S^4 \subset\R\oplus\bbC\oplus\C$ defined on an open set $\Omega \subset \bb C$ is said to be 
\begin{itemize}
\item[(i)]{\em conformal} if
\begin{equation}
\label{eq:conformality}
\partial_z u \cdot \partial_z u:= \sum_{j=0}^4 (\partial_z u_j)^2= \frac14 \abs{\partial_{x_1} u}^2-  \frac14 \abs{\partial_{x_2} u}^2- \frac{i}2\partial_{x_1}u \cdot \partial_{x_2}u \equiv 0 \,, 
\end{equation}
where ``$\,\cdot\,$'' denotes the Euclidean scalar product in $\R\oplus\bbC\oplus\C\simeq \R^5$ extended by bilinearity to $\C^5$; 
\item[(ii)] {\em isotropic} if
\begin{equation}
\label{eq:isotropicmap}
\partial^2_{z} u \cdot \partial^2_{z} u:= \sum_{j=0}^4 (\partial^2_{z} u_j)^2\equiv 0 \, ,
\end{equation}
where 
$\partial^2_{z} = \frac14 (\partial^2_{x_1}  -\partial^2_{x_2}) -\frac{i}2 \partial_{x_1x_2}$. 
\end{itemize}
\end{definition}
\vskip5pt

Here we shall not need the full definition of {\em total isotropy} from \cite[Chapter 6]{He2}{{,}} which is satisfied by the harmonic spheres discussed in \cite[Section 3]{DMP2}. Actually, under \eqref{eq:conformality}-\eqref{eq:isotropicmap}, it will be automatically satisfied for the equivariant solutions to \eqref{eq:harmonic-bdvp}, as we are going to show in the following lemma.
This extension result is the starting point of the classification of all large equivariant solutions to \eqref{eq:harmonic-bdvp}.

\begin{lemma}
\label{equiv-extension}
If $u \in W^{1,2}_{\rm sym}(\bbD; \bbS^4)$ is a weak solution of \eqref{eq:harmonic-bdvp}, then $u$ is real analytic and conformal in $\overline{\bb D}$. Moreover, $u$  uniquely extends  to a map $U \in C^\omega(\C;\bb S^4)$ which is equivariant, harmonic, conformal, and isotropic in the whole $\C$.
\end{lemma}
\begin{proof}
The map $u$ being a weak solution to  \eqref{eq:harmonic-bdvp}, it is real analytic up to the boundary, as we already remarked. Being equivariant, it is of  the form \eqref{formequivmap}, where  
the map $f=(f_0, f_1 ,f_2)$ satisfies $|f|\equiv1$ and solves system  \eqref{eq:harmonic-odes} for $r \in (0,1]$. 

Since $u$ belongs to $W^{1,2}(\bbD)$, we infer from \eqref{eq:gradsquare} that 
$$\int_0^1 \big(r^2 \abs{f^\prime}^2+\abs{f_1}^2+4\abs{f_2}^2\big) \frac{dr}{r} <\infty\,.$$
Hence $ r_j^2 \abs{f^\prime (r_j)}^2+\abs{f_1(r_j)}^2+4\abs{f_2(r_j)}^2 \to 0 $ for some sequence $r_j \downarrow 0$. Since $|f|^2\equiv 1$, we have $ f^\prime\cdot f\equiv 0$. 
Hence, taking the scalar product of \eqref{eq:harmonic-odes} with $r^2  f^\prime$ and  integrating between $r_j$ and $r$ leads to 
\[ r^2 \abs{\partial_r f (r)}^2-\abs{f_1(r)}^2-4\abs{f_2(r)}^2 = r_j^2 \abs{\partial_r f (r_j)}^2-\abs{f_1(r_j)}^2-4\abs{f_2(r_j)}^2  \to 0 \text{ as $j\to\infty$} \,.\]
Thus $\abs{\partial_r u}^2- \frac1{r^2} \abs{\partial_\phi u}^2 =\abs{\partial_r f (r)}^2-\frac{1}{r^2}\big(\abs{f_1(r)}^2+4\abs{f_2(r)}^2\big)\equiv 0$. 
On the other hand, it follows from \eqref{formequivmap} that 
$ \partial_r u \cdot \frac1r \partial_\phi u \equiv 0$. Therefore $u$ is conformal in the sense of Definition~\ref{defconformisot} since such property is independent of the chosen orthonormal frame.
\vskip3pt

Now we solve the Cauchy problem for \eqref{eq:harmonic-odes} with Cauchy data $(f(1),  f^\prime (1))$ to extend $f$ to its maximal interval of existence $(0,r_{\max}) \supseteq (0,1]$. We denote by  $\widetilde{f}$ the maximal solution. Then $\widetilde f$ is real analytic, and therefore it satisfies 
\begin{equation}
\label{eq:conservations}
|\widetilde{f}(r)|^2 =1 \quad\text{and}\quad | \widetilde{f}^\prime|^2=\frac{1}{r^2} \big( |\widetilde{f}_1|^2+4|\widetilde{f}_2|^2\big) \quad\text{for every $r\in (0, r_{\max})$}\,,
\end{equation}
since these identities hold for every $r \in (0,1]$. As a consequence of the uniform a priori bounds induced by \eqref{eq:conservations}, it follows that $r_{\max}=+\infty$, i.e., $\widetilde{f}$ solves \eqref{eq:harmonic-odes} for $r \in (0,\infty)$.  

Setting 
$$U(re^{i\phi}):=\big(\widetilde{f}_0(r),\widetilde{f}_1(r) e^{i \phi},\widetilde{f}_2(r) e^{i2\phi}\big)\,,$$
it follows  by construction that $U$ is an equivariant real analytic harmonic map from $\C$ into $\bb S^4$, extending $u$ to the whole plane. 
Repeating the argument above on $\widetilde{f}$ with $r \in (0,\infty)$, we infer that $U$ is conformal in $\C$.   
To complete the proof, it thus remains to show that $U$ is isotropic, i.e., it satisfies \eqref{eq:isotropicmap}. To this purpose, we adapt to the equivariant setting the strategy in \cite[Proposition~6.1]{He2}. First, we notice that 
$$ \partial_z U\cdot U= \frac12 \partial_z |U|^2 \equiv 0  \, , $$
$$\partial^2_{z} U\cdot U= \partial_z (\partial_z U\cdot U ) - \partial_z U \cdot \partial_z U \equiv 0 \, ,  $$
and
$$  \partial^2_{z} U\cdot \partial_z U= \frac12 \partial_z (\partial_z U\cdot \partial_z U) \equiv 0 \, , $$
since $|U|^2=1$ and $U$ is conformal.  
Then we consider $g:=\partial^2_{z}U  \cdot \partial^2_{z}U$ which is a complex-valued smooth function. 
Since $U$ is a harmonic map, we have
\begin{equation*}
\begin{split}
\partial_{\bar{z}} g&=\partial_{\bar{z}} (\partial^2_{z}U \cdot \partial^2_{z}U )= \partial^2_{z} U \cdot \partial_z ( 2\partial_{\bar{z}z}U )=-\partial^2_{z} U \cdot \partial_z \left( \big(\abs{\partial_z U}^2+ \abs{\partial_{\bar{z}} U}^2\big) U \right) \\
&= -\big(\partial^2_{z} U \cdot U\big) \partial_z \big( \abs{\partial_z U}^2+ \abs{\partial_{\bar{z}} U}^2 \big)  -\big(\partial^2_{z} U \cdot \partial_z U\big) \big( \abs{\partial_z U}^2+ \abs{\partial_{\bar{z}} U}^2 \big) \equiv 0 \, ,
\end{split}
\end{equation*}
and thus $g$ is an entire holomorphic function. On the other hand, w.r.to the $\bb S^1$-action on $\R \oplus \C \oplus \C$ given in \eqref{eq:deg-action}, the map $U$ satisfies the equivariance property $R \cdot U (z)=U ( R z)$ for all $R=e^{i \theta} \in \bb S^1$ and for all $z \in \C$.   
Long but elementary calculations now give
\begin{equation*}
\begin{split}
\partial^2_{z} U (z) \cdot \partial^2_{z}  U(z) &= \partial^2_{z} \big(R \cdot U (z)\big) \cdot \partial^2_{z} \big(R \cdot U (z)\big)\\
&= \partial^2_{z} \big(U (Rz)\big) \cdot \partial^2_{z} \big( U (Rz)\big)=R^4\big( \partial^2_{z} U (Rz) \cdot \partial^2_{z} U (Rz) \big)  \,.
\end{split}
\end{equation*}
Hence $g(e^{i\theta}z)\equiv e^{-4i \theta} g(z)$. Since $g$ is holomorphic, from the identity principle on the domain $\bb C\setminus \{0\}$, we infer that $g(z)\equiv c/z^4$ for some $c \in \bb C$. Since $g$ is smooth at the origin, we conclude that $c=0$, and thus $g \equiv 0$. Therefore \eqref{eq:isotropicmap} holds, and the proof is complete.
\end{proof}

We now are ready to classify all large solutions  to \eqref{eq:harmonic-bdvp}, i.e., solutions in the class $\widetilde{\mathcal{A}}_{\rm N}$. The proof of this classification parallels the one for harmonic mappings $\omega$ from $\C \cup \{ \infty\}\simeq\bbS^2 $ into $\bbS^4$ satisfying $\omega(0)=\eo$ (which is a combination of \cite[Proposition~3.6, Proposition~3.8, Remark~3.11 and Theorem 3.19]{DMP2}). It shows that large solutions are precisely the restriction to the unit disc of those entire harmonic maps satisfying the boundary condition and the constraint at the origin. 
 	
	\begin{proposition}
	\label{prop:largeharm} 
	If $u \in W^{1,2}_{\rm sym}(\bbD; \bbS^4)$ is a weak solution to \eqref{eq:harmonic-bdvp} satisfying $u(0)=\eo$, then 
	there exists $\mu_1 \in \C$ such that
\begin{multline}\label{eq:class-harm}
	u(z) = \frac{1}{D(z)} \left( 1 - \abs{\mu_1}^2 \abs{z}^2-  3 \abs{z}^4  +\frac{\abs{\mu_1}^2}{3}\abs{z}^6, \right. \\
	\vphantom{\frac{1}{D(z)}} \left. 2 \mu_1 z \left( 1 - \abs{z}^4 \right),   2 \sqrt{3} z^2 \left( 1 + \frac{\abs{\mu_1}^2}{3} \abs{z}^2 \right)  \right),
\end{multline}
with 
	\begin{equation}\label{eq:D(z)}
		D(z) :=  1 + \abs{\mu_1}^2 \abs{z}^2+  3 \abs{z}^4  +\frac{\abs{\mu_1}^2}{3}\abs{z}^6 \,.
	\end{equation}
In particular, 
$$\widetilde E_0(u):=\int_{\bbD}\frac{1}{2}|\nabla u|^2\,dx=6\pi\,. $$
\end{proposition}

\begin{proof}
In view of Lemma \ref{equiv-extension}, $u$ extends to a harmonic map $U \in C^\omega(\C;\bb S^4)$ which equivariant, conformal and isotropic in the whole $\C$. 
By equivariance, it writes  
 $U(re^{i\phi})=(f_0(r),f_1(r) e^{i \phi},f_2(r) e^{i2\phi})$. 
\vskip5pt

\noindent{\it Step 1.} We assume in this step that $U$ is not linearly full, and we aim to show that \eqref{eq:class-harm}-\eqref{eq:D(z)} hold with $\mu_1=0$. First we notice that, in this case,  $f_1\equiv 0$ by  \cite[Remark~2.4]{DMP2} and the boundary condition $U=g_{\overline H}$ on $\partial \bbD$.  Hence $U$ takes values in the unit $2$-sphere of $\R\oplus\{0\}\oplus\C$, that we denote by $\bbS^2_2$. Setting $\boldsymbol\sigma_2 : \bbS^2_2 \to \C \cup \{\infty\}$ to be the stereographic projection from its south pole $(-1,0,0)$, we consider  
\[\eta(re^{i\phi}):={\boldsymbol{\sigma}}_2\circ U(re^{i\phi})= \frac{{f}_2(r)}{1+{f}_0(r)}e^{i2\phi} \, .\]
Since $U(0)=(1,0,0)$, we have $U(z)\not=(-1,0,0)$ for all $z \in \C$ by \cite[Remark 3.4]{DMP2}. 
Therefore $\eta:\C \to \C$ is well defined, real analytic, and  conformal since $U$ and $\boldsymbol\sigma_2$ are. Then, $\eta$ being conformal, it is either holomorphic or anti-holomorphic. Anti-holomorphicity is easily excluded. Indeed, it would give $\eta(z)=c/\bar{z}^2$ by the identity principle on $\C \setminus \{0\}$ for a suitable $c \in \C$ (since the two functions coincide on $\{|z|=1\}$ by equivariance). But $\eta$ is smooth near the origin, so that $c=0$. In turn $\eta \equiv 0$ which is clearly impossible because $U(z)\neq (1,0,0)$ for $|z|=1$. Then, $\eta$ being holomorphic on $\C$, we have $\eta(z)=c z^2$ for a suitable $c \in \C \setminus \{0\}$, again by the identity principle and equivariance. Therefore, 
$$U(z)=\boldsymbol{\sigma}_2^{-1} \circ \eta(z)=\left(\frac{1-|c|^2|z|^4}{1+|c|^2|z|^4},0,\frac{2cz^2}{1+|c|^2|z|^4}\right)\,.$$ 
Since $U(z)=g_{\overline{H}}(z)$ for $|z|=1$, we obtain $c=\sqrt{3}$ which shows that  \eqref{eq:class-harm}-\eqref{eq:D(z)} hold with $\mu_1=0$. 
As a consequence, we have
$f_2(r)=\frac{2\sqrt{3}r^2}{1+3r^4}$, and by conformality and equivariance of $u$, 
\[\widetilde E_0(u)= \int_{\bb D} \frac12 |\nabla u|^2 dx= \int_{\bb D} \frac{1}{r^2} |\partial_\phi u|^2 dx=2\pi \int_0^1 \frac{4|f_2(r)|^2}{r^2} r dr =2\pi \int_0^1 \frac{48 r^3 }{(1+3r^4)^2} dr=6 \pi \, . \]
\vskip5pt

\noindent{\it Step 2.} We now assume $U$ is linearly full, and we claim that \eqref{eq:class-harm}-\eqref{eq:D(z)} hold for some $\mu_1\in\C\setminus\{0\}$. 
Following \cite[Section 3.3]{DMP2}, we set $\boldsymbol\sigma_4 : \bbS^4 \to \C \cup \{\infty\}$ to be the stereographic projection from the south pole $(-1,0,0)$, and we consider 
$$ \big(\xi (re^{i\phi}),\eta(re^{i\phi})\big):={\boldsymbol\sigma}_4\circ U(re^{i\phi})= \left(\frac{{f}_1(r)}{1+{f}_0(r)}e^{i\phi} ,\frac{{f}_2(r)}{1+{f}_0(r)}e^{i2\phi}\right) \, .$$
Once again, since $U(0)=(1,0,0)$, we have $U(z)\not=(-1,0,0)$ for all $z \in \C$ by \cite[Remark 3.4]{DMP2}. 
Hence $(\xi,\eta):\C \to \C^2$ is well defined and real analytic. Notice that the  conclusions of \cite[Lemma 3.12]{DMP2} still hold in the present case (although we don't know yet that $U$ extends to a harmonic sphere $U:\C\cup\{\infty\}\simeq\bb S^2 \to \bb S^4$) because $U$ is conformal and isotropic on the whole $\C$ by Lemma \ref{equiv-extension}. 

Now we can transpose word-by-word the argument in the proof of \cite[Theorem 3.19]{DMP2} to show that $U$ extends to a finite energy harmonic sphere $U \in C^\omega(\bb S^2;\bb S^4)$ (indeed, the positive lift $\widetilde{U}^+: \C \to \C P^3$ defined there extends holomorphically to the whole $\C P^1\simeq\bb S^2\simeq\C \cup \{\infty\}$ and $U=\boldsymbol{\tau} \circ \widetilde{U}^+$ on $\bb S^2$, where $\boldsymbol\tau :\C P^3 \to \bb S^4$ is the twistor fibration). As a consequence (compare with \cite[Proposition 3.8]{DMP2}), there exist $\mu_1, \mu_2 \in \C \setminus \{0\}$ such that  
\begin{multline}\label{eq:class-harm-sphere}
	U(z) = \frac{1}{D(z)}  \left( 1 - \abs{\mu_1}^2 \abs{z}^2-  \abs{\mu_2}^2 \abs{z}^4  +\frac{\abs{\mu_1}^2\abs{\mu_2}^2}{9}\abs{z}^6 , \right. \\
	\vphantom{\frac{1}{D(Z)}} \left. 2 \mu_1 z \left( 1 - \frac{\abs{\mu_2}^2}{3}\abs{z}^4 \right), 2 \mu_2 z^2 \left( 1 + \frac{\abs{\mu_1}^2}{3} \abs{z}^2 \right)  \right),
\end{multline}
with
	\begin{equation}\label{eq:D(z)-sphere}
		D(z) :=  1 + \abs{\mu_1}^2 \abs{z}^2+  \abs{\mu_2}^2 \abs{z}^4  +\frac{\abs{\mu_1}^2\abs{\mu_2}^2}{9}\abs{z}^6 \,.
	\end{equation}
The constraint $U=(U_0,U_1,U_2)\equiv g_{\overline{H}}$ on $\partial \bb D$ first implies  $U_1\equiv 0$ on $\partial\bb D$, which in turn yields 
$|\mu_2|=\sqrt{3}$.  Then $U(z)=\left( -\frac12, 0, \frac{\mu_2}{2} z^2\right)=g_{\overline{H}}(z)$ for every $z\in\partial\bbD$,
whence $\mu_2=\sqrt{3}$. Thus, \eqref{eq:class-harm-sphere}-\eqref{eq:D(z)-sphere} hold. 

To complete the proof, it remains to show that $\widetilde E_0(u)=6\pi$ for all $\mu_1 \in \C$ in \eqref{eq:class-harm}-\eqref{eq:D(z)}.
In view of \eqref{eq:gradsquare}-\eqref{eq:conservations}, the energy $\widetilde E_0(u)$ just depends on $|\mu_1|$. It is continuous with respect to $\mu_1$, and $\widetilde E_0(u)=6\pi$ for $\mu_1=0$ as already computed in the previous step. Then it is enough to check that $\widetilde E_0(u)$ is independent of $|\mu_1|$ by showing that it has zero derivative for $|\mu_1|$ positive. To see this, we first notice that $\langle u , \partial_{|\mu_1|} u \rangle\equiv 0$ since $|u|^2=1$, and $\partial_{|\mu_1|} u= 0$ on $\partial \bb D$ since $u=g_{\overline{H}}$ on $\partial\bbD$. Differentiating under integral sign, integrating by parts and using \eqref{eq:harmonic-bdvp},  we obtain
\begin{equation*}
\begin{split}
		\partial_{|\mu_1|} \widetilde E_0(u) &= \int_{\bbD} \partial_ {\abs{\mu_1}}\Big( \frac12 \abs{\nabla u}^2\Big) \,dx = \int_{\bbD}  \nabla u \cdot \nabla \big(\partial_{|\mu_1|} u\big) \,dx \\
	&= - \int_{\bbD} \Delta u \cdot \partial_{|\mu_1|} u\,dx = \int_{\bbD}  \abs{\nabla u}^2 u\cdot  \partial_{|\mu_1|} u \,dx =0 \,.
\end{split}
	\end{equation*}
which concludes the proof. 
\end{proof}


\subsection{Energy gap for the Dirichlet integral of maps into $\bbS^4$}\label{sec:gap}

In this subsection, we compute explicitly the minimum values and describe the minimizers of the minimization problems
	\begin{equation}
		\label{eq:minAN} \min_{u \in \widetilde{\mathcal{A}}_{\rm N}} \widetilde E_0(u)\,,
\end{equation}
and
\begin{equation}\label{eq:minAS} 
\min_{u \in \widetilde{\mathcal{A}}_{\rm S}} \widetilde E_0(u) \, ,
\end{equation}
where $\widetilde{E}_0$ is given in \eqref{eq:2d-LDGenergy-tilde}, thus making explicit a corresponding \emph{gap phenomenon}. The following theorem is the main result of the subsection, and it is a direct consequence of Propositions~\ref{ASminimization} and~\ref{ANminimization} below.
	\begin{theorem}
	\label{thm:gap-u} 
	The following gap holds for the Dirichlet energy \eqref{eq:2d-LDGenergy-tilde}:
	\begin{equation}\label{eq:gap}
		2\pi=\min_{u \in \widetilde{\mathcal{A}}_{\rm S}} \widetilde{E}_0(u) < \min_{u \in \widetilde{\mathcal{A}}_{\rm N}} \widetilde{E}_0(u)=6\pi\,.
	\end{equation}
In addition, the minimum value of $\widetilde{E}_0$ over $ \widetilde{\mathcal{A}}_{\rm S}$  is uniquely achieved by 
\begin{equation}\label{eq:uS}
 u_{\rm S}(z) := \left( \frac{\abs{z}^4-3}{\abs{z}^4+3}, \, \, 0 , \, \, \frac{2 \sqrt{3} z^2}{{\abs{z}^4+3}} \right) \,,
 \end{equation}
while the minimum value of $\widetilde{E}_0$ over $ \widetilde{\mathcal{A}}_{\rm N}$ is attained at $u\in \widetilde{\mathcal{A}}_{\rm N}$ iff $u$ is of the form \eqref{eq:class-harm}-\eqref{eq:D(z)}.  
\end{theorem}

\begin{remark}[``Bubbling-off'' of harmonic spheres]
\label{bubbling}
The resolution of \eqref{eq:minAN}-\eqref{eq:minAS}  suffers two main difficulties: (i) the conformal invariance of the functional $\widetilde E_0$ and the induced lack of compactness of  energy-bounded sequences; (ii) the fact that the classes $\widetilde{\mathcal{A}}_{\rm N}$ and $\widetilde{\mathcal{A}}_{\rm S}$ are not closed under weak $W^{1,2}$-convergence. To illustrate these facts, let us consider for $\mu_1 \in \C$, the mapping  $u_{\mu_1} \in \widetilde{\mathcal{A}}_{\rm N}$ given by Proposition~\ref{prop:largeharm} and satisfying $\widetilde{E}_0(u_{\mu_1})= 6 \pi$. As $|\mu_1|\to \infty$, we have $u_{\mu_1} \rightharpoonup u_{\rm S} \in \widetilde{\mathcal{A}}_{\rm S}$ weakly in $W^{1,2}(\bbD)$, where $ u_{\rm S}$ is given by \eqref{eq:uS}. 
Note that $u_{\rm S}$ solves \eqref{eq:harmonic-bdvp} and satisfies $\widetilde E_0(u_{\rm S})=2\pi$. As the convergence is smooth away from the origin,  $\frac12 \abs{\nabla u_{\mu_1}}^2 dx \rightharpoonup \frac12 \abs{\nabla u_{\rm S}}^2 dx +4\pi \delta_0$ as measures on $\overline{\bb D}$. Finally, if $\mu_1/|\mu_1| \to e^{i \theta} $ as $|\mu_1| \to \infty$,  then  $u_{\mu_1}(z/\mu_1) \to \tilde{u}(z)$ strongly in $W_{\rm loc}^{1,2}(\C ;\bb S^4)$, where
\begin{equation}\label{eq:ustar}
	\tilde{u}(z):=  \left( \frac{1 - \abs{z}^2}{1+\abs{z}^2}  ,  \frac{ 2 e^{i\theta} z}{1+\abs{z}^2}, 0	 \right)
	\end{equation}		
is a finite energy harmonic $2$-sphere (a ``bubble''), $\tilde{u}: \C\cup \{\infty\} \simeq\bb S^2 \to \bb S^4$ with $\widetilde{E}_0(\tilde{u};\C)=4\pi$.
\end{remark}

To discuss the minimization problem \eqref{eq:minAS}, we rely on existing results in the literature \cite{JK,SU}, and we actually prove that  the minimality of  $u_{\rm S}$ holds even among non symmetric competitors.

\begin{proposition}
\label{ASminimization}
The map $u_{\rm S}$ given by \eqref{eq:uS}  is the unique minimizer of $\widetilde{E}_0$ in $W^{1,2}_{g_{\overline{H}}}(\bb D; \mathbb S^4)$. As a consequence, $\min_{u \in \widetilde{\mathcal{A}}_{\rm S}} \widetilde{E}_0(u)=2\pi$ and $u_{\rm S}$ is the unique minimizer of $\widetilde{E}_0$ over $\widetilde{\mathcal{A}}_{\rm S}$. 
\end{proposition}
\begin{proof}
We shall use the real coordinates $u=(u_0, \dots, u_4)\in  \bb R^5 \simeq \R \oplus \C \oplus \C$, and we shall denote by $\bb S^4_\pm=\{ u \in \bb S^4: u_0 \gtrless 0  \}$ the upper/lower open half spheres. 

First, we observe that $u_{\rm S}(\overline\bbD)\subset \bbS^4_-$. Since $u_{\rm S}$ is a smooth harmonic map (see Remark~\ref{bubbling}), we deduce from  \cite[Lemma 2.1]{SU} that $u_{\rm S}$ minimizes  $\widetilde{E}_0$ over the whole $W^{1,2}_{g_{\overline{H}}}(\bb D; \mathbb S^4)$.  Now we claim that $u_{\rm S}$ is actually the unique minimizer over  $W^{1,2}_{g_{\overline{H}}}(\bb D; \mathbb S^4)$. Since $\bbS^4_-$ is geodesically convex, the uniqueness result from \cite{JK} tells us that $u_{\rm S}$ is the unique (smooth) solution to \eqref{eq:harmonic-bdvp} whose range is strictly included in $\bbS^4_-$. Now, if $u\in W^{1,2}_{g_{\overline{H}}}(\bb D;\bb S^4)$ is any minimizer of $\widetilde{E}_0$, then  $u$ is a harmonic map smooth up to the boundary. Hence it suffices to show that $u(\overline{\bbD})\subset\bbS^4_-$ to conclude that $u=u_{\rm S}$. 
Assume by contradiction  that $u(z)=(u_0(z), u_1(z),\ldots, u_4(z))$ satisfies $u_0(z_*)=0$ for some $z_*\in \bb D$. Then the competitor $\tilde{u}(z):=(-|u_0(z)|, u_1(z), \ldots, u_4(z))$ belongs to $W^{1,2}_{g_{\overline{H}}}(\bbD;\bbS^4)$, and $\widetilde{E}_0(\tilde u)\leq \widetilde{E}_0(u)$. Thus, $\tilde{u}$ is also a minimizer, whence a harmonic map in $\bbD$ smooth up to the boundary. 
Then the function $v(z):=|u_0(z)|$ is a smooth  solution in $\D$ to $-\Delta v =|\nabla u|^2 v \geq 0$, with $v(z) = \frac{1}{2}$ on $\partial \bbD$. By the maximum principle, we have $v\geq \frac{1}{2}$ in $\bbD$, in contradiction with the assumption $v(z_*) = 0$. Therefore $u(\overline{\bbD})\subset\bbS^4_-$, leading to $u=u_{\rm S}$. 
Finally, since $u_{\rm S}\in \widetilde{\mathcal{A}}_{\rm S}$, it obviously follows that  $u_{\rm S}$ is the unique minimizer of $\widetilde{E}_0$ over $\widetilde{\mathcal{A}}_{\rm S}$, and a direct computation yields  $\widetilde{E}_0(u_{\rm S})=2\pi$ (see the the proof of Proposition \ref{prop:largeharm}).
\end{proof}

Concerning \eqref{eq:minAN}, we have the following result.

\begin{proposition}\label{ANminimization}
		It holds 
		\begin{equation}\label{eq:inf-energies-N}
			\min_{u \in \widetilde{\mathcal{A}}_{\rm N}} \widetilde{E}_0(u) = 6\pi\,,
		\end{equation}
and the minimum is attained at a map $u$ if and only if $u$ is of the form  \eqref{eq:class-harm}-\eqref{eq:D(z)}.
\end{proposition}

The proof of Proposition \ref{ANminimization} is postponed to the end of this subsection. In contrast with the proof of Proposition \ref{ASminimization}, we now have to overcome  the possible lack of compactness of minimizing sequences and  concentration of energy. To this purpose, we shall construct suitable minimizing sequences  considering a regularization of problem \eqref{eq:minAN}. This regularization is based on the following subclasses of $\widetilde{\mathcal{A}}_{\rm N}$,  
\begin{equation}
\label{eq:ANrho} 
\widetilde{\mathcal{A}}_{\rm N}^\rho:=\big\{ u \in \widetilde{\mathcal{A}}_{\rm N} : u= (1,0,0) \mbox{ a.e. on } \bbD_\rho \big\}
 \text{ with } 0<\rho<1 \, , \quad \widetilde{\mathcal{A}}^0_{\rm N}:=\bigcup_{0<\rho<1} \widetilde{\mathcal{A}}_{\rm N}^\rho \, .
\end{equation}
As opposed to $\widetilde{\mathcal{A}}_{\rm N}$, the subsets $\widetilde{\mathcal{A}}_{\rm N}^\rho$ are closed under  weak $W^{1,2}$-convergence. The following lemma relates those different classes and their corresponding minimization problems. 

\begin{lemma}
\label{ANdensity}
The following properties hold.
$\text{ }$
\begin{itemize}
\item[(i)] $\widetilde{\mathcal{A}}_{\rm N}^0$ is a strongly dense subset of $\widetilde{\mathcal{A}}_{\rm N}$ in $W^{1,2}(\bbD)$.
\vskip5pt

\item[(ii)]   $\displaystyle{\inf_{u \in \widetilde{\mathcal{A}}_{\rm N}} \widetilde{E}_0(u) = \inf_{u \in \widetilde{\mathcal{A}}_{\rm N}^0} \widetilde{E}_0(u)=\lim_{\rho \to 0} \inf_{u \in \widetilde{\mathcal{A}}_{\rm N}^\rho} \widetilde{E}_0(u)}$\,.
\vskip5pt

\item[(iii)] For each integer $n \geq 1$, the minimization problem
\begin{equation}\label{regminppbAN}
\min_{u\in \widetilde{\mathcal{A}}_{\rm N}^{\frac{1}{n}}}  \widetilde{E}_0(u)
\end{equation}
admits a solution. In addition, for any solution $u_n\in \widetilde{\mathcal{A}}_{\rm N}^{\frac{1}{n}}$, we have 
\begin{equation}\label{limvalueregminppbAN}
\displaystyle \lim_{n\to\infty}\widetilde{E}_0(u_n)=\inf_{u \in \widetilde{\mathcal{A}}_{\rm N}} \widetilde{E}_0(u) \,.
\end{equation}
\end{itemize}	
\end{lemma}
\begin{proof}
We start proving claim (i). Let us fix $u \in \widetilde{\mathcal{A}}_{\rm N}$ arbitrary. We aim to construct  
$u_\rho \in \widetilde{\mathcal{A}}_{\rm N}^\rho$ such that $u_\rho \to u$ strongly  in $W^{1,2}(\bbD)$ as $\rho \to 0$. Writing 
$u(re^{i\phi})=(f_0(r),f_1(r) e^{i\phi},f_2(r)e^{i2\phi})$, we first set 
\[
\tilde{u}_\rho(re^{i\phi}):=
\begin{cases}
\eo  & \text{if $r \in [0 , \rho]$} \, , \\[3pt]
\displaystyle\eo +\frac{r-\rho}{\sqrt{\rho}-\rho} \big(u(\sqrt{\rho}e^{i\phi})-\eo\big)  & \text{if $r \in [\rho, \sqrt{\rho}]$} \, , \\[10pt]
u(re^{i\phi})  & \text{if $r \in [\sqrt{\rho}, 1]$} \, .	
\end{cases}
\]
Then $\tilde{u}_\rho \in W^{1,2}_{{\rm sym}}(\bb D; \R\oplus\C\oplus\C)\cap C^0(\overline\bbD)$ and $\tilde{u}_\rho=g_{\overline{H}}$ on $\partial \bbD$. Moreover, 
$\tilde{u}_\rho \to u$ uniformly in~$\overline{\bb D}$, which implies that $|\tilde{u}_\rho| \to 1$ uniformly in $\overline{\bb D}$ as $\rho \to 0$. For $\rho>0$ small enough, we thus have $|\tilde{u}_\rho|\geq 1/2 $ in $\overline{\bbD}$, and we can define 
$$u_\rho(z):= \frac{\tilde{u}_\rho(z)}{|\tilde{u}_\rho(z)|}\,.$$
By construction, we have  
$u_\rho \in \widetilde{\mathcal{A}}_{\rm N}^\rho$, and $u_\rho \to u$ uniformly in $\overline{\bb D}$ as $\rho \to 0$. In addition, 
$$\widetilde{E}_0(u_\rho) = \widetilde{E}_0\big(u_\rho; \bb D_{\sqrt{\rho}}\setminus \bbD_\rho\big)+\widetilde{E}_0(u;\bb D \setminus \bb D_{\sqrt{\rho}}) \,,$$
and 
$$ \widetilde{E}_0\big(u_\rho; \bb D_{\sqrt{\rho}}\setminus \bbD_\rho\big)\leq C  \widetilde{E}_0\big(\tilde{u}_\rho; \bb D_{\sqrt{\rho}}\setminus \bbD_\rho\big)  \\
\leq C\big(|1-f_0(\sqrt{\rho})|^2+|f_1(\sqrt{\rho})|^2+|f_2(\sqrt{\rho})|^2\big) \mathop{\longrightarrow}\limits_{\rho\to 0} 0\,.
$$
Hence $\widetilde{E}_0(u_\rho) \to \widetilde{E}_0(u)$ as $\rho\to 0$, which implies that $u_\rho \to u$ strongly in $W^{1,2}(\bbD)$.  
 \vskip3pt
 
 Concerning (ii), the first equality is an obvious consequence of (i) since $\widetilde{E}_0$ is (strongly) $W^{1,2}$-continuous. 
Then we observe that $\rho \mapsto \inf_{ \widetilde{\mathcal{A}}_{\rm N}^\rho} \widetilde{E}_0$ is non decreasing. Therefore, 
$$ \inf_{\widetilde{\mathcal{A}}_{\rm N}^0} \widetilde{E}_0=\inf_{0<\rho<1} \inf_{\widetilde{\mathcal{A}}_{\rm N}^\rho} \widetilde{E}_0= \lim_{\rho \to 0}\inf_{\widetilde{\mathcal{A}}_{\rm N}^\rho} \widetilde{E}_0\,.$$
To prove claim (iii), we recall that $ \widetilde{\mathcal{A}}_{\rm N}^{\frac{1}{n}}$ is weakly $W^{1,2}$-closed. Hence existence of solutions to \eqref{regminppbAN} follows from the direct method of calculus of variations. Finally, \eqref{limvalueregminppbAN} is a consequence of~(ii) together with the monotonicity of $\rho \mapsto \inf_{ \widetilde{\mathcal{A}}_{\rm N}^\rho} \widetilde{E}_0$. 
\end{proof}

By the previous lemma, a sequence $\{u_n\}$ of solutions to \eqref{regminppbAN} provides a minimizing sequence for \eqref{eq:minAN}. 
In the next result, we provide the key step for the asymptotic analysis  of such a sequence.   

\begin{lemma}
\label{profile-extraction}
Let $\{ u_n\} \subset \widetilde{\mathcal{A}}_{\rm N}^0$ be such that $u_n$ solves  \eqref{regminppbAN} for every $n\geq 1$. 
Assume that, for some (not relabelled) subsequence, $u_n \rightharpoonup u_*$ weakly in $W^{1,2}(\bb D)$.  Then $u_* \in \widetilde{\mathcal{A}}^{\rm sym}_{ g_{\overline{H}}}(\bbD)$ and $u_*$ is a smooth harmonic map in $\overline{\bbD}$. Moreover, 
if $u_* \in \widetilde{\mathcal{A}}_{\rm S}$, then there exists a further (not relabelled) subsequence and $r_n\to 0^+$ such that $\tilde{u}_n(z):=u_n(r_n z)$ satisfies $\tilde{u}_n \rightharpoonup \tilde{u}$ weakly in $W^{1,2}_{\rm loc} (\C)$ for some equivariant nonconstant finite energy smooth harmonic map $\tilde{u} :\C \to \bb S^4$. 
\end{lemma}

\begin{proof}
Using maps of the form \eqref{eq:class-harm}-\eqref{eq:D(z)} as competitors, we infer from  Lemma \ref{ANdensity} that
\begin{equation}\label{energbdbubminseq2d}
\lim_{n \to \infty} \widetilde{E}_0(u_n)=\inf_{u\in \widetilde{\mathcal{A}}_{\rm N}} \widetilde{E}_0(u)\leq 6\pi\,.
\end{equation}
The class $\widetilde{\mathcal{A}}^{\rm sym}_{ g_{\overline{H}}}(\bbD)$ being weakly $W^{1,2}$-closed, we have $u_*\in \widetilde{\mathcal{A}}^{\rm sym}_{ g_{\overline{H}}}(\bbD)$.  By minimality, each $u_n$ is a harmonic map in $\bbD\setminus \overline{\bbD}_{1/n}$.  Since $u_n\rightharpoonup u_*$ weakly in $W^{1,2}(\bbD)$, it classically follows that $u_*$ is a (weakly) harmonic map in $\bbD\setminus\{0\}$, see e.g. \cite[Theorem~1, p.~50]{Evans}. Moreover, since $u_*$ belongs to $W^{1,2}(\bbD)$ and the set $\{0\}$ has zero capacity,  $u_*$ is actually a weakly harmonic map in the whole disc $\bbD$, and thus a smooth harmonic map in $\overline{\bbD}$ by regularity theory.  
\vskip3pt

We now assume that $u_* \in \widetilde{\mathcal{A}}_{\rm S}$. Recalling Lemma \ref{lemma:s1eq-emb}, we write 
\[ u_n(re^{i\phi})=:\big(f^{(n)}_0(r), f^{(n)}_1(r) e^{i\phi}, f^{(n)}_2(r) e^{i 2\phi} \big)\,\,\text{ and }\,\, u_*(re^{i\phi})=:(f^*_0(r), f^*_1(r) e^{i \phi}, f^*_2(r) e^{i2\phi}) \, , \]  	
so that 
\[   \big(f^{(n)}_0(0), f^{(n)}_1(0), f^{(n)}_2(0)\big)=(1,0,0) \quad \text{ and } \quad \big(f^*_0(0), f^*_1(0), f^*_2(0)\big)=(-1,0,0) \, . \] 
The functions $f^{(n)}_0$ and $ f^*_0$ are continuous in $[0,1]$ and taking values in $[-1,1]$ by the $\bbS^4$-constraint. 
 In addition, we have $f^*_0(r) \in (-1,1)$ for every $r \in (0,1]$. Indeed, assume by contradiction that $f^*_0(t)=\pm1$ for some $t\in(0,1)$. Then, $u_*=(\pm 1,0,0)$ on $\partial \bbD_t$ which implies that $u_*=(\pm 1,0,0)$ in $\bbD_t$ by 
 Lemaire's constancy theorem \cite{Le}. Then $u_* \equiv (\pm 1,0,0)$ in $\bb D$ by unique continuation, in contradiction with the boundary condition.

By Lemma~\ref{lemma:s1eq-emb}, $u_n \to u_*$ locally uniformly in $\overline{\bb D} \setminus\{0\}$, and thus  $f^{(n)}_0 \to f^*_0$ locally uniformly in~$(0,1]$. 	
Since $f^*_0(0)=-1$, we have 
\begin{equation}
	\label{radialminima}
\lim_{n\to\infty}\, \min_{[0,1]} f^{(n)}_0=-1\,.
\end{equation}	
Recalling that $f_0^{(n)}(1)=-\frac{1}{2}$ and $f_0^{(n)}(0)=1$, each function $f_0^{(n)}$ must vanish on the interval $[0,1]$. We can thus define 
$$r_n:=\min\big\{r\in[0,1]: f_0^{(n)}(r)=0\big\}\in(0,1)\,, $$
and 
$$r^{\rm min}_n:=\min\Big\{r\in[0,1]: f_0^{(n)}(r)=  \min_{[0,1]} f^{(n)}_0\Big\}\in(0,1)\,. $$
Since $u_n \in \widetilde{\mathcal{A}}_{\rm N}^{\frac{1}{n}} $ and $f_0^*(r)>-1$ for every $r>0$, we infer from \eqref{radialminima}  that
$$\frac{1}{n} < r_n < r_n^{\rm min}\mathop{\longrightarrow}\limits_{n\to\infty}0\,,$$
whence $r_n\to 0$. Combining Cauchy-Schwarz inequality and \eqref{eq:gradsquare} leads to 
\[ 1=|f_0^{(n)}(r_n)-f_0^{(n)}(1/n)| \leq\int_{1/n}^{r_n} \sqrt{r}|\partial_r f_0^{(n)}(r)| \frac{dr}{\sqrt{r}} \leq  \sqrt{\pi^{-1} \widetilde{E}_0(u_n)} \sqrt{ \log \left( n r_n\right)}\, .\]
From the energy bound in \eqref{energbdbubminseq2d}, we conclude that 
$$r_*:=\limsup_{n\to \infty} \frac{1}{n r_n} <1\,.$$
Now we set 
$$\tilde{u}_n(z):=u_n(r_n z)\,,$$ 
so that $\tilde{u}_n \in W^{1,2}_{\rm sym}(\bb D_{1/r_n};\bb S^4)$, $\tilde{u}_n = (1,0,0)$ in $\bbD_{1/nr_n}$, and $\tilde{u}_n$ is a harmonic map in the annulus 
$$\Omega_n:=\big\{ 1/{nr_n} <|z|< 1/r_n\big\}\,.$$ 
Setting $\tilde{u}_n(re^{i\phi})=:(\tilde{f}^{(n)}_0(r), \tilde{f}^{(n)}_1(r) e^{i\phi}, \tilde{f}^{(n)}_2(r) e^{i 2\phi} ) $,  we also have $\tilde{f}^{(n)}_0(1)=0$ by construction.

In view of \eqref{energbdbubminseq2d}, we have for every $r>0$,
\[
	\limsup_{n\to\infty} \widetilde{E}_0(\tilde{u}_n,\bb D_r)\leq \limsup_{n\to\infty} \widetilde{E}_0(\tilde{u}_n,\bb D_{1/r_n}) 
=\limsup_{n\to\infty} \widetilde{E}_0(u_n,\bb D)\leq 6\pi  \,.
\]
Therefore we can extract a (not relabelled) subsequence such that  $\tilde{u}_n \rightharpoonup \tilde{u}$ in $W^{1,2}_{\rm loc}(\C)$ for some equivariant map $\tilde{u} \in W^{1,2}_{\rm loc}(\C ;\bb S^4)$ satisfying $\widetilde{E}_0(\tilde{u};\C) \leq 6 \pi$ by lower semicontinuity of the Dirichlet energy. By Lemma~\ref{lemma:s1eq-emb} again,  $\tilde{u}_n \to \tilde{u}$ locally uniformly in $\C\setminus\{0\}$, so that $\tilde{f}^{(n)}_0 \to \tilde{f}_0$ locally uniformly in~$(0,\infty)$ where $\tilde u(re^{i\phi})=:( \tilde{f}_0(r), \tilde{f}_1(r)e^{i\phi}, \tilde{f}_2(r)e^{2i\phi})$. Then $\tilde f_0(1)=\lim_n\tilde{f}^{(n)}_0(1)=0$, and $|\tilde f_0(0)|=1$ by equivariance. In particular, $\tilde u$ is nonconstant.  

Extracting a further subsequence if necessary, we have $1/nr_n\to r_*$, so that  
$$\Omega_ n \to \Omega_*:=\{ |z|> r_*\}\,.$$ 
Arguing as above, we infer that $\tilde{u}$ is a weakly harmonic map in $\Omega_*$, and hence a classical (smooth)  harmonic map in $\Omega_*$. 
Next, we claim that $r_*=0$. Indeed, assume by contradiction that  $0<r_*<1$. Then  $\tilde{u}= (1,0,0)$ in  $\bbD_{r_*}$  since $\tilde{u}_n \to \tilde{u}$ locally uniformly on $\C \setminus\{0\}$. For $z \in \overline{ \bb D}_{1/r_*}\setminus \{0\}$, we now consider the inverted map $v(z):= \tilde{u}(1/\bar{z})$. By conformal invariance of $\widetilde{E}_0$,  we have $\widetilde{E}_0(v,\bb D_{1/r_*})=\widetilde{E}_0(\tilde{u}, \Omega_*) \leq 6\pi$,  and $v$ is a weakly harmonic map in $\bb D_{1/r_*}\setminus\{0\}$ satisfying $v = (1,0,0)$ on $\partial \bb D_{1/r_*}$. Again, since $\{0\} $ has a vanishing capacity, 
 $v$ is actually a weakly harmonic map in the whole disc $\bb D_{1/r_*}$, and thus a  smooth harmonic map in  $\overline{\bb D}_{1/r_*}$. Since $v = (1,0,0)$ on $\partial \bb D_{1/r_*}$, Lemaire's theorem \cite{Le} tells us that $v\equiv (1,0,0)$ in  $\bb D_{1/r_*}$.  Hence, $\tilde u\equiv (1,0,0)$ in $\Omega_*$, in contradiction with the fact that $\tilde f_0(1)=0$. 

Since $r_*=0$, $\tilde u$ is weakly harmonic in $\C \setminus\{0\}$, and thus  weakly harmonic in the whole $\C$ as argued above. Hence $\tilde u$ is a smooth, nonconstant,  equivariant harmonic map satisfying  $\widetilde{E}_0(\tilde{u};\C) \leq 6 \pi$. 
\end{proof}

\begin{proof}[Proof of Proposition \ref{ANminimization}]
Using  maps of the form \eqref{eq:class-harm}-\eqref{eq:D(z)} as competitors, we obtain 
\begin{equation}\label{dim1704}
\inf_{u \in \widetilde{\mathcal{A}}_{\rm N}} \widetilde{E}_0(u)\leq 6 \pi\,.
\end{equation}
We are going to show that equality actually holds, so that any map of the form \eqref{eq:class-harm}-\eqref{eq:D(z)} is a minimizer. Moreover, since any minimizer is a solution of \eqref{eq:harmonic-bdvp}, it must be of the form \eqref{eq:class-harm}-\eqref{eq:D(z)} by Proposition~\ref{prop:largeharm}, so that no other minimizers exist. 

Let us now consider a sequence $\{ u_n\} \subset \widetilde{\mathcal{A}}_{\rm N}^0$ such that $u_n$ solves  \eqref{regminppbAN} for every $n\geq 1$. In view of Lemma \ref{ANdensity},  $\{ u_n\} $ is a minimizing sequence for \eqref{eq:minAN}. To show that equality holds in \eqref{dim1704}, it thus suffices to prove that $\lim_{n} \widetilde{E}_0(u_n)=6\pi$. 
By construction, $\{u_n\}$ is bounded in $W^{1,2}(\bbD)$, so that we can find a (not relabelled) subsequence such that $u_n\rightharpoonup u_*$ weakly in $W^{1,2}(\bbD)$. By Lemma~\ref{profile-extraction}, $u_* \in \widetilde{\mathcal{A}}_{\overline{H}}^{\rm sym}$ is a smooth harmonic map in $\bbD$. 
\vskip3pt

We now distinguish between two  scenarios. 
\vskip3pt

\noindent {\it Case I. Compact case: $u_* \in \widetilde{\mathcal{A}}_{\rm N}$}. Under this assumption, we have   $\widetilde{E}_0(u_*)=6\pi$ by Proposition \ref{prop:largeharm}. In addition, by
weak lower semicontinuity of the Dirichlet energy, 
$$6\pi=\widetilde{E}_0(u_*) \leq \lim_{n\to \infty} \widetilde{E}_0(u_n)= \inf_{ \widetilde{\mathcal{A}}_{\rm N}} \widetilde{E}_0\leq 6\pi\,,$$
which proves \eqref{eq:inf-energies-N}. 
\vskip3pt

\noindent {\it Case II. Noncompact case: $u_* \in \widetilde{\mathcal{A}}_{\rm S}$.} Under this assumption, we have $\widetilde{E}_0(u_*)\geq 2\pi$ by Proposition~\ref{ASminimization}. In view of Lemma~\ref{profile-extraction}, there exists a (not relabelled) subsequence and $r_n\to 0$ such that the rescaled sequence $\tilde{u}_n(z):=u_n(r_n z)$ converges weakly in $W^{1,2}_{\rm loc}(\C)$ to an entire nonconstant equivariant smooth harmonic map $\tilde u$ of finite Dirichlet energy. Being of finite energy, $\tilde u$ extends to $\C\cup \{\infty\} \simeq \bb S^2$ to  an equivariant weakly harmonic map, and thus a smooth equivariant harmonic $2$-sphere into $\bbS^4$. By the classification result in \cite[Section 3]{DMP2}, we thus have 
$\widetilde{E}_0(\tilde{u},\C)\geq 4\pi$. 

Setting  $r_n'=\sqrt{r_n}\to 0$ and using the weak lower semicontinuity of the Dirichlet energy, we infer that 
\begin{equation*}
\begin{split}
6\pi&\geq \inf_{\widetilde{\mathcal{A}}_{\rm N}} \widetilde{E}_0=\lim_{n\to \infty} \widetilde{E}_0(u_n)\geq \liminf_{n\to \infty} \widetilde{E}_0(u_n; \bb D_{r_n'})+\liminf_{n \to \infty}\widetilde{E}_0(u_n; \bb D \setminus \overline{\bb D}_{r_n'})\\
 &\geq \liminf_{n\to \infty} \widetilde{E}_0(\tilde{u}_n; \bb D_{r_n'/r_n})+ \widetilde{E}_0(u_*) \geq \widetilde{E}_0(\tilde{u};\C)+\widetilde{E}_0(u_*) \geq 6\pi \, ,
\end{split}
\end{equation*}
which again proves \eqref{eq:inf-energies-N}. 
\end{proof}


\subsection{Uniaxiality vs Biaxiality in the 2D-LdG minimization}\label{sec:biax-esc}

In the light of the previous section, we now discuss for $\lambda > 0$ the variational problems
\begin{equation}\label{eq:minANlambda} 
		\min_{Q \in \mathcal{A}_{\rm N}} E_\lambda(Q)\,,  
\end{equation}		
and
\begin{equation}\label{eq:minASlambda} 
\min_{Q \in \mathcal{A}_{\rm S}} E_\lambda(Q) \,,
\end{equation}
where $E_\lambda$ is the 2D-LdG energy in  \eqref{eq:def-2d-energy}, and $\mathcal{A}_{\rm N}$, $\mathcal{A}_{\rm S}$ are the classes defined in \eqref{eq:AN}-\eqref{eq:AS}. 

Once again we rely in an essential way on the (isometric) identification $\mathcal{S}_0 \simeq \R \oplus \C \oplus \C $ and the induced correspondence $Q\simeq u$ between 
$Q$-tensor maps and $\R \oplus \C \oplus \C $-valued maps from Corollary~\ref{cor:dec}. Recalling that ${\rm tr}(Q^3)= 3 \det Q$ for every $Q\in\mathcal{S}_0$, we infer from \eqref{detQ} that 
\begin{equation}\label{calcWusec4.3}
W(Q)=\frac{1}{3\sqrt{6}}\bigg(1 - u_0\big(u_0^2+\frac32|u_1|^2-3|u_2|^2\big)-\frac{3\sqrt{3}}{2}{\rm Re} \big(u_1^2 \overline{u_2}\big)\bigg)=:\widetilde W(u)\quad\text{for $Q\simeq u$}\,.
\end{equation}
Setting, for $u\in W^{1,2}(\bbD;\bbS^4)$, 
$$\widetilde E_\lambda(u):=  \int_{\bbD} \frac12 |\nabla u|^2 + \lambda \widetilde W(u) \,dx\,,$$
we obtain 
$$E_\lambda(Q)=\widetilde E_\lambda(u) \quad \text{for $Q\simeq u$}\,.$$
If $Q\simeq u\in W^{1,2}_{\rm sym}(\bbD;\bbS^4)$ and $u(re^{i\phi})=(f_0(r),f_1(r)e^{i\phi},f_2(r)e^{i2\phi})$, then 
 \begin{equation}
\label{eq:equiv-energy-f-2d}
E_\lambda(Q)=\widetilde E_\lambda(u) =
 \pi \int_0^1\left( \abs{f^\prime}^2+ \frac{\abs{f_1}^2 + 4\abs{f_2}^2}{r^2} +2\lambda \frac{1-\widetilde{\beta}(f)}{3\sqrt{6}}    \right) \, rdr \, ,
  \end{equation}
with $f:=(f_0,f_1,f_2)$ and $\widetilde{\beta}(f)$ given in \eqref{eq:beta-f}.  Equivariant critical points $Q\simeq u$ of the energy functional $E_\lambda$ satisfy the following system of ODEs 
 \begin{equation}\label{eq:EL-symm-ODE}
	\left\{
		\begin{aligned}
			 f^{\prime\prime}_0 + \frac{1}{r} f^\prime_0  &= - |\nabla u|^2 f_0 +  \frac{\lambda}{\sqrt{6}} \left(\abs{f_2}^2-f_0^2 - \frac{1}{2}\abs{f_1}^2 +  \widetilde{\beta}(f) f_0 \right) \,,\\
			f^{\prime\prime}_1 + \frac{1}{r}  f^\prime_1  &= -  |\nabla u|^2 f_1 - \frac{1}{r^2} f_1 + \frac{\lambda}{\sqrt{6}} \left(- \sqrt3f_2 \overline{f_1}-f_0 f_1+\widetilde{\beta}(f)  f_1    \right) \,,\\
			 f^{\prime\prime}_2 + \frac{1}{r}  f^\prime_2  &= -  |\nabla u|^2 f_2 - \frac{4}{r^2}f_2 + \frac{\lambda}{\sqrt6} \left( -\frac{\sqrt3}2 f_1^2 +  2f_0 f_2+ \widetilde{\beta}(f) f_2 \right) \,,
		\end{aligned}
	\right.
\end{equation}
with $ |\nabla u|^2$ as in \eqref{eq:gradsquare} depending also on $f$ only.
\vskip5pt
	
In the sequel, our goal is to establish  existence/nonexistence of solutions to \eqref{eq:minANlambda}-\eqref{eq:minASlambda} starting from the gap phenomenon in Theorem~\ref{thm:gap-u}. In turn, we shall derive qualitative properties of minimizers of $E_\lambda$ in $\mathcal{A}_{\overline{H}}^{\rm sym}(\bbD)=\mathcal{A}_{\rm S} \cup \mathcal{A}_{\rm N}$. The main result, Theorem~\ref{2d-biaxial-escape}, is postponed to the end of the subsection. It reveals the nature of minimizers of $E_\lambda$ in $\mathcal{A}_{\overline{H}}^{\rm sym}(\bbD)$ as $\lambda$ varies. In particular, we shall see that biaxial escape occurs for reasons of energy efficiency.
\vskip5pt

 We start with the following proposition providing the complete description of solutions to \eqref{eq:minANlambda}.
 
 \begin{proposition}
 \label{ANlambdaminimization}	
 For all $\lambda > 0$, 
		\begin{equation}\label{eq:inf-lambdaenergies-N}
			\min_{Q \in \mathcal{A}_{\rm N}} E_\lambda(Q) = 6\pi\,,
		\end{equation}
and the minimum is attained at $Q\simeq u$ if and only if  $u(z)=g_{\overline{H}}(\pm z)$ with $g_{\overline{H}}$  given by \eqref{eq:Hbar} below.
\end{proposition}

The proof is essentially based on the following preliminary lemma  of independent interest.

\begin{lemma}
\label{uniaxiality}
Let $Q \in \mathcal{A}_{\rm N}$ with $Q\simeq u$ of the form \eqref{eq:class-harm}-\eqref{eq:D(z)}. Then $Q$ is  positively uniaxial if and only if $\mu_1=\pm \sqrt{3}$, that is $u(z)=g_{\overline{H}}(\pm z)$ where 	
\begin{equation}\label{eq:Hbar}
	g_{\overline{H}}(z) := \frac{1}{(1+|z|^2)^2} \left( 1 - 4 \abs{z}^2  +\abs{z}^4 , 
	 2 \sqrt{3} z \big( 1 - \abs{z}^2 \big), 2 \sqrt{3} z^2   \right) 
\end{equation}	
extends \eqref{eq:hor-anchoring} to $\bbD$. Moreover, if $H$ denotes the unit norm nematic hedgehog in \eqref{Hsphere}, then we have  $g_{\overline{H}} \simeq (H\circ\boldsymbol{\sigma}_2^{-1})$ where $\boldsymbol{\sigma}_2:\bbS^2\setminus\{(0,0,-1)\}\to\mathbb{C}$ is the stereographic projection from the south pole of $\bbS^2$. 
\end{lemma}

\begin{proof}
Let us fix $Q\in \bbS^4$ with $Q\simeq u=(u_0,u_1,u_2)$.  For a given $\theta\in\R$, we set  $\widetilde{Q}\simeq \widetilde{u}:=(u_0, \pm |u_1|e^{i\theta} , |u_2|e^{i2\theta})$.
From \eqref{calcWusec4.3}, we derive that  $W(Q) \geq W(\widetilde{Q})$ with equality if and only if $u_1^2 \overline{u_2}\geq 0$. Hence equality  holds if and only if 
 ${\rm Re}(u_1^2 \overline{u_2})=|u_1|^2 |u_2|$ and $Q=\widetilde{Q}$ for some $\theta \in \R$.  As a consequence, $W(Q)=0$ if and only if 
\begin{align*}
0&=1 - u_0\big(u_0^2+\frac32|u_1|^2-3|u_2|^2\big)-\frac{3\sqrt{3}}{2}|u_1|^2|u_2|\\ 
&= 1-u_0\bigg(\frac{3}{2}-\frac{1}{2}u_0^2-\frac{9}{2}|u_2|^2\bigg)-\frac{3\sqrt{3}}{2}\big(1-u_0^2-|u_2|^2\big)|u_2|\\
&=\frac{3}{2}\big(u_0+\sqrt{3}|u_2|-1\big)\big(u_0+\sqrt{3}|u_2|+2\big)\,,
\end{align*}
 where we have used that $|u|^2=u_0^2+|u_1|^2+|u_2|^2=1$. Hence, $W(Q)=0$ if and only if either $u_0+\sqrt{3}|u_2|=1$, or $u_0+\sqrt{3}|u_2|=-2$.
\vskip3pt

Let us now consider $Q \in \mathcal{A}_{\rm N}$ with $Q\simeq u$ a map of the form \eqref{eq:class-harm}-\eqref{eq:D(z)}, and $Q\simeq u$. 
If $W(Q)=0$ in $\bbD$, then $u_1^2 \overline{u_2}\geq 0$ in $\bbD$ which implies that  $\mu_1\in\R$. Since $u_0+\sqrt{3}|u_2|=1$ on $\partial\bbD$, we infer that  $u_0+\sqrt{3}|u_2|=1$ in $\bbD$ by continuity. Inserting   \eqref{eq:class-harm}-\eqref{eq:D(z)} in this equation leads to $\mu_1=\pm \sqrt{3}$. The other way around, if $\mu_1=\pm \sqrt 3$,  it is now easily seen that $W(Q)=0$ in $\bbD$. 
\end{proof}

\begin{proof}[Proof of Proposition \ref{ANlambdaminimization}]
Using $H\circ\boldsymbol{\sigma}_2^{-1}\simeq g_{\overline H}$ as a competitor, we infer from Proposition~\ref{ANminimization}  and Lemma~\ref{uniaxiality} that for any $Q\in \mathcal{A}_{\rm N}$ with $Q\simeq u$,  
\begin{equation}\label{dim1815}
E_\lambda(H\circ\boldsymbol{\sigma}_2^{-1})=E_0(H\circ\boldsymbol{\sigma}_2^{-1})=\widetilde{E}_0(g_{\overline{H}}) =6\pi \leq \widetilde E_0(u) \leq \widetilde E_\lambda(u)= E_\lambda(Q) \, .
\end{equation}
Hence \eqref{eq:inf-lambdaenergies-N} holds and $\overline{H}$ is a minimizer. On the other hand, if $Q\in \mathcal{A}_{\rm N}$ is a minimizer, then $E_\lambda(Q)=6\pi$ and all inequalities in \eqref{dim1815} are equalities. Hence $W(Q)\equiv 0$ and $\widetilde{E}_0(u)=6\pi$. Finally, combining again Proposition~\ref{ANminimization} with Lemma~\ref{uniaxiality}, we deduce that  $u(z)=g_{\overline{H}}(\pm z)$. 
\end{proof}
\begin{remark}
In the previous proof,  the characterization of uniaxial minimizers can be derived in a different way.  
Indeed, if a minimizer $Q$ is (positively) uniaxial, then it must be a minimizer over the restricted class of maps $\widetilde Q\in W^{1,2}_{\rm sym}(\bb D;\R P^2)$ with trace $\widetilde Q=\overline{H}$ on $\partial \bb D$. Combining with the fact that the mapping $\Pi:\bb S^2 \to \R P^2$,  
\[  \R \oplus \C \supset \bb S^2 \ni v \to \Pi(v)=\sqrt{\frac{3}{2}}\left(v \otimes v-\frac{1}{3} {\rm Id}\right) \in \R P^2 \subset \bb S^4 \, , \]
is an isometric two-fold cover with the result in \cite{BaZa}, one can lift any such $\widetilde{Q}$ to a map $\widetilde{v} \in W^{1,2}_{\rm sym}(\bb D;\bb S^2)$ with trace $\widetilde{v}(z)=(0, z)$ on  $\partial \bb D$ (equivariance of the lift being a consequence of its uniqueness when a lift at the boundary is chosen). Then we have $E_0(\widetilde{Q})=\frac{3}{2}\int_{\bbD}|\nabla \widetilde{v}|^2\,dx$. Thus, minimizing maps are of the form $Q=\Pi \circ v$ with $v(z)=\frac{1}{1+|z|^2}\left( \pm(1- |z|^2), 2z \right)$,  the unique minimizing harmonic maps in the class $W^{1,2}_{\rm sym}(\bb D;\bb S^2)$ with trace $(0,z)$ on $\partial\bbD$ (compare \cite[Section 3.1, equations (3.5)-(3.6)]{INSZ}).
\end{remark}

We now address problem  \eqref{eq:minASlambda}, and we begin with the dependence on $\lambda$ of the associated value. Existence of solutions will be the object of Proposition \ref{ASlambdaminimization}. 

\begin{proposition}
\label{lambdainfimum}
Setting 
\begin{equation}
\label{def:minASlambda} 
		\mathfrak{e}^*_\lambda:= \inf_{Q \in \mathcal{A}_{\rm S}} E_\lambda(Q) \, , 
\end{equation}
then $2\pi \leq \mathfrak{e}^*_\lambda \leq 10 \pi$ for every $\lambda \geq 0$, and the function $\lambda \mapsto \mathfrak{e}^*_\lambda$ is continuous and nondecreasing. 
In addition, there exists $\lambda^* \in \left[\frac{48 \sqrt{2}}{2\pi -3\sqrt{3}} ,   5^2 \cdot3^6 \cdot \frac{\sqrt{6}}{4} \pi^2\right] $ such that $\lambda\mapsto \mathfrak{e}^*_\lambda$ is strictly increasing in $[0,\lambda^*]$, $ \mathfrak{e}^*_0=2\pi$, and $\mathfrak{e}^*_\lambda = 10 \pi$ for $\lambda \geq \lambda^*$. 
\end{proposition}

To prove the proposition, we shall need the following technical lemma.

\begin{lemma}[Bubble insertion]
\label{sphereattaching}
For each $\rho \in (0,1)$ there exists $v_\rho \in \widetilde{\mathcal{A}}_{\rm S}$ such that $v_\rho \equiv (1,0,0)$ for $|z|\geq \rho$, and satisfying $\widetilde{E}_0(v_\rho) \to 4\pi$ as $\rho \to 0$. 
As a consequence, for each $u \in \widetilde{\mathcal{A}}_{\rm N}$, there exists $\{ w_\rho\} \subseteq \widetilde{\mathcal{A}}_{\rm S}$ such that  $w_\rho \rightharpoonup u$ weakly in $W^{1,2}(\bb D)$, $w_\rho \to u$ strongly  in $W^{1,2}_{\rm loc}(\overline{\bb D} \setminus \{0\})$, and $\widetilde{E}_0(w_\rho) \to \widetilde{E}_0(u)+4\pi$ as $\rho \to 0$.	
\end{lemma}
\begin{proof}
Define
$$\widehat{v}(z):=\frac{1}{|z|^2+1}\big(|z|^2-1, 2z,0 \big)\,,$$ 
so that $\widehat{v} \in W^{1,2}_{\rm loc}(\C;\bb S^4)$, $\widehat v$ is $\bbS^1$-equivariant, $\widehat{v}(0)=(-1,0,0)$, and $\widetilde{E}_0(\widehat{v},\C)=4 \pi$. We rescale the map $\widehat v$ setting, for $\rho\in(0,1)$, $\widehat{v}_\rho(z):=\widehat{v}(z/\rho^3)$. Then, 
$$\max_{|z|=\rho^2}|\widehat{v}_\rho(z)+\widehat{v}(0)|\mathop{\longrightarrow}\limits_{\rho\to 0}0 \quad\text{and}\quad \widetilde{E}_0(\widehat{v}_\rho,\bb D_{\rho^2})\mathop{\longrightarrow}\limits_{ \rho\to 0} 4 \pi\,.$$
Next we consider the linear interpolation between $\widehat{v}_\rho$ and $-\widehat{v}(0)$, 
\[
\widetilde{v}_\rho(re^{i\phi}):=
\begin{cases}
\widehat{v}_\rho(re^{i\phi})   & \text{if }|z| \leq \rho^2 \, , \\[5pt]
\displaystyle \widehat{v}_\rho (\rho^2 e^{i\phi}) -\frac{|z|-\rho^2}{\rho-\rho^2} \big(\widehat{v}(0)+\widehat{v}_\rho(\rho^2 e^{i \phi})\big)  & \text{if }|z| \in [\rho^2, \rho] \, ,   \\[5pt] 
-\widehat{v}(0)   & \text{if }|z| \in [\rho, 1] \, .	
\end{cases}
\]
Since $\max_{|z|=\rho^2}|\widehat{v}_\rho(z)+\widehat{v}(0)|\to 0$,  we have $|\widehat{v}_\rho(z)|\geq 1/2$ on $\partial\bbD_{\rho^2}$ for $\rho$ small enough. It allows us to define 
$$v_\rho:=\frac{\widehat v_\rho}{|\widehat v_\rho|}\in \widetilde{\mathcal{A}}_{\rm S} \,, $$
which satisfies $v_\rho(z)=(1,0,0)$ for $|z|\geq \rho$.  
Arguing in Lemma \ref{ANdensity}, we obtain 
$\widetilde{E}_0(v_\rho; \bb D_\rho \setminus \bb D_{\rho^2}) \to 0$, and consequently 
$$\lim_{\rho \to 0}\widetilde{E}_0(v_\rho)= \lim_{\rho \to 0}\widetilde{E}_0(v_\rho;\bb D_\rho) =\lim_{\rho \to 0} \widetilde{E}_0(\widehat{v}_\rho;\bb D_{\rho^2}) =4 \pi\,,$$ 
proving the first claim.
 \vskip3pt
 
To prove the second claim, we fix $u \in \widetilde{\mathcal{A}}_{\rm N}$ and we apply Lemma~\ref{ANdensity} (i) to obtain $u_\rho \in \widetilde{\mathcal{A}}_{\rm N}^\rho$ (see~\eqref{eq:ANrho}) such that  $u_\rho \to u$ strongly in $W^{1,2}(\bb D)$ as $\rho\to0$. Then  $\widetilde{E}_0(u_\rho)=\widetilde{E}_0(u_\rho;\bb D \setminus \bb D_{\rho})\to \widetilde{E}_0(u)$. 
 Finally, we set 
 \[
w_\rho(z):=
\begin{cases}
v_\rho(z)  & \text{if } |z| \leq \rho \, , \\
u_\rho(z)  & \text{if } |z| \in [\rho, 1] \, ,	
\end{cases}
\]
and it is straightforward to check that $w_\rho\in  \widetilde{\mathcal{A}}_{\rm S}$ has all the announced properties. 
\end{proof}

\begin{proof}[Proof of Proposition \ref{lambdainfimum}] 
First, we observe that for each $Q \in \mathcal{A}_{\rm S}$, the function $\lambda \mapsto E_\lambda(Q) \in \R$ is affine and nondecreasing. Hence $\lambda \mapsto \mathfrak{e}^*_\lambda$ is nondecreasing and concave, therefore continuous in~$(0,\infty)$. In view of Proposition \ref{ASminimization}, we have $\mathfrak{e}^*_0=2\pi=E_0(Q_{\rm S})$ with $Q_{\rm S}\simeq u_{\rm S}$ given by \eqref{eq:uS}. Consequently,
$$2\pi \leq \mathfrak{e}_\lambda^* \leq E_\lambda(Q_{\rm S})=\widetilde{E}_\lambda(u_{\rm S}) \mathop{\longrightarrow}\limits_{\lambda\to 0} \widetilde{E}_0(u_{\rm S})=2\pi\,,$$ 
so that continuity also holds at  $0$. 

Next we consider $H\circ\boldsymbol{\sigma}_2^{-1}$ with $H\circ\boldsymbol{\sigma}_2^{-1}\simeq g_{\overline{H}}$ as in \eqref{eq:Hbar}, and $w_\rho\in \widetilde{\mathcal{A}}_{\rm S}$ obtained by applying Lemma~\ref{sphereattaching} to  $g_{\overline{H}}$. Then, $\int_{\bb \D} \widetilde{W}(w_\rho) \,dx \to \int_{\bb \D} \widetilde{W}(g_{\overline{H}})\, dx=0$. If $Q_\rho\simeq w_\rho$, it follows from Proposition \ref{ANlambdaminimization} and Lemma~\ref{sphereattaching}  that
\[ \mathfrak{e}^*_\lambda \leq E_\lambda(Q_\rho)=\widetilde{E}_0(w_\rho)+\lambda \int_{\bb \D} \widetilde{W}(w_\rho) \,dx\mathop{\longrightarrow}\limits_{\rho\to 0} \widetilde{E}_0(g_{\overline{H}})+4\pi=10\pi \, .\]
By monotonicity we deduce that $2\pi = \mathfrak{e}^*_0  \leq \mathfrak{e}_\lambda^* \leq 10\pi$ for every $\lambda\geq 0$. 

An elementary calculation yields
\begin{equation}
\label{uSpotentialenergy}
\int_{\bb D} \widetilde{W}(u_{\rm S})\,dx=\frac{18}{\sqrt{6}}\int_{\bb D}\frac{(|z|^4-1)^2}{(|z|^4+3)^3}\,dz=3\sqrt{6}\pi \int_0^1 \frac{(t^2-1)^2}{(t^2+3)^3}\,dt=-\frac{\sqrt{6}}{4}\pi+\frac{\sqrt{2}}6 \pi^2 \, . 	\end{equation}
As a consequence, if $\mathfrak{e}^*_\lambda=10\pi$ for some $\lambda >0$, then
\[ 10 \pi=\mathfrak{e}^*_\lambda \leq E_\lambda(Q_{\rm S})=\widetilde{E}_\lambda(u_{\rm S})=2\pi+\lambda \left(  -\frac{\sqrt{6}}{4}\pi+\frac{\sqrt{2}}6 \pi^2 \right) \, , \]
which implies that $\lambda \geq \frac{48 \sqrt{2}}{2\pi -3\sqrt{3}}$.

To complete the proof, we are going to show that if $\lambda >5^2 \cdot 3^6\cdot \frac{\sqrt{6}}{4} \pi^2$, then $E_\lambda(Q)>10\pi$ for every $Q\in \mathcal{A}_{\rm S}$. As a consequence $\mathfrak{e}^*_\lambda=10 \pi$ for $\lambda > 5^2 \cdot 3^6 \cdot \frac{\sqrt{6}}{4} \pi^2 $, so that the conclusion follows by setting $\lambda^*:=\min \{\lambda \, : \, \mathfrak{e}^*_\lambda=10 \pi \}$ and noticing that $\lambda \mapsto \mathfrak{e}^*_\lambda$ is strictly increasing on $[0,\lambda^*]$ by concavity.
To derive the previous claim, we fix $Q\in \mathcal{A}_{\rm S}$ and we observe that ${\rm B}:=\widetilde{\beta}(Q)=\sqrt{6} {\rm tr} (Q^3)$ belongs to $W^{1,2}(\bb D) \cap C(\overline{\bb D})$ with 
\[|\nabla {\rm B}| \leq 3 \sqrt{6} \big( {\rm tr} (Q^4)\big)^{1/2} |\nabla Q|= 3\sqrt{3}  |\nabla Q| \, , \]
where we used that ${\rm tr} (Q^4)=\frac12$ for $Q\in \bb S^4 \subset \mathcal{S}_0$. 
From this last inequality and  Young's inequality, we deduce that 
\[E_\lambda (Q) \geq \int_\bb{D} \frac{1}{2\cdot3^3} |\nabla {\rm B}|^2+ \frac{\lambda}{3\sqrt{6}} \left( 1-{\rm B}\right) \, dx \geq \frac{\sqrt{\lambda}}{9}\cdot \frac{1}{(24)^{1/4}}  \int_{\bb D}\sqrt{1-{\rm B}}|\nabla {\rm B}| \,dx \, .\]
Since $Q\in \mathcal{A}_{\rm S}$, we have ${\rm B}(\overline{\bb D})=[-1,1]$, so that the coarea formula yields
\[E_\lambda (Q) \geq    \frac{\sqrt{\lambda}}{9}\cdot \frac{1}{(24)^{1/4}}  \int_{\bb D}\sqrt{1-{\rm B}}\,|\nabla {\rm B}| \,dx=  \frac{\sqrt{\lambda}}{9}\cdot \frac{1}{(24)^{1/4}}  \int_{-1}^1\sqrt{1-t}\, dt=\frac{4}{3^3}\cdot \frac{\sqrt{\lambda}}{6^{1/4}} \, , \]
and the conclusion follows.
\end{proof}

We are now ready to discuss existence of solutions for variational problem \eqref{eq:minASlambda} with $\lambda \geq 0$. The proof the proposition below is postponed as it requires two auxiliary results. 

\begin{proposition}
\label{ASlambdaminimization}
Let $\lambda^*>0$ be the constant defined in Proposition \ref{lambdainfimum}. The following holds.
 
 \begin{enumerate}
 \item[(i)] If $0 \leq \lambda <\lambda^*$, then $\mathfrak{e}^*_\lambda<10 \pi$ and there exists $Q_\lambda \in \mathcal{A}_{\rm S}$ solving \eqref{eq:minASlambda}. In addition, 
 $\beta_\lambda:=\widetilde{\beta} \circ Q_\lambda$ satisfies $\beta_\lambda(\overline{\bb D})=[-1,1]$.  
 \vskip5pt
 \item[(ii)] If $\lambda>\lambda^*$, then $\mathfrak{e}^*_\lambda=10 \pi$ and \eqref{eq:minASlambda} has no solution.
  \end{enumerate}
\end{proposition}

To solve problem  \eqref{eq:minASlambda}, we proceed as for \eqref{eq:minAN} constructing an  enhanced minimizing sequence for which the eventual lack of compactness is easy to describe. It rests on the following subclasses defined for $\rho\in(0,1)$ by 
\begin{equation}
\label{eq:ASrho} 
\widetilde{\mathcal{A}}_{\rm S}^\rho:=\big\{ u \in \widetilde{\mathcal{A}}_{\rm S} : u= (-1,0,0) \mbox{ a.e. in } \bb D_\rho \big\}
 \, ,  \qquad \widetilde{\mathcal{A}}^0_{\rm S}:=\bigcup_{0<\rho<1} \widetilde{\mathcal{A}}_{\rm S}^\rho \, .
\end{equation}
Note that, as for the class $\widetilde{\mathcal{A}}^\rho_{\rm N}$  in \eqref{eq:ANrho}, the subsets $\widetilde{\mathcal{A}}_{\rm S}^\rho$ are weakly $W^{1,2}$-closed for any $\rho \in (0,1)$. 
The following lemma 
is the analogue of Lemma~\ref{ANdensity} for $\widetilde{E}_\lambda$ restricted to $\widetilde{\mathcal{A}}_{\rm S}$, instead of  $\widetilde{E}_0$ restricted to~$\widetilde{\mathcal{A}}_{\rm N}$. The proof being completely similar, it is left to the reader.

\begin{lemma}
\label{ASdensity}
The following properties hold.
\begin{itemize}
\item[(i)] $\widetilde{\mathcal{A}}_{\rm S}^0$ is a strongly dense subset of $\widetilde{\mathcal{A}}_{\rm S}$ in $W^{1,2}(\bbD)$.
\vskip5pt

\item[(ii)] $\displaystyle{\inf_{u \in \widetilde{\mathcal{A}}_{\rm S}} \widetilde{E}_\lambda(u) = \inf_{u \in \widetilde{\mathcal{A}}_{\rm S}^0} \widetilde{E}_\lambda(u)=\lim_{\rho \to 0} \inf_{u \in \widetilde{\mathcal{A}}_{\rm S}^\rho} \widetilde{E}_\lambda(u)}$\,.
\vskip5pt
 
\item[(iii)] For each integer $n \geq 1$, the minimization problem
\begin{equation}\label{regminppbAS}
\min_{u\in \widetilde{\mathcal{A}}_{\rm S}^{\frac{1}{n}}}  \widetilde{E}_\lambda(u)
\end{equation}
admits a solution. In addition, for any solution $u_n\in \widetilde{\mathcal{A}}_{\rm S}^{\frac{1}{n}}$, we have 
\begin{equation}\label{limvalueregminppbAS}
\displaystyle \lim_{n\to\infty}\widetilde{E}_\lambda(u_n)=\inf_{u \in \widetilde{\mathcal{A}}_{\rm S}} \widetilde{E}_\lambda(u) =\mathfrak{e}^*_\lambda\,.
\end{equation}
\end{itemize}	
\end{lemma}

In the vein of Lemma~\ref{profile-extraction}, we now aim to describe the lack of compactness of the minimizing sequence $\{u_n\}$ constructed in Lemma~\ref{ASdensity}.  The proof has some similarities with the one of Lemma~\ref{profile-extraction}, and we concentrate on the main differences.  

\begin{lemma}
\label{ASprofile-extraction}
Let $\{ u_n\} \subset \widetilde{\mathcal{A}}_{\rm S}^0$ be such that $u_n$ solves \eqref{regminppbAS} for every $n\geq 1$.  
Assume that for some (not relabelled) subsequence, $u_n \rightharpoonup u_*$ weakly in $W^{1,2}(\bb D)$.  Then $u_* \in \widetilde{\mathcal{A}}^{\rm sym}_{ g_{\overline{H}}}(\bbD)$ and $u_*$ is a critical point of $\widetilde{E}_\lambda$. Moreover,  
if $u_* \in \widetilde{\mathcal{A}}_{\rm N}$, then there exist a further (not relabelled) subsequence and $r_n\to 0^+$ such that $\tilde{u}_n(z):=u_n(r_n z)$ satisfies $\tilde{u}_n \rightharpoonup \tilde{u}$ weakly in $W^{1,2}_{\rm loc} (\C)$ for some equivariant nonconstant finite energy smooth harmonic map $\tilde{u} :\C \to \bb S^4$. 
\end{lemma}

	\begin{proof}
	In view of Proposition~\ref{lambdainfimum} and Lemma~\ref{ASdensity} we have  
\begin{equation}\label{energbdimprminseqEla}	
	\lim_{n \to \infty} \widetilde{E}_\lambda(u_n)=\inf_{u\in \widetilde{\mathcal{A}}_{\rm S}} \widetilde{E}_\lambda(u)=\inf_{Q\in {\mathcal{A}}_{\rm S}} {E}_\lambda(Q) =\mathfrak{e}^*_\lambda\leq 10\pi \, , 
\end{equation}
	which is the key a priori bound to obtain compactness properties.

By equivariance, $u_n$ and $u_*$ write
\[ u_n(re^{i\phi})=\big(f^{(n)}_0(r), f^{(n)}_1(r) e^{i\phi}, f^{(n)}_2(r) e^{i 2\phi} \big) \, , \quad u_*(re^{i\phi})=\big(f^*_0(r), f^*_1(r) e^{i \phi}, f^*_2(r) e^{i2\phi}\big) \, ,  \]  	
where $f^{(n)}:=(f^{(n)}_0, f^{(n)}_1, f^{(n)}_2)$ and $f^{*}:=(f^{*}_0, f^{*}_1, f^{*}_2)$ are continuous on $[0,1]$  by Lemma \ref{lemma:s1eq-emb}.

By minimality of $u_n$, each $f^{(n)}$ solves \eqref{eq:EL-symm-ODE} in the interval $(1/n,1)$. As  a consequence, 
$f^*$ solves \eqref{eq:EL-symm-ODE} in $(0,1)$. Indeed, \eqref{energbdimprminseqEla} implies a $W^{2,1}_{\rm loc}((0,1])$-bound on the sequence  $\{f^{(n)}\}$, hence a $W^{1,\infty}_{\rm loc}((0,1])$-bound by Sobolev embedding. Back to the ODE   \eqref{eq:EL-symm-ODE}, it yields a $C^2_{\rm loc}((0,1])$ bound on $\{f^{(n)}\}$. This is then enough to 
  pass to the limit $n\to\infty$ in  \eqref{eq:EL-symm-ODE} for $r \in (0,1)$. Thus, since $f^*$ solves \eqref{eq:EL-symm-ODE} in $(0,1)$, $Q_*\simeq u_*$  is a weak solution to \eqref{MasterEq} in $\bb D \setminus \{0\}$. However, $Q_*$ being of finite energy and $\{0\}$ of zero capacity, $Q_*$ weakly solves \eqref{MasterEq} in the whole $\bb D$, i.e., $Q_*$ is a critical point of $E_\lambda$ or equivalently, $u_*$ is a critical point of $\widetilde{E}_\lambda$.

We  claim that $f^*_0(r)<1$ for every $r \in (0,1]$. To prove this claim, we argue as in the proof of Lemma \ref{equiv-extension}. Since $\widetilde{E}_0(u_*)<\infty$,  we infer from \eqref{eq:equiv-energy-f-2d} that there exists a sequence $r_k \downarrow 0$ satisfying $r_k^2 \abs{(f^*)^\prime (r_k)}^2+\abs{f^*_1(r_k)}^2+4\abs{f^*_2(r_k)}^2 \to 0$ as $k \to \infty$. 
Then we multiply \eqref{eq:EL-symm-ODE} by $r^2(f^*)^\prime$ and integrate between $r_k$ and a fixed $r\in(0,1)$. Using  $(f^*)^\prime \cdot f^* = 0$ and letting $k\to \infty$,  we obtain
\[ r^2 \abs{(f^*)^\prime (r)}^2=\abs{f^*_1(r)}^2+4\abs{f^*_2(r)}^2+\frac{2\lambda}{3\sqrt6}\big(1-\widetilde{\beta}(f^*(r))\big)r^2 \, . \]
Now assume by contradiction that $f^*_0(\bar{r})=1$ for some $\bar{r}\in (0,1)$. Then $f^*(\bar{r})=(1,0,0)$ because $|f^*(\bar{r})|=1$, and the previous identity yields $ (f^*)^\prime(\bar{r})=0$. By uniqueness of the Cauchy problem for \eqref{eq:EL-symm-ODE}, it follows that $f^*(r)= f^*(\bar{r})=(1,0,0)$ for every $r \in (0,1]$. However, since $u_* =g_{\overline{H}}$ on~$\partial\bbD$, we have $f^*(1)=(-1/2,0,\sqrt{3}/2)$, a contradiction.

Let us now assume that $u_*\in\widetilde{\mathcal{A}}_{\rm N}$. Since $u_n\in\widetilde{\mathcal{A}}_{\rm S}$, we have 
\[   (f^{(n)}_0(0), f^{(n)}_1(0), f^{(n)}_2(0))=(-1,0,0)\,  \text{ and } \,(f^*_0(0), f^*_1(0), f^*_2(0))=(1,0,0) \, . \]  
Arguing as in the proof of Lemma \ref{profile-extraction}, 
$f^{(n)}_0 \to f^*_0$ locally uniformly in $(0,1]$, and 
$$\lim_{n\to\infty}\max_{[0,1]} f^{(n)}_0 =1 \, . $$
Since $f^{(n)}_0(0)=-1$, each $f_0^{(n)}$ must vanish on $[0,1]$ by continuity, at least for $n$ large enough. This allows us to define 
$$r_n:= \min \big\{ r\in[0,1] : f^{(n)}_0(r)=0\big\}\in(0,1)\,,$$
and 
$$r_n^{\max}:= \min    	\big\{ r\in[0,1] :f^{(n)}_0(r) = \max_{[ 0,1]} f^{(n)}_0 \big\}\in(0,1) \, .$$
As in the proof of Lemma \ref{profile-extraction}, we have $1/n <r_n<r_n^{\max} \to 0$ as $n \to \infty$, and 
$$r_*:=\limsup_{n\to \infty} \frac{1}{nr_n} <1\,.$$
Now we set $\tilde{u}_n(z):=u_n(r_n z)$, so that $\tilde{u}_n \in W^{1,2}_{\rm sym}(\bb D_{1/r_n};\bb S^4)$, $\tilde{u}_n(z) = -\eo$ for $|z| \leq \frac{1}{nr_n}$, and 
$\tilde u_n$ is a critical point of $\widetilde{E}_{\lambda r_n^2}$ in the domain $\Omega_n:=\left\{ \frac{1}{nr_n} <|z|< 1/r_n\right\}$.
In addition, setting $\tilde{u}_n(re^{i\phi})=:(\tilde{f}^{(n)}_0(r), \tilde{f}^{(n)}_1(r) e^{i\phi}, \tilde{f}^{(n)}_2(r) e^{i 2\phi} ) $, $\tilde{f}^{(n)}_0$ satisfies $\tilde{f}^{(n)}_0(1)=0$ by our choice of $r_n$.

By \eqref{energbdimprminseqEla}, we have $\widetilde{E}_{\lambda r_n^2}(\tilde u_n,\bbD_{1/r_n})= \widetilde{E}_\lambda(u_n)\leq 10\pi$. Hence, we can find a (not relabelled) subsequence such that $\tilde{u}_n \rightharpoonup \tilde{u}$ weakly in $W^{1,2}_{\rm loc}(\C)$ for a limiting equivariant map $ \tilde{u}\in W^{1,2}_{\rm loc}(\C ;\bb S^4)$ satisfying $\widetilde{E}_0(\tilde{u};\C) \leq 10 \pi$. Since $\Omega_ n \to \Omega_*:=\{ |z|> r_*\}$, we obtain that $\tilde{u}$ is a weakly harmonic map in $\Omega_*$. Indeed, $\tilde f_n:=(\tilde{f}^{(n)}_0, \tilde{f}^{(n)}_1, \tilde{f}^{(n)}_2)$ satisfies \eqref{eq:EL-symm-ODE} with $\lambda r_n^2$ in place of $\lambda$.  
Using the energy bound and the ODE as above, we derive that $\tilde{f}^{(n)}$ is bounded in $C^2_{\rm loc}((r_*,\infty))$. Arguing again as above and since $\lambda r_n^2\to 0$, it implies that $\tilde u$ is a critical point of $\widetilde{E}_0$ in $\Omega_*$, i.e., $\tilde{u}$ is a weakly harmonic map in $\Omega_*$.  
The rest of the proof now follows exactly as in the proof of Lemma \ref{profile-extraction}: $\tilde{u}$ is nonconstant by the normalization $\tilde{f}_0(1)=0$, $r_*=0$ by Lemaire's theorem, and $\tilde{u}$ extends to a finite energy harmonic map in the whole $\C$.
\end{proof}

\begin{proof}[Proof of Proposition \ref{ASlambdaminimization}]
We start proving (i),  arguing as in the proof of Proposition \ref{ANminimization}. We thus assume that $\lambda<\lambda^*$. We consider $\{ u_n\} \subset \widetilde{\mathcal{A}}_{\rm S}^0$ be such that $u_n$ solves  \eqref{regminppbAS} for every $n\geq 1$. From Proposition~\ref{lambdainfimum} and Lemma~\ref{ASdensity}, we infer that $\widetilde{E}_\lambda(u_n)\to\mathfrak{e}^*_\lambda<10\pi$ as $n\to\infty$.  The sequence $\{u_n\} \subset \widetilde{\mathcal{A}}_{\rm S}$ being bounded, we can find a (not relabelled) subsequence such that $u_n \rightharpoonup u_*$ weakly in $W^{1,2}(\bbD)$.  
 By  Lemma~\ref{ASprofile-extraction},  $u_* \in \widetilde{\mathcal{A}}_{g_{\overline H}}^{\rm sym}(\bbD)$ and $u_*$ is   a critical point of $\widetilde{E}_\lambda$. 
 
Now we claim that $u_* \in \widetilde{\mathcal{A}}_{\rm S}$. Assuming this claim holds,  we have  
$\mathfrak{e}^*_\lambda \leq \widetilde{E}_\lambda(u_*) \leq \widetilde{E}_\lambda(u_n)$ and $ \liminf_n \widetilde{E}_\lambda(u_n)=\mathfrak{e}^*_\lambda$ 
by weak lower semicontinuity. Hence equality holds, and since $u_*\in \widetilde{\mathcal{A}}_{\rm S}$, we conclude that $Q_*\simeq u_*$ is a minimizer  for  \eqref{eq:minASlambda}. 
In addition, $\widetilde\beta\circ Q_*(0)=-1$ and $\widetilde\beta\circ Q_* \equiv 1$ on~$\partial\bbD$, so that $\widetilde\beta\circ Q_*(\overline\bbD)=[-1,1]$ by continuity (and  Lemma~\ref{lemma:s1eq-emb}).

To show that  $u_* \in \widetilde{\mathcal{A}}_{\rm S}$, we argue by contradiction assuming that $u_* \in \widetilde{\mathcal{A}}_{\rm N}$. According to Lemma~\ref{ASprofile-extraction}, there exist a further (not relabelled) subsequence and $r_n\to 0^+$ such that $\tilde{u}_n(z):=u_n(r_n z)$ satisfies $\tilde u_n\rightharpoonup \tilde u$  weakly in $W^{1,2}_{\rm loc} (\C)$ for some equivariant nonconstant finite energy smooth harmonic map $\tilde{u}$. Setting $r_n'=\sqrt{r_n}$ and rescaling variables, we derive by weak lower semicontinuity that 
\begin{multline} \label{lowerbounde*lambda}
10\pi> \mathfrak{e}^*_\lambda=\lim_{n\to \infty} \widetilde{E}_\lambda(u_n)\geq \liminf_{n\to \infty} \widetilde{E}_\lambda(u_n; \bb D_{r_n'})+\liminf_{n \to \infty}\widetilde{E}_\lambda(u_n; \bb D \setminus \overline{\bb D_{r_n'}}) \\
\geq  \liminf_{n\to \infty} \widetilde{E}_0(\tilde{u}_n; \bb D_{r_n'/r_n})+ \widetilde{E}_\lambda(u_*) \geq  E_0(\tilde{u};\C)+\widetilde{E}_\lambda(u_*) \geq 10 \pi \, ,
 \end{multline}
a contradiction. The last inequality above combines the inequality $\widetilde{E}_\lambda(u_*)\geq 6\pi$  
from Proposition~\ref{ANlambdaminimization}, with $E_0(\tilde{u};\C)\geq 4\pi$ from the classification result in \cite[Section 3]{DMP2} ($\tilde{u}$ being of finite energy, it extends to $\C\cup \{\infty\} \simeq \bb S^2$ as a nonconstant equivariant harmonic $2$-sphere into $\bbS^4$).  Hence, $u_* \in \widetilde{\mathcal{A}}_{\rm S}$ as claimed.
\vskip3pt

To prove (ii), we first observe that Proposition \ref{lambdainfimum} yields $\mathfrak{e}^*_\lambda=10\pi$ for $\lambda>\lambda_*$. Next we argue  by contradiction assuming that a minimizer $Q_\lambda \in \mathcal{A}_{\rm S}$ for \eqref{eq:minASlambda} exists for some $\lambda>\lambda^*$. Since $W(Q_\lambda(0))=W(-\eo)>0$, we have $\int_{\bb D}W(Q_\lambda)\,dx>0$. Therefore,  
\[ 10 \pi= \mathfrak{e}^*_{\lambda'} \leq E_{\lambda'}(Q_\lambda)<E_{\lambda'}(Q_\lambda)+(\lambda-\lambda')\int_{\bb D}W(Q_\lambda)\,dx=E_\lambda(Q_\lambda)=\mathfrak{e}^*_\lambda=10\pi \]
for every $\lambda^*<\lambda'<\lambda$, which gives the contradiction.
\end{proof}

\begin{remark}
\label{ASmin-uniqueness}
It is an open problem whether the solution  $Q_\lambda\simeq u_\lambda$ of  \eqref{eq:minASlambda} is unique or not for each $\lambda \in (0,\lambda^*)$. If $u_\lambda(re^{i\phi})=(f_0^\lambda(r), f^\lambda_1(r)e^{i\phi},f_2^\lambda(r)e^{i2\phi})$,  then choosing   as  competitors  $(f_0^\lambda(r),\pm|f^\lambda_1(r)| e^{i\phi},|f_2^\lambda(r)|e^{i2\phi})$ 
 implies  that $f_2^\lambda(r)\geq 0$ (since it is positive at the boundary), and $f_1^\lambda$ is real  with constant sign. As a consequence, either $f_1^\lambda \equiv 0$ and in turn $u_\lambda$ is unique at least for $\lambda$ small (see Theorem~\ref{uniqueminimizer}, Lemma~\ref{ulambda-convergence} and Lemma~\ref{lemma:s1eq-emb}), or $f_1^\lambda \not \equiv 0$ and both $\pm f_1^\lambda$ give rise to minimizers. 
\end{remark}

\begin{remark}
\label{existence}
The previous proof obviously breaks down in the limiting case $\lambda=\lambda^*$. In this case, it is unknown if a minimizer of \eqref{eq:minASlambda} exists, or if 
the minimizing sequence $\{u_n\}$ exhibits concentration of energy and bubbling-off of a harmonic sphere at the origin according to Lemma~\ref{ASprofile-extraction}. 
\end{remark}

We are finally in the position to discuss the global minimization of the energy $E_\lambda$ in the class~\eqref{eq:def-2d-adm}. To this purpose,  we define for $\lambda\geq 0$, 
\begin{equation}
\label{def:minElambda} 
		\mathfrak{e}_\lambda:=\inf_{Q \in \mathcal{A}^{\rm sym}_{\overline{H}}(\bbD)} E_\lambda(Q) \,,
\end{equation}
and we recall that the constant $\lambda^*>0$ is defined in Proposition \ref{lambdainfimum}, and $\mathfrak{e}^*_\lambda$ is given by \eqref{def:minASlambda}. 

\begin{proposition}
\label{energyprofile}
For every $\lambda\geq 0$, we have	$\mathfrak{e}_\lambda=\min \{6\pi, \mathfrak{e}^*_\lambda \}$ with $\mathfrak{e}^*_\lambda$ given by \eqref{def:minASlambda}, so that $\lambda\mapsto \mathfrak{e}_\lambda$ is nondecreasing, continuous, and  concave. Moreover, there exists $\lambda_* \in \left[\frac{24 \sqrt{2}}{2\pi -3\sqrt{3}} ,  3^8 \cdot \frac{\sqrt{6}}{4} \pi^2\right] $ with $\lambda_*<\lambda^*$, such that  $\lambda \mapsto \mathfrak{e}_\lambda$ is strictly increasing  in $  [0,\lambda_*]$, and $\mathfrak{e}_{\lambda}=6\pi$  for $\lambda \geq \lambda_*$.
\end{proposition}

\begin{proof}
	Recalling that $\mathcal{A}^{\rm sym}_{\overline{H}}(\bbD)=\mathcal{A}_{\rm S} \cup \mathcal{A}_{\rm N}$,  combining Proposition~\ref{ANminimization}, Proposition~\ref{ANlambdaminimization}, and Proposition~\ref{lambdainfimum}, we infer that $\mathfrak{e}_\lambda=\min \{6\pi, \mathfrak{e}^*_\lambda \}$ for every $\lambda \geq 0$. It is therefore continuous, concave, and nondecreasing. Choosing $\lambda_*$ to be  the unique solution to $\mathfrak{e}^*_\lambda=6\pi$, the rest of claim follows from Proposition~\ref{lambdainfimum}. 
By obvious modifications of the proof of Proposition \ref{lambdainfimum}, we obtain the announced lower and upper bounds on $\lambda_*$. 
\end{proof}

We are finally ready to prove the main result concerning 2D-minimization, i.e., to give the full proof of Theorem \ref{2d-biaxial-escape}.

\begin{proof}[Proof of Theorem \ref{2d-biaxial-escape}]
To prove (i) we argue as follows. According to Proposition~\ref{ANminimization} and Proposition~\ref{ANlambdaminimization}, the maps  $\overline Q$ are uniaxial and minimizing $E_\lambda$ over $\mathcal{A}_{\rm N}$ for every $\lambda \geq 0$ with $E_\lambda(\overline Q)=6\pi$. As a consequence, these maps are local minimizers of $E_\lambda$ in $\mathcal{A}^{\rm sym}_{\overline{H}}(\bbD)$ because in the decomposition $\mathcal{A}^{\rm sym}_{\overline{H}}(\bbD)=\mathcal{A}_{\rm S} \cup \mathcal{A}_{\rm N}$ into open and closed sets 
(see Lemma~\ref{lemma:s1eq-emb}). Finally,  combining Propositions \ref{lambdainfimum} and \ref{energyprofile} we have $\mathfrak{e}^*_\lambda>6\pi$ and $\mathfrak{e}_\lambda=6\pi$ for $\lambda>\lambda_*$, hence these maps are the absolute minimizers of $E_\lambda$ because of Proposition \ref{ANlambdaminimization}.

In a similar way, concerning (ii), existence of a minimizer (hence, of a local minimizer) $Q_\lambda$ in the class $\mathcal{A}_{\rm S}$ follows from Proposition~\ref{ASlambdaminimization}. Moreover, we have $\mathfrak{e}_\lambda=\mathfrak{e}^*_\lambda <6\pi$ for $\lambda <\lambda_*$, and therefore $Q_\lambda$ is a minimizer over $\mathcal{A}^{\rm sym}_{\overline{H}}(\bbD)$.  
Uniqueness  for $\lambda<\lambda_0$ and $\lambda_0>0$ small enough is proved in Theorem \ref{uniqueminimizer} in the Appendix. 

Finally, concerning (iii) we have $E_{\lambda_*}(u_{\lambda_*})=\mathfrak{e}^*_{\lambda_*}=6\pi=\mathfrak{e}_{\lambda_*}=E_{\lambda_*}(\bar{u})$ for $\lambda=\lambda_*$. Hence $\overline Q$ and $Q_{\lambda_*}$ are both  global minimizers over the class  $\mathcal{A}^{\rm sym}_{\overline{H}}(\bbD)$. 
\end{proof}

\begin{remark}
\label{biaxialescape}
According to Theorem \ref{2d-biaxial-escape}, a sharp transition occurs in the qualitative properties of energy minimizers of $E_\lambda$ over    $\mathcal{A}^{\rm sym}_{\overline{H}}(\bbD)$ for $\lambda$ close to the critical value $\lambda_*$. At $\lambda=\lambda_*$, coexistence of uniaxial and biaxial minimizers occurs. For $\lambda>\lambda_*$, the influence of the potential energy is so strong that it  forces the uniaxial character of energy minimizers (and the explicit form  \eqref{eq:Hbar}), although a biaxial locally minimizing configuration exists. For $\lambda<\lambda_*$, uniaxiality is no longer energetically  convenient as the effect of the potential energy gets weaker, and  minimizers are biaxial configurations  satisfying  $Q_\lambda(0)=-\eo$ and $\widetilde{\beta} \circ Q_\lambda (\overline{\bb D})=[-1,1]$. In this latter case, we see that although the boundary condition is topologically trivial in $\pi_1(\R P^2)$ and two uniaxial local minimizers exist, biaxial escape occurs for energy minimizers. This phenomenon of purely energetic nature is of definite interest, also in comparison with \cite{Can1} where  the biaxial escape mechanism is essentially deduced from topological nontriviality of the boundary data. 
\end{remark}



\section{Split minimizers in long cylinders}\label{SectSplit}

 In this section, we return to the analysis of the LdG energy $\mathcal{E}_\lambda$ in three space dimension. We shall discuss qualitative properties of minimizers of $\mathcal{E}_\lambda$ over $\mathcal{A}_{Q_{\rm b}}^{\rm sym}(\Omega)$ for specific choices of axisymetric domains $\Omega\subset\R^3$ and boundary data $Q_{\rm b}$. Namely, we consider throughout this section the homeotropic boundary data on $\partial\Omega$ as defined in \eqref{eq:radial-anchoring} for a domain $\Omega$ of ``cigar shape'', i.e., $\Omega=\mathfrak{C}^h_{\ell,\rho}$ is the smoothed cylinder from Definition \ref{def:smooth-cyl} in a regime where the height $h$ is large and  the width $\ell$ is small.

By means of an asymptotic analysis,  our aim is to show that in the regime of parameters for which the cylinders $\mathfrak{C}^h_{\ell,\rho}$ are very long and very thin (namely, $\ell\sqrt{\lambda} \ll 1$ and $h \gg \ell$), minimizers must be singular for reasons of energy efficiency. In addition, we shall see that an \emph{energy gap} occurs between minimizers and any smooth configuration. 
This singular behavior might be surprising since the homeotropic boundary condition admits smooth $\bbS^1$-equivariant extensions, and smoothness of minimizers can't be ruled out by some topological obstruction. 
This phenomenon is clearly reminiscent of the energy gap for harmonic maps into $\bb S^2$ first observed in \cite{HL}. By the presence of singularities, minimizers in this parameter regime are thus of   split type in the sense of \cite[Section~7]{DMP2}, and their regular biaxial sets $\{\beta=t\}$, $t \in (-1,1)$, contain topological spheres according to \cite[Theorem 1.5]{DMP2}.

	
\subsection{Global energy identities for minimizers}

We start with the following general lemma based on the partial regularity result from Theorems \ref{intregthm} \& \ref{bdrregthm}. It provides a key integral identity to derive monotonicity inequalities and rigidity results in the present and next section.
 
\begin{lemma}
\label{conservation-laws}
Let $\Omega\subset\R^3$ be a bounded and axisymmetric open set with boundary of class $C^3$, and let $Q_{\rm b}\in C^{1,1}(\partial\Omega;\mathbb{S}^4)$ be an $\mathbb{S}^1$-equivariant map. Let $Q$ be a minimizer of $\mathcal{E}_\lambda$ over $\mathcal{A}^{\rm sym}_{Q_{\rm b}}(\Omega)$, and $\Omega'\subset\R^3$ a  bounded axisymmetric  open set with boundary of class $C^1$ such that $\partial \Omega$ and $\partial \Omega'$ meet transversally and $\partial \Omega' \cap {\rm sing}(Q)=\emptyset$. For every 
vector field $V \in C^1(\R^3;\R^3)$, the following identity holds, 
\begin{multline}
\label{V-identity}
\int_{\Omega \cap \Omega'} \Big[ \Big( \frac12 \abs{\nabla Q}^2+\lambda W(Q)\Big) \rmdiv{V} -\sum_{i,j} (\partial_i Q:\partial_j Q) \partial_j V_i\Big] \, dx=\\
\int_{\partial (\Omega \cap \Omega')} \Big[ \Big( \frac12 \abs{\nabla Q}^2+\lambda W(Q)\Big)V \cdot \overrightarrow{n} - \Big(\frac{\partial Q}{\partial V}:\frac{\partial Q}{\partial \overrightarrow{n}} \Big)\Big]\, d\mathcal{H}^2 \, ,
\end{multline}
where $\overrightarrow{n}$ denotes the ($\mathcal{H}^2$-a.e. defined) outer unit normal along $\partial (\Omega \cap \Omega')$.  	
\end{lemma}

\begin{proof}
We shall derive \eqref{V-identity} through the Pohozaev multiplier argument, i.e., multiplying equation \eqref{MasterEq} by $V\cdot\nabla Q$ and integrating by parts the result. However, since ${\rm sing}(Q)$ might not be empty, we shall first  integrate on a punctured domain, removing finitely many balls centered at singular points, and then let the radius of this balls go to zero. 

Recalling that $Q\in C^\infty(\Omega\setminus{\rm sing}(Q))$,  the constraint $|Q|^2=1$ implies that  $Q:(V\cdot \nabla Q)= 0$ in $\overline\Omega\setminus{\rm sing}(Q)$. Hence, taking the scalar product of \eqref{MasterEq}  with $V\cdot \nabla Q$ in $\Omega\setminus{\rm sing}(Q)$ yields 
\begin{equation}
\label{ptV-identity1}
 \Delta Q : (V \cdot\nabla Q)=\lambda \nabla_{\rm tan}W(Q) :(V\cdot \nabla Q )	\, .
\end{equation}
Direct computations lead to
\begin{multline}\label{ptV-identity2}
\rmdiv \Big(\frac12 \abs{\nabla Q}^2 V \Big)= \frac12 \abs{\nabla Q}^2\rmdiv V+\sum_{i,j} (\partial_{ij} Q :\partial_j Q)V_i \\
=\frac12 \abs{\nabla Q}^2\rmdiv V - \sum_{i,j}(\partial_i Q:\partial_j Q)\partial_j V_i  +{\rm div}\big( (V\cdot \nabla Q) : \nabla Q \big)  -\Delta Q:(V\cdot \nabla Q)\, ,
\end{multline}
and
\begin{equation}
\label{ptV-identity3}
\rmdiv \big( W(Q) V\big)=W(Q)\rmdiv V  + \nabla_{\rm tan}W(Q) :(V \cdot \nabla Q)\, .	
\end{equation}
Combining \eqref{ptV-identity1}-\eqref{ptV-identity2}-\eqref{ptV-identity3}, we obtain the following equality in $\Omega\setminus{\rm sing}(Q)$, 
\begin{multline}
\label{ptV-identity4}
	\Big( \frac12 \abs{\nabla Q}^2+\lambda W(Q)\Big) \rmdiv{V}  -\sum_{i,j}\partial_j V_i \partial_i Q:\partial_j Q  \\
=	\rmdiv \Big[  \Big( \frac12 \abs{\nabla Q}^2+\lambda W(Q)\Big)V -\frac{\partial Q}{\partial V}:\nabla Q\Big] \, .
\end{multline}

If ${\rm sing}(Q)=\emptyset$, then \eqref{V-identity}  follows as in the general case, integrating by parts the right hand side of \eqref{ptV-identity4} over $\Omega \cap \Omega'$. So we may assume that 
${\rm sing}(Q) \neq \emptyset$.
Since ${\rm sing}(Q) \cap \partial \Omega'=\emptyset$, we can find $\sigma_0>0$  small enough  that the balls $\displaystyle{\{B_{2\sigma_0}(p)\}_{p \in {\rm sing}(Q)}}$ are disjoint  and $ B_{2\sigma_0}(p) \subset \Omega \cap \Omega'$ for each $p\in{\rm sing}(Q)\cap\Omega'$. For  $0<\sigma \leq \sigma_0$, we consider punctured domain 
$$\Omega_\sigma:=(\Omega \cap \Omega') \setminus \cup_{p \in {\rm sing}(Q) \cap \Omega'} \overline{B_{\sigma}(p)} $$ 
which is obviously a piecewise smooth domain with $\partial \Omega_\sigma= \partial(\Omega \cap \Omega') \cup \left(\cup_{p \in {\rm sing}(Q) \cap \Omega'} \partial B_{\sigma}(p) \right)$.

By Theorem \ref{bdrregthm},  we have $Q\in C^1(\overline{\Omega} \setminus \cup_{p \in {\rm sing}(Q)}B_\sigma(p))$ and $Q\in  C^\omega (\cup_{p \in {\rm sing}(Q)} \partial B_{\sigma}(p))$.  Hence  $\Delta Q\in L^\infty(\Omega\setminus \cup_{p \in {\rm sing}(Q)}B_\sigma(p))$ by equation  \eqref{MasterEq}. Since $\partial\Omega$ is of class $C^3$ and  $Q_{\rm b} \in C^{1,1}(\partial \Omega) \subset W^{3/2,2}(\partial \Omega)$, it follows from standard elliptic theory that $Q \in W^{2,2}(\Omega \setminus \cup_{p \in {\rm sing}(Q)} \overline{B_\sigma(p)})$ (see e.g. \cite[Theorem 8.12]{GilbTrud}). 
As a consequence, the vector field 
\[\Phi:=\Big( \frac12 \abs{\nabla Q}^2+\lambda W(Q)\Big) V   -\frac{\partial Q}{\partial V}:\nabla Q\]
satisfies $\Phi \in W^{1,2}(\Omega_\sigma;\R^3) \cap C(\overline{\Omega_\sigma};\R^3)$ for every $0<\sigma<\sigma_0$. By the divergence theorem (on a Lipschitz regular domain), we have 
\begin{equation}
\label{Phi-identity}
\int_{\Omega_\sigma} \rmdiv{\Phi}\, dx=\int_{\partial (\Omega \cap \Omega')}\Phi \cdot\overrightarrow{n} \, d\mathcal{H}^2-\sum_{p \in {\rm sing}(Q) \cap \Omega'}\int_{\partial B_\sigma(p)}\Phi \cdot \overrightarrow{n} \, d\mathcal{H}^2 \,,
\end{equation}
while \eqref{ptV-identity4} yields 
\[
\int_{\Omega \cap \Omega'} \Big[\Big( \frac12 \abs{\nabla Q}^2+\lambda W(Q)\Big) \rmdiv{V}  -\sum_{i,j}\partial_j V_i \partial_i Q:\partial_j Q \Big]\, dx= \lim_{\sigma \to 0} \int_{\Omega_\sigma} \rmdiv{\Phi}\, dx\,.
\]
Hence \eqref{V-identity} follows  once we prove that
\begin{equation}
\label{eq:residue}
\lim_{\sigma \to 0} 	\int_{\partial B_\sigma(p)}   \Big[ \Big( \frac12 \abs{\nabla Q}^2+\lambda W(Q)\Big)V \cdot \overrightarrow{n} - \frac{\partial Q}{\partial V}:\frac{\partial Q}{\partial \overrightarrow{n}} \Big]\, d\mathcal{H}^2=0 \qquad \forall p\in {\rm sing}(Q) \, . 
\end{equation}
Let us now fix an arbitrary point $p\in {\rm sing}(Q)$, that we may assume without loss of generality to be the origin, i.e., $p=0$. By Theorem \ref{intregthm}, there exists a  $0$-homogeneous harmonic map $Q_*$ which is smooth away from the origin, and an exponent $\nu>0$ such that $\| Q_\rho -Q_*\|_{C^2(\overline{B_2}\setminus B_1)}=O(\rho^\nu)$ as $\rho\to 0$ with $Q_\rho(x):=Q(\rho x)$. In addition, the explicit expression  in \eqref{formtangmaps} yields $|\nabla Q_*(x)|^2=2|x|^{-2}$ for $x\neq0$.
As a a consequence, we easily infer the following expansions as $x \to 0$,
\[ \abs{\frac{\partial Q}{\partial \overrightarrow{n}} }=o\left(|x|^{-1}\right) \, , \quad  \abs{\frac{\partial Q}{\partial V}}=O \left(|x|^{-1}\right) \, , \quad \abs{\nabla Q}^2=2|x|^{-2} \big(1+o(1)\big) \, , \;{ \text{and }}\; W(Q)=O(1) \, . \]
In particular,  $|V-V(0)|\abs{\nabla Q}^2=o \left(|x|^{-2}\right)$ by continuity of $V$. Since $\mathcal{H}^{2}(\partial B_\sigma)=O(\sigma^2)$, the previous expansions yield
\begin{multline}
\label{eq:residue2}
\lim_{\sigma \to 0} 	\int_{\partial B_\sigma}   \Big[\Big( \frac12 \abs{\nabla Q}^2+\lambda W(Q)\Big)V \cdot \overrightarrow{n}  - \frac{\partial Q}{\partial V}:\frac{\partial Q}{\partial \overrightarrow{n}} \Big]\, d\mathcal{H}^2\\
= \lim_{\sigma \to 0} 	\int_{\partial B_\sigma}   \left( \frac12 \abs{\nabla Q}^2+\lambda W(Q)\right) V \cdot \overrightarrow{n} \, d\mathcal{H}^2= \lim_{\sigma \to 0} 	\int_{\partial B_\sigma}   \frac12 \abs{\nabla Q}^2 V \cdot \overrightarrow{n}\, d\mathcal{H}^2\\
 = \lim_{\sigma \to 0} 	\int_{\partial B_\sigma}  \frac12 \abs{\nabla Q}^2 V(0) \cdot \overrightarrow{n} \, d\mathcal{H}^2= \lim_{\sigma \to 0} \frac{1}{\sigma^2}\int_{\partial B_\sigma} V(0)\cdot \overrightarrow{n} \, d\mathcal{H}^2=0 \, , 
\end{multline}
and the last equality  holds since $\int_{\partial B_\sigma}\overrightarrow{n}\,d\mathcal{H}^2=0$ for every $\sigma>0$.  
\end{proof}

With suitable choices of the vector field $V$ in the previous lemma, we obtain the following key identities in smoothed cylinders.

\begin{corollary}
\label{global-identities-vert}	
Let $\mathfrak{C}^h_{\ell,\rho}$ be a smoothed cylinder and $Q_{\rm b}$ its homeotropic boundary data given by~\eqref{eq:radial-anchoring}. If $Q$ is minimizing $\mathcal{E}_\lambda$ over  $\mathcal{A}_{Q_{\rm b}}^{\rm sym}{(\mathfrak{C}^h_{\ell,\rho})}$, then the following  identities hold.
\begin{enumerate}
\item[(i)] ({\sl radial energy identity})
   For every $\ell\leq r_1<r_2 \leq h-\rho$, we have
\begin{multline}
			\label{eq:rad-mon-form}
			 \frac{1}{r_1}\mathcal{E}_\lambda(Q, \mathfrak{C}^h_{\ell,\rho} \cap B_{r_1}) 	+ \int_{\mathfrak{C}^h_{\ell,\rho} \cap (B_{r_2} \setminus B_{r_1}) } \frac{1}{\abs{x}}\abs{\frac{\partial Q}{\partial \abs{x}}}^2 \,dx
			 + \int_{r_1}^{r_2} \frac{1}{r^2} \int_{\mathfrak{C}^h_{\ell,\rho} \cap B_r} 2\lambda W(Q) \,{\rm d}x \,dr
			 \\
				 = \frac{1}{r_2}\mathcal{E}_\lambda(Q, \mathfrak{C}^h_{\ell,\rho} \cap B_{r_2}) +\ell\int_{r_1}^{r_2}  \frac{1}{r^2} \int_{\partial \mathfrak{C}^h_{\ell,\rho} \cap B_r} \Big[   \frac12 \abs{\nabla_{\rm tan} Q_{\rm b}}^2+\lambda W(Q_{\rm b})  -\frac12 \abs{ \frac{\partial Q}{\partial \overrightarrow{n}}}^2  \Big]\, d\mathcal{H}^2\, d r .
			\end{multline}
\vskip5pt			
\item[(ii)] ({\sl horizontal energy identity}) For any $0\leq s \leq h-\rho$ such that $(0,0,\pm s)\not \in {\rm sing}(Q)$ we have
 \begin{multline}
 \label{eq:hor-ener-id}
  \ell \int_{\partial \mathfrak{C}^s_{\ell} \cap \{|x_3|<s \}}\Big[  \frac12 \abs{\nabla_{\rm tan} Q_{\rm b}}^2+\lambda W(Q_{\rm b})  -\frac12 \abs{ \frac{\partial Q}{\partial \overrightarrow{n}}}^2  \Big]\, d\mathcal{H}^2  \\
= 	\int_{\mathfrak{C}^s_{\ell}} \Big[   \abs{ \frac{\partial Q}{\partial x_3}}^2 +2 \lambda W(Q) \Big]\, dx+ \int_{\partial \mathfrak{C}^s_{\ell}\cap\{|x_3|=s\}} (x'\cdot \nabla_{x'} Q): \frac{\partial Q}{\partial \overrightarrow{n}}  \, d\mathcal{H}^2 \, ,
 \end{multline}
where $x=:(x^\prime,x_3)\in\R^2\times\R$. 
\vskip5pt

\item[(iii)] ({\sl vertical energy identity}) For every $t_1,t_2 \in [-h+\rho,h-\rho]$ such that $(0,0,t_i)\not\in{\rm sing}(Q)$ for $i=1,2$, we have
\begin{multline}
\label{eq:vert-ener-id} 
E_\lambda\big(Q(\, \cdot \, , t_1),\bb D_\ell\big)-\frac12
\int_{\mathfrak{C}^h_{\ell,\rho}\cap \{x_3=t_1\}} \abs{\frac{\partial Q}{\partial x_3}}^2 \, d\mathcal{H}^2 \\ 
=E_\lambda\big(Q(\, \cdot \, ,t_2),\bb D_\ell\big) 
-\frac12 \int_{\mathfrak{C}^h_{\ell,\rho}\cap \{x_3=t_2\}} \abs{\frac{\partial Q}{\partial x_3}}^2 \, d\mathcal{H}^2 \, ,
\end{multline}
where $E_\lambda$ is the 2D-LdG energy in \eqref{eq:def-2d-energy}. 
\end{enumerate}
\end{corollary}

\begin{proof}
{\it Proof of (i).} For all $r \in (\ell, h-\rho]$ except finitely many if ${\rm sing}(Q)\not=\emptyset$, $\Omega =\mathfrak{C}^h_{\ell,\rho}$ and $\Omega' = B_r$ satisfy the assumptions of Lemma~\ref{conservation-laws} that we use  with $V(x)=x$. Then $V=r\overrightarrow{n}$ on $\mathfrak{C}^h_{\ell,\rho} \cap\partial B_r$, and  $V\cdot \overrightarrow{n}=\ell$ on $\partial \mathfrak{C}^h_{\ell,\rho} \cap B_r$. Noticing that  
$$\frac{\partial Q}{\partial V}:\frac{\partial Q}{\partial \overrightarrow{n}}= \ell \abs{\frac{\partial Q}{\partial \overrightarrow{n}}}^2\quad\text{on }\partial \mathfrak{C}^h_{\ell,\rho} \cap B_r\,,$$ 
because $\displaystyle\frac{\partial Q}{\partial x_3} =0$ on $\partial \mathfrak{C}^h_{\ell,\rho} \cap B_r$, we infer from identity \eqref{V-identity}, 
 \begin{multline*}
 	\int_{\mathfrak{C}^h_{\ell,\rho} \cap B_r} \Big[ \Big( \frac12 \abs{\nabla Q}^2+\lambda W(Q)\Big) +2 \lambda W(Q) \Big]\, dx=
\ell\int_{\partial \mathfrak{C}^h_{\ell,\rho} \cap B_r} \Big[  \frac12 \abs{\nabla Q}^2+\lambda W(Q) -\abs{ \frac{\partial Q}{\partial \overrightarrow{n}}}^2 \Big]\, d\mathcal{H}^2 
\\
+ r\int_{\mathfrak{C}^h_{\ell,\rho} \cap \partial B_r} \Big[  \frac12 \abs{\nabla Q}^2+\lambda W(Q) -\abs{ \frac{\partial Q}{\partial \overrightarrow{n}}}^2 \Big]
 \, d\mathcal{H}^2  \, ,
 \end{multline*}
\noindent which rewrites as 
\begin{multline} 
\label{radial-identity-vert}
 \frac1r  \int_{\mathfrak{C}^h_{\ell,\rho} \cap \partial B_r} \abs{ \frac{\partial Q}{\partial \overrightarrow{n}}}^2 
 \, d\mathcal{H}^2 + \frac{1}{r^2}  \int_{\mathfrak{C}^h_{\ell,\rho} \cap B_r}   2\lambda W(Q) \, dx =\frac{d}{dr} \left\{ \frac1r \int_{\mathfrak{C}^h_{l,\rho} \cap B_r}  \frac12 \abs{\nabla Q}^2+\lambda W(Q) \, dx \right\}  \\ + \frac{\ell}{r^2} \int_{\partial \mathfrak{C}^h_{\ell,\rho} \cap B_r} \Big[   \frac12 \abs{\nabla_{\rm tan} Q_{\rm b}}^2+\lambda W(Q_{\rm b})  -\frac12 \abs{ \frac{\partial Q}{\partial \overrightarrow{n}}}^2  \Big]\, d\mathcal{H}^2 \,\, . 
 \end{multline}
Integrating \eqref{radial-identity-vert} between $r_1$ and $r_2$ the conclusion follows.
\vskip5pt

\noindent{\it Proof of (ii).} We apply Lemma \ref{conservation-laws} with $\Omega=\mathfrak{C}^h_{\ell,\rho}$ and $\Omega'=\mathfrak{C}^s_{2\ell,\rho}$ for $s<h-\rho$, so that $\Omega\cap\Omega^\prime=\mathfrak{C}^s_{\ell}$. Choosing  $V(x)=(x',0)$, we notice that
$V\cdot \overrightarrow{n}=0$ on $\partial \mathfrak{C}^s_{\ell}\cap\{|x_3|=s\}$, and  $V\cdot  \overrightarrow{n}=\ell$ on $\partial \mathfrak{C}^s_{\ell} \cap \{|x_3|<s\}$.  
Using that  
 $\frac{\partial Q}{\partial V}:\frac{\partial Q}{\partial \overrightarrow{n}}=  \ell\abs{\frac{\partial Q}{\partial \overrightarrow{n}}}^2$ on $\partial \mathfrak{C}^s_{\ell} \cap \{|x_3|<s\}$,  
 we arrive at \eqref{eq:hor-ener-id} directly from identity \eqref{V-identity}. 
 \vskip5pt
 
\noindent{\it Proof of (iii).} We assume that $t_1<t_2$ and we apply Lemma~\ref{conservation-laws} with the domains $\Omega=\mathfrak{C}^h_{\ell,\rho}$  and $\Omega'=\mathfrak{C}^{(t_2-t_1)/2}_{2\ell-\rho}(0,0,(t_1+t_2)/2)$, so that $\Omega\cap\Omega^\prime=\mathfrak{C}^{h}_{\ell,\rho}\cap\{t_1<x_3<t_2\} $. We choose the constant vector field $V(x)=(0,0,1)$ which satisfies  
 $V= \overrightarrow{n}$ on $\mathfrak{C}^h_{\ell,\rho} \cap\{x_3=t_2\}$, $V=- \overrightarrow{n}$ on $\mathfrak{C}^h_{\ell,\rho} \cap\{x_3=t_1\}$, and  $V\cdot \overrightarrow{n}=0$ on $\partial \mathfrak{C}^h_{\ell,\rho} \cap \{t_1<x_3<t_2\}$. Using that $\partial_3 Q_{\rm b} \equiv 0$ on $\partial \mathfrak{C}^h_{\ell,\rho} \cap \{t_1<x_3<t_2\}$, we derive \eqref{eq:vert-ener-id} once again directly from \eqref{V-identity}. 
\end{proof}

\begin{remark}
\label{shiftedballs1}
It is straightforward to check that identity \eqref{eq:rad-mon-form} still holds for a ball $B_r(p)$ instead of $B_r$, whenever $p=(0,0,z) \in \Omega$, $|z|<h-\rho$, and $\ell \leq r_1 <r_2 \leq h-\rho-|z|$. 	
\end{remark}


\subsection{A priori bounds and local compactness}
 
In this subsection, we derive the necessary local boundedness and compactness properties needed in the asymptotic analysis of minimizers for cylinders of divergent height.

The following result is the fundamental tool to obtain local uniform energy bounds for energy minimizing configurations. 

\begin{proposition}
\label{local-energy-bound}
Let $\mathfrak{C}^h_{\ell,\rho}$ be a smoothed cylinder with $h-\rho>\sqrt{2}\ell$, and  
$Q_{\rm b}$ its homeotropic boundary data given by~\eqref{eq:radial-anchoring}.
If $Q$ is  minimizing  $\mathcal{E}_\lambda$ over $\mathcal{A}_{Q_{\rm b}}^{\rm sym}(\mathfrak{C}^h_{\ell,\rho})$, then 
\begin{multline}
			\label{eq:local-bounds-long-cyl}
\Big(1-\sqrt{2}\ell\Big( \frac{1}{r_1}-\frac{1}{r_2} \Big) \Big)	 \frac{1}{r_1}\mathcal{E}_\lambda\big(Q,\mathfrak{C}^h_{\ell,\rho} \cap B_{r_1}\big) 	+ \int_{\mathfrak{C}^h_{\ell ,\rho} \cap (B_{r_2} \setminus B_{r_1}) } \frac{1}{\abs{x}}\abs{\frac{\partial Q}{\partial \abs{x}}}^2 \,{\rm d}x 
			 \\
				 \leq \frac{1}{r_2}\mathcal{E}_\lambda\big(Q, \mathfrak{C}^h_{\ell ,\rho} \cap B_{r_2}\big) +\frac{3}{r_1}  \int_{ \mathfrak{C}^h_{\ell ,\rho} \cap B_{r_2}}  \abs{\frac{\partial Q}{\partial x_3} }^2{\rm d}x 
			\end{multline}
for every $\sqrt{2}\ell \leq r_1 < r_2 \leq h-\rho$. 
\end{proposition}

\begin{proof}
For $\sqrt{2}\ell \leq r_1 <r \leq r_2 \leq h-\rho$, we set 
$$s_1:=\sqrt{r_1^2-\ell^2}\in(\ell,h-\rho)\,,\;s:=\sqrt{r^2-\ell^2}\in(\ell,h-\rho)\,,\;s_2:=\sqrt{r_2^2-\ell^2}\in(\ell,h-\rho)\,,$$
and we assume that  $(0,0,\pm s) \not\in{\rm sing}(Q)$. By \eqref{eq:hor-ener-id} and Young's inequality, we estimate 
\begin{multline}
\label{eq:hor-ener-ineq1}
	 - \int_{\mathfrak{C}^s_{\ell}} 2\lambda W(Q) \,{\rm d}x   +\ell \int_{\partial \mathfrak{C}^s_{\ell} \cap \{|x_3|<s\}} \Big[   \frac12 \abs{\nabla_{\rm tan} Q_{\rm b}}^2+\lambda W(Q_{\rm b})  -\frac12 \abs{ \frac{\partial Q}{\partial \overrightarrow{n}}}^2  \Big]\, d\mathcal{H}^2\\
	 	=\int_{\mathfrak{C}^s_{\ell}}  \abs{\frac{\partial Q}{\partial x_3}}^2 \, dx+ \int_{\partial\mathfrak{C}^s_{\ell} \cap\{x_3=s\}} (x'\cdot \nabla_{x'} Q): \frac{\partial Q}{\partial \overrightarrow{n}}  \, d\mathcal{H}^2 \\
	 \leq 	\int_{\mathfrak{C}^s_{\ell}}  \abs{ \frac{\partial Q}{\partial x_3} }^2 \, dx+\ell \int_{\partial\mathfrak{C}^s_{\ell} \cap\{|x_3|=s\}} \frac12 \abs{\nabla_{x'} Q}^2  \, d\mathcal{H}^2+\frac{\ell}{2} \int_{\partial\mathfrak{C}^s_{\ell} \cap \{|x_3|=s\}} \abs{ \frac{\partial Q}{\partial x_3} }^2  \, d\mathcal{H}^2 \, .
	 \end{multline}
Averaging \eqref{eq:vert-ener-id} over $t_2\in [-s_1, s_1]$, we derive that for any $t\in [-h+\rho,h-\rho]$ such that $(0,0,t)\not\in{\rm sing(Q)}$,   
\begin{multline}
\label{eq:vert-ener-id-av} 
E_\lambda\big(Q(\, \cdot \, , t),\bb D_\ell\big)\leq\frac12
\int_{\mathfrak{C}^h_{\ell,\rho}\cap \{x_3=t\}} \abs{\frac{\partial Q}{\partial x_3}}^2 \, d\mathcal{H}^2 
+ \frac{1}{2s_1} \mathcal{E}_\lambda(Q;\mathfrak{C}^{s_1}_{\ell}) \\
\leq \frac12
\int_{\mathfrak{C}^h_{\ell,\rho}\cap \{x_3=t\}} \abs{\frac{\partial Q}{\partial x_3}}^2 \, d\mathcal{H}^2 + \frac1{\sqrt{2} r_1} \mathcal{E}_\lambda\big(Q,\mathfrak{C}^h_{\ell,\rho} \cap B_{r_1}\big)\, ,
\end{multline}
using $\mathfrak{C}^{s_1}_{\ell} \subset \mathfrak{C}^h_{\ell,\rho}\cap B_{r_1}$ in the last inequality. Summing now \eqref{eq:vert-ener-id-av} over $t\in\{\pm s\}$ yields
\begin{multline}\label{eqcorrcyl}
\ell\int_{\mathfrak{C}^h_{\ell  ,\rho} \cap\{ |x_3|=s\}} \frac12 \abs{\nabla_{x'} Q}^2  \, d\mathcal{H}^2\leq \ell\Big(  E_\lambda\big(Q(\, \cdot \, , s\big),\bb D_\ell  ) + E_\lambda\big(Q(\, \cdot \, , -s),\bb D_\ell \big ) \Big) 
 \\
\leq \frac \ell2
\int_{\mathfrak{C}^h_{\ell,\rho}\cap \{|x_3|=s\}} \abs{\frac{\partial Q}{\partial x_3}}^2 \, d\mathcal{H}^2 + \frac{\sqrt{2}\ell}{r_1} \mathcal{E}_\lambda\big(Q,  \mathfrak{C}^h_{\ell,\rho} \cap B_{r_1}\big) \, .
\end{multline}
Noticing that $\mathfrak{C}^h_{\ell,\rho}\cap \{|x_3|=s\}=\partial\mathfrak{C}^s_{\ell}\cap \{|x_3|=s\}$, we combine \eqref{eq:hor-ener-ineq1} with \eqref{eqcorrcyl} to obtain 
\begin{multline}
\label{eq:hor-ener-ineq2}
	 - \int_{\mathfrak{C}^s_{\ell}} 2\lambda W(Q) \,{\rm d}x   +\ell \int_{\partial \mathfrak{C}^s_{\ell} \cap \{|x_3|<s\}} \Big[   \frac12 \abs{\nabla_{\rm tan} Q_{\rm b}}^2+\lambda W(Q_{\rm b})  -\frac12 \abs{ \frac{\partial Q}{\partial \overrightarrow{n}}}^2  \Big]\, d\mathcal{H}^2
	 	\\
	 	\leq \int_{\mathfrak{C}^s_{\ell}}  \abs{ \frac{\partial Q}{\partial x_3} }^2 \, dx+\ell \int_{\partial\mathfrak{C}^s_{\ell} \cap \{|x_3|=s\}} \abs{\frac{\partial Q}{\partial x_3} }^2  \, d\mathcal{H}^2  +\frac{\sqrt{2}\ell}{r_1} \mathcal{E}_\lambda(Q, \mathfrak{C}^h_{{{\ell}},\rho} \cap B_{r_1}) \, .
	 \end{multline}
Next we observe that $\partial \mathfrak{C}^s_{\ell} \cap \{|x_3|<s\} =\partial \mathfrak{C}^h_{\ell,\rho} \cap B_r$ and
$\mathfrak{C}^s_{\ell} \subset \mathfrak{C}^h_{\ell,\rho} \cap B_r$. Then, multiplying \eqref{eq:hor-ener-ineq2} by $1/r^2$, integrating between $r_1$ and $r_2$, and then adding the resulting inequality   
to \eqref{eq:rad-mon-form} (term-by-term),  we obtain
			\begin{multline}
			\label{eq:rad-mon-form2}
\left(1- \sqrt{2}\ell \left( \frac{1}{r_1}-\frac{1}{r_2} \right) \right)			 \frac{1}{r_1}\mathcal{E}_\lambda(Q, \mathfrak{C}^h_{\ell,\rho} \cap B_{r_1}) 	+ \int_{\mathfrak{C}^h_{\ell,\rho} \cap (B_{r_2} \setminus B_{r_1}) } \frac{1}{\abs{x}}\abs{\frac{\partial Q}{\partial \abs{x}}}^2 \,dx\\	
\leq \frac{1}{r_2}\mathcal{E}_\lambda(Q, \mathfrak{C}^h_{\ell,\rho} \cap B_{r_2})
				  +\int_{r_1}^{r_2}  \frac{1}{r^2}  	 \Bigg(	\int_{\mathfrak{C}^{s(r)}_{\ell}}  \abs{ \frac{\partial Q}{\partial x_3} }^2 \, dx
				 +\ell \int_{\partial\mathfrak{C}^{s(r)}_{\ell} \cap \{|x_3|=s(r)\}} \abs{\frac{\partial Q}{\partial x_3} }^2  \, d\mathcal{H}^2 \Bigg) \, dr \, ,
			\end{multline}
where we write $s(r):=\sqrt{r^2-\ell^2}$. 
Since $\mathfrak{C}^{s(r)}_{\ell}\subset \mathfrak{C}^{h}_{\ell,\rho}\cap B_{r_2}$ for every $r\in(r_1,r_2)$, we obtain by a change of variable,  
\begin{align}
\nonumber\int_{r_1}^{r_2}  \frac{1}{r^2} \Big(	&\int_{\mathfrak{C}^{s(r)}_{\ell}} \abs{ \frac{\partial Q}{\partial x_3} }^2 \, dx+\ell \int_{\partial\mathfrak{C}^{s(r)}_{\ell} \cap\{|x_3|=s(r)\}} \abs{\frac{\partial Q}{\partial x_3} }^2  \, d\mathcal{H}^2 \Big) \, dr  \\
\nonumber& \leq  \Big(\frac{1}{r_1}-\frac{1}{r_2}\Big) \int_{ \mathfrak{C}^h_{\ell,\rho} \cap B_{r_2}}  \abs{\frac{\partial Q}{\partial x_3} }^2\,dx + \ell \int_{s_1}^{s_2}   \frac{s}{(s^2+\ell^2)^{3/2}} \Big(\int_{\partial\mathfrak{C}^{s}_{\ell} \cap\{|x_3|=s\}} \abs{\frac{\partial Q}{\partial x_3} }^2  \, d\mathcal{H}^2 \Big) \, ds \\
\nonumber &\leq \Big(\frac{1}{r_1}-\frac{1}{r_2}\Big) \int_{ \mathfrak{C}^h_{\ell,\rho} \cap B_{r_2}}  \abs{\frac{\partial Q}{\partial x_3} }^2\,dx +\frac{\ell}{s_1^2}\int_{\mathfrak{C}^{s_2}_{\ell} \setminus \mathfrak{C}^{s_1}_{\ell} } \abs{\frac{\partial Q}{\partial x_3} }^2  \, dx \\
 \label{eq:rad-mon-form3} &\leq \frac{3}{r_1} \int_{ \mathfrak{C}^h_{\ell,\rho} \cap B_{r_2}}  \abs{\frac{\partial Q}{\partial x_3} }^2\,dx \,.
\end{align}
Combining \eqref{eq:rad-mon-form3} with \eqref{eq:rad-mon-form2}, the conclusion follows.
\end{proof}

Combining Proposition \ref{local-energy-bound} with a comparison argument, we now derive a fundamental energy estimate for minimizers in terms of the height $h$ of a ``cigar shaped'' smoothed cylinder. 

\begin{corollary}
\label{energy-upperbound}	
Let $\mathfrak{C}^h_{\ell,\rho}$ be a smoothed cylinder with $h-\rho>2\sqrt{2}\ell$, and  $Q_{\rm b}$ its homeotropic boundary data given by~\eqref{eq:radial-anchoring}.
If $Q$ is  minimizing  $\mathcal{E}_\lambda$ over $\mathcal{A}_{Q_{\rm b}}^{\rm sym}(\mathfrak{C}^h_{\ell,\rho})$, then 
$$\mathcal{E}_\lambda(Q,\mathfrak{C}^h_{\ell,\rho})\leq 2h \mathfrak{e}_{\lambda \ell^2} +C_1\,,$$
where $\mathfrak{e}_{\lambda \ell^2}$ is defined by \eqref{def:minElambda}, and $C_1=C_1(\ell,\rho,\lambda)$ is a constant independent of $h$.  In addition, 
\begin{equation}
\label{eq:lc-local-upperbound}
\frac{1}r \mathcal{E}_\lambda\big(Q,\mathfrak{C}^h_{\ell,\rho}\cap B_r\big)+\int_{\mathfrak{C}^h_{\ell,\rho} \cap B_r} \abs{\frac{\partial Q}{\partial x_3}}^2 \, dx\leq C_2\quad \forall r\in (2\sqrt{2}\ell, h-\rho)\, ,
\end{equation}
for a constant $C_2=C_2(\ell,\rho,\lambda)$ also independent of $h$. Moreover, the dependence of $C_1$ and $C_2$ on $\lambda\geq 0$ is locally uniform.
\end{corollary}

\begin{proof}
We define $\Omega_h^\pm:=(\mathfrak{C}^h_{\ell,\rho} \setminus\mathfrak{C}^{h-\rho}_\ell) \cap \{\pm x_3> 0\}$, so that $\mathfrak{C}^h_{\ell,\rho}=\mathfrak{C}^{h-\rho}_\ell \cup \Omega_h^+\cup\Omega_h^-$, and 
$$\partial\Omega_h^\pm=\big(\mathbb{D}_\ell\times\{\pm x_3=h-\rho\}\big)\cup \big(\partial \mathfrak{C}^{h}_{\ell,\rho}\cap\{\pm x_3\geq h-\rho\}  \big)\,.$$
Setting $\widetilde{\lambda}:=\lambda \ell^2$, we fix $Q_{\widetilde{\lambda}} \in \mathcal{A}^{\rm sym}_{\overline{H}}(\bb D)$ such that $E_{\widetilde{\lambda}}(Q_{\widetilde{\lambda}})=\mathfrak{e}_{\widetilde{\lambda}}$. Since $ Q_{\widetilde{\lambda}}$ is minimizing $E_{\widetilde\lambda}$ over $\mathcal{A}^{\rm sym}_{\overline{H}}(\bb D)$, $ Q_{\widetilde{\lambda}}$ is smooth up to $\partial\mathbb{D}$ (see Section \ref{2Dminimization}). Rescaling variables, we have $E_{\lambda}(\widetilde{Q}_\lambda;\bb D_\ell)=\mathfrak{e}_{\lambda \ell^2}$ for $\widetilde Q_\lambda( \, \cdot \,):=Q_{\widetilde{\lambda}}(\cdot / \ell\,)$. We define a Lipschitz map $\widetilde{Q}^h$ on $\partial \Omega_h^\pm$ setting $\widetilde{Q}^h(x):=Q_{\rm b}(x)$ if $x\in \partial \mathfrak{C}^{h}_{\ell,\rho}\cap\{\pm x_3\geq h-\rho\} $, and $\widetilde{Q}^h(x):=\widetilde Q_\lambda(x^\prime)$ if $x=(x^\prime,x_3)\in \mathbb{D}_\ell\times\{\pm x_3=h-\rho\}$. 
Considering the points $p^\pm:=(0,0,\pm (h-\rho/2)) \in \Omega^\pm_h$,  
we extend $\widetilde{Q}^h$ to the interior of $\Omega^\pm_h$ by $0$-homogeneity from the point $p^\pm$. Then we finally extend $\widetilde{Q}^h$ to  $\mathfrak{C}^{h}_{\ell,\rho}$ setting  $\widetilde{Q}^h(x)=\widetilde Q_\lambda(x^\prime)$ if $x=(x^\prime,x_3)\in \mathfrak{C}^{h-\rho}_{\ell}$. By construction, we have $\widetilde{Q}^h \in \mathcal{A}^{\rm sym}_{Q_{\rm b}}(\mathfrak{C}^{h}_{\ell,\rho}) \cap \rmLip_{\rm loc}(\overline{\mathfrak{C}^{h}_{\ell,\rho}}\setminus \{p^\pm\})$,  
$$\mathcal{E}_\lambda(\widetilde{Q}^h,\mathfrak{C}^{h-\rho}_\ell)=2(h-\rho) E_\lambda(\widetilde Q_\lambda,\bb D_\ell)=2(h-\rho) \mathfrak{e}_{\lambda \ell^2}\,,$$
and
\begin{equation}\label{impstepcor55}
\mathcal{E}_\lambda(\widetilde{Q}^h,\Omega^\pm_h)\leq C( \|\nabla_{\rm tan} \widetilde{Q}^h\|^2_{L^2(\partial \Omega_h^\pm)} +\lambda)\leq C_1\,,
\end{equation}
for a constant $C_1=C_1(\ell,\rho,\lambda)$ independent of $h$ and continuous w.r.to $\lambda$. If $Q$ is minimizing  $\mathcal{E}_\lambda$ over $\mathcal{A}_{Q_{\rm b}}^{\rm sym}(\mathfrak{C}^h_{\ell,\rho})$, then 
\begin{equation}\label{dimcpa1}
 \mathcal{E}_\lambda(Q) \leq \mathcal{E}_\lambda(\widetilde{Q}^h)= \mathcal{E}_\lambda(\widetilde{Q}^h,\mathfrak{C}^{h-\rho}_\ell)+ \mathcal{E}_\lambda(\widetilde{Q}^h,\Omega^+_h)+ \mathcal{E}_\lambda(\widetilde{Q}^h,\Omega^-_h) \leq 2(h-\rho) \mathfrak{e}_{\lambda \ell^2} +2C_1\,.
 \end{equation}
On the other hand, by definition of 
$\mathfrak{e}_{\lambda \ell^2}$, we have 
\begin{equation}\label{dimcpa2}
 \mathcal{E}_\lambda(Q)\geq \int_{-h+\rho}^{h-\rho} E_\lambda\big({Q}(\, \cdot \,, x_3),\bb D_\ell\big) \, {d}x_3+ \int_{\mathfrak{C}^{h-\rho}_\ell} \frac12\abs{\frac{\partial Q}{\partial x_3}}^2 \, { d}x\geq 2(h-\rho) \mathfrak{e}_{\lambda \ell^2} +\int_{\mathfrak{C}^{h-\rho}_\ell} \frac12\abs{\frac{\partial Q}{\partial x_3}}^2 \, { d}x \, .
 \end{equation}
Since $\mathfrak{C}^h_{\ell,\rho} \cap B_r\subset \mathfrak{C}^{h-\rho}_\ell$, combining \eqref{dimcpa1} and \eqref{dimcpa2} leads to 
$ \int_{\mathfrak{C}^{h}_{\ell,\rho}\cap B_r} |\frac{\partial Q}{\partial x_3}|^2 \, { d}x\leq 4C_1$ for every $r\leq h-\rho$. 
In view of this estimate and \eqref{dimcpa1}, we can apply Proposition \ref{local-energy-bound} with $r_2=h-\rho$ and $r_1=r \geq 2\sqrt{2}\ell$  to obtain 
\begin{equation}\label{cppamard19}
\frac{1}{2r}\mathcal{E}_\lambda\big(Q,\mathfrak{C}^h_{\ell,\rho} \cap B_{r}\big) \leq \frac{1}{h-\rho}   \mathcal{E}_\lambda(Q)+\frac{3}{r} \int_{\mathfrak{C}^{h}_{\ell,\rho}\cap B_{h-\rho}} \abs{\frac{\partial Q}{\partial x_3}}^2 \, { d}x\leq 2 \mathfrak{e}_{\lambda \ell^2} +\frac{8C_1}{\ell}\,,
\end{equation}
 which proves  \eqref{eq:lc-local-upperbound}  once we choose $C_2 = 2 \mathfrak{e}_{\lambda \ell^2} + \frac{8C_1}{\ell}$. Since $C_2$ is continuous in its arguments, hence locally bounded w.r.to $\lambda$, the proof is complete. 
\end{proof}

Combining identity \eqref{eq:vert-ener-id} with Corollary \ref{energy-upperbound}, we obtain an energy bound as in \eqref{eq:lc-local-upperbound} for arbitrary balls centered on the vertical axis. 

\begin{corollary}\label{Corshiftedballs2}
Let $\mathfrak{C}^h_{\ell,\rho}$ be a smoothed cylinder with $h-\rho>2\sqrt{2}\ell$, and  
$Q_{\rm b}$ its homeotropic boundary data given by~\eqref{eq:radial-anchoring}.
If $Q$ is  minimizing  $\mathcal{E}_\lambda$ over $\mathcal{A}_{Q_{\rm b}}^{\rm sym}(\mathfrak{C}^h_{\ell,\rho})$, then there exists a constant $C_3=C_3(\ell,\rho,\lambda)$ independent of $h$ such that
$$\frac{1}{r}\mathcal{E}_\lambda\big(Q,\mathfrak{C}^h_{\ell,\rho}\cap B_r(p)\big)\leq C_3$$
for every $p=(0,0,z)\in \mathfrak{C}^h_{\ell,\rho}\cap \{x_3\text{-axis}\}$ and $2\sqrt{2}\ell < r < h-\rho-|z|-\ell$. 
\end{corollary}

\begin{proof}
Integrating \eqref{eq:vert-ener-id} with respect to $t_1\in[z-r,z+r]$ and dividing the result by $2r$, we obtain
\begin{equation}\label{merc3nov1}
\frac{1}{2r}\mathcal{E}_\lambda\big(Q,\mathfrak{C}^r_{\ell}(p)\big) -\frac{1}{4r}\int_{\mathfrak{C}^r_{\ell}(p)} \abs{\frac{\partial Q}{\partial x_3}}^2 \, dx \\ 
=E_\lambda\big(Q(\, \cdot \, ,t_2),\bb D_\ell\big) 
-\frac12 \int_{\mathfrak{C}^h_{\ell,\rho}\cap \{x_3=t_2\}} \abs{\frac{\partial Q}{\partial x_3}}^2 \, d\mathcal{H}^2
\end{equation}
for every $t_2\in[-h+\rho,h-\rho]$ such that $t_2\not\in{\rm sing}(Q)$. Then, integrating \eqref{merc3nov1} with respect to $t_2\in[-r,r]$, we derive that 
\begin{multline}\label{inequtiloublie}
\mathcal{E}_\lambda\big(Q,\mathfrak{C}^h_{\ell,\rho}\cap B_r(p)\big)\leq\mathcal{E}_\lambda\big(Q,\mathfrak{C}^r_{\ell}(p)\big) \leq \mathcal{E}_\lambda\big(Q,\mathfrak{C}^r_{\ell}\big)+\frac{1}{2}\int_{\mathfrak{C}^r_{\ell}(p)} \abs{\frac{\partial Q}{\partial x_3}}^2 \, dx\\
\leq \mathcal{E}_\lambda\big(Q,\mathfrak{C}^h_{\ell,\rho}\cap B_{r+\ell}\big)+\frac{1}{2}\int_{\mathfrak{C}^h_{\ell,\rho}\cap B_{r+|z|+\ell}} \abs{\frac{\partial Q}{\partial x_3}}^2 \, dx\,,
\end{multline}
since $\mathfrak{C}^h_{\ell,\rho}\cap B_r(p)\subset \mathfrak{C}^r_{\ell}(p)$, $\mathfrak{C}^r_{\ell}\subset \mathfrak{C}^h_{\ell,\rho}\cap B_{r+\ell}$, and $\mathfrak{C}^r_{\ell}(p)\subset \mathfrak{C}^h_{\ell,\rho}\cap B_{r+|z|+\ell}$.  The conclusion now follows from Corollary \ref{energy-upperbound} with $C_3=2C_2(1+1/\ell)$ and $C_2$ given by \eqref{eq:lc-local-upperbound}. 
\end{proof}

Using suitable competitors, we can now deduce from the previous corollary that the energy of minimizers remains   bounded also near the top and bottom parts of the cylinder. 

\begin{corollary}\label{Corshiftedballs3}
Let $\mathfrak{C}^h_{\ell,\rho}$ be a smoothed cylinder with $h-2\rho>4\ell$, and  
$Q_{\rm b}$ its homeotropic boundary data given by~\eqref{eq:radial-anchoring}. If $Q$ is  minimizing  $\mathcal{E}_\lambda$ over $\mathcal{A}_{Q_{\rm b}}^{\rm sym}(\mathfrak{C}^h_{\ell,\rho})$, then there exists a constant $C_4=C_4(\ell,\rho,\lambda)$ independent of $h$ such that $\mathcal{E}_\lambda\big(Q,\mathfrak{C}^h_{\ell,\rho}\setminus\mathfrak{C}^{h-\rho}_\ell\big)\leq C_4$.
\end{corollary}

\begin{proof}
Applying Corollary \ref{Corshiftedballs2} with $r=3\ell$ and $p^\pm:=(0,0,\pm t)$ and $t:=h-2\rho-4\ell$, we infer that $\mathcal{E}_\lambda\big(Q,\mathfrak{C}^\ell_{\ell}(p^\pm)\big)\leq C_3$ since $\mathfrak{C}^\ell_{\ell}(p^\pm)\subset \mathfrak{C}^h_{\ell,\rho}\cap B_r(p^\pm)$ with $C_3=C_3(\ell,\rho,\lambda)$. By Fubini's theorem, we can find a level $\bar t\in (h-2\rho-4\ell, h-2\rho-3\ell)$ such that 
$E_\lambda\big(Q(\cdot,\pm\bar t), \bb D_\ell\big)\leq C_3/\ell$. We shall now construct a competitor following an argument from the proof of Corollary \ref{energy-upperbound}. 
First, we consider the domains $\Omega^\pm_h:=(\mathfrak{C}^h_{\ell,\rho}\setminus \mathfrak{C}^{\bar t}_{\ell})\cap  \{\pm x_3>0\}$. We define a map $\widetilde Q$ on 
$\partial \Omega^\pm_h$ by setting $\widetilde Q=Q$ on $\partial \Omega^\pm_h\cap\{\pm x_3=\bar t\,\}$, and $\widetilde Q=Q_{\rm b}$ on $\partial \Omega^\pm_h\cap\{\pm x_3>\bar t\,\}$. Then we extend $\widetilde Q$ to the interior of  $\Omega^\pm_h$ by $0$-homogeneity from the point $q^\pm:=(0,0,\pm(h-\rho-2\ell))$. As in the proof of Corollary \ref{energy-upperbound} (see \eqref{impstepcor55}), we have $\mathcal{E}_\lambda(\widetilde Q,\Omega^\pm_h)\leq C$ for some constant $C$ independent of $h$, thanks to our choice of $\bar t$. Now we extend $\widetilde Q$ to $\mathfrak{C}^h_{\ell,\rho}$ setting $\widetilde Q=Q$ in $\mathfrak{C}^h_{\ell,\rho}\setminus \Omega^\pm_h$. In this this way, $\widetilde Q\in \mathcal{A}_{Q_{\rm b}}^{\rm sym}(\mathfrak{C}^h_{\ell,\rho})$ is a competitor to test the minimality of $Q$ which leads to $\mathcal{E}_\lambda(Q,\Omega^+_h\cup\Omega^-_h)\leq \mathcal{E}_\lambda(\widetilde Q,\Omega^+_h\cup\Omega^-_h)\leq 2C$. Since $\mathfrak{C}^h_{\ell,\rho}\setminus\mathfrak{C}^{h-\rho}_{\ell}\subset \Omega^+_h\cup\Omega^-_h$, the conclusion follows. 
\end{proof}

The next result will be useful to turn the local boundedness in Corollaries \ref{energy-upperbound} \& \ref{Corshiftedballs2} into a local compactness property up to ``the lateral boundary''. 
The arguments here are suitable  modifications of \cite[Theorem~5.1 and~5.2]{DMP2}, taking advantage of the translation invariance of the Dirichlet boundary data.  
 Before stating the result, let us define precisely the notion of local minimality we shall use in the sequel. 

\begin{definition}
Let $\mathfrak{C}^h_\ell$ be a cylinder with $\ell<\infty$.  We call {\it lateral boundary} of the cylinder $\mathfrak{C}^h_\ell$, the set 
\begin{equation}\label{deflatbound}
\partial^{\rm lat} \mathfrak{C}^h_\ell:=\partial \mathfrak{C}^h_\ell\cap\{|x_3|<h\}=\partial\mathbb{D}_\ell\times(-h,h)\,.
\end{equation}
 An equivariant map $Q\in W^{1,2}_{\rm loc}(\mathfrak{C}^h_\ell;\bb S^4)$ is said to be an {\it equivariant local minimizer of $\mathcal{E}_{\lambda}$ in $\mathfrak{C}^h_\ell$ up to the lateral boundary} if  for every $\eta\in(0,h)$, $Q\in W^{1,2}_{\rm sym}(\mathfrak{C}^\eta_\ell;\bb S^4)$ and $\mathcal{E}_{\lambda}(Q,\mathfrak{C}^\eta_\ell)\leq \mathcal{E}_{\lambda}(\widetilde Q,\mathfrak{C}^\eta_\ell)$ for every  $\widetilde{Q}\in W^{1,2}_{\rm sym}(\mathfrak{C}^\eta_\ell;\bb S^4)$ satisfying $\widetilde{Q}=Q$ on $\partial \mathfrak{C}^\eta_\ell$. 
\end{definition}

\begin{lemma}
\label{verticalcompactness}
Let $\mathfrak{C}^h_\ell$ be a bounded cylinder and  $Q_{\rm b}$ its homeotropic boundary data given by~\eqref{eq:radial-anchoring}. Let 
 $\lambda_j\to \lambda$ and $\{Q_j\} \subset W^{1,2}_{\rm sym}(\mathfrak{C}^h_\ell;\bb S^4)$  a sequence such that each $Q_j$ is  an equivariant local minimizer of $\mathcal{E}_{\lambda_j}$ in $\mathfrak{C}^h_\ell$ up to the lateral boundary and $Q_j=Q_{\rm b}$ on $\partial^{\rm lat} \mathfrak{C}^h_\ell$. If $\sup_j \mathcal{E}_{\lambda_j}(Q_j,\mathfrak{C}^h_\ell)<\infty$, then there exists a (not relabeled) subsequence such that $Q_j \to Q_*$ strongly in $W^{1,2}( \mathfrak{C}^\eta_\ell)$ for every $\eta\in(0,h)$, where  $Q_* \in W^{1,2}_{\rm sym}(\mathfrak{C}^h_\ell;\bb S^4)$ is an equivariant local minimizer of $\mathcal{E}_\lambda$ up  to the lateral boundary satisfying $Q_*=Q_{\rm b}$ on $\partial^{\rm lat} \mathfrak{C}^h_\ell$. 
\end{lemma}

\begin{proof}  
By the uniform energy bound, the sequence $\{Q_j\}$ is bounded in $ W^{1,2}(\mathfrak{C}^h_\ell)$. Hence, we can find a (not relabeled) subsequence such that 
 $Q_j \rightharpoonup Q_{*}$ weakly in $W^{1,2}(\mathfrak{C}^h_\ell)$, strongly in $L^2(\mathfrak{C}^h_\ell)$, and also a.e. in $\mathfrak{C}^h_\ell$, for some $Q_* \in W^{1,2}_{\rm sym}(\mathfrak{C}^h_\ell;\bb S^4)$. By $W^{1,2}$-weak continuity  and locality of the trace operator,  $Q_*=Q_{\rm b}$ on $\partial^{\rm lat} \mathfrak{C}^h_\ell$.  In addition, \cite[Theorem 5.1]{DMP2} implies that  $Q_j \to Q_{*}$ strongly in $W^{1,2}_{\rm loc}(B_r(p))$ for every $p\in  \mathfrak{C}^h_\ell\cap\{x_3\text{-axis}\}$ and $r > 0$  such that $B_r(p)\subset \mathfrak{C}^h_\ell$. 
As a consequence, given an arbitrary $\delta>0$ with $2\delta<\min\{h,\ell\}$, we have $Q_j \to Q_{*}$ strongly in $W^{1,2}$ in the set $\bb D_{\delta/2}\times\{h-\delta<|x_3|<h-\delta/2\}$. By a standard application of Fubini's theorem and  Fatou's lemma, extracting a further subsequence if necessary, we can find 
 $\eta \in (h-\delta, h-\delta/2)$ such that the restrictions $\widehat{Q}^\pm_j$ and $\widehat{Q}^\pm_*$ of $Q_j$ and $Q_*$ to $\mathfrak{C}^h_\ell \cap \{x_3=\pm \eta\}$ satisfy 
 $\widehat{Q}^\pm_j \rightharpoonup \widehat{Q}_*^\pm$ weakly in $W^{1,2}(\bb D_\ell)$ and strongly in $W^{1,2}(\bb D_{\delta/3})$. By  Lemma~\ref{lemma:s1eq-emb}, we conclude that $\widehat{Q}^\pm_j, \widehat{Q}_*^\pm \in C^0(\overline{\bb D}_\ell)$ and  $\widehat{Q}^\pm_j \to \widehat{Q}_*^\pm$ uniformly in $\overline{\bb D}_\ell$.

Let us now fix an arbitrary $\widetilde{Q}\in W^{1,2}_{\rm sym}(\mathfrak{C}^{h-\delta}_\ell;\bb S^4)$ satisfying $\widetilde{Q}=Q_*$ on $\partial \mathfrak{C}^{h-\delta}_\ell$. We extend  $\widetilde{Q}$ to $\mathfrak{C}^{\eta}_\ell$ setting $\widetilde{Q}=Q_*$ in   $\mathfrak{C}^{\eta}_\ell\setminus \mathfrak{C}^{h-\delta}_\ell$, 
and we set $\sigma_j:=\| \widehat{Q}^+_j-\widehat{Q}_*^+\|_\infty+ \| \widehat{Q}^-_j-\widehat{Q}_*^-\|_\infty+2^{-j} \to 0$ as $j \to \infty$.  For $j$ large enough we have  $\sigma_j<1$, and we define  $v_j \in W^{1,2}_{\rm sym}(\mathfrak{C}^\eta_\ell;\mathcal{S}_0)$ as
\begin{equation}
\label{eq:def-vmaps} 
v_j(x^\prime,x_3):=
\begin{cases}
\displaystyle  \frac{x_3- (1-\sigma_j )\eta}{\sigma_j \eta}\big(\widehat{Q}^+_j(x^\prime) -\widehat{Q}^+_*(x^\prime)\big)+\widehat{Q}^+_*(x^\prime) 
&
  \text{if }  (1-\sigma_j)\eta \leq x_3 \leq \eta \, ,
\\[8pt] 	
\widetilde{Q}\big(x^\prime, x_3/(1-\sigma_j)\big)   & \text{if } |x_3|< (1-\sigma_j)\eta \,,
\\[8pt]
\displaystyle  \frac{-x_3+ (\sigma_j-1 )\eta}{\sigma_j \eta}\big(\widehat{Q}^-_j(x^\prime) -\widehat{Q}^-_*(x^\prime)\big)+\widehat{Q}^-_*(x^\prime) 
& \text{if }   -\eta \leq x_3 \leq -(1-\sigma_j)\eta\, .
\end{cases}
\end{equation}
Since the restriction of $Q_{\rm b}$ to $\partial^{\rm lat} \mathfrak{C}^h_\ell$ is independent of $x_3$, we have $v_j=Q_{\rm b}$ on $\partial^{\rm lat} \mathfrak{C}^\eta_\ell$. Hence $v_j=Q_j$ on $\partial \mathfrak{C}^\eta_\ell$. A simple calculation yields
\begin{equation}\label{mardvaccpa1}
 \int_{\mathfrak{C}^{(1-\sigma_j)\eta}_\ell} |\nabla v_j|^2 \, dx \leq  \frac{1}{1-\sigma_j} \int_{\mathfrak{C}^\eta_\ell} |\nabla \widetilde{Q}|^2 \, dx
 \end{equation}
and 
\begin{multline}\label{mardvaccpa2}
 \int_{\mathfrak{C}^\eta_\ell\setminus\mathfrak{C}^{(1-\sigma_j)\eta}_\ell} |\nabla v_j|^2 \, dx\leq  C \sigma_j   \int_{\bb D_\ell} |\nabla \widehat{Q}^+_j|^2+|\nabla \widehat{Q}^-_j|^2+ |\nabla \widehat{Q}^+_*|^2+|\nabla \widehat{Q}^-_*|^2 \, dx^\prime\\
 +\frac{C}{ \sigma_j}\big(\| \widehat{Q}^+_j-\widehat{Q}_*^+\|^2_\infty+ \| \widehat{Q}^-_j-\widehat{Q}_*^-\|^2_\infty \big) \leq C\sigma_j\,,
\end{multline}
for some constant $C=C(\delta,\eta)$ independent $j$. 
By construction, we have $|v_j|=1$ in $\mathfrak{C}^{(1-\sigma_j)\eta}_\ell$, and $0\leq 1-|v_j|\leq \sigma_j$ in $\mathfrak{C}^{\eta}_\ell\setminus \mathfrak{C}^{(1-\sigma_j)\eta}_\ell$. 
Therefore, $|v_j|\geq 1/2$ for $j$ large enough, and we can define the competitor
$$\widetilde{Q}_j:=\frac{v_j}{|v_j|}\in W^{1,2}_{\rm sym}(\mathfrak{C}^{\eta}_\ell;\bb S^4)$$ 
which satisfies $\widetilde{Q}_j= Q_j$ on $\partial \mathfrak{C}^{\eta}_\ell$. Since $\widetilde Q_j=v_j$ in $\mathfrak{C}^{(1-\sigma_j)\eta}_\ell$ and $|v_j|\geq 1/2$, we infer from \eqref{mardvaccpa1} and \eqref{mardvaccpa2} that 
\begin{equation}\label{mardvaccpa3} 
 \int_{\mathfrak{C}^\eta_\ell} |\nabla \widetilde{Q}_j|^2 \, dx  \leq \frac{1}{1-\sigma_j} \int_{\mathfrak{C}^\eta_\ell} |\nabla \widetilde{Q}|^2 \, dx+ C \sigma_j \, .
 \end{equation}
On the other hand, $Q_j \to Q_*$ and $\widetilde{Q}_j \to \widetilde{Q}$ a.e. in $\mathfrak{C}^\eta_\ell$. Then, 
\begin{equation}\label{mardvaccpa4} 
\lim_{j\to\infty} \lambda_j\int_{\mathfrak{C}^\eta_\ell} W( Q_j)\,dx= \lambda\int_{\mathfrak{C}^\eta_\ell} W( Q_*)\,dx\quad\text{and}\quad \lim_{j\to\infty} \lambda_j\int_{\mathfrak{C}^\eta_\ell} W(\widetilde Q_j)\,dx= \lambda\int_{\mathfrak{C}^\eta_\ell} W(\widetilde Q)\,dx
\end{equation}
by dominated convergence. By minimality of $Q_j$,  \eqref{mardvaccpa3}-\eqref{mardvaccpa4}, and weak lower semicontinuity of the Dirichlet integral, we finally deduce that  
\begin{equation}\label{mardvaccpa5} 
 \mathcal{E}_\lambda(Q_*,\mathfrak{C}^\eta_\ell) \leq \liminf_{j \to \infty} \mathcal{E}_{\lambda_j}(Q_j,\mathfrak{C}^\eta_\ell) \leq \limsup_{j \to \infty} \mathcal{E}_{\lambda_j}(Q_j, \mathfrak{C}^\eta_\ell) \leq  \limsup_{j \to \infty} \mathcal{E}_{\lambda_j}(\widetilde{Q}_j,\mathfrak{C}^\eta_\ell)=\mathcal{E}_\lambda(\widetilde{Q},\mathfrak{C}^\eta_\ell) \, . 
 \end{equation}
Since $\widetilde Q=Q_*$ in  $\mathfrak{C}^{\eta}_\ell\setminus \mathfrak{C}^{h-\delta}_\ell$, it follows that $ \mathcal{E}_\lambda(Q_*,\mathfrak{C}^{h-\delta}_\ell) \leq \mathcal{E}_\lambda(\widetilde{Q},\mathfrak{C}^{h-\delta}_\ell)$ proving the minimality of $Q_*$ in $\mathfrak{C}^{h-\delta}_\ell$. Moreover, choosing $\widetilde Q=Q_*$ leads to  $\lim_j\mathcal{E}_{\lambda_j}(Q_j,\mathfrak{C}^\eta_\ell)=  \mathcal{E}_\lambda(Q_*,\mathfrak{C}^\eta_\ell)$ which, in view of \eqref{mardvaccpa4}, implies that $Q_j\to Q_*$ strongly in   $W^{1,2}(\mathfrak{C}^\eta_\ell)$ (and thus strongly in  $W^{1,2}(\mathfrak{C}^{h-\delta}_\ell)$). The conclusion now follows from the arbitrariness of $\delta$. 
\end{proof}


\subsection{Rigidity in infinite cylinders and proof of Theorem \ref{thm:vertical-cylinders}}

The following rigidity result will be the key ingredient to deduce qualitative properties for minimizers of $\mathcal{E}_\lambda$ on expanding cylinders $\mathfrak{C}^h_{\ell,\rho}$ as $h \to +\infty$. To this purpose, we will heavily use results from Section~\ref{2Dminimization}, to which the reader is referred also for some of the notation employed here.  We only recall from  Theorem~\ref{2d-biaxial-escape}  that in the case $\lambda\ell^2<\lambda_0$, the functional $E_{\lambda}$ admits a unique minimizer $\widehat Q_\ell$ over  $\mathcal{A}^{\rm sym}_{\overline{H}}(\bb D_\ell)$, and it satisfies $\widehat Q_\ell(0)=-{\bf e}_0$.

\begin{proposition}
\label{vertical-liouville}
Let $\ell>0$ be such that $\lambda\ell^2<\lambda_0$ with $\lambda_0$ the constant given by Theorem~\ref{2d-biaxial-escape}. Assume that  
$Q \in W^{1,2}_{\rm loc}(\mathfrak{C}^\infty_\ell;\bb S^4)$ is an equivariant local minimizer of $\mathcal{E}_\lambda$ in $\mathfrak{C}^\infty_\ell$ up to the lateral boundary satisfying $Q=Q_{\rm b}$ on $\partial^{\rm lat}\mathfrak{C}^\infty_\ell$, where $Q_{\rm b}$ denotes the homeotropic boundary data given by~\eqref{eq:radial-anchoring}. 
If $\mathcal{E}_\lambda(Q,\mathfrak{C}_\ell^h)=O(h)$ as $h \to \infty$, then $Q(x)\equiv \widehat Q_\ell(x^\prime)$  
where $\widehat Q_\ell$ denotes the unique minimizer of the 2D-functional $E_{\lambda}$ over  $\mathcal{A}^{\rm sym}_{\overline{H}}(\bb D_\ell)$. In particular, $Q$ is smooth, independent of $x_3$, and $Q= -\eo$ on the $x_3$-axis. 
\end{proposition}

\begin{proof}
We first notice that 
$$\int_{2^n}^{2^{n+1}} \bigg(E_\lambda\big(Q(\, \cdot \,, t),\bb D_\ell\big)+E_\lambda(Q(\, \cdot \,, -t),\bb D_\ell\big)\bigg)\,dt\leq  \mathcal{E}_\lambda\big(Q,\mathfrak{C}^{2^{n+1}}_\ell\big)\leq C 2^n\qquad \forall n\in\mathbb{N}\,.$$
Hence, for each integer $n$, there exists $h_n\in(2^n,2^{n+1})$ such that $Q(\cdot,\pm h_n)\in W^{1,2}(\bb D_\ell)$ with   $E_\lambda\big(Q(\, \cdot \,, \pm h_n),\bb D_\ell\big)=O(1)$ as $n\to \infty$. 

We claim that 
\begin{equation}\label{mercrfoiebad}
\mathcal{E}_\lambda(Q,\mathfrak{C}^{h_n}_\ell)=2h_n \mathfrak{e}_{\lambda \ell^2}+O(1)\quad\text{as $n\to\infty$}\,.
\end{equation}
We argue as in Corollary~\ref{energy-upperbound} to construct competitors, and we set $\Omega_n^\pm:=(\mathfrak{C}^{h_n}_\ell \setminus \overline{\mathfrak{C}^{h_n-1}_\ell } ) \cap \{\pm x_3> 0\}$. We define a map $\widetilde{Q}_n$ in $\mathfrak{C}^{h_n}_\ell\cap\{|x_3|\leq h_n-1\}$ setting $\widetilde{Q}_n(x):=\widehat Q_\ell(x^\prime)$. For $x\in \partial \Omega_n^\pm\cap \{|x_3|>h_n-1\}$, we set $\widetilde{Q}_n(x):=Q(x)$ and we then extend $\widetilde{Q}_n$ inside $\Omega_n^\pm$ by $0$-homogeneity from the point 
 $p_n^\pm=(0,0,\pm (h_n-1/2))$. As in the proof of Corollary ~\ref{energy-upperbound} (see \eqref{impstepcor55}), our choice of $h_n$ ensures that $\mathcal{E}_\lambda(\widetilde{Q}_n,\Omega_n^\pm)=O(1)$ as $n\to\infty$. Since $\mathcal{E}_\lambda(\widetilde{Q}_n,\mathfrak{C}^{h_n-1}_\ell)=2(h_n-1) \mathfrak{e}_{\lambda \ell^2}$, the claim follows. 

In view of  \eqref{mercrfoiebad}, we now have  
\begin{multline}
\label{eq:two-sided-bound-cyl}
2h_n \mathfrak{e}_{\lambda \ell^2}+\int_{-h_n}^{h_n} \Big(E_\lambda\big(Q(\, \cdot \,,x_3),\bb D_\ell\big)-\mathfrak{e}_{\lambda \ell^2}\Big)\, dx_3 +\int_{\mathfrak{C}^{h_n}_\ell}\frac12 \abs{\frac{\partial Q}{\partial x_3}}^2\,dx_3\\
\leq\mathcal{E}_\lambda(Q,\mathfrak{C}^{h_n}_\ell)=2h_n \mathfrak{e}_{\lambda \ell^2}+O(1)\, .
\end{multline}
Recalling that $E_\lambda\big(Q(\, \cdot \,,x_3),\bb D_\ell\big)-\mathfrak{e}_{\lambda \ell^2}\geq 0$, letting $n\to\infty$ in \eqref{eq:two-sided-bound-cyl} yields 
 \[ \int_{-\infty}^{\infty} \Big(E_\lambda(Q\big(\, \cdot \,,x_3),\bb D_\ell\big)-\mathfrak{e}_{\lambda \ell^2}\Big) \, dx_3 +\int_{\mathfrak{C}^{\infty}_\ell}\frac12 \abs{\frac{\partial Q}{\partial x_3}}^2\, dx_3 <\infty \, .\]
As a consequence, there exists $\tilde{h}_n \nearrow +\infty$ such that $E_\lambda\big(Q(\, \cdot \, ,\pm \tilde{h}_n),\bb D_\ell\big) \to \mathfrak{e}_{\lambda \ell^2}$ as $n \to \infty$. Since $\lambda \ell^2 <\lambda_0$, it follows from Theorem~\ref{2d-biaxial-escape} that $Q(\, \cdot \, ,\pm \tilde{h}_n) \to \widehat{Q}_\ell$ strongly in $W^{1,2}(\bb D_\ell)$. 
 Indeed, by weak lower semicontinuity of   $E_\lambda$, any weak limit is a minimizer of $E_\lambda$ over $\mathcal{A}_{\overline{H}}^{\rm sym}(\bb D_\ell)$ so that convergence is in fact $W^{1,2}$-strong. Convergence of the full sequence follows from the uniqueness of the limit.  In addition, Theorem~\ref{2d-biaxial-escape}  also ensures that the smooth map $\widehat Q_\ell$ satisfies $\widehat Q_\ell(0)=-{\bf e}_0$. 
Applying Lemma~\ref{lemma:s1eq-emb}, we also infer that $Q(\, \cdot \, ,\pm \tilde{h}_n) \to \widehat{Q}_\ell$ uniformly on $\overline{\bb D}_\ell$. 

Finally, we construct a further competitor $\widehat Q_n$ testing the minimality of $Q$ following the construction in the proof of  Lemma \ref{verticalcompactness}. We first define a sequence a map $v_n$ as in \eqref{eq:def-vmaps} with $\tilde h_n$ in place of $\eta$, $\widetilde Q$ and $\widehat Q_*^\pm$ replaced by $\widehat Q_\ell$,  and $Q(\cdot,\pm\tilde h_n)$ instead of $\widehat Q_j^\pm$. Then $|v_n|\geq 1/2$ for $n$ large enough which allows us to define   $\widehat Q_n:=v_n/|v_n|$. Then $\widehat Q_n\in W_{\rm sym}^{1,2}(\mathfrak{C}^{\tilde{h}_n}_\ell;\bb S^4)$ satisfies $\widehat Q_n=Q$ on $\partial \mathfrak{C}^{\tilde{h}_n}_\ell$. As in the proof of of  Lemma \ref{verticalcompactness}, the minimality of $Q$ implies that $\mathcal{E}_\lambda(Q,\mathfrak{C}^{\tilde{h}_n}_\ell)\leq \mathcal{E}_\lambda(\widehat Q_n,\mathfrak{C}^{\tilde{h}_n}_\ell)= 2\tilde{h}_n \mathfrak{e}_{\lambda \ell^2}+o(1)$ as $n\to\infty$. 
Combining this upper bound with the lower bound \eqref{eq:two-sided-bound-cyl} with $\tilde h_n$ instead of $h_n$, and letting $n\to\infty$ we conclude that  $\frac{\partial Q}{\partial x_3}\equiv 0$ in $\mathfrak{C}^\infty_\ell$ and $E_\lambda(Q(\, \cdot \, ,x_3),\bb D_\ell)\equiv \mathfrak{e}_{\lambda \ell^2}$. By uniqueness of~$\widehat Q_\ell$,  the conclusion follows. 
\end{proof}

\begin{remark}
\label{heteroclinic}
It is not known whether Proposition \ref{vertical-liouville} still holds for $\lambda \ell^2\geq \lambda_0$, or if there exists a map $Q \in W^{1,2}_{\rm loc}(\mathfrak{C}^\infty_\ell;\bb S^4)$ which is an equivariant local minimizer of $\mathcal{E}_\lambda$ up to the lateral boundary connecting two different minimizers $\widehat{Q}^\pm_\ell $ of $E_\lambda$ over $\mathcal{A}^{\rm sym}_{\overline{H}}(\bb D_\ell)$ as $x_3 \to \pm \infty$. One may expect that such local minimizer do exist for $\lambda \ell^2>\lambda_*$ with $\widehat{Q}_\ell^\pm(x^\prime)=g_{\overline{H}}(\pm x^\prime/\ell)$ and $g_{\overline{H}}$ defined through \eqref{eq:Hbar}. We have not pursued these issues, and these questions remain open.  
\end{remark}
\vskip5pt
We are now in position to prove the main result of this section, that is Theorem \ref{thm:vertical-cylinders}.

\begin{proof}[Proof of Theorem \ref{thm:vertical-cylinders}] 
To prove claim (i), we argue by contradiction assuming  that $Q^{(n)}$ is smooth for some subsequence. Notice that $I_n:=\overline{\Omega}_n\cap \{ x_3\hbox{-axis}\}$ is a closed interval and that $Q_{\rm b}^{(n)}(x)=\eo$ for each $x\in \partial\Omega_n\cap \{ x_3\hbox{-axis}\}=\partial I_n$. 
Hence $Q^{(n)}\equiv \eo$ on $I_n$ by continuity, which implies that $Q^{(n)}(\, \cdot \, /\ell,x_3) \in \mathcal{A}_{\rm N}$ whenever $|x_3|<h-\rho$. 
Combining Proposition~\ref{ANlambdaminimization} with Theorem~\ref{2d-biaxial-escape} yields 
\begin{equation}\label{energefflongcylarg}
\mathcal{E}_\lambda(Q^{(n)},\Omega_n)\geq \int_{-h_n+\rho}^{h_n-\rho} E_\lambda\big(Q^{(n)}(\, \cdot \, ,x_3),\bb D_\ell\big) \, dx_3 \geq 6\pi \cdot 2(h_n-\rho)=2(h_n-\rho) \mathfrak{e}_{\lambda_*} \, . 
\end{equation}
On the other hand, $\mathcal{E}_\lambda(Q^{(n)},\Omega_n) \leq 2h_n \mathfrak{e}_{\lambda \ell^2}+O(1)$ as $n\to \infty$ by Corollary~\ref{energy-upperbound}. Since $\lambda \ell^2<\lambda_*$, we have $\mathfrak{e}_{\lambda \ell^2}<\mathfrak{e}_{\lambda_*}$ by Proposition \ref{energyprofile}. Hence this upper bound contradicts \eqref{energefflongcylarg} for $n$ large enough. 

Therefore $\rmsing(Q^{(n)})\neq \emptyset$ for $n$ large enough. According to \cite[Theorem 1.1]{DMP2}, we then have $\beta_n (\overline{\Omega_n})=[-1,1]$ since this property  holds for the tangent map at any singular point (see also \cite[Remark~7.18]{DMP2}). Finally, since $\Omega_n$ is connected, simply connected, with boundary of class $C^3$, and $Q^{(n)}_{\rm b}$ is the homeotropic boundary data,  assumptions ($HP_1$)-($HP_3$) in \cite{DMP1,DMP2} are satisfied and each $Q^{(n)}$ is a split minimizer in the sense of \cite[Definition~7.11]{DMP2}.
\vskip5pt 

We now prove claim (ii). According to Corollary \ref{energy-upperbound},  $Q^{(n)}$ satisfies the uniform bound \eqref{eq:lc-local-upperbound} whenever $h_n -\rho>r>2\sqrt{2}\ell$. On the other hand, for each $\eta>0$ such that $\mathfrak{C}^\eta_\ell\subset \mathfrak{C}^{h_n}_{\ell,\rho}\cap B_r$, $Q^{(n)}$ is obviously an equivariant local minimizer of $\mathcal{E}_{\lambda_j}$ in $\mathfrak{C}^\eta_\ell$ up to the lateral boundary, so that  Lemma \ref{verticalcompactness} applies. 
By a standard diagonal argument, we infer the existence of a (not relabeled) subsequence such that  $Q^{(n)}\to Q^*$ strongly in $W^{1,2}(\mathfrak{C}^\eta_\ell)$ for every $\eta>0$ as $n \to \infty$, where $Q^*\in W^{1,2}_{\rm loc}(\mathfrak{C}^\infty_\ell;\bbS^4)$ an equivariant local minimizer of $\mathcal{E}_{\lambda}$ in $\mathfrak{C}^\infty_\ell$ up to the lateral boundary agreeing with the homeotropic boundary data \eqref{eq:radial-anchoring} on $\partial\mathfrak{C}^\infty_\ell$. Then, letting $n\to\infty$  in \eqref{eq:lc-local-upperbound}, we deduce that $\mathcal{E}_\lambda(Q^*,\mathfrak{C}^h_\ell)=O(h)$ as $h \to \infty$. Applying Proposition \ref{vertical-liouville}, it follows that $Q^*=\widehat{Q}_\ell$. By uniqueness of the the limit, the full sequence actually converges to $\widehat{Q}_\ell$ as claimed. 

To prove the locally smooth convergence, we rely on the regularity results in Section \ref{secregth}. We fix an arbitrary $\eta>0$, and we aim to prove that $Q^{(n)}$ is bounded  
in  $C^k(\overline{\mathfrak{C}^\eta_\ell})$ for every $k\in\N$, which is clearly enough for our purposes. Let us first fix an arbitrary point $x^*\in\bbD_\ell\times[-\eta,\eta]$. By smoothness of $\widehat{Q}_\ell$,  we can find $\delta>0$ small enough such that $B_\delta(x^*)\subset \mathfrak{C}^\infty_\ell$ and 
$\frac{1}{\delta} \mathcal{E}_\lambda\big(\widehat Q_\ell, B_\delta(x^*)\big)\leq \boldsymbol{\varepsilon}_{\rm in}/16$, where $ \boldsymbol{\varepsilon}_{\rm in}>0$ denotes the universal constant from Proposition~\ref{intepsregprop}. By the strong $W^{1,2}$-convergence of $Q^{(n)}$, we have $\frac{1}{\delta} \mathcal{E}_\lambda\big(Q^{(n)}, B_\delta(x^*)\big)\leq \boldsymbol{\varepsilon}_{\rm in}/8$ for $n$ large enough. By Proposition \ref{intepsregprop}, it implies that $Q^{(n)}$ is bounded in $C^k(B_{\delta/16}(x^*))$ for every $k\in\N$. Next we fix  $x^*\in \partial \bb D_\ell\times[-\eta,\eta]$ and a radius $r_*\in(0,\ell)$. By $\bbS^1$-equviariance, without loss of generality we can assume that $x^*=(x^*_1,0,x^*_3)\in\{x_2=0\}$.  
For $n$ large enough, we have $\Omega_n\cap B_{r_*}(x^*)=\mathfrak{C}^\infty_\ell\cap B_{r_*}(x^*)$ and $\partial\Omega_n\cap B_{r_*}(x^*)=\partial \mathfrak{C}^\infty_\ell\cap B_{r_*}(x^*)$, so that $\partial\Omega_n\cap B_{r_*}(x^*)$ and the restriction of $Q_{\rm b}^{(n)}$ to $\partial\Omega_n\cap B_{r_*}(x^*)$ are independent of $n$. Accordingly, the constants $\bar{\boldsymbol{\varepsilon}}_{\rm bd}>0$ and $\bar{\boldsymbol{\kappa}}>0$ from Proposition \ref{bdrepsregprop}  only depends on $\ell$. Arguing as in the proof of \cite[Proposition 6.9]{DMP2}, the equivariance of $Q^{(n)}$ implies that for   $r\in (0,r_*/4)$ and every ball $B_\rho(\bar x)\subset B_r(x^*)$, 
$$\frac{1}{\rho}\int_{B_\rho(\bar x)\cap \mathfrak{C}^\infty_\ell}|\nabla Q^{(n)}|^2\,dx \leq  \frac{C_*}{\ell} \int_{\mathcal{D}_r(x^*)} |\nabla Q^{(n)}|^2\,d\mathcal{H}^2\,,$$
where $C_*>0$ is a universal constant and $\mathcal{D}_r(x^*):=\mathfrak{C}^\infty_\ell\cap B_{r}(x^*)\cap \{x_2=0\}$. 
By the strong $W^{1,2}$-convergence of $Q^{(n)}$ and equivariance, the restriction of $|\nabla Q^{(n)}|^2$ to the slice 
$\mathcal{D}_{r_*}(x^*)$ is strongly converging in $L^1(\mathcal{D}_{r_*}(x^*))$. By the Vitali-Hahn-Sacks Theorem (see e.g. \cite[Theorem 1.30]{AFP}), we can find $r\in(0,r_*/4)$ such that  $\int_{\mathcal{D}_r(x^*)} |\nabla Q^{(n)}|^2\,d\mathcal{H}^2\leq \ell \bar{\boldsymbol{\varepsilon}}_{\rm bd}/C_*$ for $n$ large enough. Hence, 
$$ \sup_{B_\rho(\bar x)\subset B_r(x^*)}\frac{1}{\rho}\int_{B_\rho(\bar x)\cap \mathfrak{C}^\infty_\ell}|\nabla Q^{(n)}|^2\,dx \leq  \bar{\boldsymbol{\varepsilon}}_{\rm bd}\,,$$
and we infer from Proposition  \ref{bdrepsregprop} that  $Q^{(n)}$ is bounded in $C^k(B_{\bar{\boldsymbol{\kappa}} r/2}(x^*)\cap\overline{\mathfrak{C}^\infty_\ell})$ for every $k\in\N$. In view of the arbtrariness of $x^*$ (either in the interior or at the boundary), by a standard covering argument we finally conclude that $Q^{(n)}$ is bounded in  $C^k(\overline{\mathfrak{C}^\eta_\ell})$ for every $k\in\N$. 
\vskip5pt
 
   It remains to prove claim (iii). Writing  $\Sigma_n:=\rmsing(Q^{(n)})$ to ease the notation, we first observe that the convergence of $Q^{(n)}$ towards $\widehat Q_\ell$ established in claim (ii) implies that $\Sigma_n\cap\{|x_3|<1\}=\emptyset$ and  
$Q^{(n)}=-{\bf e}_0$ on $\{x_3\text{-axis}\}\cap\{|x_3|<1\}$ for $n$ large enough. Since $Q^{(n)}(q^\pm_n)={\bf e}_0$ at $q^\pm_n:=(0,0,\pm h_n)$ and $Q^{(n)}(x)\in\{{\bf e}_0,-{\bf e}_0\}$ for every $x\in (\Omega_n\cap\{x_3\text{-axis}\})\setminus \Sigma_n$, we deduce that both sets 
$\Sigma_n^+:=\Sigma_n\cap \{x_3\geq 0\}$ and $\Sigma_n^-:=\Sigma_n\cap \{x_3< 0\}$ are nonempty (recall that  $\Sigma_n$ is a finite subset of $\Omega \cap \{x_3\text{-axis}\}$). In view of Theorem \ref{intregthm}, the restriction of $Q^{(n)}$ to $(\Omega_n\cap\{x_3\text{-axis})\setminus \Sigma_n$ is constant on each connected component and jumps from ${\bf e}_0$ to $-{\bf e}_0$ at each point of $\Sigma_n$. It easily implies that both $\Sigma_n^+$ and $\Sigma_n^-$ contain an odd number of points. 

Let us now set $t^{\rm min}_n:=\min\big \{|p|: p\in \rmsing(Q^{(n)})\big\}\in (0,h_n)$. We claim that $h_n-t_n\leq \alpha$ for some constant $\alpha> 0$. To prove this claim, we argue by contradiction assuming that for some (not relabeled) subsequence, we have $h_n-t^{\rm min}_n\to + \infty$.  Next we consider a point $p^{\rm min}_n\in\rmsing(Q^{(n)})$ such that $|p^{\rm min}_n|=t^{\rm min}_n$. Notice that claim (ii) implies that $|p^{\rm min}_n|=t^{\rm min}_n\to \infty$, so that the translated domain $\widetilde\Omega_n:=\Omega_n-p^{\rm min}_n$ satisfies $\widetilde\Omega_n\to \mathfrak{C}^\infty_\ell$ as $n\to\infty$. By Corollary \ref{Corshiftedballs2}, for every $r>2\sqrt{2}\ell$ we have $\mathcal{E}_\lambda\big(Q^{(n)},\Omega_n\cap B_r(p^{\rm min}_n)\big)\leq Cr$ for $n$ large enough (so that $r<h_n-\rho-|p^{\rm min}_n|-\ell$), where the constant $C$ is independent of $n$. Considering the translated map $\widetilde Q^{(n)}(x):=Q^{(n)}(x+p^{\rm min}_n)$, we then have  $\mathcal{E}_\lambda\big(\widetilde Q^{(n)},\widetilde\Omega_n\cap B_r\big)\leq Cr$. Arguing as in the proof of claim (ii), we infer from Lemma~\ref{verticalcompactness} and Proposition \ref{vertical-liouville} that $\widetilde Q^{(n)}\to \widehat Q_\ell$ strongly in $W^{1,2}(\mathfrak{C}^\eta_\ell)$ for every $\eta>0$. In particular, $\widetilde Q^{(n)}\to \widehat Q_\ell$ strongly in $W^{1,2}(B_{\ell})$. Since $\widehat Q_\ell$ is smooth, we deduce from Lemma \ref{persissmoothloc} (applied in the ball $B_\ell$) that ${\rm sing}(\widetilde Q^{(n)})\cap B_{\ell/2}=\emptyset$ for $n$ large enough, contradicting the fact that $0\in {\rm sing}(\widetilde Q^{(n)})$. Hence $h_n-t^{\rm min}_n$ remains bounded from above. 

Next we consider $t_n^{\rm max}:=\max\big \{|p|: p\in \rmsing(Q^{(n)})\big\}\in (0,h_n)$, and we claim that $t_n^{\rm max}\leq h_n-\delta$ for some $\delta>0$ independent of $n$. Without loss of generality, we can assume $t_n^{\rm max}$ is achieved at a singular point $p_n^{\rm max}$ belonging to $\{x_3< 0\}$ (the other case being analoguous). To prove the claim, we argue by contradiction assuming that $\tau^2_n:=h_n-t_n^{\rm max}\to 0$ as $n\to \infty$ for some (not relabeled) subsequence. We observe that $\Omega_n\cap B_\rho(q^-_n)=q^-_n+B^+_\rho$ with $B^+_\rho:=B_\rho\cap\{x_3>0\}$, and $Q^{(n)}={\bf e}_0$ on $\partial\Omega_n\cap B_\rho(q^-_n)=q^-_n+B_\rho\cap\{x_3=0\}$.  According to Remark~\ref{specgeomremreg}, we have $\frac{1}{\tau_n}\mathcal{E}_\lambda\big(Q^{(n)},\Omega_n\cap B_{\tau_n}(q^-_n)\big)\leq \frac{8}{\rho}\mathcal{E}_\lambda(Q^{(n)},\Omega_n\cap B_\rho(q^-_n))$. Since $\Omega_n\cap B_\rho(q^-_n)\subset \Omega_n\setminus\mathfrak{C}^{h-\rho}_\ell$,  Corollary \ref{Corshiftedballs3} tells us that $\mathcal{E}_\lambda(Q^{(n)},\Omega_n\cap B_\rho(q^-_n))=O(1)$ as $n\to\infty$. Therefore,  $\mathcal{E}_\lambda\big(Q^{(n)},\Omega_n\cap B_{\tau_n}(q^-_n)\big)= O(\tau_n)$ as $n\to\infty$.  
Then we consider the translated and rescaled map $\widehat Q^{(n)}(x):=Q^{(n)}(\tau_nx+q^-_n)$ which satisfies  $\widehat Q^{(n)}={\bf e}_0$ on $B_\rho\cap\{x_3=0\}$.   Since $\Omega_n\cap B_\rho(q^-_n)\subset \Omega_n\setminus\mathfrak{C}^{h-\rho}_\ell$, we deduce from Corollary~\ref{Corshiftedballs3} that $\mathcal{E}_{\lambda\tau_n^2}(\widehat Q^{(n)},B^+_{1})= \frac{1}{\tau_n}\mathcal{E}_\lambda(Q^{(n)},\Omega_n\cap B_\rho(q^-_n))\leq C$ for a constant $C$ independent of $n$. 
According to \cite[Theorem 5.5]{DMP2}, there exists a (not relabeled) subsequence and $\widehat Q^*\in W^{1,2}_{\rm sym}(B_1^+;\bb S^4)$ such that 
$\widehat Q^{(n)}\to \widehat Q^*$ strongly in $W^{1,2}(B_r^+)$ for every $r\in(0,1)$.   By continuity of the trace operator, we have $\widehat Q^*={\bf e}_0$ on $B_1\cap\{x_3=0\}$. 
Rescaling variables, we realize that $\widehat Q^{(n)}$ is a weak solution of \eqref{MasterEq} in $B_1^+$ with $\lambda\tau_n^2\to 0$ in place of $\lambda$. From the locally strong $W^{1,2}$-convergence of $\widehat Q^{(n)}$, we deduce that $\widehat Q^*$ is a weakly harmonic map into $\bb S^4$ in $B_1^+$. Now we observe that (after rescaling variables), $\widehat Q^{(n)}$ satisfies the interior monotonicity formula from Proposition \ref{intmonoform} in $B_1^+$ with  $\lambda\tau_n^2\to 0$ in place of $\lambda$. Once again, by the established locally strong $W^{1,2}$-convergence, we infer that $\widehat Q^*$ satisfies the same interior monotonicity formula with $\lambda=0$. In view of Remark~\ref{specgeomremreg} (applied to $Q^{(n)}$), the same argument shows that $\widehat Q^*$ satisfies the boundary monotonicity formula \eqref{specifbdmonotform}    with $\lambda=0$ for balls centered on $B_1\cap\{x_3=0\}$. This is then enough to apply the boundary regularity theory from \cite{Scheven} (see also \cite[Section 2]{DMP1}) and conclude that $\widehat Q^*$ is smooth in a neighborhood of  $B_1\cap\{x_3=0\}$. In particular, we can find a radius $\eta\in(0,1)$ such that $\frac{1}{\eta}\mathcal{E}_0(\widehat Q^*, B_\eta^+)\leq \boldsymbol{\varepsilon}^\sharp_{\rm bd}/4$, where $\boldsymbol{\varepsilon}^\sharp_{\rm bd}>0$ is the universal constant provided by   Remark~\ref{specgeomremreg}. By strong $W^{1,2}$-convergence, we then have $\frac{1}{\eta}\mathcal{E}_0(\widehat Q^{(n)}, B_\eta^+)\leq \boldsymbol{\varepsilon}^\sharp_{\rm bd}/2$ for $n$ large enough. According to Remark~\ref{specgeomremreg}, it implies that $\widehat Q^{(n)}$ is smooth in $B_{\boldsymbol{\kappa}^\sharp\eta}\cap\{x_3\geq 0\}$ for $n$ large enough where $\boldsymbol{\kappa}^\sharp>0$ is a further universal constant. On the other hand, by construction $\widehat Q^{(n)}$ is singular at $\bar p_n:=(p_n^{\rm max}-q_n^-)/\tau_n=(0,0,\tau_n)\to 0$ as $n\to\infty$, a contradiction. This proves the upper bound $t_n^{\rm max}\leq h_n-\delta$ for a constant $\delta>0$ that we can choose to be equal to $1/\alpha$, taking the constant $\alpha$ larger if necessary. 

It finally remains to prove that ${\rm Card}\,\Sigma_n=O(1)$ as $n\to\infty$. By  inequality in \eqref{inequtiloublie} (applied with $r=(\rho+\alpha+1)/2$ and $z=h_n-r$), we have $\mathcal{E}_\lambda\big(Q^{(n)}, \mathfrak{C}^{h_n-\rho}_\ell\cap\{|x_3|>h_n-\rho-\alpha-1\}\big)=O(1)$, which in view of  Corollary~\ref{Corshiftedballs3} yields $\mathcal{E}_\lambda\big(Q^{(n)}, \Omega_n\cap\{|x_3|>h_n-\rho-\alpha-1\}\big)=O(1)$ as $n\to \infty$. Hence, there exists a constant $M>0$ independent of $n$ such that   $\mathcal{E}_\lambda\big(Q^{(n)}, B_{1/\alpha}(x)\big)\leq M$ for every $x\in \{x_3\text{-axis}\}\cap\{h_n-\alpha\leq |x_3|\leq h_n-\frac{1}{\alpha}\}$. In turn, applying Lemma~\ref{infdistsing} in such a ball $B_{1/\alpha}(x)$ shows that there exists a constant ${\bf c}={\bf c}(M,\lambda,\alpha)>0$ (independent of $n$ and $x$) such that $|p-p^\prime|\geq {\bf c}$ for every $p,p^\prime\in\Sigma_n\cap B_{1/(2\alpha)}(x)$ with $p\neq p^\prime$. 
Since  $\Sigma_n\subset \{x_3\text{-axis}\}\cap\{h_n-\alpha\leq |x_3|\leq h_n-\frac{1}{\alpha}\}$, we conclude that ${\rm Card}\,\Sigma_n\lesssim \alpha/{\bf c}$. Since this holds for every $n$ large enough, the proof is complete. 
\end{proof}


\subsection{Instability and symmetry breaking in long cylinders}

To conclude this section, we discuss two important consequences of Theorem \ref{thm:vertical-cylinders}. We first present a general result about the instability of singular configurations minimizing $\mathcal{E}_\lambda$ among $\bbS^1$-equivariant maps. This instability is  essentially issued from the instability of singular tangent maps for the Dirichlet energy (see \cite{LiWa}).

\begin{proposition}\label{split-instability}
Let $\Omega\subset\R^3$ be a bounded and axisymmetric open set with boundary of class~$C^3$, and let 
$Q_{\rm b} \in C^{1,1}(\partial \Omega;\mathbb{S}^4)$ be an $\mathbb{S}^1$-equivariant map. 
If $Q_\lambda$ is a minimizer of $\mathcal{E}_\lambda$ in the class $\mathcal{A}^{{\rm sym}}_{Q_{\rm b}}(\Omega)$ such that $\rmsing(Q_\lambda)\not=\emptyset$, then $Q_\lambda$ is an unstable critical point of $\mathcal{E}_\lambda$ in the class $\mathcal{A}_{Q_{\rm b}}(\Omega)$.
More precisely, for every radial function $\eta \in C^\infty_c(B_1 \setminus \{0\})$ satisfying  $\int_{B_1} |\nabla \eta |^2- \frac{2}{|x|^2} \eta^2 \, dx<0$ and  every $p \in {\rm sing}(Q_\lambda)$,  there exists  a small $r>0$ such that for every  $\bar{v}\in \bb S^4 \cap L_2$,  $Q_\lambda$ is unstable
  along the variations   $\Phi^{p,r}(x):=\frac1{\sqrt{r}}\eta\left(\frac{x-p}r \right)\bar{v}$, i.e., $\mathcal{E}_\lambda''(\Phi^{p,r};Q_\lambda)<0$.
\end{proposition}

\begin{proof}
 According to Theorem \ref{intregthm}, if  $p\in \rmsing(Q_\lambda)$, then there exist a degree-zero and equivariant homogeneous harmonic map $Q_* \in C^\infty(\R^3\setminus\{ 0\};\bb S^4)$ and $\nu>0$ such that  $Q_*$ is taking values in $L_0 \oplus L_1$, and   
\begin{equation}\label{asymptnearsingthm}
\|Q_\lambda^{p,r}-Q_*\|_{C^2(\overline{B_2}\setminus B_1)}=O(r^\nu)\quad\text{as $r\to 0$}\,,
\end{equation}
where $Q_\lambda^{p,r}(x):=Q_\lambda(p+r x)$. By formula \eqref{formtangmaps}, we have 
$|\nabla Q_*|^2=\frac{2}{|x|^2}$. In turn, \eqref{asymptnearsingthm} implies that $|\nabla Q_\lambda^{p,r}|^2 \to \frac{2}{|x|^2}$ locally uniformly in $\R^3 \setminus \{0\}$.
 
Recall that the second variation of the energy at a general map $\Phi \in C^\infty_c( \Omega; \mathcal{S}_0)$ is defined as  
\[ \mathcal{E}^{''}_\lambda(\Phi;Q_\lambda):=\left[ \frac{d^2}{dt^2} \mathcal{E}_\lambda \left( \frac{Q_\lambda+t\Phi}{|Q_\lambda+t\Phi|} \right) \right]_{t=0}  \, . \]
Using \eqref{MasterEq}, one may proceed as for second variation formula for harmonic maps (see e.g. \cite[Chapter 1]{LiWa2} or \cite[Section~4.3]{DMP1}), to obtain 
\begin{equation}
\label{secondvarElambda}
 \mathcal{E}^{''}_\lambda(\Phi;Q_\lambda)=\int_{\Omega} |\nabla \Phi_T|^2- |\nabla Q_\lambda|^2 |\Phi_T|^2 +\lambda (D^2 W (Q_\lambda) \Phi_T ):\Phi_T \, dx   \, ,
 \end{equation}
where $\Phi_T :=\Phi- (Q_\lambda:\Phi)Q_\lambda$ denotes the tangential component of $\Phi$ along $Q_\lambda$. Choosing $r>0$ small enough in such a way that $B_r(p)\subset \Omega$ and $B_r(p)\cap{\rm sing}(Q_\lambda)=\{p\}$, we have $\Phi^{p,r} \in C^\infty_c( \Omega; L_2)$, and rescaling/translating  $B_r(p)$ to the unit ball $B_1(0)$ yields 
$$  \mathcal{E}^{''}_\lambda(\Phi^{p,r};Q_\lambda)= \int_{B_1} |\nabla \Phi^{0,1}_T|^2- |\nabla Q^{p,r}_\lambda|^2 |\Phi^{0,1}_T|^2+\lambda r^2  (D^2 W (Q^{p,r}_\lambda) \Phi^{0,1}_T):\Phi^{0,1}_T  \, dx
$$
Since $\Phi^{0,1}=\eta \bar v $ and $Q_*$ is taking values in $L_0\oplus L_1=L_2^\perp$, we infer that $\Phi^{0,1}: Q_\lambda^{p,r} \to \Phi^{0,1}: Q_*=0$ in $C^1_{\rm loc}(B_1\setminus\{0\})$
as $r\to 0$. Hence $\Phi^{0,1}_T\to \Phi^{0,1}=\eta \bar v$ in  $C^1_{\rm loc}(B_1\setminus\{0\})$. Since  $D^2W(Q_\lambda^{p,r}) \to D^2W (Q_*)$ in $C^0_{\rm loc}(B_1\setminus\{0\})$ and $\eta$ is compactly supported in $B_1\setminus\{0\}$, we have 
\[ \lim_{r \to 0} \mathcal{E}^{''}_\lambda(\Phi^{p,r};Q_\lambda)= \int_{B_1} |\nabla \left(\bar{v}\eta\right) |^2- \abs{\nabla Q_*}^2 \abs{\bar{v}\eta}^2 \, dy =\mathcal{E}_0''(\eta \bar v;Q_*) = \int_{B_1} |\nabla \eta |^2- \frac{2}{|x|^2} \eta^2 \, dx \, , \]
where we used the fact that $|\bar v|=1$ in the last equality.  
By the sharp Hardy inequality in $\R^3$ and the last equality above, there exist radial functions $\eta \in C^\infty_c (B_1 \setminus\{0\})$ such that $\mathcal{E}_0''(\eta\bar{v};Q_*)<0$ (see e.g.  \cite[proof of Proposition~4.7]{DMP1}). Then, for any such function, the conclusion follows for $r$ small enough.  
  \end{proof}

Combining Theorem~\ref{thm:vertical-cylinders} with Proposition~\ref{split-instability} and  the full regularity of global minimizers (without symmetry constraint) from \cite[Theorem~1.1]{DMP1}, we readily obtain the following consequences for minimizers of $\mathcal{E}_\lambda$ in sufficiently long cylinders.

\begin{corollary}
\label{inst+symmbreaking}
Assume that  $\lambda \ell^2 <\lambda_0$. Let $\Omega_n:=\mathfrak{C}^{h_n}_{\ell,\rho}$ for varying $h_n>\ell$, $h_n \nearrow +\infty$ as in Theorem~\ref{thm:vertical-cylinders},  and  $Q_{\rm b}^{(n)}:=Q_{\rm b}^{h_n}$ the homeotropic boundary data given by~\eqref{eq:radial-anchoring}. 
If $Q^{(n)}_{\rm sym}$ and  ${Q}_{\rm glob}^{(n)}$ are minimizers of $\mathcal{E}_\lambda$ in the respective classes $\mathcal{A}^{\rm sym}_{Q_{\rm b}^{(n)}}(\Omega_n)$ and $\mathcal{A}_{Q_{\rm b}^{(n)}}(\Omega_n)$, 
then the following properties hold  for $n$ large enough : 
\begin{enumerate}
\item[(i)]  $\rmsing (Q_{\rm sym}^{(n)})\neq \emptyset $ and  $Q_{\rm sym}^{(n)}$ is an unstable critical point of $\mathcal{E}_\lambda$ (under some symmetry breaking perturbations);
\vskip5pt

\item[(ii)]  $\rmsing (Q_{\rm glob}^{(n)})= \emptyset $ and ${Q}_{\rm glob}^{(n)}$ is not $\mathbb{S}^1$-equivariant.  In particular,  the orbit  under the $\bb S^1$-action $\{R \cdot Q_{\rm glob}^{(n)} (R^\trans \cdot )\}_{R\in\bb S^1}$ 
provides infinitely many  $\mathcal{E}_\lambda$-minimizers in the  class $\mathcal{A}_{Q_{\rm b}^{(n)}}(\Omega_n)$, and nonuniqueness holds. 
\end{enumerate}
\end{corollary}



\section{Torus minimizers in large cylinders}\label{sec:fat-cyl}

 We consider in this section the homeotropic boundary condition for large (smoothed) cylinders, the geometry opposite to the  one in the previous section. We shall prove that for sufficiently large cylinders, any equivariant minimizers is smooth and thus of torus type, as claimed in Theorem~\ref{thmlargecyl}. This result is the counterpart of the main result of the previous section in case of long (smoothed) cylinders but the conclusions are in the opposite direction. In the spirit of Section \ref{sectcoexball}, we shall exploit these two extreme cases to show that smooth and singular minimizers actually coexist for intermediate cylinders, i.e., Theorem \ref{thmcoexmovdom},  and deduce that symmetry breaking occurs when minimization is performed without the symmetry constraint (see Corollary~\ref{symbreakintermcyl}).
	
The analysis in case of large cylinders resembles the one in  Section \ref{SectSplit}. It essentially relies on a monotonicity formula, the construction of suitable competitors, local compactness of minimizers, and regularity theory. In the  last subsections, we obtain the coexistence result of Theorem \ref{thmcoexmovdom} applying the 
persistence of smoothness and the persistence of singularities for minimizers developed in Section \ref{sectcoexball}, and symmetry breaking follows by continuity w.r.t. the thickness of the infimum values of the energy functional.


\subsection{A priori energy bounds and local compactness}\label{subsec:hor-compactness}
In this subsection, we establish some preliminary results starting with the following monotonicity formula.

\begin{lemma}\label{monotformpacake}
Let $\mathfrak{C}^h_{\ell,\rho}$ be a smoothed cylinder with $2\rho < h < \ell -\rho$ and $Q_{\rm b}$ its homeotropic boundary data given by~\eqref{eq:radial-anchoring}. If $Q$ is minimizing $\mathcal{E}_\lambda$ over  $\mathcal{A}_{Q_{\rm b}}^{\rm sym}{(\mathfrak{C}^h_{\ell,\rho})}$, then 
			\begin{multline}\label{eq:mon-form-full}
				\frac{1}{r_2}\mathcal{E}_\lambda(Q,\mathfrak{C}^h_{\ell,\rho} \cap B_{r_2}) =  \frac{1}{r_1}\mathcal{E}_\lambda(Q, \mathfrak{C}^h_{\ell,\rho} \cap B_{r_1}) 	+ \int_{\mathfrak{C}^h_{\ell,\rho} \cap (B_{r_2} \setminus B_{r_1}) } \frac{1}{\abs{x}}\abs{\frac{\partial Q}{\partial \abs{x}}}^2 \,dx \\
				+ \int_{r_1}^{r_2} \frac{1}{r^2} \left(\int_{\mathfrak{C}^h_{\ell,\rho} \cap B_r} 2\lambda W(Q) \, dx \right) \,dr + \int_{r_1}^{r_2} \frac{h}{r^2} \left(\int_{\partial \mathfrak{C}^h_{\ell,\rho} \cap B_r} \frac{1}{2}\abs{\frac{\partial Q}{\partial \vec{n}}}^2 \,{\rm d}\mathcal{H}^2 \right) \,dr
			\end{multline}
 for every $h < r_1 < r_2 \leq \ell-\rho$.
\end{lemma}

\begin{proof}
For $h < r \leq \ell-\rho$, the boundaries $\partial\mathfrak{C}^h_{\ell,\rho}$ and $\partial B_r$ are transversal and $\partial B_r \cap {\rm sing}(Q) = \emptyset${{, h}}ence  Lemma~\ref{conservation-laws} applies. Choosing the vector field $V(x) = x$ in  Lemma~\ref{conservation-laws}, computations analogous  to those leading to \eqref{radial-identity-vert} yield
\begin{multline*}
	\frac{d}{dr} \left\{ \frac{1}{r} \int_{\mathfrak{C}^h_{\ell,\rho} \cap B_r} \left( \frac{1}{2} \abs{\nabla Q}^2 + \lambda W(Q) \right) \,dx \right\} = \\
	\frac{1}{r}\int_{\mathfrak{C}^h_{\ell,\rho} \cap \partial B_r} \abs{\pd{Q}{\vec{n}}}^2 \, d\mathcal{H}^2 + \frac{1}{r^2} \int_{\mathfrak{C}^h_{\ell,\rho} \cap B_r} 2\lambda W(Q)\, dx + \frac{h}{r^2} \int_{\partial \mathfrak{C}^h_{\ell,\rho} \cap B_r} \frac{1}{2} \abs{{\pd{Q}{\vec{n}}}}^2 \,d\mathcal{H}^2\,,
\end{multline*}
since $Q={\bf e}_0$ and $V\cdot \vec{n}=h$ on $\partial \mathfrak{C}^h_{l,\rho} \cap B_r$. 
Integrating now over $r \in (r_1,r_2)$ yields \eqref{eq:mon-form-full}.
\end{proof}

The next result provides a first a priori  estimate for the energy of minimizers in large cylinders.

\begin{lemma}\label{lemma:energy-estimates-large}
Let $\mathfrak{C}^h_{\ell,\rho}$ be a smoothed cylinder and  $Q_{\rm b}$ its homeotropic boundary data given by~\eqref{eq:radial-anchoring}.
If $Q$ is  minimizing  $\mathcal{E}_\lambda$ over $\mathcal{A}_{Q_{\rm b}}^{\rm sym}(\mathfrak{C}^h_{\ell,\rho})$, then $\mathcal{E}_\lambda(Q, \mathfrak{C}^h_{\ell,\rho}) \leq K \ell$ for a constant $K$ independent of $\ell$. 
\end{lemma}

\begin{proof}
We prove the announced energy estimate by constructing a suitable competitor $\widetilde Q$. To this purpose, we introduce the subdomains 	$\Omega'_\ell :=\mathfrak{C}^h_{\ell,\rho} \cap \{ \abs{x'} < \ell-\rho \}$ and  $\Omega''_\ell := \mathfrak{C}^h_{\ell,\rho} \cap \{\abs{x'} > \ell-\rho\}$. Noticing that $Q_{\rm b} = \eo$ on $\partial \mathfrak{C}^h_{\ell,\rho} \cap \{|x_3| = h\}$, we set $\widetilde Q(x)={\bf e}_0$ for $x\in {{\overline{\Omega'_\ell}}}$. To define $\widetilde Q$ in ${{\overline{\Omega''_\ell}}}$, we consider the vertical slice  $\mathcal{D}^+_{\Omega''_\ell}$ defined in  \eqref{vertsectsymdom}. Notice {{that $\ell > 2\rho$, so}} that $\mathcal{D}^+_{\Omega''_\ell}$ is a translate of $\mathcal{D}^+_{\Omega''_{2\rho}}$. Moreover, the shape of $\mathcal{D}^+_{\Omega''_\ell}$ is independent of $\ell$. We set  $\widetilde Q(x)=Q_{\rm b}(x)$ for $x\in {{\partial \Omega''_\ell}}\cap \partial \mathfrak{C}^h_{\ell,\rho}$  and{{, since}} $\widetilde Q(x)={\bf e}_0$ for $x\in {{\partial \Omega''_\ell}}\cap\{|x^\prime|=\ell-\rho\}$, we have $\widetilde Q\in  \rmLip({{\partial \Omega''_\ell}}; \bbS^4)$ and the translated map $\psi: (x_1,0,x_3)\in \partial \mathcal{D}^+_{\Omega''_{2\rho}}\mapsto\widetilde Q(x_1-\ell+2\rho,0,x_3)\in \bbS^4$ is independent of $\ell$. Since $\bbS^4$ is simply connected, $\psi$ admits an extension $\Psi\in {\rm Lip} ({{\overline{ \mathcal{D}^+_{\Omega''_\ell}}}}; \bbS^4)$. Then we set $\widetilde Q(x_1,0,x_3)=\Psi(x_1-\ell+2\rho,0,x_3)$ for $(x_1,0,x_3)\in  \mathcal{D}^+_{\Omega''_\ell}$.  
Finally, we extend $\widetilde{Q}$ to $\Omega_\ell''$ by $\bbS^1$-equivariance, that is {{setting}} $\widetilde Q(Rx)=R\widetilde Q(x)R^\trans$ for every $R \in \mathbb{S}^1$ and $x\in  \mathcal{D}^+_{\Omega''_\ell}$. By construction $\widetilde Q\in {\rm Lip}(\mathfrak{C}^h_{\ell,\rho};\bbS^4)$ with a Lipschitz norm independent of $\ell$, $\widetilde Q$ is $\bbS^1$-equivariant, and $\widetilde Q=Q_{\rm b}$ on $\partial \mathfrak{C}^h_{\ell,\rho}$. Hence, as in \eqref{Eq:equiv-energy-E0}, we have
$$\mathcal{E}_\lambda(\widetilde{Q},\mathfrak{C}^h_{\ell,\rho})=\mathcal{E}_\lambda(\widetilde{Q},\Omega''_\ell) =\pi \int_{\mathcal{D}_{\Omega_\ell''}} \bigg(|\nabla  \Psi|^2  + \frac{|\Psi_1|^2 + 4|\Psi_2|^2}{x_1^2}+{{2}}\lambda W(\Psi) \bigg)\,x_1\,dx_1dx_3 \leq K\ell \,,$$
for some $K{{=K(h,\rho)}}>0$ independent of $\ell$. By minimality of $Q$, we have $\mathcal{E}_\lambda(Q, \mathfrak{C}^h_{\ell,\rho}) \leq \mathcal{E}_\lambda(\widetilde{Q},\mathfrak{C}^h_{\ell,\rho})$ and the conclusion follows. 
\end{proof}

\begin{definition}
\label{hor-minimnizers}
Let $\mathfrak{C}^h_\ell$ be a cylinder with $h<\infty$.  We call {\it top/bottom boundary} of the cylinder~$\mathfrak{C}^h_\ell$, the set 
$$\partial^{=} \mathfrak{C}^h_\ell:=\partial \mathfrak{C}^h_\ell\cap\{|x_3|=h\}=\mathbb{D}_\ell\times\{-h,h\}\,.$$  
 An equivariant map $Q\in W^{1,2}_{\rm loc}(\mathfrak{C}^h_\ell;\bb S^4)$ is said to be an {{\it equivariant local minimizer of $\mathcal{E}_{\lambda}$ in $\mathfrak{C}^h_\ell$ up to the top/bottom boundary}} if  for every $\eta\in(0,\ell)$, $Q\in W^{1,2}_{\rm sym}(\mathfrak{C}^h_\eta;\bb S^4)$ and $\mathcal{E}_{\lambda}(Q,\mathfrak{C}^h_\eta)\leq \mathcal{E}_{\lambda}(\widetilde Q,\mathfrak{C}^h_\eta)$ for every  $\widetilde{Q}\in W^{1,2}_{\rm sym}(\mathfrak{C}^h_\eta;\bb S^4)$ satisfying $\widetilde{Q}=Q$ on $\partial \mathfrak{C}^h_\eta$. 
\end{definition}

\begin{remark}\label{remmonotformlocminpanc}
According to Remark \ref{specgeomremreg}, the regularity theory from Section \ref{secregth} applies to an equivariant local minimizer $Q$ of $\mathcal{E}_{\lambda}$ in $\mathfrak{C}^h_\ell$  up to the top/bottom boundary satisfying $Q={\bf e}_0$ on~$\partial^{=} \mathfrak{C}^h_\ell$. It shows that $Q$ is smooth in the interior of $ \mathfrak{C}^h_\ell$ 
and up to $\partial^{=} \mathfrak{C}^h_\ell$ away from finitely many points located on $\{x_3\text{-axis}\}\cap\mathfrak{C}^h_\ell$. As a consequence, the computations from the proof of  Lemma~\ref{conservation-laws} can be performed, and as in Lemma \ref{monotformpacake}, we infer that identity \eqref{eq:mon-form-full} holds for $h<r_1<r_2<\ell$ and $\mathfrak{C}^h_\ell$ instead of $\mathfrak{C}^h_{\ell,\rho}$. 
\end{remark}

The following compactness lemma will be repeatedly used in the sequel.

\begin{lemma}\label{lemma:local-hor-cpt}
Let $\mathfrak{C}^h_\ell$ be a bounded cylinder, and let  $\{Q_j\} \subset W^{1,2}_{\rm sym}(\mathfrak{C}^h_\ell;\bb S^4)$  be a sequence such that each $Q_j$ is  an equivariant local minimizer of $\mathcal{E}_{\lambda}$ in $\mathfrak{C}^h_\ell$ up to the top/bottom boundary and $Q_j={\bf e}_0$ on $\partial^{\rm =} \mathfrak{C}^h_\ell$. If $\sup_j \mathcal{E}_{\lambda}(Q_j,\mathfrak{C}^h_\ell)<\infty$, then there exists a (not relabeled) subsequence such that $Q_j \to Q_*$ strongly in $W^{1,2}( \mathfrak{C}^h_\eta)$ for every $\eta\in(0,\ell)$, where  $Q_* \in W^{1,2}_{\rm sym}(\mathfrak{C}^h_\ell;\bb S^4)$ is an equivariant local minimizer of $\mathcal{E}_\lambda$ up  to the top/bottom boundary satisfying $Q_*={\bf e}_0$ on $\partial^{\rm =} \mathfrak{C}^h_\ell$. 
\end{lemma}

\begin{proof}
	Since  the sequence $\{Q_j\}$ has equibounded $\mathcal{E}_\lambda$-energy, $\{Q_j\}$ is  bounded in $W^{1,2}_{\rm sym}(\mathfrak{C}^h_\ell)$, whence the existence of a (not relabeled) subsequence and $Q_* \in W^{1,2}_{\rm sym}(\mathfrak{C}^h_\ell;\bbS^4)$ such that $Q_j \rightharpoonup Q_*$ weakly in $W^{1,2}(\mathfrak{C}^h_\ell)$. 
	In addition, $Q_*=\eo$ on $\partial^{\rm =} \mathfrak{C}^h_\ell$ by  locality and  weak continuity of the trace operator.
	
We now argue as in Lemma \ref{verticalcompactness}. Fix an arbitrary $r \in (0,\ell)$ and $\delta \in (0,\ell-r)$. Extracting a further subsequence if necessary, by Fatou's lemma and Fubini's theorem there exists $\eta \in (r, r+\delta)$ such that	\[
		\lim_{j \to \infty}\int_{\Gamma_\eta} \abs{Q_j - Q_*}^2\, d\mathcal{H}^2 = 0\quad\text{and}\quad \int_{\Gamma_\eta} |\nabla Q_j|^2 + |\nabla Q_*|^2\, d\mathcal{H}^2 \leq C\,,
	\]
where $\Gamma_\eta=\partial^{\rm lat} \mathfrak{C}^h_{\eta}$  (see \eqref{deflatbound}), and $C > 0$ does not depend on $j$. 
Setting $\gamma_\eta:=\{(\eta, 0,x_3) : |x_3|<h\}$, we observe that $\Gamma_\eta=\bigcup_{R\in\bbS^1} R\cdot\gamma_\eta$. By $\bbS^1$-equivariance, we deduce that the restriction of $Q_j$ to $\gamma_\eta$ is weakly convergent to $Q_*$ in $W^{1,2}(\gamma_\eta)$. 
By the compact embedding 
$W^{1,2}(\gamma_\eta)\hookrightarrow C^0(\overline{\gamma_\eta})$, we infer that 
$Q_j \to Q_*$ uniformly on $\overline{\gamma_\eta}$.  By equivariance again, $Q_j \to Q_*$ uniformly on $\overline{\Gamma}_\eta$.
\vskip3pt
	
We fix an arbitrary $\bar Q\in W^{1,2}_{\rm sym}(\mathfrak{C}^h_r; \bbS^4)$ satisfying  $\bar Q=Q_*$ on $\partial\mathfrak{C}^h_r$.  We extend $\bar Q$ to $\mathfrak{C}^h_\eta$ setting 
$\bar Q=Q_*$ in  $\mathfrak{C}^h_\eta\setminus \mathfrak{C}^h_r$, and we set $\sigma_j := \| {Q}_j - {Q}_* \|_{L^\infty(\Gamma_\eta)} + 2^{-j}\to 0$.
 For $j$ large enough so that $\sigma_j < 1$ and $r< (1-\sigma_j)\eta$, we define 
\[
		v_j (x):= 
\begin{cases}
\displaystyle  \frac{\abs{x'}-(1-\sigma_j)\eta}{\sigma_j\eta} \left({Q}_j\left(\eta\frac{x^\prime}{|x^\prime|},x_3\right) - {Q}_*\left(\eta\frac{x^\prime}{|x^\prime|},x_3\right)\right)+ {Q}_*\left(\eta\frac{x^\prime}{|x^\prime|},x_3\right) & \mbox{if } x\in \mathfrak{C}^h_{\eta} \setminus \mathfrak{C}^h_{(1-\sigma_j)\eta} \,,\\[8pt] 
\displaystyle\bar{Q} \left(\frac{x'}{1-\sigma_j},x_3 \right) & \mbox{if } x\in \mathfrak{C}^h_{(1-\sigma_j)\eta} \,.
\end{cases}
	\]
Then $v_j\in W^{1,2}_{\rm sym}(\mathfrak{C}^h_\eta; \mathcal{S}_0)$, $|v_j|=1$ in  $\mathfrak{C}^h_{(1-\sigma_j)\eta}$, and
 $v_j = Q_j$ on $\partial \mathfrak{C}^h_\eta$ {{(indeed, $v_j = Q_j = \eo$ on $\partial^= \mathfrak{C}^h_\eta$)}}. Since $\sigma_j\to 0$, 
 we have $|v_j| \to 1$ uniformly in $ \mathfrak{C}^h_{\eta} \setminus \mathfrak{C}^h_{(1-\sigma_j)\eta}$, and thus $\||v_j|-1\|_{L^\infty(\mathfrak{C}^h_{\ell})}\to 0$. 
 In addition, $v_j \to \bar{Q}$ a.e. in $\mathfrak{C}^h_\eta$ because $\bar{Q}(\cdot,x_3) \in C^0(\overline{\bbD_\eta})$ for a.e. $x_3$ by Lemma~\ref{lemma:s1eq-emb}. For $j$ large enough we have $|v_j|\geq 1/2$ and we can define the competitor $\widetilde{Q}_j := v_j/|v_j| \in W^{1,2}_{\rm sym}(\mathfrak{C}^h_{\eta};\bbS^4)$ which satisfies  $\widetilde{Q}_j =Q_j$ on $\partial \mathfrak{C}^h_{\eta}$. As in the proof of Lemma \ref{verticalcompactness}, we have
\[
		\int_{\mathfrak{C}^h_\eta} |\nabla \widetilde{Q}_j|^2 \, dx \leq \int_{\mathfrak{C}^h_\eta} |\nabla \bar{Q}|^2 \,dx + C \sigma_j\,, 
	\]
which implies that $\mathcal{E}_\lambda(\widetilde{Q}_j,\mathfrak{C}^h_\eta)\to \mathcal{E}_\lambda(\bar {Q},\mathfrak{C}^h_\eta)$ as $j\to \infty$.  {{In addition, b}}y minimality of $Q_j$ we have 
${{\limsup_j}}\mathcal{E}_\lambda({Q}_j,\mathfrak{C}^h_\eta)\leq {{\limsup_j}} \mathcal{E}_\lambda(\widetilde{Q}_j,\mathfrak{C}^h_\eta)$. Letting $j\to \infty$ yields 
$\mathcal{E}_\lambda({Q}_*,\mathfrak{C}^h_\eta)\leq \liminf_j \mathcal{E}_\lambda({Q}_j,\mathfrak{C}^h_\eta)\leq  \mathcal{E}_\lambda(\bar {Q},\mathfrak{C}^h_\eta)$ by weak lower semicontinuity of $\mathcal{E}_\lambda$. Since $\bar Q=Q_*$ in  $\mathfrak{C}^h_\eta\setminus \mathfrak{C}^h_r$, it implies that $\mathcal{E}_\lambda({Q}_*,\mathfrak{C}^h_r)\leq \mathcal{E}_\lambda(\bar {Q},\mathfrak{C}^h_r)$ proving the minimality of $Q_*$ in $\mathfrak{C}^h_r$. As in the proof of  Lemma \ref{verticalcompactness} again, choosing $\bar Q=Q_*$ implies the strong $W^{1,2}$-convergence of $Q_j$ in  $\mathfrak{C}^h_r$.
\end{proof}


\subsection{Sublinear energy growth and proof of Theorem \ref{thmlargecyl}}\label{subsec:sublinear+torus}

The next result is a  key step in proving Theorem \ref{thmlargecyl}, in particular to control the asymptotic location of the biaxiallity sets. 

\begin{lemma}\label{lemma:lin-sublin}
Let $\ell_j\to+\infty$ be an increasing sequence. For each $j\in \N$, let  $Q_j \in W^{1,2}_{\rm sym}(\mathfrak{C}^h_{\ell_j};\bbS^4)$ be an equivariant local minimizer of $\mathcal{E}_{\lambda}$ in $\mathfrak{C}^h_{\ell_j}$ up to the top/bottom boundary satisfying $Q_j ={\bf e}_0$ on~$\partial^{\rm =} \mathfrak{C}^h_{\ell_j}$. If $\mathcal{E}_\lambda(Q_j, \mathfrak{C}^h_{\ell_j}) = O(\ell_j)$ as $j \to \infty$, then there exists a constant $\bar{\veps} > 0$ independent of $j$ such that the following holds:  for every $\veps \in (0,\bar{\veps})$, there exist $d_\veps > 0$ and $j_\veps \in \N$ such that 
\begin{equation}\label{mard1936main}
\mathcal{E}_\lambda\big(Q_j,B_{\ell_j-d_\veps} \cap \mathfrak{C}^h_{\ell_j}\big) \leq C_* \,\veps (\ell_j-d_\veps)\qquad\forall j\geq j_\veps\,,
\end{equation}
where $C_*$ denotes a constant independent of $j$ and $\veps$. In particular,
\begin{equation}\label{mard1936main2}
 \mathcal{E}_\lambda\big(Q_j,B_{\sigma \ell_j} \cap \mathfrak{C}^h_{\ell_j}\big) = o(\ell_j)\quad\text{as $j\to\infty$}
 \end{equation}
for every $\sigma\in(0,1)$. 
\end{lemma}

\begin{proof}
By assumption, $\mathcal{E}_\lambda(Q_j,\mathfrak{C}^h_{\ell_j}) \leq \bar{C} \ell_j$  for some $\bar{C}>0$ independent of $j$. We claim that for every $\veps\in(0,1/2)$, there exists an integer $j_\veps\geq 1$ and $d_\veps>1$ independent of $j$ such that 
$$\inf_{r\in(\ell_j-d_\veps,\ell_j)}\frac{1}{r}\mathcal{E}_\lambda(Q_j,\Gamma_r)<\veps \qquad\forall j\geq j_\veps\,,$$
where $\Gamma_r:=\partial^{\rm lat}\mathfrak{C}^h_{r}$ (see \eqref{deflatbound}). 
Indeed, given $\veps \in (0,1/2)$ and $j_\veps\geq 1$ to be chosen,  we have for  $0<d\leq\ell_{j_\veps}$ and $j\geq j_\veps$, 
	\begin{multline*}
		\bar{C} \geq \frac{1}{\ell_j}\mathcal{E}_\lambda\big(Q_j,\mathfrak{C}^h_{\ell_j} \setminus \mathfrak{C}^h_{\ell_j-d}\big) = \frac{1}{\ell_j}\int_{\ell_j-d}^{\ell_j} \bigg(\frac{\mathcal{E}_\lambda(Q_j,\Gamma_t)}{t} \bigg) t \,dt \\
		\geq \left( \inf_{r \in (\ell_j-d,\ell_j)} \frac{\mathcal{E}_\lambda(Q_j, \Gamma_r)}{r} \right) \frac{1}{\ell_j} \int_{\ell_j-d}^{\ell_j}  t\,dt \geq \left( \inf_{r \in (\ell_j-d,\ell_j)} \frac{\mathcal{E}_\lambda(Q_j, \Gamma_r)}{r} \right)  \frac{d}2 \,,
	\end{multline*}
and the claim {{follows}} whenever {{we choose}} $d_\veps > \frac{2\bar{C}}{\veps}$ and $j_\veps$ such that $\ell_{j_\veps}\geq d_\veps$. 
	 
As a consequence, 	for an arbitrary $\veps\in(0,1/2)$ and $j\geq j_\veps$,  there exists $r^\veps_j\in(\ell_j-d_\veps,\ell_j)$ such that 
\begin{equation}\label{eq:small-2}
		 \frac{1}{r_j^\veps}\mathcal{E}_\lambda(Q_j, \Gamma_{r_j^\veps}) < \veps  \,.
\end{equation}
Note that $r_j^\veps\to {{+}} \infty$ since $\ell_j\to {{+}}\infty$. 	From \eqref{eq:small-2} and $\bbS^1$-equivariance, we infer that for $j\geq j_\veps$, 
\[
	\pi \int_{\gamma_{r_j^\veps}} \abs{\pd{Q_j}{x_3}}^2 \,dx_3 = \frac{1}{r_j^\veps} \int_{\Gamma_{r_j^\veps}} \frac12\abs{\pd{Q_j}{x_3}}^2 \,d \mathcal{H}^2 \leq\frac{1}{r_j^\veps}\mathcal{E}_\lambda(Q_j, \Gamma_{r_j^\veps}) < \veps\,,
	\] 
where $\gamma_{r_j^\veps}:=\{(r_j^\veps, 0,x_3) : |x_3|<h\}$. Since $Q_j(r_j^\veps,0,\pm h) = \eo$, we deduce (again by $\bbS^1$-equivariance) that 
\begin{equation}\label{eq:small-3}
		\|{Q}_j - \eo \|_{L^\infty(\Gamma_{r_j^\veps})} \leq C_*\sqrt{h\veps} \qquad \forall j\geq j_\veps \,,
\end{equation}
	for some universal  constant $C_*>0$. 
	
Next we define 	for $j \geq j_\veps$ and $x\in\mathfrak{C}^h_{r_j^\veps}$, 
	\[
		v_j^\veps(x) :=
		\begin{cases}
			\displaystyle \big(|{x'}| - r_j^\veps + 1\big) \bigg({{Q_j}}\left(r_j^\veps\frac{x^\prime}{|x^\prime|},x_3\right)-\eo\bigg) + \eo & \mbox{if } x\in\mathfrak{C}^h_{r_j^\veps} \setminus \mathfrak{C}^h_{r_j^\veps-1}\,, \\[5pt]
			\eo & \mbox{if } x\in\mathfrak{C}^h_{r_j^\veps-1}\,.
		\end{cases}
	\]
Then $v_j^\veps \in W^{1,2}_{\rm sym}(\mathfrak{C}^h_{r_j^\veps};\mathcal{S}_0)$ satisfies $v_j^\veps=Q_j$ on $\partial \mathfrak{C}^h_{r_j^\veps}$, and $|v_j^\veps|=1$ in $\mathfrak{C}^h_{r_j^\veps-1}$. Moreover, combining \eqref{eq:small-2} and \eqref{eq:small-3}, we obtain 
$$  \int_{\mathfrak{C}^h_{r_j^\veps} \setminus \mathfrak{C}^h_{r_j^\veps-1}} |\nabla v_j^\veps|^2 \,dx\leq C \veps r_j^\veps\,,$$
for some constant $C$ independent of $\veps$ and $j$.  

Now we choose $\bar \veps\in(0,1/2)$ in such a way that $C_*\sqrt{h\bar\veps}{{<}}1/2$. Then, for $\veps\in(0,\bar\veps)$ arbitrary, we have $|1-|v_j^\veps||\leq |v_j^\veps-{\bf e}_0|<1/2$ in $\mathfrak{C}^h_{r_j^\veps}$ by \eqref{eq:small-3}. Thus we can define for $j\geq j_\veps$, 
$$\widetilde{Q}_j^\veps:=\frac{v_j^\veps}{|v^\veps_j|}\in W^{1,2}_{\rm sym}(\mathfrak{C}^h_{r_j^\veps};\bbS^4){{,}}$$
which satisfies $\widetilde{Q}_j^\veps=Q_j$ on $\partial \mathfrak{C}^h_{r_j^\veps}$, and 
\begin{equation}\label{mard1923}
\int_{\mathfrak{C}^h_{r_j^\veps} \setminus \mathfrak{C}^h_{r_j^\veps-1}} |\nabla \widetilde{Q}_j^\veps|^2 \,dx \leq C \veps r_j^\veps 
\end{equation}
for a further constant $C$ independent of $\veps$ and $j$. In addition, $|\widetilde{Q}_j^\veps-{\bf e}_0|\leq 3 |v_j^\veps-{\bf e}_0|\leq 3C_*\sqrt{h\veps}$ in $ \mathfrak{C}^h_{r_j^\veps}$ once again by \eqref{eq:small-3}. Consequently, 
\begin{equation}\label{mard1923bis}
W(\widetilde{Q}_j^\veps)\leq C^\prime \veps\quad \text{in $ \mathfrak{C}^h_{r_j^\veps}$}\,,
\end{equation}
still for a constant $C^\prime$ independent of $\veps$ and $j$, by Taylor expansion of $W$ near ${\bf e}_0$ and choosing $\bar \veps$ smaller if necessary. 

By minimality of $Q_j$, we conclude from \eqref{mard1923} and  \eqref{mard1923bis}  that 
	\[
		\mathcal{E}_\lambda(Q_j,\mathfrak{C}^h_{r_j^\veps}) \leq \mathcal{E}_\lambda(\widetilde{Q}_j^\veps,\mathfrak{C}^h_{r_j^\veps})= \mathcal{E}_\lambda\big(\widetilde{Q}_j^\veps,\mathfrak{C}^h_{r_j^\veps} \setminus \mathfrak{C}^h_{r_j^\veps-1}\big) \leq C \veps r_j^\veps \,, 
	\]
	where $C > 0$ is still independent of $j$ and $\veps$. Noticing that $B_{r^\veps_j} \cap \mathfrak{C}^h_{\ell_j} \subset \mathfrak{C}^h_{r^\veps_j}$, we deduce from Remark ~\ref{remmonotformlocminpanc} that 
$$\frac{1}{\ell_j-d_\veps}\mathcal{E}_\lambda\big(Q_j,B_{\ell_j-d_\veps}\cap \mathfrak{C}^h_{\ell_j}\big)\leq \frac{1}{r^\veps_j} \mathcal{E}_\lambda\big(Q_j,B_{r^\veps_j}\cap \mathfrak{C}^h_{\ell_j}\big)\leq C\veps\,,$$
proving \eqref{mard1936main}.

To complete the proof, we fix an arbitrary $\sigma\in(0,1)$. Then $\ell_j - d_\veps > \sigma \ell_j$ for $j$ large enough, and by  	Remark ~\ref{remmonotformlocminpanc} again, 
	\[
	\frac{1}{\sigma \ell_j} \mathcal{E}_\lambda\big(Q_j,B_{\sigma \ell_j} \cap \mathfrak{C}^h_{\ell_j}\big) \leq \frac{1}{\ell_j-d_\veps }\mathcal{E}_\lambda\big(Q_j, B_{\ell_j-d_\veps} \cap \mathfrak{C}^h_{\ell_j}\big) \leq C \veps \, . 
	\]
Then \eqref{mard1936main2}	follows from the arbitrariness of $\veps$ letting $j\to\infty$. 
\end{proof}

The following rigidity result is an immediate consequence of Lemma~\ref{lemma:lin-sublin}. 

\begin{corollary}\label{cor:rigidity-large}
Let $Q \in W^{1,2}_{\rm loc}(\mathfrak{C}^h_\infty;\bbS^4)$ be an equivariant  local minimizer of $\mathcal{E}_{\lambda}$ in $\mathfrak{C}^h_\infty$ up to the top/bottom boundary 
satisfying $Q={\bf e}_0$ on $\partial^{\rm =} \mathfrak{C}^h_{\infty}$. If  $\mathcal{E}_\lambda(Q, \mathfrak{C}^h_\ell) = O(\ell)$ as $\ell \to \infty$, then  $Q\equiv{\bf e}_0$. 
\end{corollary}

\begin{proof}
Let $\ell_j\to \infty$ be an increasing sequence, and set $Q_j := Q_{|{\mathfrak{C}^h_{\ell_j}}}$. Then $Q_j$ is an equivariant  local minimizer of $\mathcal{E}_{\lambda}$ in $\mathfrak{C}^h_{\ell_j}$ up to the top/bottom boundary 
satisfying $Q={\bf e}_0$ on $\partial^{\rm =} \mathfrak{C}^h_{\ell_j}$ and $\mathcal{E}_\lambda(Q_j,\mathfrak{C}^h_{\ell_j}) = O(\ell_j)$. According to Lemma~\ref{lemma:lin-sublin}, we have  
$\mathcal{E}_\lambda(Q_j,B_{\ell_j/2} \cap \mathfrak{C}^h_{ \ell_j}) = o(\ell_j)$. Let us now fix an arbitrary $\ell\gg h$. From Remark~\ref{remmonotformlocminpanc}, we deduce that for $j$ large enough so that $\ell_j>2\ell$, 
$$ 0 \leq \frac{1}{\ell}\mathcal{E}_\lambda\big(Q,B_\ell \cap \mathfrak{C}^h_\infty\big)= \frac{1}{\ell}\mathcal{E}_\lambda\big(Q,B_\ell \cap \mathfrak{C}^h_{\ell_j}\big) \leq \frac{2}{ \ell_j}\mathcal{E}_\lambda\big(Q_j,B_{\ell_j/2} \cap \mathfrak{C}^h_{\ell_j}\big) \to 0 \quad \mbox{as } j \to \infty \,.$$
Hence $	\mathcal{E}_\lambda(Q,B_\ell \cap \mathfrak{C}^h_\infty)=0$. From the arbitraniness of $\ell$, it follows that $Q$ is constant, and thus $Q \equiv \eo$ in view of its values on $\partial \mathfrak{C}^h_\infty$.
\end{proof}

We are now ready to prove Theorem \ref{thmlargecyl}.

\begin{proof}[Proof of Theorem \ref{thmlargecyl}] 
{\it Step 1.} We start  proving that $Q^{(n)} \to \eo$ strongly in $W^{1,2}(\mathfrak{C}^h_\ell)$ for every $\ell>0$. By Lemma~\ref{lemma:energy-estimates-large}, we have $\mathcal{E}_\lambda(Q^{(n)},\Omega_n) \leq K \ell_n$ for a constant $K$ independent of $n$. Given an arbitrary $\ell > h$, we consider $n$ large enough in such a way that $\ell_n >2 \ell+\rho$. Then 
$\mathfrak{C}^h_\ell \subset B_{2\ell}\cap \Omega_n\subset B_{\ell_n-\rho} \cap \Omega_n$. 
Applying the monotonicity formula \eqref{eq:mon-form-full} we obtain 
	\begin{equation}\label{Qn-growt} 
	\mathcal{E}_\lambda(Q^{(n)},\mathfrak{C}^h_\ell) \leq \mathcal{E}_\lambda\big(Q^{(n)},B_{2\ell}\cap \Omega_n\big)  \leq \frac{2\ell}{\ell_n-\rho} \mathcal{E}_\lambda\big(Q^{(n)},B_{\ell_n-\rho}\cap \Omega_n\big)
	\leq  \frac{2K\ell \ell_n}{\ell_n-\rho}\leq 4K \ell \, .
	\end{equation} 
In view of 	Lemma~\ref{lemma:local-hor-cpt}, we conclude that, up to a (not relabeled) subsequence, $Q^{(n)} \to Q_*$ strongly in $W^{1,2}(\mathfrak{C}^h_{\ell})$ for every $0<\ell<\infty$, where $Q_* \in W^{1,2}_{\rm loc}(\mathfrak{C}^h_\infty;\bbS^4)$ is  an equivariant local minimizer of $\mathcal{E}_\lambda$ up  to the top/bottom boundary satisfying $Q_*={\bf e}_0$ on $\partial^{\rm =} \mathfrak{C}^h_\infty$. By lower semicontinuity of $\mathcal{E}_\lambda$, letting $n\to\infty$ in \eqref{Qn-growt} yields $\mathcal{E}_\lambda(Q_*,\mathfrak{C}^h_\ell) \leq 4K\ell$. Hence Corollary~\ref{cor:rigidity-large} applies and $Q_*\equiv {\bf e}_0$. By uniqueness of the limit, we now infer that the full sequence $\{Q^{(n)}\}$ strongly converges to ${\bf e}_0$ in $W^{1,2}(\mathfrak{C}^h_{\ell})$ for every $0<\ell<\infty$. 
\vskip3pt

\noindent{\it Step 2.} 	We now {{accomplish the proof of $(ii)$ proving}} that ${\rm sing}(Q^{(n)})=\emptyset$ for $n$ large enough and that $Q^{(n)}$ converges smoothly to~${\bf e}_0$ locally in $\overline{\mathfrak{C}^h_\infty}$. To this purpose, we fix an arbitrary $\ell>h$ and we consider $n$ large enough so that $\ell_n\gg \ell+h$.  

\noindent{\sl Case 1: convergence near $\partial^=\mathfrak{C}^h_\ell$.} We fix an arbitrary point $x_0\in \partial^=\mathfrak{C}^h_\ell$. Given a radius $0<r_*<h/2$,  $B_{2r_*}(x_0)\cap\Omega_n\subset\mathfrak{C}^h_{\ell+h}$ is a half ball for $n$ large enough and Remark \ref{specgeomremreg} applies  since $Q^{(n)}={\bf e}_0$ on $B_{2r_*}(x_0)\cap\partial\Omega_n=B_{2r_*}(x_0)\cap \partial^=\mathfrak{C}^h_{\ell+h}$. We fix a radius $r_0\in(0,r_*/4)$ such that the conclusion of Remark \ref{specgeomremreg} holds (it only depends on $\lambda$). According to Step 1, we  have $\mathcal{E}_\lambda(Q^{(n)},\mathfrak{C}^h_{\ell+h}\big)\to 0$ as $n\to\infty$. 
Therefore, $\frac{1}{r_0}\mathcal{E}_\lambda(Q^{(n)},B_{r_0}(x_0)\cap\Omega_n)\leq \boldsymbol{\veps}^\sharp_{\rm bd}/2$ whenever $n$ is large enough (independently of $x_0\in \partial^=\mathfrak{C}^h_\ell$), where $\boldsymbol{\veps}^\sharp_{\rm bd}>0$ is the universal constant provided by Remark \ref{specgeomremreg}. Then Remark~\ref{specgeomremreg}  tells us that $Q^{(n)}$ in {{smooth and}} bounded in $C^k(B_{\boldsymbol{\kappa}^\sharp r_0/2}(x_0)\cap{{\overline{\Omega_n}}})$ for every $k\in\mathbb{N}$ (independently of $x_0$), where $\boldsymbol{\kappa}^\sharp\in(0,1)$ is a universal constant. By arbitrariness of $x_0$,  we deduce that ${\rm sing}(Q^{(n)})\cap \{h-\delta_*{{<}}|x_3|\leq h\}=\emptyset$ for $n$ large enough (recall that ${\rm sing}(Q^{(n)})\subset \{x_3-\text{axis}\}$) and that $Q^{(n)}$ is bounded in $C^k({{\overline{\mathfrak{C}^h_\ell}}}\cap\{h-\delta_*{{<}}|x_3|\leq h\})$ for every $k\in\mathbb{N}$ with 
 $\delta_*:=\boldsymbol{\kappa}^\sharp r_0/2>0$.  As a consequence, $Q^{(n)}\to {\bf e}_0$ in $C^k({{\overline{\mathfrak{C}^h_\ell}}}\cap\{h-\delta_*{{<}}|x_3|\leq h\})$ for every $k\in\mathbb{N}$. 
 \vskip3pt
 
\noindent{\sl Case 2: convergence in the interior of $\mathfrak{C}^h_\ell$.} We fix a radius $0<r_1<\delta_*$. In view of  Step 1, we have 
$\mathcal{E}_\lambda(Q^{(n)},\mathfrak{C}^h_{\ell+\delta_*})\leq (\boldsymbol{\veps}_{\rm in}r_1)/8$ for $n$ large enough, where $\boldsymbol{\veps}_{\rm in}>0$ is the universal constant provided by Proposition \ref{intepsregprop}. Choosing  $r_1$ small enough (depending only on $\lambda$) and an arbitrary point $x_0\in \mathfrak{C}^h_\ell\cap\{|x_3|\leq h-\delta_*\}$, we then have $\frac{1}{r_1}\mathcal{E}_\lambda\big(Q^{(n)},B_{r_1}(x_0)\big)\leq \boldsymbol{\veps}_{\rm in}/8$ so that Proposition~\ref{intepsregprop} applies. It shows that $Q^{(n)}$ is {{smooth and}} bounded in $C^k(B_{r_1/16}(x_0))$ for every $k\in\N$ (independently of $x_0$).  Once again,  it implies that $Q^{(n)}$ is smooth in  $\mathfrak{C}^h_\ell\cap\{|x_3|\leq h-\delta_*\}$ for $n$ large enough, so that  ${\rm sing}(Q^{(n)})=\emptyset$, and 
$Q^{(n)}\to {\bf e}_0$ in $C^k(\mathfrak{C}^h_\ell\cap\{|x_3|\leq h-\delta_*\})$ for every $k\in\mathbb{N}$.
\vskip3pt

\noindent{\it Step 3: proof of {{(i)}}.} We observe that the assumptions ($HP_0$)-($HP_3$) from \cite{DMP1,DMP2} are satisfied by $Q^{(n)}$ for $n$ large enough. Indeed, ${\rm sing}(Q^{(n)})=\emptyset $ for $n$ large enough so that ($HP_0$) holds (recall Theorem \ref{bdrregthm}).  Since the boundary condition $Q^{(n)}_{\rm b}$ is {{positively}} uniaxial, ($HP_1$) holds. Then, $\Omega_n$ being a topological ball and  $Q^{(n)}_{\rm b}$ the homeotropic boundary data \eqref{eq:radial-anchoring}, ($HP_2$) and ($HP_3$) trivially hold.  Hence $Q^{(n)}$ is a torus minimizer in the sense of \cite[Definition~7.6]{DMP2} (for $n$ large), and  \cite[Theorem~1.4]{DMP2}  provides the announced properties of the function $\beta_n:=\widetilde\beta\circ Q^{(n)}$. 
\vskip3pt 
	
\noindent{\it Step 4.}	Now it only remains to prove {{($iii$). To this purpose, we fix an arbitrary $t\in[-1,1)$ and, since}} $\widetilde\beta({\bf e}_0)=1$, we infer from the previous step that there is an integer 
$\bar{n}_t$ such that $\{\beta_n \leq t \} \cap \mathfrak{C}^h_{2h} = \emptyset$ for all $n \geq \bar{n}_t$. 

Since $Q^{(n)}$ is minimizing $\mathcal{E}_\lambda$ over  $\mathcal{A}_{Q_{\rm b}}^{\rm sym}(\Omega_n)$, Lemma~\ref{lemma:lin-sublin} applies in $\mathfrak{C}^h_{\ell_n-\rho}$, and we consider the constant $\bar\veps>0$ (independent of $n$) provided by this lemma. We fix a value $\veps=\veps(t)$ to be chosen later such that 
\begin{equation}\label{choiceeps1biaxesc}
0<\veps<\frac{1}{2}\min\big\{\boldsymbol{\veps}^\sharp_{\rm bd}/(2C_*), \boldsymbol{\veps}_{\rm in}/(8C_*),\bar\veps\big\}\,.
\end{equation}
where $C_*$ denotes the constant in inequality \ref{mard1936main}. 
According to Lemma~\ref{lemma:lin-sublin}, we can find $d_\veps>0$ and an integer $n_\veps\geq \bar n_t$ such that 
\begin{equation}\label{keyestmer15}
\mathcal{E}_\lambda\big(Q^{(n)}, B_{\ell_n-\rho-d_\veps} \cap \mathfrak{C}^h_{\ell_n-\rho}\big) \leq C_*\veps (\ell_n-\rho-d_\veps)\qquad\forall n\geq n_\veps\,,
\end{equation}
Enlarging $n_\veps$ and $d_\veps$ if necessary (see the proof of Lemma~\ref{lemma:lin-sublin}), we can assume that $\ell_n>2d_\veps+\rho+h$  for $n\geq n_\veps$, and $d_\veps>2h+\rho$ (so that 
$B_{2h}(x)\cap \partial B_{\ell_n-\rho-d_\veps}=\emptyset$ for every $x\in B_{\ell_n-2d_{\veps}}$). 

Let us now fix a  point $x_*\in B_{\ell_n-2d_\veps} \cap (\overline{\mathfrak{C}^h_{\ell_n-\rho}}\setminus \mathfrak{C}^h_{2h})$ (possibly depending on $n$) that either belongs to $\partial^={\mathfrak{C}^h_{\ell_n-\rho}}$ or to  $\mathfrak{C}^h_{\ell_n-\rho}\cap\{|x_3|\leq h-\delta_*\}$. By $\bbS^1$-equivariance, we may assume without loss of generality that $x_*=(x_{*,1},x_{*,2},x_{*,3})$ satisfies $x_{*,2}=0$ and $x_{*,1}\geq 2h$. If $x_*\in \partial^={\mathfrak{C}^h_{\ell_n-\rho}}$, we set 
$s:=r_0\in(0,h)$, and $s:=r_1\in(0,{{\delta_*}})$ if  $x_*\in \mathfrak{C}^h_{\ell_n-\rho}\cap\{|x_3|\leq h-\delta_*\}$ (note that $B_{r_1}(x_*)\subset\mathfrak{C}^h_{\ell_n-\rho}$ in this case). Next we denote $\ell_*:=x_{*,1}\in[2h,\ell_n-2d_\veps)$, $\Sigma^*_s:=B_s(x_*)\cap {\mathfrak{C}^h_{\ell_n-\rho}}\cap\{x_2=0\}$, and we consider the sets   
$$T^*_s := \displaystyle\bigcup_{\phi \in \left(-\frac{2s}{\ell_*},\frac{2s}{\ell_*}\right)} R_\phi \cdot \Sigma^{*}_s\;\,\,\text{ and }\,\,\; \mathfrak{T}^{*} :=  \displaystyle\bigcup_{\phi \in \left(0,2\pi\right)} R_\phi \cdot \Sigma^{*}_s\,.$$
Notice that $B_s(x_*)\cap {\mathfrak{C}^h_{\ell_n-\rho}}\subset T^*_s$ and $\mathfrak{T}^{*}\subset B_{\ell_*+2h}\cap {\mathfrak{C}^h_{\ell_n-\rho}}$.  Using the $\bbS^1$-equivariance and the monotonicity formula from Lemma \ref{monotformpacake}, we derive that
\begin{multline}\label{chainineqmer15}
		\frac{1}{s}\mathcal{E}_\lambda\big(Q^{(n)}, B_s(x_*) \cap \mathfrak{C}^h_{\ell_n-\rho}\big) \leq \frac{1}{s} \mathcal{E}_\lambda\big(Q^{(n)},T^*_s) =  \frac{2}{\pi \ell_*} \mathcal{E}_\lambda(Q^{(n)}, \mathfrak{T}^{*}) \\
	\leq \frac{4}{\pi}\cdot\frac{1}{\ell_*+2h} \mathcal{E}_\lambda\big(Q^{(n)},B_{\ell_*+2h} \cap \mathfrak{C}^h_{\ell_n-\rho}\big) \leq \frac{2}{\ell_n-\rho-d_\veps}\mathcal{E}_\lambda\big(Q^{(n)}, B_{\ell_n-\rho-d_\veps}\cap \mathfrak{C}^h_{\ell_n-\rho}\big)\,.
	\end{multline}
In view of \eqref{keyestmer15} and our choice of $\veps$, we conclude that for $n\geq n_\veps$, 
$$\frac{1}{s}\mathcal{E}_\lambda\big(Q^{(n)}, B_s(x_*) \cap \mathfrak{C}^h_{\ell_n-\rho}\big) \leq \min\big\{\boldsymbol{\veps}^\sharp_{\rm bd}/2, \boldsymbol{\veps}_{\rm in}/8\big\}\,. $$
As in Step 2, by Proposition~\ref{intepsregprop}  and Remark \ref{specgeomremreg}, it implies that  $|\nabla Q^{(n)}|\leq M$ in $B_{\delta_*}(x_*)$ if  $x_*\in\partial^={\mathfrak{C}^h_{\ell_n-\rho}}$, and $|\nabla Q^{(n)}|\leq M$ in $B_{r_1/16}(x_*)$ if  $x_*\in\mathfrak{C}^h_{\ell_n-\rho}\cap\{|x_3|\leq h-\delta_*\}$, where $M$ denotes a constant depending only on $\lambda$ and $h$. By arbitrariness of $x_*$ and in view of Step 2, we conclude that for $n\geq n_\veps$, 
\begin{equation}\label{gradbdmer15}
|\nabla Q^{(n)}|\leq M\quad\text{in $B_{\ell_n-2d_\veps} \cap \mathfrak{C}^h_{\ell_n-\rho}$} \,,
\end{equation}
for some constant $M$ depending only on $\lambda$ and $h$. 

We now claim that a suitable choice of $\veps=\veps(t)$ yields 
\begin{equation}\label{biaxescmer15}
\{\beta_n\leq t\}\cap(B_{\ell_n-2d_\veps} \cap \mathfrak{C}^h_{\ell_n-\rho})=\emptyset\qquad \forall n\geq n_\veps\,.
\end{equation}
To prove this claim, we assume by contradiction that for $n\geq n_\eps$ {{(more precisely, for a not relabeled subsequence)}},  there exists $x_t\in B_{\ell_n-2d_\veps} \cap \mathfrak{C}^h_{\ell_n-\rho}$ such that $\beta_n(x_t)\leq t$. Since $n_\veps\geq \bar n_t$, we must have $x_t\not\in  \mathfrak{C}^h_{2h}$. In view of \eqref{gradbdmer15}, we can find a radius $\tau\in(0,h)$ depending only on $t$, $\lambda$, and $h$ such that $\beta_n\leq \frac{1+t}{2}$ in $B_\tau(x_t)\cap \mathfrak{C}^h_{\ell_n-\rho}$. By $\bbS^1$-equivariance, it implies that $\beta_n\leq \frac{1+t}{2}$ in the set 
$$ \mathfrak{T}^{t} := \displaystyle\bigcup_{\phi \in \left(0,2\pi\right)} R_\phi \cdot \big(B_\tau(x_t)\cap \mathfrak{C}^h_{\ell_n-\rho}\big)\,.$$
Note that the volume {{of}} $ \mathfrak{T}^{t}$ is at least half of the volume of the solid torus $\bigcup_{\phi \in \left(0,2\pi\right)} R_\phi \cdot B_\tau(x_t)$. 
Setting $\ell_t:=|x^\prime_t|$ with $x_t=:(x_t^\prime,x_{t,3})$, we thus have 
\begin{equation}\label{lwbdpotmer15}
\frac{1}{\ell_t}\int_{ \mathfrak{T}^{t}}W(Q^{(n)})\,dx \geq \pi^2\tau^2\frac{1-t}{6\sqrt{6}}=:c_t\qquad\forall n\geq n_\veps\,.
\end{equation}
In addition to \eqref{choiceeps1biaxesc}, we now choose $\veps$ such that 
$$\veps< \frac{\lambda c_t}{2C_*}\,.$$ 
As in \eqref{chainineqmer15}, it follows from \eqref{keyestmer15} and \eqref{lwbdpotmer15} that for $n\geq n_\veps$, 
\begin{multline*}
 \lambda c_t\leq \frac{1}{\ell_t} \mathcal{E}_\lambda(Q^{(n)}, \mathfrak{T}^{t}) 
	\leq \frac{2}{\ell_t+2h} \mathcal{E}_\lambda\big(Q^{(n)},B_{\ell_t+2h} \cap \mathfrak{C}^h_{\ell_n-\rho}\big) \\
	\leq \frac{2}{\ell_n-\rho-d_\veps}\mathcal{E}_\lambda\big(Q^{(n)}, B_{\ell_n-\rho-d_\veps}\cap \mathfrak{C}^h_{\ell_n-\rho}\big)\leq C_*\veps<\lambda c_t\,,
\end{multline*}
a  contradiction proving \eqref{biaxescmer15}. Setting $d_t:=2d_\veps+h$ and noticing that $\mathfrak{C}^h_{\ell_n-d_t}\subset B_{\ell_n-2d_\veps} \cap \mathfrak{C}^h_{\ell_n-\rho}$, the conclusion follows with $n_t:=n_\veps$. 
\end{proof}


\subsection{Intermediate cylinders and coexistence results}\label{subsec:coex}

The purpose of this subsection is to prove coexistence of smooth/torus and singular/split minimizers for intermediate cylinders.  As a first step, we establish in 
Propositions \ref{perssmoothmovdom} and \ref{complemmovdom} the {\sl persistence of regularity} and {\sl persistence of singularities} properties when changing the shape of a smoothed cylinder. 

\begin{proposition}[\bf persistence of regularity]\label{perssmoothmovdom}
Let $h,\ell_\flat,\rho>0$ be fixed{{ with $0<4\rho<\ell_\flat < h$}}, and $\{\ell_n\}$ a sequence of positive numbers such that  $\ell_n\geq 3 \ell_\flat$. Assume that $\ell_n\to \ell_*$ as $n\to\infty$. Setting $\Omega_\flat:=\mathfrak{C}^h_{\ell_\flat,\rho}$, $\Omega_n:=\mathfrak{C}^h_{\ell_n,\rho}$, and $\Omega_*:=\mathfrak{C}^h_{\ell_*,\rho}$ as well as $Q^{(n)}_{\rm b}$ and $Q^*_{\rm b}$ to be the corresponding homeotropic boundary conditions given by \eqref{eq:radial-anchoring}, let $Q_n$ and $Q_*$ be minimizers of $\mathcal{E}_\lambda$ over  $\mathcal{A}^{\rm sym}_{Q_{\rm b}^{(n)}}(\Omega_n)$ and 
 $\mathcal{A}^{\rm sym}_{Q_{\rm b}^{*}}(\Omega_*)$ respectively. Assume that $Q_n\to Q_*$ strongly in $W^{1,2}(\Omega_{\flat})$ as $n\to\infty$. 
 \begin{enumerate}
 \item[(i)] If ${\rm sing}(Q_*)=\emptyset$, then there exists an integer $n_*$ such that ${\rm sing}(Q_n)=\emptyset$ for every $n\geq n_*$.
 \vskip5pt
 \item[(ii)]  If ${\rm sing}(Q_n)=\emptyset$ for every integer $n$, then ${\rm sing}(Q_*)=\emptyset$. 
\end{enumerate}
\end{proposition}

\begin{proof}
To simplify the notation, we write $\Omega_{\flat/2}:=\mathfrak{C}^h_{\ell_\flat/2,\rho}$. Since $\partial\Omega_{\flat/2}\setminus\{|x_3|=h\}\subset \Omega_\flat$, the restrictions $Q_n$ and $Q_*$ to $\partial\Omega_{\flat/2}$ belong at least to $C^2(\partial \Omega_{\flat/2})$ by Theorem \ref{bdrregthm} and Corollary \ref{bdrepsregcor} (applied at  {{balls}} centered on  $\partial\Omega_{\flat/2}\cap\{|x_3|=h\}$). To prove the proposition, we only have to show that $Q_n\to Q_*$ in $C^2(\partial \Omega_{\flat/2})$. Indeed, once this $C^2$-convergence is established, the conclusion follows from Corollary \ref{corolpersissmooth} in the domain $\Omega_{\flat/2}$.

First, we observe that  for $\ell\leq 2\ell_\flat$, we have 
$\mathfrak{C}^h_{\ell}\subset \Omega_n$, so that $\partial \mathfrak{C}^h_{\ell}\cap\partial\Omega_n=
\partial^= \mathfrak{C}^h_{\ell}$.  Setting $r_*:=\ell_\flat/2$, it implies that for every $x_*=(x_{*,1},x_{*,2},x_{*,3})\in \partial^= \mathfrak{C}^h_{\ell_\flat}$, the set $\Omega_n\cap B_{2r_*}(x_*)$ is a half ball  
and $Q^{(n)}_{\rm b}={\bf e}_0$ on $\partial\Omega_n\cap B_{2r_*}(x_*)$ so that Remark \ref{specgeomremreg} applies. 

By Theorem \ref{bdrregthm}, $Q_*$ is smooth in a neighborhood of $\partial \Omega_*$. Therefore, we can find  $r_1\in (0,r_*/4)$ (depending on $\lambda$) such that the conclusion of Remark \ref{specgeomremreg}  holds and 
$$\frac{1}{r_1}\int_{B_{r_1}(x_0)\cap \Omega_{{*}}}|\nabla Q_*|^2\,dx\leq \frac{\boldsymbol{\varepsilon}^\sharp_{\rm bd}}{2} \quad\text{for every $x_0\in  \partial^= \mathfrak{C}^h_{3\ell_\flat/4}$}\,,$$
where the universal constant $\boldsymbol{\varepsilon}^\sharp_{\rm bd}>0$ is given by Remark \ref{specgeomremreg}. Then we consider a finite covering of $\partial^= \mathfrak{C}^h_{3\ell_\flat/4}$  by open balls  $B_{\boldsymbol{\kappa}^\sharp r_1/2}(x_j)$, $j=1,\ldots, J$, with $x_j \in \partial^= \mathfrak{C}^h_{3\ell_\flat/4}$ and $\boldsymbol{\kappa}^\sharp\in(0,1)$ the further universal constant given by  Remark \ref{specgeomremreg}. Since  $Q_n\to Q_*$ strongly in $W^{1,2}(\Omega_{\flat})$, we have for $n$ large enough,
$$ \frac{1}{r_1}\int_{B_{r_1}(x_j)\cap \Omega_n}|\nabla Q_n|^2\,dx\leq \boldsymbol{\varepsilon}^\sharp_{\rm bd} \quad\text{for every $j=1,\ldots,J$}\,.$$
Applying Remark \ref{specgeomremreg}, we deduce that $Q_n$ is bounded in $C^{2,\alpha}(B_{\boldsymbol{\kappa}^\sharp r_1/2}(x_j)\cap \Omega_n)$ for every $\alpha\in(0,1)$ and each $j=1,\ldots,J$. Hence $Q_n$ is bounded in the  $C^{2,\alpha}$-topology in
$$L_\delta:=\overline{\mathfrak{C}^h_{3\ell_\flat/4}}\cap \big\{h-\delta\leq |x_3|\leq h\big\}$$
for some $\delta\in(0,\boldsymbol{\kappa}^\sharp r_1/2)$. 

By Theorem \ref{bdrregthm} again, $Q_*$ is smooth away from $\{x_3\text{-axis}\}$. Hence we can find $r_2\in(0,\delta/4)$ such that 
$$ \frac{1}{r_2}\int_{B_{r_2}(x_0)}|\nabla Q_*|^2\,dx\leq \frac{\boldsymbol{\varepsilon}_{\rm in}}{8}\quad\text{for every $x_0\in \big(\overline{\mathfrak{C}^h_{3\ell_\flat/4}}\setminus \mathfrak{C}^h_{\ell_\flat/4}\big)\cap\big\{|x_3|\leq h-\delta/2\big\}$}\,,$$
where $\boldsymbol{\varepsilon}_{\rm in}>0$ is the universal constant  given by Proposition \ref{intepsregprop}. Choosing $r_2$ small enough (depending on $\lambda$), and using the strong convergence in $W^{1,2}(\Omega_\flat)$ of $Q_n$ toward $Q_*$ combined with a covering argument (as above), we conclude from Proposition \ref{intepsregprop} that $Q_n$ is bounded in the  $C^{2,\alpha}$-topology in $ \big(\overline{\mathfrak{C}^h_{3\ell_\flat/4}}\setminus \mathfrak{C}^h_{\ell_\flat/4}\big)\cap\big\{|x_3|\leq h-\delta/2\big\}$ for every $\alpha\in(0,1)$.   

To summarize, $Q_n$ is thus bounded in the $C^{2,\alpha}$-topology in the set
$$ L_\delta\cup  \bigg(\big(\overline{\mathfrak{C}^h_{3\ell_\flat/4}}\setminus \mathfrak{C}^h_{\ell_\flat/4}\big)\cap\big\{|x_3|\leq h-\delta/2\big\}\bigg)=L_\delta\cup\big(\overline{\mathfrak{C}^h_{3\ell_\flat/4}}\setminus \mathfrak{C}^h_{\ell_\flat/4}\big)=:N_\delta \,.$$
From the strong $W^{1,2}(\Omega_\flat)$-convergence to $Q_*$, we conclude that $Q_n\to Q_*$ in $C^2(N_\delta)$. Observing  that 
$\partial \Omega_{\flat/2}\subset N_\delta$, the conclusion follows. 
\end{proof}

\begin{proposition}[\bf persistence of singularities]\label{perssingmovdom}
Let $Q_n$ and $Q_*$ be as in Proposition~\ref{perssmoothmovdom}. If ${\rm sing}(Q_*)=\{a_1^*,\ldots,a_K^*\}$, then there exists an integer $n_*$ such that for every $n\geq n_*$,  ${\rm sing}(Q_n)=\{a_1^n,\ldots,a_K^n\}$ for some distinct points $a_1^n,\ldots,a_K^n$ satisfying $|a^n_j-a^*_j|\to 0$ {{as $n\to\infty$}}.
\end{proposition}

\begin{proof}
As in the proof of Proposition \ref{perssmoothmovdom}, $Q_n\to Q_*$ in $C^2(\partial \Omega_{\flat/2})$, and the conclusion follows from Corollary \ref{corosinggotosingglobal} in the domain $\Omega_{\flat/2}$. 
\end{proof}

In combination with the previous propositions, we now provide the required compactness property of minimizers as $\Omega_n\to\Omega_*$.  

\begin{lemma}\label{complemmovdom}
Under the assumptions (and notations) of Proposition~\ref{perssmoothmovdom},  assume that $\ell_n \to \ell_*$ as $n \to \infty$.  
There exists a (not relabeled) subsequence and $Q_*$ minimizing $\mathcal{E}_\lambda$ over $\mathcal{A}^{\rm sym}_{Q_{\rm b}^{*}}(\Omega_{*})$ such that $Q_n\to Q_*$ strongly in $W^{1,2}(\Omega_{\flat})$ as $n\to\infty$. 
\end{lemma}

\begin{proof}
Notice that{{,}} by our choice of the parameters, we have 
$\{x\in\Omega_n: \ell_\flat\leq r< 2\ell_\flat\}=\mathfrak{C}^h_{2\ell_\flat}\setminus \mathfrak{C}^h_{\ell_\flat}$ where $r^2:=x_1^2+x_2^2$, and the mapping $\Phi_n:\overline\Omega_n\to\overline\Omega_*$ given in cylindrical coordinates by 
\begin{equation} \label{formbiliphomeocyl}
\Phi_n(r,x_3):=\begin{cases}
(r,x_3) & \text{if $r<\ell_\flat$}\,,\\
(\sigma_n r-\tau_n,x_3) & \text{if $\ell_\flat\leq r<2\ell_\flat$}\,,\\
(r+\tau_n,x_3) & \text{if $2\ell_\flat\leq r\leq \ell_n$}\,,
\end{cases}
\quad \text{with }\sigma_n:=1+\frac{\ell_*-\ell_n}{\ell_\flat}\text{ and }\tau_n:=\ell_*-\ell_n\,,
\end{equation}
is one-to-one and biLipschitz{{, $\bbS^1$-equivariant and such that $\Phi_n(\overline{\mathfrak{C}^h_{2 \ell_\flat}} \setminus \mathfrak{C}^h_{\ell_\flat}) = \overline{\mathfrak{C}^h_{2 \ell_\flat+\tau_n}} \setminus \mathfrak{C}^h_{\ell_\flat}$}}. 

For an arbitrary map $\widehat{Q} \in \mathcal{A}_{Q_{\rm b}^{(n)}}(\Omega_n)$, we define $\widehat{Q}_n := \widehat{Q} \circ \Phi_n^{-1}$ and we observe that $\widehat{Q}_n \in \mathcal{A}_{Q_{\rm b}^*}(\Omega_*)$. Combining the chain rule, a change of variables, and \eqref{formbiliphomeocyl} we obtain
\begin{equation}\label{eq:bilip-energy-est}
	\frac{1}{C_n}\mathcal{E}_\lambda ({{\widehat{Q}}},\Omega_n)\leq  \mathcal{E}_\lambda ({{\widehat{Q}_n}},\Omega_*)\leq C_n \mathcal{E}_\lambda ({{\widehat{Q}}},\Omega_n)\,,
\end{equation}
for a constant $C_n\to 1$ as $n\to\infty${{ depending only on $\Phi_n$. In addition, we notice that, if $\widehat{Q} \in \mathcal{A}^{\rm sym}_{Q_n}(\Omega_n)$, then $\widehat{Q}_n \in \mathcal{A}^{\rm sym}_{Q_{\rm b}^*}(\Omega_*)$ because the $\Phi_n$'s are equivariant. Therefore, t}}esting the minimality of $Q_n$ with the $0$-homogenous extension of $Q^{(n)}_{\rm b}$, we infer from direct computations that 
$\sup_n\mathcal{E}_\lambda(Q_n,\Omega_n)<\infty$, and thus{{, defining $\widetilde{Q}_n := Q_n \circ \Phi_n^{-1}$ and using \eqref{eq:bilip-energy-est} with $\widehat{Q}_n = \widetilde{Q}_n$, it follows}} $\sup_n\mathcal{E}_\lambda(\widetilde Q_n,\Omega_*)<\infty$. As a consequence, we can find a (not relabeled) subsequence such that $\widetilde Q_n\rightharpoonup Q_*$ weakly in $W^{1,2}(\Omega_*)$.  {{Since $\widetilde{Q}_n |_{\partial \Omega_*} = Q_{\rm b}^*$ independently of $n$ and since the symmetry and unit norm constraints are weakly closed,}} we have $Q_*\in \mathcal{A}^{\rm sym}_{Q_{\rm b}^{*}}(\Omega_*)$. In addition, by lower semi-continuity of the energy,
\begin{equation}\label{liminfcompmovdom}
\mathcal{E}_\lambda(Q_*,\Omega_*)\leq \liminf_{n\to\infty}   \mathcal{E}_\lambda (\widetilde Q_n,\Omega_*) \,.
\end{equation}
On the other hand, as above we have $Q_*\circ \Phi_n\in  \mathcal{A}^{\rm sym}_{Q_{\rm b}^{(n)}}(\Omega_n)$, and the minimality of $Q_n$ {{together with \eqref{eq:bilip-energy-est} applied twice (once with $\widehat{Q}_n = \widetilde{Q}_n$ and once with $Q_* \circ \Phi_n$ in place of $\widehat{Q}$)}} yields
$$\frac{1}{C_n}  \mathcal{E}_\lambda(\widetilde Q_n,\Omega_*) \leq  \mathcal{E}_\lambda(Q_n,\Omega_n)\leq   \mathcal{E}_\lambda(Q_*\circ\Phi_n,\Omega_n)\leq C_n \mathcal{E}_\lambda(Q_*,\Omega_*)\,.$$
Since $C_n\to 1$, taking the $\limsup_n$ above we deduce from \eqref{liminfcompmovdom} that $\lim_n \mathcal{E}_\lambda(\widetilde Q_n,\Omega_*)=\mathcal{E}_\lambda(Q_*,\Omega_*)$. By the compact embedding $W^{1,2}(\Omega_*)\hookrightarrow L^4(\Omega_*)$, we have $W(\widetilde Q_n)\to W(Q_*)$ strongly in $L^1(\Omega_*)$. Hence, $\int_{\Omega_*}|\nabla \widetilde Q_n|^2\,dx\to \int_{\Omega_*}|\nabla  Q_*|^2\,dx$ so that $\widetilde Q_n\to Q_*$ strongly in $W^{1,2}(\Omega_*)$. Since $\widetilde Q_n=Q_n$ in $\Omega_\flat\subset\Omega_*$, we conclude that  $Q_n\to Q_*$ strongly in $W^{1,2}(\Omega_\flat)$. 

It now remains to show the minimality of $Q_*$. To this purpose, let us fix an arbitrary competitor $Q\in \mathcal{A}^{\rm sym}_{Q_{\rm b}^{*}}(\Omega_*)$. Once again, we observe that 
$Q\circ \Phi_n\in  \mathcal{A}^{\rm sym}_{Q_{\rm b}^{(n)}}(\Omega_n)$, and by minimality of $Q_n$ {{along with \eqref{eq:bilip-energy-est}}}, 
$$\frac{1}{C_n}  \mathcal{E}_\lambda(\widetilde Q_n,\Omega_*) \leq  \mathcal{E}_\lambda(Q_n,\Omega_n)\leq  \mathcal{E}_\lambda(Q\circ\Phi_n,\Omega_n)\leq C_n \mathcal{E}_\lambda(Q,\Omega_*)\,.$$
Letting $n\to\infty$, we thus  obtain $\mathcal{E}_\lambda(Q_*,\Omega_*)\leq \mathcal{E}_\lambda(Q,\Omega_*)$, which completes the proof. 
\end{proof}

We are finally ready to prove our coexistence result for torus and split minimizers under homeotropic boundary data.

\begin{proof}[Proof of Theorem \ref{thmcoexmovdom}]
Throughout the proof we set $\ell_\flat:=\ell_0/3<h$. 
\vskip3pt

\noindent {\it Step 1.} Define 
$$\ell_1:=\sup\bigg\{\bar \ell \geq \ell_0 : \text{every minimizer of $\mathcal{E}_\lambda$ over $\mathcal{A}^{\rm sym}_{Q_{\rm b}^{(\ell)}}(\Omega_{\ell})$ is  split for every $\ell_0\leq \ell\leq \bar \ell$} \bigg\}\,, $$
and observe that $\ell_1<\infty$ by Theorem  \ref{thmlargecyl}.
We claim that $\ell_1>\ell_0$. Indeed, assume by contradiction that $\ell_1=\ell_0$. Then, there exists a strictly decreasing sequence $\{\ell_n\}$ such that $\ell_n\to \ell_0$, and for each integer $n$,   $\mathcal{E}_\lambda$ admits a minimizer $Q_n$ over $\mathcal{A}^{\rm sym}_{Q_{\rm b}^{(\ell_n)}}(\Omega_{\ell_n})$ such that ${\rm sing}(Q_n)=\emptyset$. By Lemma~\ref{complemmovdom}, there exists a (not relabeled) subsequence such that $Q_n\to Q_*$ strongly in $W^{1,2}(\Omega_{\ell_\flat})$ where $Q_*$ minimizes $\mathcal{E}_\lambda$ over $\mathcal{A}^{\rm sym}_{Q_{\rm b}^{(\ell_0)}}(\Omega_{\ell_0})$.  Applying Proposition \ref{perssmoothmovdom}, we infer that ${\rm sing}(Q_*)=\emptyset$, i.e., $Q_*$ is torus, contradicting our assumption on $\ell_0$. Hence $\ell_1>\ell_0$.  

We now claim that $\mathcal{E}_\lambda$ admits both a split and a torus minimizer over $\mathcal{A}^{\rm sym}_{Q_{\rm b}^{(\ell_1)}}(\Omega_{\ell_1})$. Indeed, assume first by contradiction that every minimizer is split. Arguing as above with $\ell_1$ in  place of $\ell_0$, it would lead to the existence of $\delta>0$ such that for $\ell_1\leq \ell<\ell_1+\delta$, every minimizer is split, contradicting the definition of $\ell_1$. Whence the existence of a torus minimizer. To prove the existence of a split minimizer, let us consider a strictly increasing sequence $\ell_0<\ell_n<\ell_1$ such that $\ell_n\to \ell_1$.   For each integer $n$, let $Q_n$ be a minimizer of $\mathcal{E}_\lambda$ over $\mathcal{A}^{\rm sym}_{Q_{\rm b}^{(\ell_n)}}(\Omega_{\ell_n})$, which must be split by definition of $\ell_1$.  Applying Lemma~\ref{complemmovdom}, we can find a (not relabeled) subsequence such that $Q_n\to Q_\sharp $ strongly in $W^{1,2}(\Omega_{\ell_\flat})$ where $Q_\sharp$ minimizes $\mathcal{E}_\lambda$ over $\mathcal{A}^{\rm sym}_{Q_{\rm b}^{(\ell_1)}}(\Omega_{\ell_1})$. Since ${\rm sing}(Q_n)\not=\emptyset$, we deduce from Proposition~\ref{perssingmovdom} that ${\rm sing}(Q_\sharp)\not=\emptyset$, i.e., $Q_\sharp$ is a split solution.  
\vskip5pt

\noindent{\it Step 2.} Define
$$\ell_2 :=\inf\bigg\{\bar \ell \geq \ell_0: \text{every minimizer of $\mathcal{E}_\lambda$ over $\mathcal{A}^{\rm sym}_{Q_{\rm b}^{(\ell)}}(\Omega_{\ell})$ is  torus for every $\ell\geq \bar \ell$}  \bigg\}\,,$$
and observe that it is indeed well defined and finite by Theorem~\ref{thmlargecyl} (as the set above is not empty). Clearly, $\ell_2\geq \ell_1$ {{by definition of $\ell_1$}}. Interchanging the roles of split and torus, we can argue exactly as in the previous step to infer that there exists a minimizer of $\mathcal{E}_\lambda$ over $\mathcal{A}^{\rm sym}_{Q_{\rm b}^{(\ell_2)}}(\Omega_{\ell_2})$ which is split (assume by contradiction it does not exist, then use Proposition \ref{perssmoothmovdom} and Lemma \ref{complemmovdom} along an increasing sequence $\ell_n\to\ell_2$ to deduce that for some $\delta>0$, every minimizer of $\mathcal{E}_\lambda$ over $\mathcal{A}^{\rm sym}_{Q_{\rm b}^{(\ell)}}(\Omega_{\ell})$ is torus for $\ell_2\geq \ell>\ell_2-\delta$, hence contradicting the definition of $\ell_2$). The existence of a torus minimizer  of $\mathcal{E}_\lambda$ over $\mathcal{A}^{\rm sym}_{Q_{\rm b}^{(\ell_2)}}(\Omega_{\ell_2})$ also follows as in Step 1. We consider a strictly decreasing sequence $\ell_n\to \ell_2$ and corresponding torus minimizers  of $\mathcal{E}_\lambda$ over $\mathcal{A}^{\rm sym}_{Q_{\rm b}^{(\ell_n)}}(\Omega_{\ell_n})$. By Lemma \ref{complemmovdom} and Proposition  \ref{perssmoothmovdom}, we can extract a subsequence strongly converging in $W^{1,2}(\Omega_{\ell_\flat})$ toward a minimizer over $\mathcal{A}^{\rm sym}_{Q_{\rm b}^{(\ell_2)}}(\Omega_{\ell_2})$ which must by be torus. 
\end{proof}


\subsection{Symmetry breaking in intermediate cylinders}

We complete this section exploiting Theorem~\ref{thmcoexmovdom} to show that a symmetry breaking occurs for intermediate cylinders of thickness $\ell$ close to the critical values $\ell_1$ and $\ell_2$. As in Corollary \ref{inst+symmbreaking}, it relies on the full regularity of global energy minimizers \cite[Theorem~1.1]{DMP1}  among nonsymmetric competitors, and on the continuity of the energy infimum with respect to the thickness of the cylinder stated in the following lemma. 

\begin{lemma}\label{contvalfcts}
Let $h>0$ and $\rho>0$ be fixed with $h>2\rho$. For a smoothed cylinder $\mathfrak{C}^h_{\ell,\rho}$, let $Q^{(\ell)}_{\rm b}$ be its homeotropic boundary data given by~\eqref{eq:radial-anchoring}. The functions
$$\ell\in(2\rho,+\infty)\mapsto {\rm Val}(\ell):=\inf\Big\{\mathcal{E}_\lambda(Q,\mathfrak{C}^h_{\ell,\rho}): Q\in \mathcal{A}_{Q^{(\ell)}_{\rm b}}(\mathfrak{C}^h_{\ell,\rho})\Big\} $$
and
$$\ell\in(2\rho,+\infty)\mapsto {\rm Val}^{\rm sym}(\ell):=\inf\Big\{\mathcal{E}_\lambda(Q,\mathfrak{C}^h_{\ell,\rho}): Q\in \mathcal{A}^{\rm sym}_{Q^{(\ell)}_{\rm b}}(\mathfrak{C}^h_{\ell,\rho})\Big\} $$
are continuous. 
\end{lemma}

\begin{proof}
Let $\ell_n\to\ell_*$ be an arbitrary converging sequence satifying with $\ell_n>2\rho$ and $\ell_*>2\rho$. Applying the Direct Method of Calculus of Variations, we can find for each $n$ a map $Q_n\in \mathcal{A}_{Q^{(\ell_n)}_{\rm b}}(\mathfrak{C}^h_{\ell_n,\rho})$ and  $Q_*\in \mathcal{A}_{Q^{(\ell_*)}_{\rm b}}(\mathfrak{C}^h_{\ell_*,\rho})$ such that 
$\mathcal{E}_\lambda(Q_n,\mathfrak{C}^h_{\ell_n,\rho})={\rm Val}(\ell_n)$ and $\mathcal{E}_\lambda(Q_*,\mathfrak{C}^h_{\ell_*,\rho})={\rm Val}(\ell_*)$ (see \cite{DMP1}). 
Now, we consider the sequence of equivariant biLipschitz homeomorphisms $\Phi_n:\overline{\mathfrak{C}^h_{\ell_n,\rho}}\to \overline{\mathfrak{C}^h_{\ell_*,\rho}}$ from the proof of Lemma~\ref{complemmovdom}, recalling that their biLipschitz constants go to 1 as $n \to \infty$ and that $Q_{\rm b}^{(\ell_n)}\circ \Phi^{-1}_n= Q_{\rm b}^{(\ell_*)}$ for all $n \in \N$.

We set $\widetilde Q_n:= Q_n\circ  \Phi^{-1}_n\in \mathcal{A}_{Q^{(\ell_*)}_{\rm b}}(\mathfrak{C}^h_{\ell_*,\rho})$ and $\widehat Q_n:= Q_*\circ  \Phi_n\in \mathcal{A}_{Q^{(\ell_n)}_{\rm b}}(\mathfrak{C}^h_{\ell_n,\rho})$. Then, {{\eqref{eq:bilip-energy-est} and energy minimality yield}}
\begin{multline*}
{\rm Val}(\ell_*)\leq \mathcal{E}_\lambda(\widetilde Q_n,\mathfrak{C}^h_{\ell_*,\rho}) \leq C_n \mathcal{E}_\lambda(Q_n,\mathfrak{C}^h_{\ell_n,\rho})= C_n {\rm {{Val}}}(\ell_n)\\
\leq 
C_n \mathcal{E}_\lambda(\widehat Q_n,\mathfrak{C}^h_{\ell_n,\rho}) \leq C^2_n \mathcal{E}_\lambda(Q_*,\mathfrak{C}^h_{\ell_*,\rho})= C_n^2{\rm Val}(\ell_*)
\,,
\end{multline*}
for a constant $C_n\to 1$ as $n\to\infty$. Hence, $\lim_n {\rm Val}(\ell_n)={\rm Val}(\ell_*)$ showing that $ {\rm Val}$ is continuous at $\ell_*$. 
The same argument applies  to $ {\rm Val}^{\rm sym}$ since the $\Phi_n$'s are equivariant. 
\end{proof}

\begin{corollary}\label{symbreakintermcyl}
Under the assumptions (and notations) of Theorem \ref{thmcoexmovdom}, there exists $\delta>0$ such that 
$${\rm Val}(\ell)<{\rm Val}^{\rm sym}(\ell) 
\qquad \forall \ell\in[\ell_0,\ell_1+\delta)\cup(\ell_2-\delta,\ell_2+\delta)\,.$$
In particular, for $\ell\in(\ell_0,\ell_1+\delta)\cup(\ell_2-\delta,\ell_2+\delta)$, any minimizer of $\mathcal{E}_\lambda$ over $\mathcal{A}_{Q^{(\ell)}_{\rm b}}(\mathfrak{C}^h_{\ell,\rho})$ is not $\mathbb{S}^1$-equivariant and there exists infinitely many minimizers. 
\end{corollary}

\begin{proof}
By \cite[Theorem 1.1]{DMP1}, any map realizing ${\rm Val}(\ell)$ is smooth. By definition of $\ell_0$, $\ell_1$, and $\ell_2$ (see Theorem  \ref{thmcoexmovdom}),  
for each $\ell\in[\ell_0,\ell_1]\cup\{\ell_2\}$ there exists a singular map realizing  ${\rm Val}^{\rm sym}(\ell)$. Hence ${\rm Val}(\ell)<{\rm Val}^{\rm sym}(\ell)$ for every $\ell\in[\ell_0,\ell_1]\cup\{\ell_2\}$. By the continuity of ${\rm Val}$ and ${\rm Val}^{\rm sym}$ provided by Lemma \ref{contvalfcts}, it follows that  ${\rm Val}<{\rm Val}^{\rm sym}$ in a neighborhood of $[\ell_0,\ell_1]\cup\{\ell_2\}$. Then the orbit under the $\mathbb{S}^1$-action of a minimizer provides infinitely many other minimizers. 
\end{proof}



\appendix

\section{Uniqueness of 2D-minimizers for $\lambda$ small}

The aim of this appendix is to complete the proof of Theorem~\ref{biaxialescape}, showing that the minimizer of the 2D-LdG energy $E_\lambda$ in the class $\mathcal{A}^{\rm sym}_{\overline{H}}(\bbD)$ is unique whenever $\lambda>0$ small enough. According to Proposition~\ref{ASminimization}, the claim holds for $\lambda=0$ where the harmonic map $u_{\rm S}$ given by \eqref{eq:uS} is the unique minimizer even without the symmetry constraint. In Theorem~\ref{uniqueminimizer}, we shall prove that the same unconstrained uniqueness holds for every $\lambda>0$ sufficiently small, and therefore in the restricted class $\mathcal{A}^{\rm sym}_{\overline{H}}(\bbD)$ as well. Our argument is inspired by the recent interesting paper \cite{INSZ}  addressing a similar question for minimizers of the 2D-LdG energy in a more elaborated asymptotic analysis without the norm constraint. 
\vskip5pt

We start with the following preliminary result (recall that the constant $\lambda_*>0$ is defined in Theorem \ref{2d-biaxial-escape}).

\begin{lemma}
\label{ulambda-convergence}
Let $u_{\rm S} \in \widetilde{\mathcal{A}}_{\rm S}$ given by \eqref{eq:uS}, $\lambda \in [0,\lambda_*)$, and $u_\lambda$  any minimizer of $\widetilde{E}_\lambda$ over the class $\widetilde{\mathcal{A}}_{g_{\overline{H}}}(\bbD)$. 
 The family $\{u_\lambda\}_{0<\lambda\leq\frac{\lambda_*}{2}} \subset C^2 (\overline{\bbD};\bb S^4)$ is bounded, and $u_\lambda \to u_{\rm S}$ in $C^1(\overline{\bbD})$ as $\lambda \to 0$.
\end{lemma}
\begin{proof}
By	Proposition~\ref{ASminimization} and the minimality of each $u_\lambda$, we have 
$$\widetilde{E}_0(u_{\rm S})\leq \widetilde{E}_0(u_\lambda)\leq\widetilde{E}_\lambda(u_\lambda)\leq \widetilde{E}_\lambda(u_{\rm S}) \mathop{\longrightarrow}\limits_{\lambda\to 0} \widetilde{E}_0(u_{\rm S})\,.$$ 
The family $\{ u_\lambda\}_{0<\lambda\leq\frac{\lambda_*}{2}}$ is thus bounded in $W^{1,2}(\bbD)$. Each  $W^{1,2}$-weak limit $u_*$ along an arbitrary sequence $\lambda_n\to 0$ 
 belongs to $\widetilde{\mathcal{A}}_{g_{\overline{H}}}(\bbD)$ and 
 satisfies $\widetilde{E}_0(u_{\rm S})\leq \widetilde{E}_0(u_*)\leq \widetilde{E}_0(u_{\rm S})$, again by  Proposition~\ref{ASminimization} 
 and the weak lower semicontinuity of $\widetilde{E}_0$. By uniqueness of $u_{\rm S}$ in  Proposition~\ref{ASminimization}, we deduce that $u_*=u_{\rm S}$. 
 In addition, $\widetilde{E}_{\lambda_n}(u_{\lambda_n})\to \widetilde{E}_0(u_{\rm S})$ also yields the strong $W^{1,2}$-convergence of $u_{\lambda_n}$ toward $u_{\rm S}$ as $\lambda_n \to 0$.

To conclude the proof, it is enough to establish a $C^{2}(\overline\bbD)$-bound on $u_\lambda$ since the embedding $C^{2}(\overline\bbD) \hookrightarrow C^1(\overline\bbD)$ is compact and  $C^1(\overline{\bbD}) \subset W^{1,2}(\bbD)$ is continuous. To obtain this $C^{2}$-bound, we rely on the regularity results from  \cite{DMP1} in three dimensions.

We consider a fixed cylinder $\mathcal{C}=\bb D \times (-1,1)$ and for each 2D-minimizer $u_\lambda$ we consider the boundary map $v_\lambda \in W^{1,2}(\partial \mathcal{C}; \bb S^4)$ as the trace of $u_\lambda$, the latter extended to the whole $\mathcal{C}$ independently of $x_3$. Clearly, $u_\lambda \in W^{1,2}_{v_\lambda}(\mathcal{C};\bb S^4)$ and it is easy to see that it is indeed the unique minimizer because of its $\widetilde{E}_\lambda$-minimality for each $x_3 \in (-1,1)$. Thus, we may apply the results in \cite{DMP1} to infer full interior regularity, i.e., that $u_\lambda \in C^\omega(\bb D;\bb S^4)$, and the full boundary regularity up to the lateral boundary $\partial \bb D \times (-1,1)$, so that $u_\lambda \in C^\omega (\overline{\bb D};\bb S^4)$. Finally, as $u_\lambda \to u_{\rm S}$ in $W^{1,2}(\bb D)$ and in turn in $W^{1,2}(\mathcal{C})$ we can apply interior and boundary $\veps$-regularity results on the whole family $\{u_\lambda\}$, as the scaled energy on balls centered at $\bar{x} \in \overline{\bb{D}}\times \{0\}$ can be made uniformly small for $\lambda>0$ small enough, to derive uniform $C^2$-bounds for the minimizers $\{u_\lambda \}$ for $\lambda>0$ small enough.  
\end{proof}

In order to discuss the uniqueness property of $u_\lambda$, we first recall that its energy minimality and smoothness properties yield the criticality condition
\begin{equation}
\label{eq:der-Elambda}
 \widetilde{E}^{\prime}_\lambda(\Phi;u_\lambda):=\left[ \frac{d}{dt} \widetilde{E}_\lambda \left( \frac{u_\lambda+t\Phi}{|u_\lambda+t\Phi|} \right) \right]_{t=0}  = \int_{\bb D} \left( -\Delta u_\lambda- \abs{\nabla u_\lambda}^2 u_\lambda + \lambda \nabla_{\rm tan} \widetilde{W}(u_\lambda) \right) \cdot \Phi \, dx =0 \, ,
\end{equation}
together with the positivity of the second variation 
\begin{equation}
\label{eq:secder-Elambda}
\widetilde{E}^{\prime\prime}_\lambda(\Phi;u_\lambda):=\left[ \frac{d^2}{dt^2} \widetilde{E}_\lambda \left( \frac{u_\lambda+t\Phi}{|u_\lambda+t\Phi|} \right) \right]_{t=0}  = \int_{\bb D} |\nabla \Phi_T|^2- |\nabla u_\lambda|^2 |\Phi_T|^2 +\lambda D^2 \widetilde{W} (u_\lambda) \Phi_T .\Phi_T \, dx \, ,
\end{equation}
defined for $\Phi \in C^\infty_c( \bb D;\R\oplus\C\oplus\C)$ with $\Phi_T :=\Phi-u_\lambda (u_\lambda\cdot\Phi)$ denoting the tangential component of $\Phi$ along $u_\lambda$.

 The following lemma guarantees injectivity of for the linearization of equations \eqref{eq:der-Elambda}, i.e., strict positivity of the quadratic forms \eqref{eq:secder-Elambda} for $\lambda\geq 0$ small enough.
 
 \begin{lemma}
 \label{pos-second-variation}
 	Let $u_{\rm S} \in \mathcal{A}_{\rm S}$ be as in \eqref{eq:uS}, $\lambda \in [0,\lambda_*)$ and $u_\lambda$ be any minimizer of $\widetilde{E}_\lambda$ over the class	$\widetilde{\mathcal{A}}_{g_{\overline{H}}}(\bbD)$. 
	Then there exists $m_0>0$ such that
 	\begin{equation}
\label{eq:strongpositivity}
 	\int_{\bb D} |\nabla \zeta|^2- |\nabla u_{\rm S}|^2 |\zeta|^2 \, dx \geq m_0 \int_{\bb D} \abs{\nabla \zeta}^2 \, dx \, , \end{equation}
 	for any $\zeta \in W^{1,2}_0( \bb D;\R\oplus\C\oplus\C)$. As a consequence, for $\lambda\geq0$ small enough we have
 	\begin{equation}
 	\label{eq:pos-2ndvar-ulambda}
 	\widetilde{E}^{\prime\prime}_\lambda(\Phi;u_\lambda) \geq \frac{m_0}2 \int_{\bb D} \abs{\nabla \Phi}^2 \, dx \, , 
 	\end{equation}
 for any $\Phi \in \mathcal{H}_\lambda:=\{ \Psi \in W^{1,2}_0(\bb D;\R\oplus\C\oplus\C) : \Psi \cdot u_\lambda \equiv 0 \, \}$.
 \end{lemma}

 \begin{proof}
 In view of \eqref{eq:uS} we have $u_{\rm S}(z)=(f_0(r),f_1(r)e^{i\theta},f_2(r)e^{i2\theta})$, where $z=re^{i\phi}\in \bb D$ and $f_0(r)=\frac{r^4-3}{r^4+3} \leq -\frac12$ in $\overline{\bb D}$. Since $u_{\rm S}$ is a harmonic map we have $-\Delta f_0= \abs{\nabla u_{\rm S}}^2 f_0$, where $\abs{\nabla u_{\rm S}}^2=\frac{96r^2}{(1+3r^4)^2}$ is bounded in $\overline{\bb D}$. 	
 
 Since every $\zeta \in C^\infty_0(\bb D;\R\oplus\C\oplus\C)$ can be written as $\zeta=f_0 \xi$ for some $\xi  \in C^\infty_0(\bb D;\R\oplus\C\oplus\C)$, a classical integration by parts argument using the equation for $f_0$ gives
 \[\int_{\bb D} \abs{\nabla \zeta}^2-\abs{\nabla u_{\rm S}}^2 |\zeta|^2 \, dx= \int_{\bb D} \abs{\nabla f_0}^2 |\xi|^2+ f_0^2 \abs{\nabla \xi}^2 -\abs{\nabla u_{\rm S}}^2 f_0^2 |\xi|^2 \, dx +\frac12 \int_{\bb D} \nabla f_0^2 \cdot \nabla |\xi|^2 \, dx=\]
 \[ \int_{\bb D} \abs{\nabla f_0}^2 |\xi|^2+ f_0^2 \abs{\nabla \xi}^2 -\abs{\nabla u_{\rm S}}^2 f_0^2 |\xi|^2 \, dx - \int_{\bb D} \Delta f_0 |\xi|^2 f_0+\abs{\nabla f_0}^2|\xi|^2 \, dx=\int_{\bb D} f_0^2 \abs{\nabla \xi}^2 \, dx \, .\]
 The previous identity extends by density to any $\zeta \in W^{1,2}_0(\bb D;\R\oplus\C\oplus\C)$ (correspondingly, to any $\xi\in W^{1,2}_0(\bb D;\R\oplus\C\oplus\C)$), so that in particular $F(\zeta):= \int_{\bb D} \abs{\nabla \zeta}^2-\abs{\nabla u_{\rm S}}^2 |\zeta|^2 \, dx >0$ whenever $\zeta \neq 0$.
 
 Now we set
 \[ \sigma_*:=\inf \{F(\zeta) \, , \, \, \| \zeta\|_{L^2}=1 \,  \, , \, \zeta \in  W^{1,2}_0(\bb D;\R\oplus\C\oplus\C)\} \, . \]
 By the direct method in the Calculus of Variations it is easy to check that $\sigma_*$ is attained and it is nonnegative. Moreover, the previous observation shows that actually $\sigma_*>0$ because of the norm constraint. Thus, for any $\zeta \in  W^{1,2}_0(\bb D;\R\oplus\C\oplus\C)$ we have
 \[  \int_{\bb D} \abs{\nabla \zeta}^2-\abs{\nabla u_{\rm S}}^2 |\zeta|^2 \, dx\geq \sigma_* \int_{\bb D} |\zeta|^2 \, dx\geq \frac{\sigma_*}{\| |\nabla u_{\rm S}|^2 \|_{L^\infty}} \int_{\bb D} \abs{\nabla u_{\rm S}}^2 |\zeta|^2 \, dx \, ,\]
 so that for $\eta:= \frac{\sigma_*}{\| |\nabla u_{\rm S}|^2 \|_{L^\infty}}>0$ and $m_0:= \frac{\eta}{1+\eta}$ inequality \eqref{eq:strongpositivity} follows.
 
 Finally, inequality \eqref{eq:pos-2ndvar-ulambda} follows easily from \eqref{eq:strongpositivity}. Indeed, for $\lambda>0$ small enough to be chosen later and $\Phi \in \mathcal{H}_\lambda$, so that $\Phi=\Phi_T$, \eqref{eq:secder-Elambda} can be rewritten and estimated as follows:
 \[ \widetilde{E}_\lambda^{\prime\prime}(\Phi;u_\lambda)=  \int_{\bb D} |\nabla \Phi|^2- \abs{\nabla u_{\rm S}}^2 |\Phi|^2 +(\abs{\nabla u_{\rm S}}^2 -|\nabla u_\lambda|^2) |\Phi|^2 +\lambda D^2 \widetilde{W} (u_\lambda) \Phi \cdot\Phi \, dx  \]
 \[\geq  m_0 \int_{\bb D} \abs{\nabla \Phi}^2 \, dx  - \left( \| \abs{\nabla u_{\rm S}}^2 -|\nabla u_\lambda|^2   \|_{L^\infty(\bb D)}+\lambda \| D^2 \widetilde{W} ( \, \cdot \,)\|_{L^\infty(\bbS^4)} \right) \int_{\bb D} |\Phi|^2 \, dx \, .\]
 Then, applying 2D-Poincar\'e inequality the lower bound \eqref{eq:pos-2ndvar-ulambda} follows from the $C^1$-convergence in Lemma \ref{ulambda-convergence} for $\lambda>0$ small enough. 
 \end{proof}

We are finally ready for the main result of the appendix.

\begin{theorem}
\label{uniqueminimizer}	
Let $\lambda \in [0,\lambda_*)$ and $u_\lambda$ a minimizer for the energy $\widetilde{E}_\lambda$ over the class $\widetilde{\mathcal{A}}_{g_{\overline{H}}}(\bbD)$. Then for $\lambda$ sufficiently small the minimizer is unique. As a consequence, $u_\lambda$ is $\bbS^1$-equivariant and it is the unique minimizer of $\widetilde{E}_\lambda$ over the class $\widetilde{\mathcal{A}}^{\rm sym}_{g_{\overline{H}}}(\bbD)$. 
\end{theorem}
\begin{proof}

We aim to show that for all pairs of minimizers $u_\lambda, v_\lambda \to u_{\rm S}$ we have $\| u_\lambda -v_\lambda \|_{L^2}\equiv 0$ for every $\lambda>0$ small enough. The main ingredient in the proof is equation \eqref{eq:pos-2ndvar-ulambda} in Lemma~\ref{pos-second-variation}, i.e., the uniform strict positivity of the second variation $\widetilde{E}_\lambda^{\prime\prime}(\cdot; u_\lambda)$ along tangent vector fields in $W^{1,2}_0$ for $\lambda > 0$ small enough. Once uniqueness holds, then $\bb S^1$-equivariance of $u_\lambda$ obviously follows and in turn its minimality in the subclass $\widetilde{\mathcal{A}}^{\rm sym}_{g_{\overline{H}}}(\bbD)$, because of the invariance property of $\widetilde{\mathcal{A}}_{g_{\overline{H}}}(\bbD)$ under the $\bb S^1$-action, namely, $(R * u) (z):=R u(R^\trans z)R^\trans$ for any $(u,R) \in \widetilde{\mathcal{A}}_{g_{\overline{H}}}(\bbD)\times \bb S^1$, combined with constancy of the energy functional $\widetilde{E}_\lambda$ along its orbits.

We start by decomposing $v_\lambda$ along $u_\lambda$ as $v_\lambda=u_\lambda+w_\lambda$, where in turn the difference $w_\lambda$ is pointwise decomposed into its tangential and its orthogonal part along $u_\lambda$, i.e.,
\begin{equation}
\label{vlambda-dec}
  v_\lambda=u_\lambda+w_\lambda^T+w_\lambda^\perp \,, \qquad w_\lambda^\perp:=[ (v_\lambda-u_\lambda)\cdot u_\lambda] u_\lambda \, , \quad w_\lambda^T:=w_\lambda-w_\lambda^\perp.
\end{equation}
From now on we assume that $|w_\lambda|<1/4$ uniformly on $\overline{\bb D}$, which is always the case for $\lambda$ small enough by Lemma~\ref{ulambda-convergence}.
Note that $|w_\lambda^\perp|^2+|w_\lambda^T|^2=|w_\lambda|^2=-2 u_\lambda \cdot w_\lambda^\perp=2|w_\lambda^\perp|$, whence $|w_\lambda^\perp|=1-\sqrt{1-|w_\lambda^T|^2}$ and in turn $w_\lambda^\perp=u_\lambda \left(-1+\sqrt{1-|w_\lambda^T|^2}\right)$.

Combining the uniform convergence and $C^1$-bounds from Lemma~\ref{ulambda-convergence} with \eqref{vlambda-dec},we see that the following pointwise inequalities hold uniformly on $\overline{\bb D}$ for every $\lambda > 0$ small enough (the symbol $\lesssim$ will mean inequality up to multiplicative constants independent of $\lambda$), namely, 
\begin{equation}
\label{pt-wlambda-bd}
|w_\lambda^\perp| \lesssim |w_\lambda^T|^2<1/4 \, , \quad |\nabla w_\lambda^\perp| \lesssim |w_\lambda^T| \left(|w_\lambda^T|+|\nabla w_\lambda^T|\right) \, , \quad |w_\lambda^T| \approx |w_\lambda| \, , \quad |\nabla w_\lambda^T| \lesssim |\nabla w_\lambda|+|w_\lambda| \, .
\end{equation}

In view of \eqref{eq:der-Elambda} and \eqref{eq:secder-Elambda}, it is convenient to extend $\widetilde{W}$ to a degree-zero homogeneous function of $\R\oplus\C\oplus\C \setminus \{0\}$ and to introduce the following operator,

\begin{equation}
\label{der-Elambda-op}
E'_\lambda[\Psi]=	 -\Delta \Psi- \abs{\nabla \Psi}^2 \Psi + \lambda D \widetilde{W}(\Psi)   \, , \qquad \Psi \in C^2(\overline{\bb D};\R\oplus\C\oplus\C\setminus\{0\}) \,, 
\end{equation}
together with its formal linearization at $u_\lambda$, namely, 
\begin{equation}
\label{secder-Elambda-op}
 E^{''}_\lambda[u_\lambda] \Phi:= -\Delta \Phi - |\nabla u_\lambda|^2 \Phi -2 \Big( \nabla u_\lambda \cdot \nabla \Phi \Big) u_\lambda+\lambda D^2 \widetilde{W} (u_\lambda) \Phi  \, ,
 \quad \Phi \in C^2(\overline{\bb D}; \R\oplus\C\oplus\C) \cap \mathcal{H}_\lambda \, ,
 \end{equation}
 so that by \eqref{eq:der-Elambda}, \eqref{eq:secder-Elambda} and pointwise orthogonality we have

\begin{equation}
\label{der-Elambda-rel}
 \widetilde{E}^{\prime}_\lambda(\Phi;u_\lambda)= \int_{\bb D} E'_\lambda[u_\lambda ] \cdot \Phi \, dx  \, , \qquad \widetilde{E}_\lambda^{\prime\prime}(\Phi ; u_\lambda)= \int_{\bb D} E''[u_\lambda] \Phi \cdot \Phi \, dx \, .
\end{equation}
for any $\Phi \in C^2( \overline{\bb D}; \R\oplus\C\oplus\C) \cap \mathcal{H}_\lambda$. Notice that $D\widetilde{W}(u_\lambda)\cdot\Phi = \nabla_{\rm tan}\widetilde{W}(u_\lambda)\cdot\Phi$ and $D^2 \widetilde{W}(u_\lambda)\Phi\cdot\Phi = D^2_{\rm tan} \widetilde{W}(u_\lambda)\Phi \cdot\Phi$ whenever $\Phi$ is tangent to $\bb S^4$ at $u_\lambda$ but, although these terms could be easily computed from \eqref{redpotential} and \eqref{signedbiaxiality}, exact formulas are irrelevant, as for our purposes the corresponding contributions will be negligible as $\lambda \to 0$.

Since both $u_\lambda$ and $v_\lambda$ are solutions, we have $E'_\lambda[v_\lambda]\equiv E'_\lambda[u_\lambda] \equiv 0$, hence for $w_\lambda=w_\lambda^\perp+w_\lambda^T=v_\lambda-u_\lambda$ as above and $\Phi=w_\lambda^T \in C^2( \overline{\bb D}; \R\oplus\C\oplus\C) \cap \mathcal{H}_\lambda $,  from \eqref{der-Elambda-op}-\eqref{secder-Elambda-op} we infer

\[0=\int_{\bbD} \left( E_\lambda'[v_\lambda] -E_\lambda'[u_\lambda]\right)\cdot w_\lambda^T dx=  \int_{\bbD} \left( E_\lambda'[v_\lambda]\mp E_\lambda'[u_\lambda+w_\lambda^T]\mp E_\lambda''[u_\lambda]w^T_\lambda -E_\lambda'[u_\lambda]\right) \cdot w_\lambda^T dx \, , \]
so that
\[\int_{\bb D} E_\lambda''[u_\lambda]w^T_\lambda \cdot w_\lambda^T dx=:I =II+III:= \]
\[   \int_{\bbD} \left( -E_\lambda'[u_\lambda+w_\lambda] + E_\lambda'[u_\lambda+w_\lambda^T] \right)\cdot  w_\lambda^T dx+\int_{\bb D} \left( E_\lambda'[u_\lambda] +    E_\lambda''[u_\lambda]w^T_\lambda -E_\lambda'[u_\lambda+w_\lambda^T] \right)\cdot  w_\lambda^T dx                                                                                                                                                                                         \, .\]

Combining \eqref{der-Elambda-rel}, \eqref{eq:pos-2ndvar-ulambda} and Poincar\'e inequality we obtain
$\| w_\lambda^T\|^2_{W^{1,2}} \lesssim I$ , 
with uniform constant for $\lambda$ small enough. Thus, in order to conclude it is enough to show that $II+III \lesssim o(1) \| w_\lambda^T\|^2_{W^{1,2}}$ as $\lambda \to 0$, to obtain $w_\lambda^T\equiv 0$ and in view of \eqref{pt-wlambda-bd} also $w_\lambda^\perp\equiv 0$, i.e., $v_\lambda=u_\lambda$ for $\lambda$ small enough.

Concerning $III$, an elementary calculation gives
\[ E_\lambda'[u_\lambda] +    E_\lambda''[u_\lambda]w^T_\lambda -E_\lambda'[u_\lambda+w_\lambda^T]= \lambda \Big( D  \widetilde{W}(u_\lambda) + D^2 \widetilde{W} (u_\lambda) w_\lambda^T -  D \widetilde{W}(u_\lambda+w_\lambda^T)  \Big) \]
\[	 - \abs{\nabla u_\lambda}^2 u_\lambda  - |\nabla u_\lambda|^2 w_\lambda^T -2 \Big( \nabla u_\lambda\cdot\nabla w_\lambda^T\Big)u_\lambda + \abs{ \nabla (u_\lambda+w_\lambda^T)}^2 (u_\lambda+w_\lambda^T) =
\]
\[ \lambda \Big( D \widetilde{W}(u_\lambda) + D^2 \widetilde{W} (u_\lambda) w_\lambda^T -  D  \widetilde{W}(u_\lambda+w_\lambda^T)  \Big) + 2 \Big( \nabla u_\lambda\cdot\nabla w_\lambda^T \Big)w_\lambda^T +\abs{\nabla w_\lambda^T}^2 (u_\lambda+w_\lambda^T) \, .\]
Using the uniform $C^1$-bounds for $u_\lambda$, $v_\lambda$ and $w_\lambda$ together with Taylor's theorem on $D  \widetilde{W}$ we easily obtain the pointwise bound
\[ \left| E_\lambda'[u_\lambda] +    E_\lambda''[u_\lambda]w^T_\lambda -E_\lambda'[u_\lambda+w_\lambda^T] \right| \lesssim (1+\lambda)\abs{w_\lambda^T}^2+\abs{\nabla w_\lambda^T}^2 \, ,\] 
so that for $\lambda<1$ small enough we obtain $III \lesssim \| w^T_\lambda\|_{L^\infty} \| w_\lambda^T\|^2_{W^{1,2}}=o(1) \| w_\lambda^T\|^2_{W^{1,2}} $ as $\lambda \to 0$ because of Lemma \ref{ulambda-convergence} and \eqref{pt-wlambda-bd}.

Concerning $II$, another simple calculation leads to
\[ -E_\lambda'[u_\lambda+w_\lambda] + E_\lambda'[u_\lambda+w_\lambda^T]=\Delta w_\lambda^\perp +\lambda \Big( D  W(u_\lambda+w_\lambda)-D  W(u_\lambda+w_\lambda^T) \Big) \]
\[+\Big(2 \nabla (u_\lambda+w_\lambda^T)\cdot\nabla w_\lambda^\perp + \abs{\nabla w_\lambda^\perp}^2 \Big)v_\lambda +
\abs{\nabla (u_\lambda+w_\lambda^T)}^2 w_\lambda^\perp\, . \]
In view of this last identity, the resulting terms in $II$ can be pointwise estimated as follows.  Since $w_\lambda^\perp=u_\lambda \left( -1+ \sqrt{1-|w_\lambda^T|^2} \right)$, by orthogonality we also have
\[ \Delta w_\lambda^\perp \cdot w_\lambda^T= \left( -1+ \sqrt{1-|w_\lambda^T|^2} \right) \Delta u_\lambda \cdot w_\lambda^T+2 \left(\nabla \left( -1+ \sqrt{1-|w_\lambda^T|^2} \right) \cdot \nabla u_\lambda\right) \cdot w_\lambda^T \, , \]
so that the $C^2$-bounds in Lemma \ref{ulambda-convergence} together with \eqref{pt-wlambda-bd} yield 
\begin{equation}
\label{IIbd1}
 \abs{\Delta w_\lambda^\perp\cdot w_\lambda^T} \lesssim |w_\lambda^T|^2 (|w_\lambda^T|+|\nabla w_\lambda^T|) \, ,
\end{equation}
uniformly on $\overline{\bb D}$ for $\lambda$ small enough.

Next, using the Lipschitz property of $D  \widetilde{W}$ on compact sets in combination with \eqref{pt-wlambda-bd} we have the pointwise bounds
\begin{equation}
\label{IIbd2}
 \lambda \Big( D  \widetilde{W}(u_\lambda+w_\lambda)-D  \widetilde{W}(u_\lambda+w_\lambda^T) \Big) \cdot w_\lambda^T \lesssim \lambda |w_\lambda^\perp| |w_\lambda^T| \lesssim |w_\lambda^T|^3 \, , \end{equation} 
uniformly on $\overline{\bb D}$ for $\lambda<1$ small enough. 
Finally, the uniform $C^1$-bound from Lemma \ref{ulambda-convergence} together with \eqref{pt-wlambda-bd} also yield the pointwise bound

\[
 \Big( \Big(2 \nabla (u_\lambda+w_\lambda^T)\cdot\nabla w_\lambda^\perp + \abs{\nabla w_\lambda^\perp}^2 \Big)v_\lambda +
\abs{\nabla (u_\lambda+w_\lambda^T)}^2 w_\lambda^\perp \Big) \cdot w_\lambda^T \]

\begin{equation}
\label{IIbd3}
\lesssim |w_\lambda^T| \left(|\nabla w_\lambda^\perp| +|w_\lambda^\perp|\right) \lesssim |w_\lambda^T|^2 (|w_\lambda^T|+|\nabla w_\lambda^T|) \, , 
\end{equation} 
uniformly on $\overline{\bb D}$ for $\lambda$ small enough.

Collecting together \eqref{IIbd1}-\eqref{IIbd3} we finally obtain the pointwise estimate
\[ \abs{\left( -E_\lambda'[u_\lambda+w_\lambda] + E_\lambda'[u_\lambda+w_\lambda^T] \right) \cdot w_\lambda^T} \lesssim |w_\lambda^T|^2 \left( |w_\lambda^T|+|\nabla w_\lambda^T|\right) \lesssim  |w_\lambda^T| \left( |w^T_\lambda|^2+|\nabla w^T_\lambda|^2\right) \, ,\]
so that integrating and arguing as above we easily obtain $II \lesssim \| w^T_\lambda\|_{L^\infty} \| w_\lambda^T\|^2_{W^{1,2}} =  o(1)  \|  w_\lambda^T\|^2_{W^{1,2}} $ as $\lambda \to 0$, which completes the proof. 
\end{proof}





\begin{thebibliography}{100}

\bibitem{ABL} \textsc{S. Alama, L. Bronsard, X. Lamy} : \textit{Spherical particle in a nematic liquid crystal under an external field: the Saturn ring regime.} J. Nonlinear Sci. \textbf{28} (2018), 1443-1465.

\bibitem{ABGL} \textsc{S. Alama, L. Bronsard, D. Golovaty, X. Lamy} : \textit{Saturn ring defect around a spherical particle immersed in nematic liquid crystal.} arXiv:2004.04973  [math.AP] (2020).


\bibitem{AlLi} \textsc{F.J. Almgren, E.H. Lieb} : Singularities of energy minimizing maps from the ball to the sphere: examples, counterexamples, and bounds, {\it Ann. of Math.}  {\bf 128} (1988),  483--530.

\bibitem{ACS} \textsc{F. Alouges, A. Chambolle, D. Stantejsky} : \textit{The saturn ring effect in nematic liquid crystals with external field: effective energy and hysteresis.} Arch. Ration. Mech. Anal. \textbf{241} (2021), 1403-1457. 

\bibitem{AFP} {\sc L. Ambrosio, N. Fusco, D. Pallara} : {\it Functions of Bounded Variation and Free Discontinuity Problems}, Oxford University Press, New York (2000).


\bibitem{AZB} \textsc{M. Asad-uz-Zaman, D. Blume} : \textit{Aligned dipolar Bose-Einstein condensate in a double-well potential: From cigar shaped to pancake shaped.} Phys. Rev. A \textbf{80}, 053622, (2009).

\bibitem{BaZa} \textsc{J. Ball, A. Zarnescu} : \textit{Orientability and energy minimization in liquid crystals model}. Arch. Ration. Mech. Anal. \textbf{202} (2011), 493-535. 



\bibitem{BrCo} \textsc{H. Brezis, J.M. Coron} : \textit{Large solutions for harmonic maps in two dimensions}. Comm. Math. Phys. \textbf{92} (1983), 203-215.


\bibitem{Can1} \textsc{G. Canevari} : \textit{Biaxiality in the asymptotic analysis of a 2D Landau-de Gennes model for liquid crystals.} Esaim: COCV \textbf{21} (2005) 101-137.




\bibitem{DLR} \textsc{G. De Luca, A.D. Rey} : \textit{Point and ring defects in nematics under capillary confinement}  J. Chem. Phys. \textbf{127} (2007), 104902.



\bibitem{DMP1} \textsc{F. Dipasquale, V. Millot, A. Pisante} : \textit{Torus-like  solutions for the Landau-de Gennes model. \\ Part I: The Lyuksyutov regime.} Arch. Rational Mech. Anal. \textbf{239} (2021), 599-678.
Digital 

\bibitem{DMP2} \textsc{F. Dipasquale, V. Millot, A. Pisante} : \textit{Torus-like  solutions for the Landau-de Gennes model. \\ Part II: Topology of $\mathbb{S}^1$-equivariant minimizers.} https://arxiv.org/pdf/
2008.13676.pdf

\bibitem{Evans} \textsc{L.C. Evans} : \textit{Weak convergence methods in nonlinear partial differential equations.} Regional Conference Series in mathematics, 0160-7642, v. 74, American Mathematical Society, 1990.

\bibitem{GaMk} \textsc{E. C. Gartland Jr., S. Mkaddem} : \textit{Fine structure of defects in radial nematic droplets.}  Phys. Rev. E \textbf{62} (2000),  6694-6705.


\bibitem{GilbTrud} \textsc{D. Gilbarg, N.S. Trudinger} : \textit{Elliptic Partial Differential Equations of Second Order.} Classics in mathematics, Springer-Verlang, Berlin (2001). 

\bibitem{HKL} \textsc{R. Hardt, D. Kinderlehrer, F.~H. Lin} : \textit{The variety of configurations of static liquid crystals.} Variational methods (Paris, 1988), 115--131, 
Progr. Nonlinear Differential Equations Appl. {\bf 4}, Birkh\"auser, Boston (1990). 

\bibitem{HL} \textsc{R. Hardt, F.~H. Lin} : \textit{A remark on $H^1$-mappings.}  Manuscripta Math. \textbf{56} (1986), 1-10.

\bibitem{HL2} \textsc{R. Hardt, F.~H. Lin} : \textit{Stability of singularities of minimizing harmonic maps.}  J. Differential Geometry \textbf{29} (1989), 1113--123.


\bibitem{He} \textsc{F. H\'elein} : \textit{Regularity of a weakly harmonic map from a surface into a manifold with symmetries.} Manuscripta Math. \textbf{70} (1991), 203-218. 

\bibitem{He2} \textsc{F. H\'elein} : \textit{Constant Mean Curvature Surfaces, Harmonic Maps and Integrable Systems.} Lectures in Mathematics: ETH Z\"urich, Birkh\"auser Basel, 2001. 

\bibitem{HQZ} \textsc{Y. Hu, T. Qu, P. Zhang} : \textit{On the Disclination Lines of Nematic Liquid Crystals.} Commun. Comp. Phys. \textbf{19}(2) (2016), 354-379.

\bibitem{INSZ} \textsc{R. Ignat, L. Nguyen, V. Slastikov,   A. Zarnescu} : \textit{Symmetry and multiplicity of solutions in a two-dimensional Landau-de Gennes model for liquid crystals.} Archive for Rational Mechanics and Analysis  \textbf{237} (2020), 1421-1473.



\bibitem{INSZ2} \textsc{R. Ignat, L. Nguyen, V. Slastikov, A. Zarnescu} : \textit{Instability of point defects in a two-dimensional nematic liquid crystal model.} Ann. Inst. H. Poincar\'{e} Anal. Nonlin\'{e}are \textbf{33} (2016), 1131-1152. 

\bibitem{INSZ3} \textsc{R. Ignat, L. Nguyen, V. Slastikov, A. Zarnescu} : \textit{Stability of point defects of degree $\pm \frac12$ in a two-dimensional nematic liquid crystal model.} Calc. Var. Partial Differential Equations \textbf{55} (2016), 55-119.

\bibitem{JK} \textsc{W. J\"ager, H. Kaul} : \textit{Uniqueness and stability of harmonic maps and their Jacobi fields.}
Manuscripta Math. \textbf{28} (1979), 269-291.

\bibitem{Jo} \textsc{J. Jost} : \textit{The Dirichlet problem for harmonic maps from a surface with boundary onto a 2-sphere with nonconstant boundary values.}
J. Differential Geom. \textbf{19} (1984), 393-401.


\bibitem{KV} \textsc{S. Kralj, E. G. Virga} : \textit{Universal fine structure of nematic hedgehogs.} J. Phys. A \textbf{34} (2001), 829-838.

\bibitem{KVZ} \textsc{S. Kralj, E. G. Virga, S. \v{Z}umer} : \textit{Biaxial torus around nematic point defects.} Phys. Rev. E \textbf{60} (1999), 1858-1866.

\bibitem{Le} \textsc{L. Lemaire} : \textit{Applications harmoniques de surfaces Riemanniennes.} J. Differential Geometry \textbf{13} (1978), 51--78.

\bibitem{LiWa} \textsc{F.H. Lin, C.Y. Wang}: \textit{Stable Stationary Harmonic Maps to Spheres.}  Acta Math. Sin. (Engl. Ser.) \textbf{22} (2006), 319--330.

\bibitem{LiWa2} \textsc{F.H. Lin, C.Y. Wang}: \textit{The analysis of harmonic maps and their heat flows.} World Scientific, 2008.

\bibitem{LoTr} \textsc{L. Longa, H.-R. Trebin} : \textit{Structure of elastic free energy for chiral nematic liquid crystals}, Phys. Rev. A \textbf{39}(4) (1989), 2160--2168.


\bibitem{Morrey} \textsc{C. B. Morrey, Jr.} : \textit{Multiple integrals in the calculus of variations.}  Springer Science \& Business Media, 1966.

\bibitem{MoNe} \textsc{N.J. Mottram, C.J.P. Newton} : \textit{Introduction to Q-tensor theory.} https://arxiv.org/pdf/1409.3542.pdf

\bibitem{MuNi} \text{D. Mucci, L. Nicolodi} : \textit{On the Landau–de Gennes elastic energy of a Q-tensor model for soft biaxial nematics}, J. Nonlinear Sci., 27 (2017), no. 3, pp. 16874–1724. 

\bibitem{Pi} \textsc{A.~Pisante} : \textit{Torus-like  solutions for the Landau-de Gennes model.} Annales de la Facult\'e des sciences de Toulouse: Math\'ematiques \textbf{30} (2021), 301-326.


\bibitem{Quing} \textsc{J. Quing} : \textit{Boundary regularity for weakly harmonic maps from surfaces.} J. Functional Analysis \textbf{114} (1993), 458--466.





\bibitem{Scheven} \textsc{C. Scheven} : \textit{Variational harmonic maps with general boundary conditions: Boundary regularity.} Calc. Var.
Partial Differential Equations {\bf 4} (2006), 409--429.

\bibitem{SU} \textsc{R. Schoen and K. Uhlenbeck.} : \textit{Regularity of minimizing harmonic maps into the sphere.} Invent. Math. \textbf{78} (1984), 89-100.




\bibitem{TaYu} \textsc{H.M.Tai, Y. Yu} : \textit{Pattern Formation in Landau-de Gennes Theory.} https://arxiv.org/pdf/2107.01440.pdf

\bibitem{Vi} \textsc{E.G. Virga.} : \textit{Variational theories for liquid crystals.} Vol. \textbf{8}. CRC Press, 1995.


\bibitem{Yu} \textsc{Y. Yu} : \textit{Disclinations in limiting Landau-de Gennes theory.} Arch. Ration. Mech. Anal. \textbf{237} (2020), 147-200.


%
%
%
%
%
%
%
%
%
%
%
%
%
%
%
%
%
%
%
%
%
%
%
%
%
%
%
%
%
%
%
%
\end{thebibliography}
\end{document}